\documentclass[12pt]{amsart}
\usepackage{amsfonts,amsthm,amsmath,amssymb,amscd, tikz}

\usepackage{upref}
\usepackage[backref,bookmarks=false]{hyperref}
\usepackage{mathrsfs}
\usepackage{url}
\usepackage{graphicx}

\numberwithin{equation}{section}

\newif\ifdraft\drafttrue

\drafttrue

\newcommand{\lab}{\label}

\newcommand{\ben}{\begin{enumerate}}
\newcommand{\een}{\end{enumerate}}

\newcommand{\bea}{\begin{eqnarray}}
\newcommand{\ba}{\begin{array}}
\newcommand{\bean}{\begin{eqnarray*}}
\newcommand{\ea}{\end{array}}
\newcommand{\eea}{\end{eqnarray}}
\newcommand{\eean}{\end{eqnarray*}}
\newcommand{\beq}{\begin{equation}}
\newcommand{\eeq}{\end{equation}}
\newcommand{\bthm}{\begin{thm}}
\newcommand{\ethm}{\end{thm}}
\newcommand{\blem}{\begin{lem}}
\newcommand{\elem}{\end{lem}}
\newcommand{\bprop}{\begin{prop}}
\newcommand{\eprop}{\end{prop}}
\newcommand{\bcor}{\begin{cor}}
\newcommand{\ecor}{\end{cor}}
\newcommand{\bdfn}{\begin{dfn}}
\newcommand{\edfn}{\end{dfn}}
\newcommand{\brem}{\begin{rem}}
\newcommand{\erem}{\end{rem}}
\newcommand{\bpf}{\begin{proof}}
\newcommand{\epf}{\end{proof}}
\newcommand{\bfact}{\begin{fact}}
\newcommand{\efact}{\end{fact}}
\newcommand{\bobs}{\begin{observation}}
\newcommand{\eobs}{\end{observation}}

\newcommand{\nl}{\newline}

\alph{enumii} \roman{enumiii}

\newtheorem{thm}{Theorem}[section]
\newtheorem{prop}[thm]{Proposition}
\newtheorem{lem}[thm]{Lemma}

\newtheorem{cor}[thm]{Corollary}

\theoremstyle{definition}
\newtheorem{dfn}[thm]{Definition}
\newtheorem{rem}[thm]{Remark}
\newtheorem{fact}[thm]{Fact}
\newtheorem{observation}[thm]{Observation}

\newtheorem{example}[thm]{Example}

\def\cA{\mathcal A}

\def\cS{\mathcal S}             \def\cE{\mathcal E}       \def\cD{\mathcal D}

\def\N{{\mathbb N}}            \def\Z{{\mathbb Z}}      \def\R{{\mathbb R}}            \def\cR{{\mathcal R}}          \def\bH{{\mathbb H}}
\def\C{{\mathbb C}}                  \def\oc{\widehat\C}

\def\1{1\!\!\text{{\rm 1}}}
\def\and{\text{ and }}        
         
          \def\Con{\text{Con}} \def\Per{\text{{\rm Per}}}
        \def\diam{\text{\rm {diam}}}
  
\def\Crit{\text{Crit}}

\def\h{{\rm h}}             \def\res{\text{{\rm res}}}
           
\def\H{\text{{\rm H}}}     \def\HD{\text{{\rm HD}}}   
            \def\PC{\text{{\rm PC}}}
\def\re{\text{{\rm Re}}}    \def\im{\text{{\rm Im}}}  
\def\Int{\text{{\rm Int}}}    
 \def\Conv{{\rm Conv}}
  
\def\Ker{\text{{\rm Ker}}}  \def\Leb{{\rm Leb}}
         \def\P{\text{{\rm P}}}     \def\Id{\text{{\rm Id}}}

\def\a{\alpha}                \def\b{\beta}             \def\d{\delta}
\def\De{\Delta}               \def\e{\varepsilon}          \def\f{\phi}
\def\g{\gamma}                \def\Ga{\Gamma}           
\def\La{\Lambda}              \def\om{\omega}           \def\Om{\Omega}
\def\Sg{\Sigma}               \def\sg{\sigma}
               \def\th{\theta}           
\def\ka{\kappa}               \def\vp{\varphi}          \def\phi{\varphi}

\def\bi{\bigcap}              \def\bu{\bigcup}
\def\({\big(}                \def\){\big)}
\def\lt{\left}                \def\rt{\right}

\def\ld{\ldots}               \def\bd{\partial}         \def\^{\widetilde}

      \def\du{\bigoplus}

\def\es{\emptyset}            \def\sms{\setminus}
\def\sbt{\subseteq}             \def\spt{\supseteq}
             \def\lek{\preceq}
\def\eqv{\Leftrightarrow}     
      
\def\comp{\asymp}
\def\upto{\nearrow}           \def\downto{\searrow}
\def\sp{\medskip}             \def\fr{\noindent}        \def\nl{\newline}
\def\ov{\overline}            
\def\ess{{\rm ess}}           
\def\om{\omega}

\def\re{\text{{\rm Re}}}

\def\supp{\text{{\rm supp}}}

\def\D{{\mathbb D}}

\newcommand{\lam}{\lambda}
\newcommand{\ep}{\varepsilon}

\newcommand{\pf}{{\mathcal{L}}}

\begin{document}

\date{\today}
 \title[Asymptotic Counting in Conformal Dynamical Systems]{Asymptotic Counting \\ in \\ Conformal Dynamical Systems}

\author{Mark Pollicott}
\address{University of Warwick, Institute of Mathematics, UK} 
\email{masdbl@warwick.ac.uk}

\author{Mariusz Urba\'nski}
\address{University of North Texas, Department of Mathematics, 1155
  Union Circle \#311430, Denton, TX 76203-5017, USA} 
\email{urbanski@unt.edu  \newline \hspace*{0.3cm} Web:\!\!\!
www.math.unt.edu/$\sim$urbanski}

%
\thanks{The research of both authors supported in part by the NSF
Grant DMS 0700831. The first author would like to thank Richard Sharp for useful discussions. Both authors thank Tushar Das for careful reading the first draft of our manuscript. His comments and suggestions improved the final exposition. The authors also wish to thank Hee Oh, whose valuable comments helped them to make the text more accurate  and to supplement the historical background  and references on circle packings.} 
\keywords{}
\subjclass{Primary:}

\begin{abstract}
In this monograph we consider the  general setting of conformal graph directed Markov systems modeled by countable state symbolic subshifts of finite type. We deal with two classes of such systems: attracting and parabolic. The latter being treated by means of the former.

We prove fairly complete asymptotic counting results for multipliers and diameters associated with preimages or periodic orbits
ordered hy a natural geometric weighting. We also prove the corresponding Central Limit Theorems describing the further features of the distribution of their weights.  

These results  have direct applications to a variety of examples, including the case of Apollonian Circle Packings, Apollonian Triangle, expanding and parabolic rational functions, Farey maps, continued fractions, Mannenville-Pomeau maps, Schottky groups, Fuchsian groups, and many more. A fairly complete collection of asymptotic counting results for them is presented.

Our new approach is founded on spectral properties of complexified Ruelle--Perron--Frobenius operators and Tauberian theorems as used in classical problems of prime number theory. 
\end{abstract}


\maketitle

\tableofcontents

\section{{\bf{\large Introduction}}}

\

\subsection{Short General Introduction} We begin with a simple problem formulated for iterated function systems (schemes). Let 
$$
\phi_e:X \to X, \  \  e\in E,
$$
be a countable, either finite or infinite, family of $C^{1+\alpha}$ contracting maps. We can associate to a point $\xi\in X$ the images 
$$
\phi_\om(\xi):= \phi_{\om_1} \circ \cdots \circ \phi_{\om_n}(\xi)
$$
where $\om_i\in E$, and then we associate two natural weights 
$$
\lam_\xi(\om):= -\log|(\phi_\om)'(\xi)|
$$
and
$$
\De_\xi(\om):=-\log\diam(\phi_\om(X)).
$$
Since there is no obvious way to order and count these images  in terms of their combinatorial weight (the length of $\om = (\omega_1, \cdots, \omega_n)$) we use instead the two weights introduced above: $\lam_\xi(\om)$ and $\De_\xi(\om)$. 

Under mild natural hypotheses we show that there exist two constants $C_1, C_2 > 0$ (we provide explicit dynamical expressions for them) and $\delta \in(0,+\infty)$ such that 
$$
\lim_{T\to+\infty}\frac{\#\{\om \hbox{ : } \lam_\xi(\om) \leq T\}}{e^{\d T}}= C_1 
$$
and 
$$
\lim_{T\to+\infty}\frac{\#\{\om \hbox{ : } \De_\xi(\om)\leq T\}}{e^{\d T}}= C_2.
$$
These are the most transparent and simplest highlights of our results; but we prove more. For example, we provide the corresponding asymptotic results when in addition one requires that the points $\phi_\om(\xi)$ are to fall into a prescribed Borel subset $B$ of $X$. We also count multipliers and diameters if the points $\phi_\om(\xi)$ are replaced by periodic points of the system, i.e. by fixed points $x_\om$ of the maps $\phi_\om$. We denote
$$
\lam_p(\om) = -\log|(\phi_\om)'(x_\om)|.
$$

A fuller  description of our results is provided below in further subsections of this introduction and in complete detail in appropriate technical sections of the manuscript.

There are natural and instructive parallels of our work and the classical 
approach to the  prime number theorem, as well as 
with known results on the Patterson-Sullivan orbit counting technology and the asymptotics of Apollonian circles. There are also applications to both expanding and parabolic rational functions, complex continued fractions, Farey maps, Mannenville-Pomeau maps, Schottky groups, Fuchsian groups, including Hecke groups, and more examples. We apply our general results to all of them, thus giving a unified approach which yields both new results 
and  a new  approach to established results.   

All of these are based on our results for conformal graph directed Markov systems over a countable alphabet. 
Our counting results (on the symbolic level) are close in spirit to those of Steve Lalley from \cite{lalley}. These would directly apply to our counting on the symbolic level if the graph directed Markov systems we considered had finite alphabets. However, we need to deal with those systems with a countable alphabet and we obtain our counting results via the study of spectral properties of complexified Ruelle--Perron--Frobenius operators, as used by William Parry and the first--named author, rather than the renewal theory approach of Lalley. It is worth mentioning that our results on the symbolic level could have been  formulated and proved with no real additional  difficulties in terms of ergodic sums of 
summable H\"older continuous potentials rather than merely the functions $\lam_\xi(\om)$ from the next subsection.

We would also like to add that our work was partly inspired by counting results of Kontorovich and Oh for Apollonian packings from \cite{KO} (see also \cite{OS1}--\cite{OS3}), which in our monograph are recovered and ultimately follow from our more general results for conformal graph directed Markov systems. Nevertheless the level of generality our approach is still entirely different than that of Kontorovich and Oh. We have recently received an interesting preprint \cite{Heersink} of Byron Heersink where he studies the counting problems for the Farey map, Gauss map, and closed geodesics on the modular surface. We would also like to note that a part of the classical work of the first named author and William Parry (including \cite{P_RPF_2}, \cite{Pollicott}, \cite{PP}, \cite{PP2}) the method of the complex Perron--Frobenius operator to approach various counting problems in geometry and dynamics has been used by several authors including \cite{Morris}, 
\cite{Naud}, \cite{PS}, \cite{AHS}.

We now discuss our results below in more detail.

\subsection{Asymptotic Counting Results}
In Sections~\ref{Attracting_GDMS_Prel} and \ref{section:parabolic} we recall from \cite{MU_GDMS} the respective concepts of attracting and parabolic countable alphabet conformal graph directed Markov systems. 
This symbolic viewpoint is convenient for keeping track of  the quantities  we  want to  counting.
 Let $A$ be the associated transition matrix
 and $\pi_{\cS}(\rho) \in X$ is a reference point coded by an infinite sequence $\rho$.
Fix any Borel set $B \subset X$ 
then for  $T>0$ we define:
\begin{align*}
N_\rho(B, T)&:=\# 
\left\{\om\in E_\rho^*  \hbox{ : }\phi_{\om}(\pi_\cS(\rho)) \in B 
 \hbox{ and }  
 \lambda_\rho(\om) \leq T \right\} 
\\ \hbox{ and }  \\
N_p(B,T)&:=\#
\left\{\om\in E_p^*  \hbox{ : } x_\om \in B  \hbox{ and }   \lambda_p(\om) \leq T \right\},
\end{align*}
where
$$
E_\rho^*:=\{\om\in E_A^*:\om\rho\in E_A^*\},
$$
and 
$$
E_p^*=\{\om\in E_A^*:A_{\om_{|\om|}\om_1}=1\},
$$
are finite words of symbols,
i.e. we count the number of words $\om\in E_i^*$ for which the 
weight $\lambda_i(\om)$ doesn't exceed $T$ and, additionally, the image $\phi_\om(\pi_\cS(\rho))$ is in $B$ if $i=\rho$, or the fixed point $x_\om$ 
of $\phi_\om$, is in $B$ if $i=p$. The following result comprises both Theorem~\ref{dyn} for attracting conformal GDMSs and Theorem~\ref{t2pc6_B} for parabolic systems.
We refer the reader to the appropriate sections for the detailed  definitions of any unfamiliar hypotheses (or to the next subsection for concrete  examples where these are known to  hold).

\sp
\begin{thm}[Asymptotic Equidistribution Formula for Multipliers] \label{t1_2017_04_17}
Suppose that $\cS$ is either a strongly regular finitely irreducible D-generic attracting conformal GDMS or finite alphabet parabolic conformal GDMS.  

Fix $\rho\in E_A^\infty$. If $B \subset X$ is a Borel set such that $\^m_\d(\bd B)=0$ (equivalently $\^\mu_\d(\bd B)=0$) then,
$$
\lim_{T \to +\infty} \frac{N_\rho(B,T)}{e^{\d T}} 
= \frac{\psi_\d(\rho)}{\d\chi_{\mu_\d}}\^m_\d(B)
$$
and 
$$
\lim_{T \to +\infty} \frac{N_p(B,T)}{e^{\d T}} 
= \frac{1}{\d\chi_{\mu_\d}}\^\mu_\d(B).
$$
Here we use the following notation.
\begin{itemize}
\item $\d=\HD(J_\cS)$ is the Hausdorff dimension of the limit set (attractor) of the GDMS $\cS$. 

\item $\^m_\d$ is the $\d$-conformal measure for $\cS$.

\item $\^\mu_\d$ is its $\cS$--invariant version.

\item $\psi_\d$ is essentially the Radon--Nikodym of $\^\mu_\d$ with respect $\^m_\d$ but on the symbolic level. 

\item The quantity $\chi_{\mu_\d}$ is the corresponding Lyapunov exponent.
\end{itemize}
\end{thm}

\sp\fr Our proof of this theorem for attracting systems is based on following five steps:  
\begin{enumerate}

\sp\item Describing the spectrum of an associated complexified Ruelle-Perron-Frobenius (RPF) operator; done at the symbolic level, culminating in the results in Section~\ref{CRPFOSDG},

\sp\item Using this information  on the RPF operator  to find meromorphic extensions of associated complex $\eta$ functions, i.e., Poincar\'e functions (or series), see Section~\ref{Local_Poincare},

\sp\item Using the information on the domain of the Poincar\'e series to deduce the asymptotic formulae (Theorem~\ref{dynA}) for $\lam_\om(\xi)$ on the mixture of the symbolic level (the words $\om\rho$ are required to belong to a symbolic cylinder $[\tau]$ rather than $\phi_{\om}(\pi_\cS(\rho))$ or $x_\om$ to belong to $B$) and GDMS level, by classical methods from prime number theory based on Tauberian theorems. 

\sp\item Having (3) derive the asymptotic formulae for $-\log|\phi_\om'(x_\om)|$; i.e. for periodic points $x_\om$ of $\phi_\om$ by means of sufficiently fine approximations. 

\sp\item Deducing the asymptotic formulae for the Borel sets $B\sbt X$ (Theorem~\ref{dyn}) from those  
at the symbolic level 
(Theorem~\ref{dynA}).
\end{enumerate}  

\bigskip

We can leverage our results for attracting systems to prove the
corresponding results for the more delicate case of parabolic systems.
This is done by associating with a parabolic system (by a form of
inducing) a countable alphabet attracting GDMSs and expressing the
corresponding Poincar\'e series for parabolic systems as infinite sums
of the Poincar\'e series for those associated attracting systems.

Furthermore,  the $D$-generic hypothesis of Theorem~\ref{t1_2017_04_17} needed for attracting systems is very mild.  
Moreover, parabolic systems, or more precisely the attracting systems associated to them, are automatically D--generic (see Theorem~\ref{t1pc3}), so no genericity hypothesis is needed for them.

Finally, parabolic systems are of equal importance to the attracting systems. Indeed, many of the applications, such as to Farey maps or Apollonian packings for example, are based on parabolic GDMSs. The parabolic systems generate more complex and intriguing counting phenomena, particularly in regard to counting diameters. 

\sp We now describe the results for asymptotic counting of diameters. These are more geometrical and more complex than for multipliers, and  counting multipliers is intrinsically more of a  ``dynamical process''. The following theorem comprises Theorem~\ref{t1da7}, Theorem~\ref{t1ma1}, Remark~\ref{r1_2017_03_20},   Theorem~\ref{t1dp13}, Theorem~\ref{t1dp13B}, and Remark~\ref{r1_2017_03_20B}.
We again refer the reader to the appropriate section for the  detailed definitions of the hypotheses (and to the next subsection for specific examples where these are known to hold).  However, for  the present, we note that
  $\Omega$ denotes the set of all parabolic points
and 
$\Omega_{\rho_1} \subset \Omega$ denotes  the subset
whose coding by an infinite sequence $\rho_1 \in E_A^\infty$ begins with the symbol $\rho_1$.  Finally, 
 $\Omega_\infty$ denotes the set of 
 parabolic  points $x_a$ whose corresponding index $p(a)$ 
 (see Proposition \ref{p1c5.13}) satisfies $\delta > 2p(a)/(1+ p(a))$.

\bthm [Asymptotic Equidistribution Formula for Diameters]\label{t2_2017_04_17} 
Suppose that $\cS$ is either a strongly regular finitely irreducible D-generic attracting conformal GDMS or a finite irreducible parabolic conformal GDMS.

Denote by $\d$ the Hausdorff dimension of its limit set $J_\cS$. 
Fix $\rho\in E_A^\infty$ and then a set $Y\sbt X_{i(\rho)}$ having at least two elements. 
 If $B \subset X$ is a Borel set such that $m_\d(\bd B)=0$ (equivalently $\mu_\d(\bd B)=0$) then,
$$
\lim_{T \to +\infty} \frac{D^{\rho}_Y(B,T)}{e^{\d T}} 
=C_{\rho_1}(Y)m_\d(B)
=\lim_{T \to +\infty} \frac{E^{\rho}_Y(B,T)}{e^{\d T}} ,
$$
where $C_{\rho_1}(Y)\in (0,+\infty]$ is a constant depending only on the system $\cS$, the letter $\rho_1$ and the set $Y$. 

In addition $C_{\rho_1}(Y)$ is finite if 
and only if either

\begin{enumerate}
\item 
$
\ov Y\cap \Om_\infty=\es
$
or 
\item 
$
\d>\max\big\{p(a):a\in\Om_{\rho_1} \  \  {\rm and} \  \  x_a\in \ov Y \big\}.
$
\end{enumerate}
In particular $C_{\rho_1}(Y)$ is finite if the system $\cS$ is attracting.
\ethm

\fr The proof of the results 
in Theorem~\ref{t2_2017_04_17} for diameters
are based on Theorem~\ref{t1_2017_04_17} for multipliers. The subtlety in the attracting case is that the basic bounded distortion property alone does not suffice to pass from the case of multipliers to the case of diameters; one needs additional approximating steps. For parabolic systems even the basic  bounded distortion property is weaker and more involved and a careful analysis of parabolic behavior is needed. 

It is worth emphasizing once again the importance of parabolic systems for many  applications and classes of examples, including that of  Apollonian packings. This is even more pronounced in the case of diameters than multipliers, since the  diameters often appear more frequently in the geometric setting.

\subsection{Examples}\label{examples_I} Now we would like to describe some classes of conformal dynamical systems to which we can  apply Theorem~\ref{t1_2017_04_17} and Theorem~\ref{t2_2017_04_17}. Often applying these results  requires some non-trivial preparation.

Our first class of examples is formed by conformal expanding repellers, see Definition~\ref{exprep}. The appropriate consequences of Theorem~\ref{t1_2017_04_17} and Theorem~\ref{t2_2017_04_17} are stated as Theorem~\ref{t1ex4}. The primary examples of non-linear conformal expanding repellers are formed by expanding rational functions of the Riemann sphere $\oc$. The consequences of Theorem~\ref{t1_2017_04_17} and Theorem~\ref{t2_2017_04_17} in this context, are given by Theorem~\ref{t1ex6}.

Perhaps the the most obvious example related to attracting GDMSs are the Gauss map 
$$
G(x) = x - [x],
$$
and the corresponding Gauss IFS $\mathcal G$ consisting of the maps 
$$
[0,1]\ni x \longmapsto g_n(x):=\frac{1}{x+n}, \  \ n \in \mathbb N.
$$
Theorem~\ref{t1_2017_04_10} summarizes the consequences of Theorem~\ref{t1_2017_04_17} and Theorem~\ref{t2_2017_04_17} stated for the Gauss map $G$ itself.

\sp Now let  describe some well known parabolic GDMSs to which our results apply. We start with $1$-dimensional systems. Our primary classes of such systems, defined and analyzed in Subsection~\ref{CPDS}, are illustrated by following.

\begin{enumerate}
\item[a)]
Manneville--Pomeau maps $f_\a:[0,1] \to [0,1]$ defined by
$$
f_\a(x) = x + x^{1+\alpha} \hbox{ (mod $1$)},
$$
where $\alpha > 0$ is a fixed number, and the Farey map 
$f:[0,1] \to [0,1]$ defined by
$$
f(x)
= 
\begin{cases}
\frac{x}{1-x} &\hbox{ if } 0 \leq x \leq \frac{1}{2}\\
\frac{1-x}{x} &\hbox{ if } \frac{1-x}{x} \leq x \leq 1.
\end{cases}
$$
The appropriate asymptotic counting results, stemming from Theorem~\ref{t1_2017_04_17} and Theorem~\ref{t2_2017_04_17}, are provided by Theorem~\ref{t1ex18} and Theorem~\ref{1ex5par}.

\item[b)]
A large class of conformal parabolic systems is provided by parabolic rational functions of the Riemann sphere $\oc$. These are those rational functions (see Subsection~\ref{PRF}) that have no critical points in the Julia sets but do have rationally indifferent periodic points. The appropriate asymptotic counting results, consequences of Theorem~\ref{t1_2017_04_17} and Theorem~\ref{t2_2017_04_17}, are stated as Corollary~\ref{c1ex19}. Probably the best known example of a parabolic rational function is the polynomial
$$
\widehat \C \ni z \longmapsto f_{1/4}(z) := z^2 + \frac{1}{4}\in \widehat \C.
$$
It has only one parabolic point, namely $z = 1/2$. In fact this is a fixed point of $f_{1/4}$ and $f_{1/4}'(1/2) =1$. Another explicit class of such functions is given by the maps of the form 
$$
\widehat {\mathbb C} \ni z \longmapsto 2 + 1/z + t
$$ 
where $t \in \mathbb R$.

\item[c)]
A separate large class of examples is provided by Kleinian groups, namely by finitely generated Shottky groups and essentially all finitely generated Fuchsian groups. 

Convex co-compact (no tangencies) Schottky groups are described and analyzed in detail in Section~\ref{Schottky-No Tangencies} while general Schottky groups (allowing tangencies) are dealt with in Subsection~\ref{tangent Schottky}. The appropriate asymptotic counting results, stemming from Theorem~\ref{t1_2017_04_17} and Theorem~\ref{t2_2017_04_17}, are provided by Theorem~\ref{t1ex16} and Theorem~\ref{tangent Schottky}.

\end{enumerate}

As a particularly famous example, the  counting problem of circles in a full Apollonian packing reduces to an appropriate counting problem for a finitely generated Schottky group with tangencies. The full presentation of asymptotic counting in this context, stemming from Theorem~\ref{t1_2017_04_17} and Theorem~\ref{t2_2017_04_17}, is given by Corollary~\ref{c2ex28}. We present below a more  restricted form (see Theorem~\ref{t1ex29}) involving only the counting of diameters; it overlaps with results from \cite{KO} (see also \cite{OS1}--\cite{OS3}), obtained by entirely different methods. 

\begin{thm}\label{t1ex29_I}
Let $C_1, C_2, C_3$ be three mutually tangent circles in the Euclidean plane having mutually disjoint interiors. Let $C_4$ be the circle tangent to all the circles $C_1, C_2, C_3$ and having all of them in its interior; we then refer to the configuration $C_1, C_2, C_3, C_4$ as bounded. Let $\mathcal A$ be the corresponding circle packing.

Let $\delta = 1.30561 \ldots$ be the Hausdorff dimension of the residual set of $\mathcal A$ and let $m_\delta$ be the Patterson-Sullivan measure of the corresponding parabolic Schottky group $\Gamma$.  

If $N_{\mathcal A}(T)$ denotes the number of circles in $\mathcal A$ of diameter at least $\frac{1}{T}$ then the limit
$$
\lim_{T\to +\infty} \frac{N_{\mathcal A}(T)}{e^{\delta T}}
$$
exists, is positive, and finite. Moreover, there exists a constant $C \in (0, +\infty)$ such that 
if $N_{\mathcal A}(T, B)$ denotes the number of circles in $\mathcal A$ of diameter at least $\frac{1}{T}$ and lying in $B$
 then
$$
\lim_{T \to +\infty} 
 \frac{N_{\mathcal A}(T; B)}{e^{\delta T}} = C m_\delta(B)
$$
for every open ball  $B \subset \mathbb C$.
\end{thm}

\fr Closely related to $\mathcal A$ is the curvilinear triangle $\mathcal T$ (or Apollonian triangle) formed by the three edges joining the three tangency points of $C_1, C_2, C_3$
and lying on these circles. The collection 
$$
\mathcal G := \{ C \in \mathcal A \hbox{ : } C \subset \mathcal T\}
$$
is called the Apollonian gasket generated by the circles $C_1, C_2, C_3$. 
As a consequence of Theorem~\ref{t1ex29_I} we get the following (see Corollary~\ref{c1ex30}); it overlaps with results from \cite{KO} (see also \cite{OS1}--\cite{OS3}), obtained with entirely different methods.

\begin{cor}\label{c1ex30_I}
Let $C_1, C_2, C_3$ be three mutually tangent circles in the Euclidean plane having mutually disjoint interiors. Let $C_4$ be the circle tangent to all the circles $C_1, C_2, C_3$ and having all of them in its interior; we then refer to  the configuration $C_1, C_2, C_3, C_4$ as bounded. Let $\mathcal A$ be the corresponding circle packing.

If $\mathcal T$ is the curvilinear triangle formed by $C_1$, $C_2$ and $C_3$, then the limit
$$
\lim_{T\to +\infty} \frac{N_{\mathcal A}(T; \mathcal T)}{e^{\delta T}}
$$
exists, is positive, and finite and   counts the elements of $\mathcal G$. Moreover, there exists a constant $C \in (0, +\infty)$, in fact the one of Theorem~\ref{t1ex29}, such that 
$$
\lim_{T \to +\infty}  \frac{N_{\mathcal A}(T; B)}{e^{\delta T}} = C m_\delta(B)
$$
for every Borel set $B \subset \mathcal T$ with $m_\delta(\partial B) = 0$.
\end{cor}

\fr In fact we can provide a more direct proof of Corollary 
\ref{c1ex30_I}, by appealing directly to the theory
 of parabolic conformal IFSs and avoiding the intermediate 
step of parabolic Schottky groups. Indeed, it follows immediately from  Theorem~\ref{c1da12.1J}

\begin{figure}[h]
\includegraphics[height=6cm]{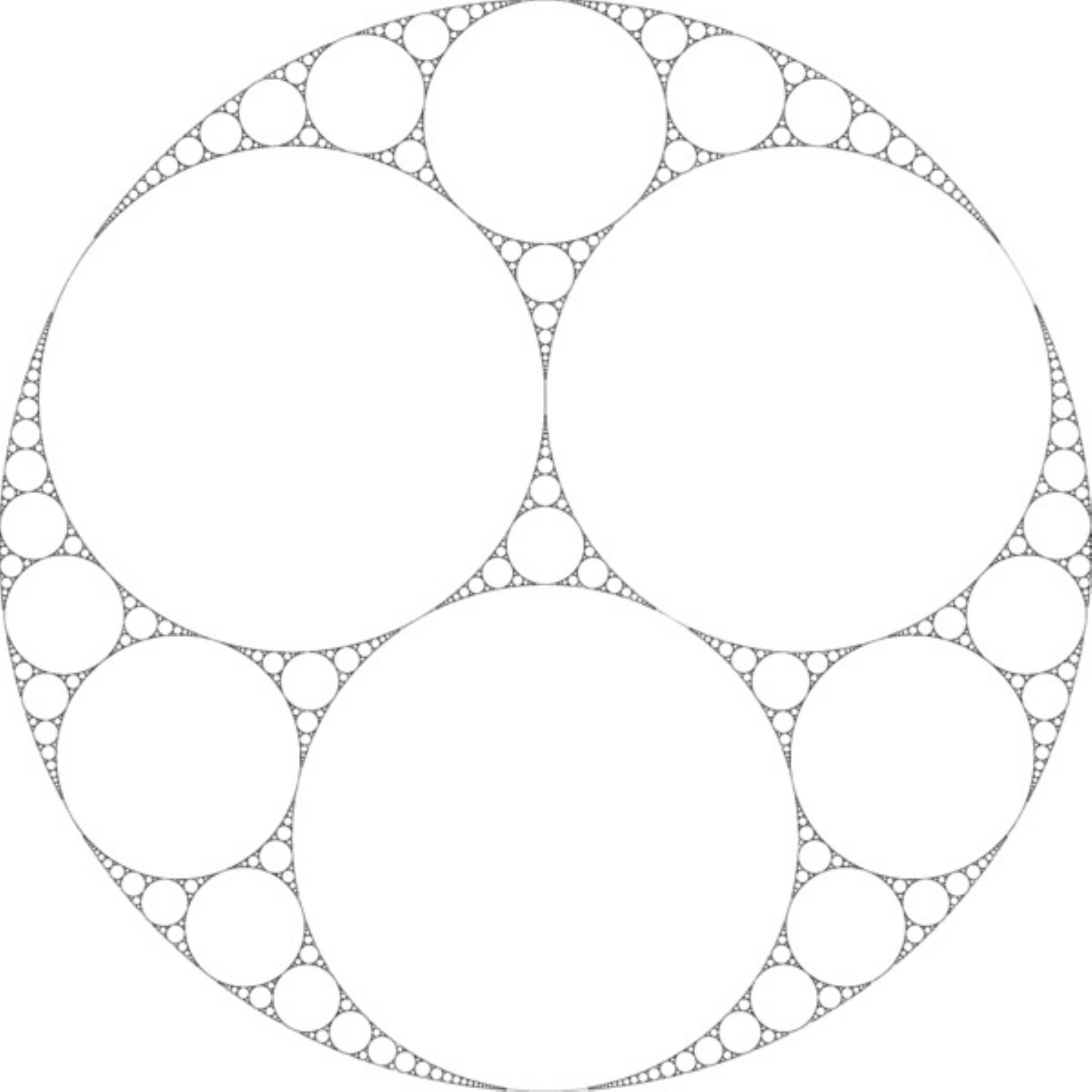}
\hskip 1cm
\includegraphics[height=6cm]{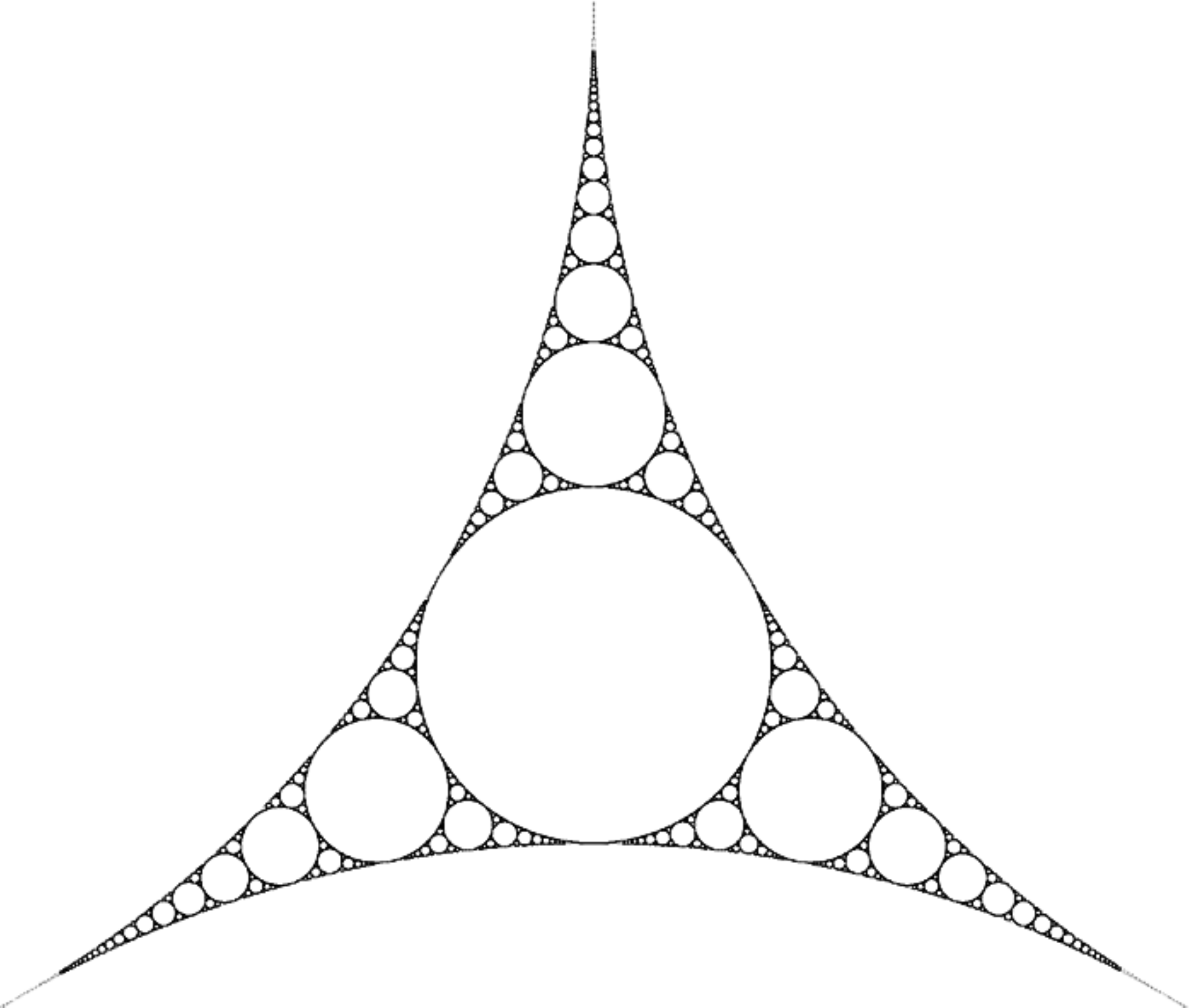}
\caption{(i) The Standard Apollonian Packing; (ii) The Apollonian Gasket}
\end{figure}

\subsection{Statistical results}
A second aim of this monograph is to consider the statistical properties of the distribution of the different weights $\lam_\rho(\om)$ and $\diam(\phi_\om(X))$  corresponding to words $\om$ with the same length $n$. In the context of attracting and parabolic GDMSs we have the following Central Limit Theorem, see Part~\ref{Section-CLT}. We refer the reader to the appropriate section for a detailed definitions of the hypothesis.

\begin{thm}\label{t1ms1_I}  
If $\cS$ is either a strongly regular finitely irreducible
D--generic conformal GDMS or a finite alphabet irreducible parabolic GDMS with $\delta > \frac{2p}{p+1}$ \footnote{this hypothesis means that the corresponding invariant measure $\mu_\d$ is finite, thus a probability after normalization},
then there exists $\sigma^2 > 0$ such that if $G \subset \mathbb R$
is a Lebesgue measurable set with $\hbox{{\rm Leb}}(\partial G) = 0$, then 
$$
\lim_{n \to +\infty}
\mu_\d\left(
\left\{
\omega \in E_A^\infty:\frac{-\log \big|\phi_{\omega|_n}' (\pi_{\cS}(\sigma^n(\omega)) )\big| - \chi_{\mu_\d} n}{\sqrt{n}}
\in G
\right\}
\right)
= \frac{1}{\sqrt{2\pi}\sigma} \int_G e^{-\frac{t^2}{2\sigma^2}} \,dt.
$$  
In particular, for any $\alpha <\beta$
$$
\lim_{n \to +\infty}
\mu_\d\left(
\left\{
\omega \in E_A^\infty:\alpha \leq \frac{-\log\big|\phi_{\omega|_n}' (\pi_{\cS}(\sigma^n(\omega)))\big| - \chi_{\mu_\d} n}{\sqrt{n}}
\leq \beta
\right\}
\right)
= \frac{1}{\sqrt{2\pi}\sigma} \int_\a^\b e^{-\frac{t^2}{2\sigma^2}}\, dt.
$$  
\end{thm}

The following  result is an alternative Central Limit Theorem 
considering  instead the logarithms of the  diameters of the images of  reference sets.

\begin{thm}\label{t1ms1-again_I}
Suppose that there $\cS$ is either a strongly regular finitely irreducible D--generic conformal GDMS or a finite alphabet irreducible parabolic GDMS with $\delta > \frac{2p}{p+1}$. Let $\sigma^2:=\P''(0)(\neq  0)$. For every $v \in V$ let $Y_v \subset X_v$ be a set with at least two points. If $G \subset \mathbb R$ is a Lebesgue measurable set with $\hbox{\rm Leb}(\partial G) = 0$, then 
$$
\lim_{n \to +\infty}
\mu_\d\left(
\left\{
\omega \in E_A^\infty:
 \frac{
 -\log \hbox{\rm diam}(\phi_{\omega|_n}(Y_{t(\omega_n)}))  - \chi_{\mu_\d} n
 }{\sqrt{n}}
\in G
\right\}
\right)
= \frac{1}{\sqrt{2\pi}\sigma} \int_G e^{-\frac{t^2}{2\sigma^2}}\, dt.
$$  
In particular, for any $\alpha < \beta$
$$
\lim_{n \to +\infty}
\mu_\d\left(
\left\{
\omega \in E_A^\infty: \alpha \leq 
 \frac{-\log \hbox{\rm diam}(\phi_{\omega|_n}(Y_{t(\omega_n)}))  - \chi_{\mu_\d} n}{\sqrt{n}}
\leq \beta
\right\}
\right)
= \frac{1}{\sqrt{2\pi}\sigma} \int_\alpha^\beta e^{-\frac{t^2}{2\sigma^2}}\,dt.
$$  
\end{thm}

\fr There are more theorems in this vein proven in Part~\ref{Section-CLT},  for example the Law of Iterated Logarithm. In order to formulate 
other  statistical results of  a slightly different flavor, we define the following measures
$$
\mu_n(H) := \frac{\sum_{\omega \in H} e^{-\delta \lambda_\rho(\omega)}}{\sum_{\omega \in E^n_\rho} e^{-\delta \lambda_\rho(\omega)}}
$$
for integers $n\ge 1$ and $H \subset E^n_\rho$.
We also consider the function $\Delta_n:E_A^n\to\R$ given by 
$$
\Delta_n(\omega)= \frac{\lambda_\rho(\omega) - \chi_\delta n}{\sqrt{n}}.
$$ 		
\begin{thm}\label{t1ms5.2_I} 
If $\cS$ is either a finitely irreducible strongly regular conformal GDMS or a finite alphabet irreducible parabolic GDMS with $\delta > \frac{2p}{p+1}$, then for every $\rho \in E_A^\infty$ we have that
$$
\lim_{n \to +\infty} \int_{E_\rho^n} \frac{\lambda_\rho}{n} d\mu_n 
= \chi_{\mu_\d}=\int_{E_\rho^\infty}\lambda d\mu_\delta.
$$
\end{thm}

\fr The following  theorem describes  precisely the magnitude of deviations in this convergence, and is another form of Central Limit Theorem. 

\begin{thm}\label{t1m58_I} 
If $\cS$ is either a strongly regular finitely irreducible D--generic attracting conformal graph directed Markov system or a finite alphabet irreducible parabolic GDMS with $\delta > \frac{2p}{p+1}$ , then the sequence of random variables $(\Delta_n)_{n=1}^\infty$ converges in distribution
to the normal (Gaussian) distribution 
$\mathcal N_0(\sigma)$ with mean value zero and the variance $\sigma^2 = \P''(\delta)$.  
Equivalently, the sequence 
$(\mu_n \circ \Delta_n^{-1})_{n=1}^\infty$
converges weakly to the normal distribution $\mathcal N_0(\sigma^2)$.  This means that for every Borel set $F \subset \mathbb R$ with 
$\hbox{\rm Leb}(\partial F) = 0$, we have 
$$
\lim_{n \to +\infty}\frac{\sum_{\omega\in  E^n_\rho}  
\big|\phi_\om'(\pi_\cS(\rho)\big|^\d\1_F\lt(\frac{\lambda_\rho(\omega|_{n}) - \chi_\delta n}{\sqrt{n}}\rt)}
{\sum_{\omega\in E^n_\rho}\big|\phi_\om'(\pi_\cS(\rho|)\big|^\d}
=\lim_{n \to +\infty} \mu_n(\Delta_n^{-1}(F))
= \frac{1}{\sqrt{2\pi} \sigma}
\int_F e^{- t^2/2\sigma^2} dt.
$$
\end{thm}

\fr In particular all these theorems hold for all classes of examples described in 
subsection~\ref{examples_I}, in the case of parabolic systems under the additional hypothesis that $\delta > \frac{2p}{p+1}$, which ensures  that the corresponding invariant measure $\mu_\d$ is finite, thus probabilistic after normalization. In the case of continued fractions these take on exactly the same form, in the case of Kleinian groups, including Apollonian circle packings, the same form for associated GDMSs.

However, in giving  statements of the Central Limit Theorems for examples we have chosen  rational functions to best illustrate them. 
 The first result is a Central Limit Theorem for the distribution of the derivatives of $T$ along orbits.
\begin{thm}\label{t51_2017_04_03} 
Let $f:\oc\to\oc$ be either an expanding rational function of the Riemann sphere $\widehat{\mathbb{C}}$ or a parabolic rational function of $\widehat{\mathbb{C}}$ with $\d>\frac{2p}{p+1}$. Then
there exists $\sigma^2 > 0$  such that if $G \subset \mathbb R$ is a Lebesgue measurable set with $\hbox{{\rm Leb}}(\partial G) = 0$, then 
$$
\lim_{n \to +\infty}
\mu_\d\left(
\left\{
z\in J(f) \hbox{ : } \frac{\log \big|(f^n)'(z)\big| - \chi_{\mu_\d} n}{\sqrt{n}}
\in G
\right\}
\right)
= \frac{1}{\sqrt{2\pi}\sigma} \int_G e^{-\frac{t^2}{2\sigma^2}} \,dt.
$$  
In particular, for any $\alpha <\beta$
$$
\lim_{n \to +\infty}
\mu_\d\left(
\left\{
z\in J(f) \hbox{ : } \alpha \leq \frac{\log\big|(f^n)'(z)\big| - \chi_{\mu_\d} n}{\sqrt{n}}
\leq \beta
\right\}
\right)
= \frac{1}{\sqrt{2\pi}\sigma} \int_\a^\b e^{-\frac{t^2}{2\sigma^2}}\, dt.
$$  
\end{thm}

The second result is  a Central limit Theorem describing the diameter of the preimages of reference sets.

\begin{thm}\label{t12_2017_04_03_I} 
Let $f:\oc\to\oc$ be either an expanding rational function of the Riemann sphere $\widehat{\mathbb{C}}$ or a parabolic rational function of $\widehat{\mathbb{C}}$ with $\d>\frac{2p}{p+1}$. 
Then for every $e \in F$ let $Y_e \subset R_e$ be a set with at least two points. If $G \subset \mathbb R$ is a Lebesgue measurable set with $\hbox{\rm Leb}(\partial G) = 0$, then 
$$
\lim_{n \to +\infty}
\mu_\d\left(
\left\{
z\in J(f) \hbox{ : }
 \frac{
 -\log \hbox{\rm diam}\(f_x^{-n}(Y_{e(z,n)})\)  - \chi_{\mu_\d} n
 }{\sqrt{n}}
\in G
\right\}
\right)
= \frac{1}{\sqrt{2\pi}\sigma} \int_G e^{-\frac{t^2}{2\sigma^2}}\, dt
$$  
where $f_x^{-n}$ is a local inverse for $f^n$ in a neighbourhood of $x= f^n(z)$.
In particular, for any $\alpha < \beta$
$$
\lim_{n \to +\infty}
\mu_\d\left(
\left\{
z\in J(f) \hbox{ : } \alpha \leq 
 \frac{-\log \hbox{\rm diam}\(f_x^{-n}(Y_{e(z,n)})\) - \chi_{\mu_\d} n}{\sqrt{n}}
\leq \beta
\right\}
\right)
= \frac{1}{\sqrt{2\pi}\sigma} \int_\alpha^\beta e^{-\frac{t^2}{2\sigma^2}}\,dt.
$$  
\end{thm}

\begin{thm}\label{t54_2017_04_03_I}
If $f:\oc\to\oc$ is either an expanding rational function of the Riemann sphere $\widehat{\mathbb{C}}$ or a parabolic rational function of $\widehat{\mathbb{C}}$ with $\d>\frac{2p}{p+1}$, then for every $\xi\in J(f)$, we have that
$$
\lim_{n \to +\infty} \int_{f^{-n}(\xi)} \frac{\log\big|(f^n)'\big|}{n} d\mu_n 
= \chi_\delta.
$$
\end{thm} 

The final result is a Central Limit Theorem which describes the distribution of preimages of a reference point.

\begin{thm}\label{t54_2017_04_04_I} 
If $f:\oc\to\oc$ is either an expanding rational function of the Riemann sphere $\widehat{\mathbb{C}}$ or a parabolic rational function of $\widehat{\mathbb{C}}$ with $\d>\frac{2p}{p+1}$, then the sequence of random variables $(\Delta_n)_{n=1}^\infty$ converges in distribution to the normal (Gaussian) distribution 
$\mathcal N_0(\sigma)$ with mean value zero and the variance $\sigma^2>0$.  
Equivalently, the sequence 
$(\mu_n \circ \Delta_n^{-1})_{n=1}^\infty$
converges weakly to the normal distribution $\mathcal N_0(\sigma^2)$.  This means that for every Borel set $F \subset \mathbb R$ with 
$\hbox{\rm Leb}(\partial F) = 0$, we have 
$$
\lim_{n \to +\infty}\!\!\frac{\sum_{z\in f^{-n}(\xi)}|(f^n)'(z)|^{-\d}\1_F\lt(\frac{\log|(f^n)'(z)|-\chi_\delta n}{\sqrt{n}}\rt)}
{\sum_{z\in f^{-n}(\xi)}|(f^n)'(z)|^{-\d}}
=\lim_{n \to +\infty} \mu_n(\Delta_n^{-1}(F))
= \frac{1}{\sqrt{2\pi} \sigma}
\int_F e^{- t^2/2\sigma^2} dt.
$$
\end{thm}

\part{{\Large  Attracting Conformal Graph Directed Markov Systems}}

\sp\section{{\large Thermodynamic Formalism of Subshifts of Finite Type with Countable Alphabet; Preliminaries}\label{sec:subshifts}}

\

In this section we introduce the basic symbolic setting in which we will be working.
We will describe the fundamental thermodynamic concepts, ideas and results, particularly those related to the associated Ruelle-Perron-Frobenius operators, which will play a crucial role throughout the monograph.

Let $\mathbb{N}=\{1, 2, \ldots \}$  be  the set of all positive
integers  and let $E$ be  a  countable  set, either finite or infinite,
called in the sequel  an {\it alphabet}. Let
$$
\sg: E^\mathbb{N} \to E^\mathbb{N}
$$
be the  shift map, \index{shift map} i.e. cutting off the first coordinate and shifting one place to the left. It is given by the
formula
$$
\sg\( (\om_n)^\infty_{n=1}  \) =  \( (\om_{n+1})^\infty_{n=1}  \).
$$
We also set
$$
E^*=\bigcup_{n=0}^\infty E^n.
$$
to be the set of finite strings.
\fr For every $\om \in E^*$, we denote by $|\om|$  the unique integer
$n \geq 0$ such that $\om \in E^n$. We  call $|\om|$ the length of
$\om$. We make the convention that $E^0=\{\es\}$. If $\om \in
E^\mathbb{N}$ and $n \geq 1$, we put
$$
\om |_n=\om_1\ldots \om_n\in E^n.
$$
If $\tau \in E^*$ and $\om \in E^* \cup E^\mathbb{N}$, we define the concatenation of $\tau$ and $\omega$ by:  
$$
\tau\om:=
\begin{cases}
\tau_1\dots\tau_{|\tau|}\om_1\om_2\dots\om_{|\om|} \  \  &{\rm if } \ \om \in E^*, \\
\tau_,\dots\tau_{|\tau|}\om_1\om_2\dots \  \  &{\rm if } \ \om \in E^\N,.
\end{cases}
$$
Given $\om,\tau\in E^{\mathbb N}$, we define $\omega\wedge\tau  \in
E^{\mathbb N}\cup E^*$ to be the longest
initial block common to both $\om$ and $\tau$. For each $\alpha >
0$, we define a metric $d_\alpha$ on
$E^{\mathbb N}$ by setting
\begin{equation}\label{d-alpha}
d_\alpha(\om,\tau) ={\rm e}^{-\alpha|\om\wedge\tau|}.
\end{equation}
All these metrics induce the same topology, known to be the product (Tichonov) topology. A real or complex valued function defined on a subset of $E^\N$ is uniformly
continuous with respect to one of these metrics if and only if it is
uniformly continuous with respect to all of them. Also, this function is
H\"older with respect to one of these metrics if and only if it is
H\"older with respect to all of them although, of course, the H\"older exponent depends
on the metric. If no metric is specifically mentioned, we take it to
be $d_1$.

Now consider an arbitrary matrix $A:E \times E \to \{0,1\}$. Such a matrix will be called the incidence matrix in the sequel. Set
$$
E^\infty_A
:=\{\om \in E^\mathbb{N}:  \,\, A_{\om_i\om_{i+1}}=1  \,\, \mbox{for
  all}\,\,   i \in \N
\}.
$$
Elements \index{A-admissible matrices@$A$-admissible matrices} of $E^\infty_A$ are called {\it $A$-admissible}. We also set
$$
E^n_A
:=\{\om \in E^\mathbb{N}:  \,\, A_{\om_i\om_{i+1}}=1  \,\, \mbox{for
  all}\,\,  1\leq i \leq
n-1\}, \ n \in \N,
$$
and
$$
E^*_A:=\bigcup_{n=0}^\infty E^n_A.
$$
The elements of these sets are also called $A$-admissible. For
every  $\om \in E^*_A$, we put
$$
[\om]:=\{\tau \in E^\infty_A:\,\, \tau_{|_{|\om|}}=\om \}.
$$
The set $[\om]$ is called the cylinder generated by the word $\om$. The collection of all such cylinders forms a base for the product topology relative to $E^\infty_A$.
The following fact is obvious.

\bprop\label{p1j83}
The set $E^\infty_A$ is  a closed subset of
$E^\mathbb{N}$, invariant under the shift map $\sg: E^\mathbb{N}\to E^\mathbb{N}$, the latter meaning that
$$
\sg(E^\infty_A)\sbt E^\infty_A.
$$
\eprop

The matrix $A$ is said to be {\it finitely irreducible} if there
exists a finite set $\La \sbt E_A^*$ such that for all $i,j\in E$
there exists $\om\in \La$ for which $i\om j\in E_A^*$. If all elements of some such $\La$ are of the same length, then $A$ is called {\it finitely primitive} (or aperiodic).

\sp The topological pressure of a continuous function $f:E_A^\infty\to\R$ with respect to the shift map $\sg:E_A^\infty\to E_A^\infty$ is defined to be
\beq\lab{2.1.1}
\P(f)
:=\lim_{n\to\infty}\frac{1}{n}\log \sum_{\om \in E_A^n}\exp \lt(\sup_{\tau\in [\om]}\sum_{j=0}^{n-1}f(\sg^j(\tau))\rt).
\eeq
The existence of this limit, following from the observation that the ``$\log$'' above forms a subadditive sequence, was established in \cite{MU_Israel}, comp. \cite{MU_GDMS}. Following the common usage we abbreviate
$$
S_nf:=\sum_{j=0}^{n-1}f\circ\sg^j
$$
and call $S_nf(\tau)$ the $n$th Birkhoff's sum of $f$ evaluated at a word $\tau\in E_A^\infty$.

\sp Observe that a function $f:E_A^\infty\to\R$ is (locally)  {\it H\"older continuous
with an exponent $\a>0$} if and only if
$$
V_\a(f):=\sup_{n\ge 1}\{V_{\a,n}(f)\}<+\infty,
$$
where 
$$
V_{\a,n}(f)=\sup\{|f(\om)-f(\tau)|e^{\a(n-1)}:\om,\tau\in E_A^\infty
\text{ and } |\om\wedge \tau|\ge n\}.
$$
Observe further that $\H_\a(A)$, the vector space of all bounded H\"older continuous functions $f:E_A^\infty\to\R\, ({\rm or \ }\C)$ with an exponent $\a>0$ becomes a Banach space with the norm $||\cdot||_\a$ defined as follows:
$$
||f||_\a:=||f||_\infty+V_\a(f).
$$
The following theorem has been proved in \cite{MU_Israel}, comp. \cite{MU_GDMS}, for the class of acceptable functions defined there. Since H\"older continuous ones are among them, we have the following.

\bthm[Variational Principle]\lab{t2.1.6}
If the incidence matrix $A:E\times E\to\{0,1\}$ is 
finitely irreducible and if $f:E_A^\infty\to\R$ is H\"older continuous, then
$$
\P(f)=\sup\Big\{\h_{\mu}(\sg)+\int f\,d\mu\Big\},
$$
where the supremum is taken over all $\sg$-invariant (ergodic) Borel
probability measures $\mu$ such that $\int f\,d\mu >-\infty$.
\ethm

We call a $\sg$-invariant probability measure $\mu$ on $E_A^\infty$ an {\it equilibrium
state of a H\"older continuous function $f:E_A^\infty\to\R$} if $\int -f\,d\mu<+\infty$ and 
\begin{equation}\lab{2.2.9}
\h_{\mu}(\sg)+\int \!\! f\, d\mu=\P(f). 
\end{equation}
If $f:E_A^\infty\to\R$ is a H\"older continuous function, then following \cite{MU_Israel}, and \cite{MU_GDMS} a Borel probability
measure $\mu$ on $E_A^\infty$ is called a {\it Gibbs state} for $f$ (comp. also \cite{Bo1}, \cite{HMU}, \cite{PU}, \cite{Ru1}, \cite{Wa} and \cite{U1}) if there exist constants
$Q\ge 1$ and $\P_\mu\in\R$ such that for every $\om\in E_A^*$ and every
$\tau\in [\om]$
\begin{equation}\lab{2.2.1}
Q^{-1}\le {\mu([\om])\over \exp\(S_{|\om|}f(\tau)-\P_{\mu}|\om|\)}
\le Q. 
\end{equation}
If additionally $\mu$ is shift-invariant, it is then called an
{\it invariant Gibbs state}. It is readily seen from this definition that if a H\"older continuous function $f:E_A^\infty\to\R$ admits a Gibbs state $\mu$, then 
$$
\P_\mu=\P(f).
$$
From now on  throughout this section $f:E_A^\infty\to\R$ is assumed to be a
H\"older continuous function with an exponent $\a>0$, and it is also assumed to satisfy the following requirement 
\begin{equation}\lab{2.3.1}
\sum_{e\in E}\exp(\sup(f|_{[e]}))<+\infty. 
\end{equation}
Functions $f$ satisfying this condition are called (see \cite{MU_Israel}, and \cite{MU_GDMS}) in the sequel
{\it summable}. We note that if $f$ has a Gibbs state, then $f$ is summable.
This requirement of  summability, allows us to define the {\it Perron-Frobenius operator} 
$$
\pf_f:C_b(E_A^\infty)\to C_b(E_A^\infty),
$$
acting on the space of bounded 
continuous functions $C_b(E_A^\infty)$ endowed with $\|\cdot\|_\infty$, the 
supremum norm, as follows:
$$
\pf_f(g)(\om):=\sum_{e\in E:A_{e\om_1}=1}\exp(f(e\om)g(e\om).
$$
Then $\|\pf_f\|_\infty\le \sum_{e\in E}\exp(\sup(f|_{[e]}))<+\infty$ and for every
$n\ge 1$
$$
\pf_f^n(g)(\om)=\sum_{\tau\in E_A^n:A_{\tau_n\om_1}=1}\exp\(S_nf(\tau\om)\)
g(\tau\om).
$$
The conjugate operator $\pf_f^*$ acting on the space $C_b^*(E_A^\infty)$ has the following form:
$$
\pf_f^*(\mu)(g):=\mu(\pf_f(g)) = \int\pf_f(g)\,d\mu.
$$
Observe that the operator $\pf_f$ preserves the space $\H_\a(A)$, of all H\"older continuous functions with an exponent $\a>0$. More precisely
$$
\pf_f(\H_\a(A))\sbt \H_\a(A).
$$

We now provide a brief account of those elements of the spectral theory that we will need and use in the sequel. Let $B$ be a Banach space and let $L:B\to B$ be a bounded linear operator. A point $\lam\in\C$ is said to belong to the spectral set (spectrum) of the operator $L$ if the operator $\lam I_B-L:B\to B$ is not invertible, where $I_B:B\to B$ is the identity operator on $B$. The spectral radius $r(L)$ of $L$ is defined to be the supremum of moduli of all elements in the spectral set of $L$. It is known that $r(L)$ is finite and
$$
r(L)=\lim_{n\to\infty}\|L^n\|^{1/n}.
$$
A point $\lam$ of the spectrum of $L$ is said to belong to the essential spectral set (essential spectrum) of the operator $L$ if $\lam$ is not an isolated eigenvalue of $L$ of finite multiplicity. The essential spectral radius $r_\ess(L)$ of $L$ is defined to be the supremum of moduli of all elements in the essential spectral set of $L$. It is known (see \cite{Nussbaum}) that 
$$
r_\ess(L)=\varlimsup_{n\to\infty}\inf\big\{\|L^n-K\|^{1/n}\big\},
$$
where for every $n\ge 1$ the infimum is taken over all compact operators $K:B\to B$. 
The operator $L:B\to B$ is called quasi-compact if either $r(L)=0$ or 
$$
r_\ess(L)<r(L).
$$
The proof of the following theorem can be found in \cite{MU_GDMS}. For the items (a)--(f) see also Corollary~4.3.8 in \cite{CTU}.

\bthm\lab{c2.7.5} 
Suppose that $f:E_A^\infty\to\R$ is a H\"older continuous summable function
and the incidence matrix $A$ is finitely irreducible. Then

\sp\begin{itemize}
\item[(a)] There exists a unique Borel probability eigenmeasure $m_f$ of the conjugate 
Perron-Frobenius operator $\pf_f^*$ and the corresponding eigenvalue is equal
to $e^{\P(f)}$.

\sp\item[(b)] The eigenmeasure $m_f$ is a Gibbs state for $f$.

\sp\item[(c)] The function $f:E_A^\infty\to\R$ has a unique $\sg$-invariant Gibbs
state $\mu_f$.

\sp\item[(d)] The measure $\mu_f$ is ergodic, equivalent to $m_f$ and if $\psi_f=
d\mu_f/dm_f$ is the Radon--Nikodym derivative of $\mu_f$ with respect to $m_f$, then
$\log\psi_f$ is uniformly bounded.

\sp\item[(e)] If $\int-f\,d\mu_f<+\infty$, then the $\sg$-invariant Gibbs
state $\mu_f$ is the unique equilibrium state for the potential $f$.

\sp\item[(f)] In case the incidence matrix $A$ is finitely primitive, the Gibbs
state $\mu_f$ is completely ergodic. 

\sp\item[(g)] The spectral radius of the operator $\pf_f$ considered as acting either on $C_b(E_A^\infty)$ or $\H_\a(A)$ is in both cases equal to $e^{\P(f)}$.

\sp\item[(h)] In either case of {\rm(g)} the number $e^{\P(f)}$ is a simple (isolated in the case of $\H_\a(A)$) eigenvalue of $\pf_f$ and the Radon--Nikodym derivative $\psi_f\in \H_\a(A)$ generates its eigenspace.

\sp\item[(i)] The remainder of the spectrum of the operator $\pf_f:\H_\a(A)\to\H_\a(A)$ is contained in a union of finitely many eigenvalues of finite multiplicity (different from $e^{\P(f)}$) of modulus $e^{\P(f)}$ and a closed disk centered at $0$ with radius strictly smaller than $e^{\P(f)}$. \nl
In particular, the operator $\pf_f:\H_\a(A)\to\H_\a(A)$ is quasi-compact.

\sp\fr In the case where the incidence matrix $A$ is finitely primitive a stronger statement holds: namely, apart from $e^{\P(f)}$, the spectrum of $\pf_f:\H_\a(A)\to\H_\a(A)$ is contained in a closed disk centered at $0$ with radius strictly smaller than $e^{\P(f)}$.\nl
In particular, the operator $\pf_f:\H_\a(A)\to\H_\a(A)$ is quasi-compact.
\end{itemize}
\ethm

\section{Attracting Conformal Countable Alphabet Graph Directed Markov Systems (GDMSs) \\ and \\ Countable Alphabet Attracting Iterated Function Systems (IFSs);  \\ Preliminaries}\label{Attracting_GDMS_Prel}

In this article we consider a slightly more general setting than just the usual conformal iterated function systems, namely the ones better 
suited to modeling  the examples in which we are interested.
In later sections we will prove the results in this context and explain  how they can be used to derive  
 different geometric and dynamical  results, such as those already mention in the introduction.

Let us define a graph directed Markov system (abbr. GDMS) relative to a directed multigraph $(V,E,i,t)$ and an incidence matrix $A:E\times E\to\{0,1\}$. Such systems are defined and studied at length in \cite{MU_London} and \cite{MU_GDMS}. A \emph{directed multigraph} consists of a finite set $V$ of vertices, a countable (either finite or infinite) set $E$ of directed edges, two functions 
$$
i,t:E\to V,
$$
and an \emph{incidence matrix} $A:E\times E\to \{0,1\}$ on $(V,E,i,t)$ such that 
$$
A_{ab} = 1 \  \  \ {\rm implies } \  \  \  t(b) = i(a).
$$
Now suppose that in addition, we have a collection of nonempty compact metric spaces $\{X_v\}_{v\in V}$ and a number $\ka \in (0,1)$, such that for every $e\in E$, we have a one-to-one contraction $\phi_e:X_{t(e)}\to X_{i(e)}$ with Lipschitz constant (bounded above by) $\ka$. Then the collection
\[
\cS = \{\phi_e:X_{t(e)}\to X_{i(e)}\}_{e\in E}
\]
is called an \emph{attracting graph directed Markov system} (or GDMS). We will frequently refer to it just as a graph directed Markov system or GDMS. We will however always  keep the  adjective "parabolic" when, in later sections, we will also speak about parabolic graph directed Markov systems. We extend the functions $i,t:E\to V$ in a natural way to $E_A^*$ as follows:
$$
t(\om):=t\(\om_{|\om|}\) \  \  \  {\rm and }  \  \  \  i(\om):=i(\om_1).
$$
For every word $\om\in E_A^*$, say $\om\in E_A^n$, $n\ge 0$, let us denote 
$$
\phi_{\omega}:= \phi_{\omega_1} \circ \cdots \circ  \phi_{\omega_n}:X_{t(\om)}\to X_{i(\om)}.
$$
We now describe the limit set, also frequently called the  attractor, of the system $\cS$. For any $\omega \in E^\infty_A$, the sets $\{\phi_{\omega|_n}\left(X_{t(\omega_n)}\right)\}_{n \geq 1}$ form a descending sequence of nonempty compact sets and therefore $\bigcap_{n \geq 1}\phi_{\omega|_n}\left(X_{t(\omega_n)}\right)\ne\emptyset$. Since for every $n \geq 1$, 
$$
\diam\left(\phi_{\omega|_n}\left(X_{t(\omega_n)}\right)\right)
\le \ka^n\diam\left(X_{t(\omega_n)}\right)
\le \ka^n\max\{\diam(X_v):v\in V\}, 
$$
we conclude that the intersection 
\[
\bigcap_{n \in \N}\phi_{\omega|_n}\left(X_{t(\omega_n)}\right)
\]
is a singleton and we denote its only element by $\pi_{\cS}(\omega)$ or simpler, by $\pi(\omega)$. In this way we have defined a map
\[
\pi_{\cS}=\pi:E^\infty_A\to X:=\coprod_{v\in V}X_v,
\]
where $\coprod_{v\in V} X_v$ is the disjoint union of the compact topological spaces $X_v$, $v\in V$. The map $\pi$ is called the \emph{coding map}, and the set
\[
J = J_\cS = \pi(E^\infty_A)
\]
is called the \emph{limit set} of the GDMS $\cS$. The sets
\[
J_v = \pi(\{\omega \in E_A^\infty : i(\omega_1) = v\}), \;\; v\in V,
\]
are called the \emph{local limit sets} of $\cS$.

\sp\fr We call the GDMS $\cS$ \emph{finite} if the alphabet $E$ is finite. Furthermore, we call $\cS$ \emph{maximal} if for all $a,b\in E$, we have $A_{ab}=1$ if and only if $t(b)=i(a)$.  In \cite{MU_GDMS}, a maximal GDMS was called a \emph{graph directed system} (abbr. GDS).
Finally, we call a maximal GDMS $\cS$ an \emph{iterated function system} (or IFS) if $V$, the set of vertices of $\cS$, is a singleton. Equivalently, a GDMS is an IFS if and only if the set of vertices of $\cS$ is a singleton and all entries of the incidence matrix $A$ are equal to $1$.

\bdfn\label{definitionsymbolirred}
We call the GDMS $\cS$ and its incidence matrix $A$ \emph{finitely irreducible} if there exists a finite set $\Omega\subset E_A^*$ such that for all $a, b\in E$ there exists a word $\omega\in\Omega$ such that the concatenation $a\omega b$ is in $E_A^*$. $\cS$ and $A$ are called \emph{finitely primitive} if the set $\Omega$ may be chosen to consist of words all having the same length. If such a set $\Omega$ exists but is not necessarily finite, then $\cS$ and $A$ are called irreducible and primitive, respectively. Note that all IFSs are symbolically irreducible.
\edfn

\brem\label{r1_2017_04_01}
For every integer $q\ge 1$ define $\cS^q$, the $q$th iterate of the system $\cS$,  to be 
$$
\{\phi_\om:X_{t(\om)}\to X_{i(\om)}:\om\in E_A^q\}
$$
and its alphabet is $E_A^q$. All the theorems proved in this monograph hold under the formally weaker hypothesis that all the elements of some iterate $\cS^q$, $q\ge 1$, of the system $\cS$, are uniform contractions. This in particular pertains to the Gauss system of Example~\ref{r1ex5} for which $q=2$ works. 
\erem

\sp With the aim of moving on to geometric applications,  and following \cite{MU_GDMS}, we call a GDMS \emph{conformal} if for some $d\in\N$, the following conditions are satisfied.

\sp
\begin{itemize}
\item[(a)] For every vertex $v\in V$, $X_v$ is a compact connected subset of $\R^d$, and $X_v=\overline{\Int(X_v)}$.
\item[(b)] (Open Set Condition) For all $a,b\in E$ such that $a\ne b$,
\[
\phi_a(\Int(X_{t(a)}))\cap \phi_b(\Int(X_{t(b)}))=\emptyset.
\]
\item[(c)] (Conformality)  There exists a family of open connected sets $W_v \subset X_v$, $v\in V$, such that for every $e\in E$, the map $\phi_e$ extends to a $C^1$ conformal diffeomorphism from $W_{t(e)}$ into $W_{i(e)}$ with Lipschitz constant $\leq \ka$.

\item[(d)] (Bounded Distortion Property (BDP)) There are two constants $L\ge 1$ and $\alpha>0$ such that for every $e\in E$ and every pair of points $x,y\in X_{t(e)}$,
\[
\bigg|\frac{|\phi_e'(y)|}{|\phi_e'(x)|}-1 \bigg| \le L\|y-x\|^\alpha,
\]
where $|\phi_\omega'(x)|$ denotes the scaling factor of the derivative $\phi_\omega'(x):\R^d\to\R^d$ which is a similarity map.
\end{itemize}

\brem\label{p1.033101}
When $d=1$ the conformality is automatic.  
If $d\ge 2$ and a family $\cS = \{\phi_e\}_{e\in E}$ satisfies the conditions (a) and (c), then it also satisfies condition (d) with $\alpha=1$. When $d = 2$ this is due to the well-known Koebe's Distortion Theorem (see for example, \cite[Theorem 7.16]{Conway_II}, \cite[Theorem 7.9]{Conway_II}, or \cite[Theorem 7.4.6]{Hille_II}). When $d \geq 3$ it is due to \cite{MU_GDMS} depending heavily on Liouville's representation theorem for conformal mappings; see \cite{Iwaniec_Martin} for a detailed development of this theorem leading up to the strongest current version, and also including exhaustive references to the historical background. 
\erem

For every real number $s\ge 0$, let (see \cite{MU_London} and \cite{MU_GDMS})
$$ 
\P(s) := \lim_{n \to +\infty} \frac{1}{n} \log 
\left(\sum_{|\om| = n}\|\phi_\om'\|_\infty^s\right),
$$
where $\|\phi'\|_\infty$ denotes the supremum norm of the derivative of a conformal map $\phi$ over its domain; in our context these domains will be always the sets $X_v$, $v\in V$. The above limit always exists because the corresponding sequence is clearly subadditive. The number $\P(s)$ is called the topological pressure of the parameter $s$. Because of the Bounded Distortion Property (i.e., Property (d)), we have also the following characterization of topological pressure:
$$ 
\P(s) := \lim_{n \to +\infty} \frac{1}{n} \log 
\left(\sum_{|\om| = n}  |\phi_\om'(z_\om)|^s\right),
$$
where $\{z_\om: \om\in E_A^*\}$ is an entirely arbitrary set of points such that $z_\om\in X_{t(\om)}$ for every $\om\in E_A^*$. Let $\zeta: E^\infty_A \to \mathbb{R}$ be defined by the formula
\beq\label{1MU_2014_09_10}
\zeta(\om):= \log|\vp'_{\om_1}(\pi(\sg(\om))|.
\eeq
The following proposition is easy to prove; see \cite[Proposition 3.1.4]{MU_GDMS} for complete details. 

\bprop\label{l1j85}
For every real $s\geq 0$ the function $s\zeta:E^\infty_A \to
\mathbb{R}$ is H\"older continuous and 
$$
\P(s\zeta)=\P(s).
$$
\eprop

\bdfn\label{fins}
We say that a nonnegative real number $s$ belongs to $\Ga_\cS$ if
\beq\label{finite_parameters}
\sum_{e\in E}\|\vp'_e\|_\infty^s<+\infty.
\eeq
\edfn

\fr Let us record the following immediate observation.

\bobs\label{o1_2016_01_20}
A nonnegative real number $s$ belongs to $\Ga_\cS$ if and only if the H\"older continuous potential $s\zeta:E_A^\infty\to\R$ is summable.
\eobs

\fr We recall from \cite{MU_London} and \cite{MU_GDMS} the following definitions:
$$
\g_\cS:=\inf\Ga_\cS=\inf\lt\{s\ge 0: \sum_{e\in E}\|\phi_e'\|_\infty^s<+\infty\rt\}.
$$
The proofs of the following two statements can be found in \cite{MU_GDMS}.

\bprop\label{p1_2016_01_12}
If $\cS$ is an irreducible conformal GDMS, then for every $s\ge 0$ we have that 
$$
\Ga_\cS=\{s\ge 0: \P(s)<+\infty\}
$$
In particular,
$$
\g_\cS:=\inf\lt\{s\ge 0: \P(s)<+\infty\rt\}.
$$
\eprop

\bthm\label{t3_2016_01_12}
If $\cS$ is a finitely irreducible conformal GDMS, then the function $\Ga_\cS\ni s\mapsto \P(s)\in\R$ is 

\sp\begin{enumerate}
\item strictly decreasing, 

\sp\item real-analytic, 

\sp\item convex, and 

\sp\item $\lim_{s\to+\infty}\P(s)=-\infty$.
\end{enumerate}
\ethm

\fr We denote
$$
\pf_s:=\pf_{s\zeta}
$$
acting either on $C_b(E_A^\infty)$ or  on $\H_a(A)$.
Because of Proposition~\ref{l1j85} and Observation~\ref{o1_2016_01_20}, our Theorem~\ref{c2.7.5} applies to all functions $s\zeta:E_A^\infty\to\R$ giving the following. 

\bthm\lab{thm-conformal-invariant}
Suppose that the system $\cS$ is finitely irreducible and $s\in\Ga_\cS$. Then
\begin{itemize}
\item[(a)] There exists a unique Borel probability eigenmeasure $m_s$ of the conjugate 
Perron-Frobenius operator $\pf_s^*$ and the corresponding eigenvalue is equal
to $e^{\P(s)}$.

\sp\item[(b)] The eigenmeasure $m_s$ is a Gibbs state for $t\zeta$.

\sp\item[(c)] The function $s\zeta:E_A^\infty\to\R$ has a unique $\sg$-invariant Gibbs state $\mu_s$. 

\sp\item[(d)] The measure $\mu_s$ is ergodic, equivalent to $m_s$ and if $\psi_s=
d\mu_s/dm_s$ is the Radon--Nikodym derivative of $\mu_s$ with respect to $m_s$, then
$\log\psi_s$ is uniformly bounded.

\sp\item[(e)] If $\chi_{\mu_s}:=-\int\zeta\,d\mu_s<+\infty$, then the $\sg$-invariant Gibbs
state $\mu_s$ is the unique equilibrium state for the potential $s\zeta$.

\sp\item[(f)] In case the the system $\cS$ is finitely primitive, the Gibbs
state $\mu_s$ is completely ergodic.

\sp\item[(g)] The spectral radius of the operator $\pf_s$ considered as acting either on $C_b(E_A^\infty)$ or $\H_\a(A)$ is in both cases equal to $e^{\P(s)}$.

\sp\item[(h)] In either case of {\rm(g)} the number $e^{\P(s)}$ is a simple (isolated in the case of $\H_\a(A)$) eigenvalue of $\pf_s$ and the Radon--Nikodym derivative $\psi_s\in \H_\a(A)$ generates its eigenspace.

\sp\item[(i)] The reminder of the spectrum of the operator $\pf_s:\H_\a(A)\to\H_\a(A)$ is contained in a union of finitely many eigenvalues (different from $e^{\P(s)}$) of modulus $e^{\P(s)}$ and a closed disk centered at $0$ with radius strictly smaller than $e^{\P(s)}$ (if $A$ is finitely primitive, then these eigenvalues of modulus smaller than $e^{\P(s)}$ disappear). In particular, the operator $\pf_s:\H_\a(A)\to\H_\a(A)$ is quasi-compact.
\end{itemize}
\ethm

\fr Given $s\in\Ga_\cS$ it immediately follows from this theorem and the definition of Gibbs states that
\begin{equation}\label{515pre_1}
C_s^{-1} e^{-\P(s)|\om|}\|\phi_\om'\|_\infty^s
\leq  m_s([\om])
\comp \mu_s([\om])
\leq C_se^{-\P(s)|\om|}\|\phi_\om'\|_\infty^s
\end{equation}
for all $\om\in E_A^*$, where $C_s\ge 1$ denotes some constant. We put
\beq\label{1_2016_12_14}
\^m_s:=m_s\circ\pi_\cS^{-1} 
\  \  \  {\rm and} \  \  \
\^\mu_s:=\mu_s\circ\pi_\cS^{-1}.
\eeq
The measure $\^m_s$ is characterized (see \cite{MU_GDMS}) by the following two properties:
\beq\label{2_2016_12_14}
\^m_s(\phi_e(F))=e^{-\P(s)}\int_F|\phi_e'|\,d\^m_s
\eeq
for every $e\in E$ and every Borel set $F\sbt X_{t(e)}$, and
\beq\label{3_2016_12_14}
\^m_s\(\phi_a(X_{t(a)})\cap \phi_b(X_{t(b)})\)=0
\eeq
whenever $a, b\in E$ and $a\ne b$. By a straightforward induction these extend to 
\beq\label{3_2016_12_14-again}
\^m_s(\phi_\om(F))=e^{-\P(s)|\om|}\int_F|\phi_\om'|\,d\^m_s
\eeq
for every $\om\in E_A^*$ and every Borel set $F\sbt X_{t(\om)}$, and
\beq\label{4_2016_12_14}
\^m_s\(\phi_\a(X_{t(\a)})\cap \phi_\b(X_{t(\b)})\)=0
\eeq
whenever $\a,\b\in E_A^*$ and are incomparable.

The following theorem, providing a geometrical interpretation of the parameter $\d_\cS$, has been proved in  \cite{MU_GDMS} (\cite{MU_London} in the case of IFSs).

\bthm\label{t2_2016_01_12}
If $\cS$ is an finitely irreducible conformal GDMS, then 
$$
\d=\d_\cS:=\HD(J_\cS)=\inf\{s\ge 0: \P(s)\le 0\}\ge \g_\cS.
$$
\ethm

Following \cite{MU_London} and \cite{MU_GDMS} we call the system $\cS$ regular if there exists $s\in (0,+\infty)$ such that 
$$
\P(s)=0.
$$
Then by Theorems~\ref{t2_2016_01_12} and \ref{t3_2016_01_12}, such zero is unique and is equal to $\d_\cS$. So,
\beq\label{1_2016_08_19}
\P(\d_\cS)=0.
\eeq
Formula \eqref{515pre_1} then takes  the following form:
\begin{equation}\label{515pre}
C_{\d_\cS}^{-1}\|\phi_\om'\|_\infty^{\d_\cS}
\leq  m_{\d_\cS}([\om])
\comp \mu_{\d_\cS}([\om])
\leq C_{\d_\cS}\|\phi_\om'\|_\infty^{\d_\cS}
\end{equation}
for all $\om\in E_A^*$. The measure $\^m_{\d_\cS}$ is then referred to as the $\d_\cS$--conformal measure of the system $\cS$.

Also following \cite{MU_London} and \cite{MU_GDMS}, we call the system $\cS$ strongly regular if there exists $s\in [0,+\infty)$ (in fact in $(\g_\cS,+\infty)$) such that 
$$
0<\P(s)<+\infty. 
$$
Because of Theorem~\ref{t3_2016_01_12} each strongly regular conformal GDMS is regular. Furthermore, we record the following two immediate observations.

\bobs\label{o5_2016_01_20}
If $s\in\Int(\Ga_\cS)$, then $\chi_{\mu_s}<+\infty$.
\eobs

\bobs\label{o1_2016_01_20-plus}
A finitely irreducible conformal GDMS $\cS$ is strongly regular if and only if
$$
\g_\cS<\d_\cS.
$$
In particular, if the system $\cS$ is a strongly regular, then $\d_\cS\in\Int(\Ga_\cS)$.
\eobs

\fr These two observations yield the following.

\bcor\label{c1_2016_01_21}
If a finitely irreducible conformal GDMS $\cS$ is strongly regular, then $\chi_{\mu_\d}<+\infty$.
\ecor 

\fr We will also need the following fact, well-known in the case of finite alphabets $E$, and proved for all countable alphabets in \cite{MU_GDMS}.

\bthm\label{t1_2016_01_29}
If $s\in\Int(\Ga_\cS)$, then 
$$
\P'(s)=-\chi_{\mu_s}.
$$
In particular this formula holds if the system $\cS$ is strongly regular and $s=\d_\cS$.
\ethm

\fr We end this section by noting that each finite irreducible system is strongly regular.

\section{Complex Ruelle--Perron--Frobenius Operators; Spectrum and D--Genericity}
\label{CRPFOSDG}
A key ingredient when analyzing the Poincar\'e series $\eta_\xi(s)$ 
and $\eta_p(s)$ mentioned in the introduction is to use complex Ruelle-Perron-Frobenius or Transfer operators. These are closely related to the RPF operators already introduced, except that we now allow the weighting function to take complex values.
More precisely,  we  extend the definition of operators $\pf_s$, $s\in \Ga_\cS$, to the complex half-plane
$$
\Ga_\cS^+:=\{s\in\C:\re s>\g_\cS\},
$$
in a most natural way; namely, for every $s\in\Ga_\cS^+$, we set
\beq\label{5_2016_01_21}
\pf_s(g)(\om)=\sum_{e\in E:A_{e\om_1}=1}|\phi_e'(\pi(\om))|^sg(e\om).
\eeq
Clearly these linear operators $\pf_s$ act on both Banach spaces $C_b(E_A^\infty)$ and $\H_a(A)$, are bounded, and we have the following.

\bobs\label{o1_2016_01_21}
The function 
$$
\Ga_\cS^+\ni s\mapsto \pf_s\in L(\H_a(A))
$$
is holomorphic, where $L(\H_a(A))$ is the Banach space of all bounded linear operators on $\H_a(A)$ endowed with the operator norm. 
\eobs
\bprop\label{p1nh15}
Let $\cS$ be a finitely irreducible conformal GDMS. Then for every $s=\sg+it\in\Ga_\cS^+$ 

\sp
\begin{enumerate}

\item the spectral radius $r(\pf_s)$ of the operator $\pf_s:\H_\a(A)\to \H_a(A)$ is not larger than $e^{\P(\sg)}$ and

\sp\item the essential spectral radius $r_\ess(\pf_s)$ of the operator $\pf_s:\H_\a(A)\to \H_a(A)$ is not larger than $e^{-\a}e^{\P(\sg)}$.
\end{enumerate}
\eprop

\bpf
Assume without loss of generality that $E=\N$. For every $\om\in E_A^*$ choose arbitrarily $\hat\om\in[\om]$. Now for every integer $n\ge 1$ define the linear operator 
$$
E_n:\H_\a(A)\to \H_a(A)
$$
by the formula
\beq\label{1nh16}
E_n(g):=\sum_{\om\in E_A^n}g(\hat\om)\1_{[\om]}.
\eeq
Equivalently
$$
E_n(g)=g(\hat\om), \  \om\in E_A^\infty.
$$
Of course $||E_n(g)||_\a\le||g||_\a$ and $E_n$ is a bounded operator with $||E_n||_\a\le 1$. However, the series \eqref{1nh16} is not uniformly convergent, i.e. it is not convergent in the supremum norm $||\cdot||_\infty$, thus not in the H\"older norm $||\cdot||_\a$ either. For all integers $N\ge 1$ and $n\ge1$ denote
$$
E_A^n(N):=\{\om\in E_A^n:\forall_{j\le n} \, \om_j\le N\}
$$
and
$$
E_A^n(N+):=\{\om\in E_A^n:\exists_{j\le n} \, \om_j> N\}.
$$
Let  us further write 
$$
E_{n,N}g:=\sum_{\om\in E_A^n(N)}g(\hat\om)\1_{[\om]}
$$
and 
$$
E_{n,N}^+g:=\sum_{\om\in E_A^n(N+)}g(\hat\om)\1_{[\om]}.
$$
Of course $E_{n,N}:\H_\a(A)\to \H_a(A)$ is a finite--rank operator, thus compact. Therefore, the composite operator $\pf_sE_{n,N}:\H_\a(A)\to \H_a(A)$ is also compact. We know that
\beq\label{2nh16}
\begin{aligned}
\|\pf_s^n-\pf_s^nE_{n,N}\|_\a
&=\|(\pf_s^n-\pf_s^nE_n)+\pf_s^n(E_n-E_{n,N})\|_\a
=\|\pf_s^n(I-E_n)+\pf_s^nE_{n,N}^+\|_\a \\
&\le \|\pf_s^n(I-E_n)\|_\a+\|\pf_s^nE_{n,N}^+\|_\a.
\end{aligned}
\eeq
We will estimate from above each of the last two terms separately.  We begin first with the first of these two terms. In the same way as for real parameters $s$, which is done in \cite{MU_GDMS}, one proves for all operators $\pf_s:\H_\a(A)\to \H_a(A)$ the following form of the Ionescu--Tulcea--Marinescu inequality:
\beq\label{1_2016_08_17}
\|\pf_s^ng\|_\a\le Ce^{\P(\sg)n}\(\|g\|_\infty+e^{-\a n}\|g\|_\a\)
\eeq
with some constant $C>0$. This establishes item (1) of our theorem. Since a straightforward calculation shows that $\|g-E_ng\|_\a\le 2\|g\|_\a$ and $\|g-E_ng\|_\infty\le \|g|_\a e^{-\a n}$ , we therefore get that
$$
\|\pf_s^n(I-E_n)g\|_\a
\le Ce^{\P(\sg)n}\(\|g\|_\a e^{-\a n}+2e^{-\a n}\|g\|_\a\)
=3Ce^{\P(\sg)n}e^{-\a n}\|g\|_\a.
$$
Thus,
\beq\label{1nh17}
||\pf_s^n(I-E_n)||_\a\le3Ce^{\P(\sg)n}e^{-\a n}.
\eeq
Passing to the estimate of the second term, we have
$$
\pf_s^nE_{n,N}^+g(\om)
=\sum_{{\tau\in E_A^n(N+)\atop\tau\om\in E_A^\infty}}g(\hat\tau)|\phi_\tau'(\pi(\sg(\om)))|^s.
$$
Therefore, 
$$
\begin{aligned}
\|\pf_s^nE_{n,N}^+g\|_\a
&\le\sum_{\tau\in E_A^n(N+)}|g(\hat\tau)
     |\Big\|\big|\phi_\tau'\circ\pi\circ\sg\big|^s\Big\|_\a
     \cr
&\le\|g\|_\infty\sum_{\tau\in E_A^n(N+)}\Big\|\big|   
     \phi_\tau'\circ\pi\circ\sg\big|^s\Big\|_\a\\
&\le\|g\|_\infty\sum_{\tau\in E_A^n(N+)}\Big\|\big|   
     \phi_\tau'\circ\pi\circ\sg\big|^s\Big\|_\a.
\end{aligned}
$$
Hence,
\beq\label{2nh17}
\|\pf_s^nE_{n,N}^+\|_\a
\le\sum_{\tau\in E_A^n(N+)}
    \Big\|\big|\phi_\tau'\circ\pi\circ\sg\big|^s\Big\|_\a.
\eeq
But 
$$
\Big\|\big|\phi_\tau'\circ\pi\circ\sg\big|^s\Big\|_\a
\le C \|\phi_\tau'\|_\infty^\sg.
$$
for all $\tau\in E_A^*$ with some constant $C>0$. Since the matrix $A:E\times E\to \{0,1\}$ is finitely irreducible, there exists a finite set 
$\La_\infty\sbt E_A^\infty$ such that for every $e\in E$ there exists (at least one) $\hat e\in\La_\infty$ such that $e\hat e\in E_A^\infty$. We further set for every $\tau\in E_A^*$,
$$
\hat\tau:=\widehat{\tau_{|\tau|}}.
$$
For every $k\in E=\N$ let
\beq\label{1nh17.2.1}
\xi_k:=\sup\{\|\phi_n'\|_\infty:n\ge k\}\longrightarrow 0  \  \  {\rm as }  \  \
   k\to \infty.
\eeq
Fix an arbitrary $\e>0$ so small that $\sg-\e>\g_{\cS}$. By the Bounded Distortion Property and \eqref{1nh17.2.1}, we then have

\beq\label{1nh17.2}
\begin{aligned}
\sum_{\tau\in E_A^n(N+)}\|\phi_\tau'\|_\infty^\sg
&\le K^\sg\sum_{\tau\in E_A^n(N+)}|\phi_\tau'(\pi(\hat\tau))|^\sg \cr
& \le K^\sg\sum_{\om\in\La_\infty}\sum_{\tau\in E_A^n(N+)\atop 
       \tau\om\in  E_A^\infty}|\phi_\tau'(\pi(\om))|^\sg \\
&=   K^\sg\sum_{\om\in\La_\infty}\sum_{\tau\in E_A^n(N+)\atop 
       \tau\om\in E_A^\infty}|\phi_\tau'(\pi(\om))|^\e
       |\phi_\tau'(\pi(\om))|^{\sg-\e} \\
&\le K^\sg\xi_N^\e\sum_{\om\in\La_\infty}\sum_{\tau\in E_A^n(N+)\atop 
       \tau\om\in E_A^\infty}|\phi_\tau'(\pi(\om))|^{\sg-\e} \\
&\le K^\sg\#\La_\infty\xi_N^\e\pf_{\sg-\e}^n\1(\om)
 \le K^\sg\#\La_\infty\xi_N^\e\|\pf_{\sg-\e}^n\|_\infty \\
&\le K^\sg\#\La_\infty\xi_N^\e\|\pf_{\sg-\e}^n\|_\a \\
&\le CK^\sg\#\La_\infty\xi_N^\e e^{\P(\sg-\e)n},
\end{aligned}
\eeq
where the last inequality was written due to \eqref{1_2016_08_17} applied with $s=\sg-1$ and $g=\1$. Inserting this to \eqref{1nh17.2.1} and \eqref{1nh17.2}, we thus get that
$$
\|\pf_s^nE_{n,N}^+\|_\a\le CK^\sg\#\La_\infty\xi_N^\e e^{\P(\sg-\e)n}.
$$
Now, take an integer $N_n\ge 1$ so large that $\xi_N^\e\le (K^\sg\#\La_\infty)^{-1}e^{-\a n}$. Inserting this to the above display, we get that 
$$
\|\pf_s^nE_{n,N_n}^+\|_\a\le Ce^{\P(\sg-\e)n}e^{-\a n}.
$$
Along with (\ref{1nh17}), (\ref{2nh16}), and the fact that $\P(\sg)\le \P(\sg-\e)$, this gives that
$$
\|\pf_s^n-\pf_s^nE_{n,N_n}\|_\a\le 4Ce^{\P(\sg-\e)n}e^{-\a n}.
$$
Therefore,
$$
r_\ess(\pf_s)
\le \varlimsup_{n\to\infty}\|\pf_s^n-\pf_s^n\circ E_{n,N_n}\|_\a^{1/n}
\le e^{\P(\sg-\e)}e^{-\a}.
$$
Letting $\e\to 0$ and using continuity of the pressure function $\Ga_\cS^+\ni t\mapsto\P(t)\in\R$, we thus get that 
$$
r_\ess(\pf_s) \le e^{-\a}e^{\P(\sg)}.
$$
The proof of item (2) is thus complete, and we are done.
\epf

\sp We recall that if $\lam_0$ is an isolated point of the spectrum of a bounded  linear operator $L$ acting on a Banach space $B$, then the Riesz projector $P_{\lam_0}:B\to B$ of $\lam_0$ (with respect to $L$) is defined as
$$
\frac1{2\pi i}\int_\g(\lam I-L)^{-1}d\lam
$$
where, $\g$ is any simple closed rectifiable Jordan curve enclosing $\lam_0$ and enclosing no other point of the spectrum of $L$. We recall that $\lam_0$ is called simple if the range $P_{\lam_0}(B)$ of the projector $P_{\lam_0}$ is $1$-dimensional. Then $\lam_0$ is necessarily an eigenvalue of $L$. We recall the following well-known fact.

\bthm\label{t1_2016_03_10} 
Let $\lam_0$ be an eigenvalue of a bounded linear operator $L$ acting on a Banach space $B$. Assume that the Riesz projector $P_{\lam_0}$ of $\lam_0$ (and $L$) is of finite rank. If there exists a constant $C\in [0,+\infty)$ such that 
$$
\|L^n\|\le C|\lam_0|^n
$$
for all integers $n\ge 0$,
then (of course) $r(L)=|\lam_0|$, and moreover
$$
P_{\lam_0}(B)=\Ker(\lam_0I-L).
$$
\ethm

\fr What we will really need in conjunction with Proposition~\ref{p1nh15} is the following.

\blem\label{l1nh18.1}
If $\cS$ is a finitely irreducible conformal GDMS and if $s=\sg+it\in\Ga_\cS^+$, then every eigenvalue of $\pf_s:\H_\a(A)\to \H_a(A)$ with modulus equal to  $e^{\P(\sg)}$ is simple.
\elem

\bpf
Since $\|\pf_s^n\|_\a\le 3||\pf_\sg^n||_\a\le Ce^{\P(\sg)n}$ for every $n\ge 0$ and some constant $C>0$ independent of $n$, and since the Riesz projector of every eigenvalue of modulus $e^{\P(\sg)}$ of $\pf_s$ is of finite rank (as by Proposition~\ref{p1nh15} such an eigenvalue does not belong to the essential spectrum of $\pf_s$), we conclude from Theorem~\ref{t1_2016_03_10} that in order to prove our lemma it suffices to show that
$$
\dim\(\Ker(\lam I-\pf_s)\)=1
$$
for any such eigenvalue $\lam$. Consider two operators $\hat\pf_\sg, \hat\pf_s:\H_\a(A)\to \H_a(A)$ given by the formulae
\beq\label{1nh18.1}
\hat\pf_\sg g(\om):=e^{-\P(\sg)}\frac1{\psi_\sg(\om)}\pf_\sg(g\psi_\sg)(\om)
\eeq
and 
\beq\label{2nh18.1}
\hat\pf_s g(\om):=e^{-\P(\sg)}\frac1{\psi_\sg(\om)}\pf_s(g\psi_\sg)(\om)
\eeq
Both these operators are conjugate respectively to the operators $e^{-\P(\sg)}\pf_\sg$ and $e^{-\P(\sg)}\pf_s$, $r(\hat\pf_\sg)=1$, 
\beq\label{3nh18.1}
\hat\pf_\sg\1=\1 \  \  (\text{{\rm so}} \  \  \hat\pf_\sg^n\1=\1 \  \
\text{{\rm for all}} \  \  n\ge 0),
\eeq
and in order to prove our lemma it is enough to show that
$$
\dim\(\Ker(\lam I-\hat\pf_s)\)=1
$$
for every eigenvalue $\lam$ of $\hat\pf_s$ with modulus equal to $1$. We shall prove the following.

\sp{\sl Claim~$1^0$:} If $u\in \H_\a(A)$, then the sequence 
$$
\lt(\frac1n\sum_{j=0}^{n-1}\hat\pf_\sg^ju\rt)_{n=1}^\infty
$$
converges uniformly on compact subsets of $E_A^\infty$ to the constant function equal to $\int_{E_A^\infty}u\, d\mu_\sg$.

\bpf
The same proof as that of Theorem~4.3 in \cite{MU_GDMS} asserts that any subsequence of the sequence $\lt(\frac1n\sum_{j=0}^{n-1} \hat\pf_\sg^ju\rt)_{n=1}^\infty$ has a subsequence converging uniformly on compact subsets of $E_A^\infty$ to a function which is a fixed point of $\hat\pf_\sg$. By \eqref{3nh18.1} and Corollary~7.5 in \cite{MU_GDMS} each such function is a constant. Since the operator $\hat\pf_\sg$ preserves integrals ($\hat\pf_\sg^*\mu_\sg=\mu_\sg$) against Gibbs/equilibrium measure $\mu_\sg$, it follows that all these constants must be equal to $\int_{E_A^\infty}u\, d\mu_\sg$. The proof of Claim~$1^0$ is thus complete.
\epf

Now, fix $\lam\in \Ker(\lam I-\hat\pf_s)$ arbitrary and let $g\ne 0\in\Ker(\lam I-\hat\pf_s)$ be arbitrary.

\sp{\sl Claim~$2^0$:} The function $E_A^\infty\ni\om\mapsto|g(\om)|\in\R$ is constant.

\bpf
For every $\om\in E_A^\infty$ and every integer $n\ge 0$ we have $|g(\om)|=|\hat\pf_s^ng(\om)|\le \hat\pf_\sg^n|g|(\om)$, and therefore
$$
|g(\om)|\le \frac1n\sum_{j=0}^{n-1}\hat\pf_\sg^j|g|(\om).
$$
So, invoking Claim~$1^0$, we get that
$$
|g(\om)|\le \int_{E_A^\infty}|g|\, d\mu_\sg.
$$
Since $g$ is continuous and $\supp(\mu_\sg)=E_A^\infty$, this implies that 
$$
|g(\om)|=\int_{E_A^\infty}|g|\, d\mu_\sg
$$
for all $\om\in E_A^\infty$. The proof of Claim~$2^0$ is thus complete.
\epf

\sp\fr Formulae \eqref{1nh18.1}--\eqref{3nh18.1} give for every $\tau\in E_A^\infty$ that
$$
\hat\pf_\sg^ng(\tau)
=\sum_{{\om\in E_A^n\atop A_{\om_n\tau_1}=1}}\exp\(S_nh(\om\tau)\)g(\om\tau)
$$
and 
$$
\lam^ng(\tau) 
=\hat\pf_s^ng(\tau)
=\sum_{{\om\in E_A^n\atop A_{\om_n\tau_1}=1}}\exp\(S_nh(\om\tau)\)
|\phi_\om'(\pi(\tau))|^{it}g(\om\tau),
$$
where $h:E_A^\infty\to(-\infty,0)$ is some H\"older continuous function resulting from \eqref{3nh18.1} and 
$$
\sum_{{\om\in E_A^n\atop A_{\om_n\tau_1}=1}}\exp\(S_nh(\om\tau)\)=1.
$$
Since $\lam^n=1$, it follows from the last two formulas and Claim~$1^0$ that
$$
|\phi_\om'(\pi(\tau))|^{it}g(\om\tau)=\lam^ng(\tau)
$$
for all $\om\in E_A^n$ with $A_{\om_n\tau_1}=1$. Equivalently:
$$
g(\om\tau)=\lam^n|\phi_\om'(\pi(\tau))|^{-it}g(\tau).
$$
This implies that if $g_1$, $g_2$ are two arbitrary functions in $\Ker(\lam I-\pf_s)$ such that
$$
g_1(\tau)=g_2(\tau),
$$
then $g_1$ coincides with $g_2$ on the set $\{\om\tau:\om\in E_A^* \  \and  \, A_{\om_{|om|}\tau_1}=1\}$. But since this set is dense in $E_A^\infty$ and both $g_1$ and $g_2$ are continuous, it follows that 
$$
g_1=g_2.
$$
Thus the vector space $\Ker(\lam I-\pf_s)$  is $1$-dimensional and the proof is complete.
\epf

\fr Now we define
$$
E_p^*:=\{\om\in E_A^*:A_{\om_{|\om|}\om_1}=1\}.
$$
This set will be treated in greater detail in the forthcoming sections and will play an important role throughout the monograph.

For all $t,a\in\R$ we denote by $G_a(t)$ and $G_a^i(t)$ the multiplicative subgroups respectively of positive reals $(0,+\infty)$ and of the unit circle $S^1:=\{z\in\C:|z|=1\}$ that are respectively generated by the sets
$$
\big\{e^{-a|\om|}|\phi_\om'(x_\om)|^t:\om\in E_p^*\big\}\sbt (0,+\infty) 
\  \and  \
\big\{e^{-ia|\om|}|\phi_\om'(x_\om)|^{it}:\om\in E_p^*\big\}\sbt S^1,
$$
where $x_\omega$ is the only fixed point for $\phi_\omega:X_{i(\om_1)}\to X_{i(\om_1)}$.
The following proposition has been proved in \cite{P_RPF_2} in the context of finite alphabets $E$, but the proof carries through without any change to the case of countable infinite alphabets as well.

\bprop\label{p1nh12}
Let $\cS=\{\phi_e\}_{e\in E}$ be a finitely irreducible conformal GDMS. If $t\in\R$ and $a\in\R$, then the following conditions are equivalent.
\begin{itemize}

\item [(a)] $G_a(t)$ is generated by $e^{2\pi k}$ with some $k\in\N_0$.

\sp\item [(b)] $\exp(ia+\P(\sg))$ is an eigenvalue for $\pf_{\sg+it}: C_b(E_A^{\mathbb N})\to C_b(E_A^{\mathbb N})$ for some $\sg\in \Ga_\cS$.

\sp\item [(c)] $\exp(ia+\P(\sg))$ is an eigenvalue for $\pf_{\sg+it}:\H_\a(A)\to \H_\a(A)$ for all $\sg\in \Ga_\cS$.

\sp\item [(d)] There exists $u\in C_b(E_A^\infty)\, (\H_a(A))$ such that the function
$$
E_A^\infty\ni \om\mapsto t\zeta(\om)-a+u(\om)-u\circ\sg(\om)
$$
belongs to $C_b(E_A^\infty,2\pi\Z)\, (\H_a(E_A^\infty,2\pi\Z))$.

\sp\item [(e)]  $G_a^i(t)=\{1\}$.
\end{itemize}
\eprop

\fr As a matter of fact \cite{P_RPF_2} establishes equivalence (in the case of finite alphabet) of conditions (a)--(d) but the equivalence of (a) and (e) is obvious.

\sp\fr We call a parameter $t\in\R$ $\cS$-generic if the above condition (a) fails for $a=0$
 and we call it strongly $\cS$--generic if it fails for all $a\in\R$. We call the system $\cS$ D--generic if each parameter $t\in\R\sms\{0\}$ is $\cS$--generic and we call it strongly D-generic if each parameter $t\in\R\sms\{0\}$ is strongly $\cS$-generic. 
 
\brem\label{r2_2017_02_17}
We would like to remark that if the GDMS $\cS$ is D-generic, then no function $t\zeta:E_A^\infty\to\R$, $t\in\R\sms\{0\}$, is cohomologous to a constant. Precisely, there is no function $u\in C_b(E_A^\infty)$ such that
$$
t\zeta(\om)+u(\om)-u\circ\sg(\om)
$$
is a constant real-valued function.
\erem

The concept of D--genericity will play a pivotal role throughout our whole article. We start dealing with it by proving the following.

\bprop\label{c1nh18}
If $\cS$ is a finitely irreducible strongly D-generic conformal GDMS and if $s=\sg+it\in\Ga_\cS^+$ with $t\in\R\sms\{0\}$, then $r(\pf_s)<e^{\P(\sg)}$.
\eprop
\begin{proof}
By Proposition~\ref{p1nh15} the set 
$$
\sg(\pf_s)\cap\(\C\sms\ov B(0,e^{-\a/2}e^{\P(\sg)})\)
$$
is finite and consists only of eigenvalues of $\pf_s$. So, by Proposition~\ref{p1nh12},
$$
\sg(\pf_s)\cap \(\C\sms \ov B(0,e^{-\a/2}e^{\P(\sg)})\)\cap \{\lam\in\C:|\lam|=e^{\P(\sg)}\}=\es.
$$
Therefore, using also Theorem~\ref{thm-conformal-invariant} (g), we get that
$$
r(\pf_s)\le\max\Big\{e^{-\a/2}e^{\P(\sg)},\max\big\{|\lam|:\lam\in \sg(\pf_s)\cap \(\C\sms \ov B(0,e^{-\a/2}e^{\P(\sg)})\)\big\}\Big\} < e^{\P(\sg)}.
$$
The proof is complete.
\end{proof}

We now shall provide a useful characterization of D-generic and strongly D-generic systems. 

\bprop\label{p1nh13}
A finitely irreducible conformal GDMS $\cS=\{\phi_e\}_{e\in E}$ is D--generic if and only if the additive group generated by the set
$$
\big\{\log|\phi_\om'(x_\om)|:\om\in E_p^*\big\} 
$$
is not cyclic.
\eprop

\bpf
Suppose first that the system $\cS=\{\phi_e\}_{e\in E}$ is not D--generic. This means that there exists $t\in\R\sms\{0\}$ which is not $\cS$-generic. This in turn means that the group $G_0(t)$ is generated by some non-negative integral power of $e^{2\pi}$, say by $e^{2q\pi}$, $q\in\N_0$. And this means that for every $\om\in E_p^*$,
$$
|\phi_\om'(x_\om)|^t=\exp\(2\pi qk_\om\)
$$
with some (unique) $k_\om\in\Z$. But then $t\log|\phi_\om'(x_\om)|=2\pi qk_\om$ or equivalently
$$
\log|\phi_\om'(x_\om)|= \frac{2\pi q}{t}k_\om. 
$$
This implies that the additive group generated by the set
$$
\big\{\log|\phi_\om'(x_\om)|:\om\in E_p^*\big\}\sbt\R 
$$
is a subgroup of $\langle\frac{2\pi q}{t}\rangle$, the cyclic group generated by $\frac{2\pi q}{t}$, and is therefore itself cyclic.

\sp For the converse implication suppose that the additive group generated by the set
$$
\big\{\log|\phi_\om'(x_\om)|:\om\in E_p^*\big\} 
$$
is cyclic. This means that there exists $\g\in(0,+\infty)$ such that 
$$
\log|\phi_\om'(x_\om)|=2\pi \g l_\om
$$
for all $\om\in E_p^*$ and some $l_\om\in-\N_0$. There then exists $t\in\R\sms\{0\}$ such that $t\g\in\N$. But then
$$
|\phi_\om'(x_\om)|^t=\exp\((2\pi t\g)l_\om\),
$$
implying that the multiplicative group generated by the set
$$
\{|\phi_\om'(x_\om)|^t:\om\in E_p^*\big\} 
$$
is a subgroup of $<e^{2\pi t\g}>$, the cyclic group generated by $e^{2\pi t\g}$, and is therefore itself cyclic. This means that $t\in\R\sms\{0\}$ is not $\cS$-generic, and this finally means that the system $\cS$ is not D-generic. We are done.
\epf

\begin{rem}\label{generic}
The D--genericity assumption is fairly generic.
For example, it  holds if
there are two values $i,j\in E$ (or the weaker condition $i, j\in E_A^*$) such that $\frac{\log|\phi_i'(x_i)|}{\log|\phi_j'(x_j)|}$ is irrational; where we recall  that $x_i$ and $x_j$ are the unique fixed points, respectively, of $\phi_i$ and $\phi_j$.
On the other hand, it is easy to construct specific conformal GDMSs for which it fails.  For example, we can consider maps $\phi_i(x) = \frac{x+1}{2^i}$ for $i \geq 1$ and than we can deduce that 
$\log |\phi_i'(x)| \in (\log 2)\Z$.
\end{rem}

\bprop\label{p1nh13-again}
A finitely irreducible conformal GDMS $\cS=\{\phi_e\}_{e\in E}$ is strongly D--generic if and only if the additive group generated by the set
$$
\big\{\log|\phi_\om'(x_\om)|-\b|\om|:\om\in E_p^*\big\} 
$$
is not cyclic for any $\b\in\R$.
\eprop

\bpf
Suppose first that the system $\cS=\{\phi_e\}_{e\in E}$ is not strongly D--generic. This means that there exists $t\in\R\sms\{0\}$ which is not $\cS$-generic. This in turn means that for some $a\in\R$ the group $G_a(t)$ is generated by some non-negative integral power of $e^{2\pi}$, say by $e^{2q\pi}$, $q\in\N_0$. And this means that for every $\om\in E_p^*$,
$$
e^{-a|\om|}|\phi_\om'(x_\om)|^t=\exp\(2\pi qk_\om\)
$$
with some (unique) $k_\om\in\Z$. But then $t\log|\phi_\om'(x_\om)|-a|\om|=2\pi qk_\om$ or equivalently
$$
\log|\phi_\om'(x_\om)|-\frac{a}{t}|\om|= \frac{2\pi q}{t}k_\om. 
$$
This implies that the additive group generated by the set
$$
\big\{\log|\phi_\om'(x_\om)|- \frac{a}{t}|\om|:\om\in E_p^*\big\} 
$$
is a subgroup of $<\frac{2\pi q}{t}>$, the cyclic groups generated by $\frac{2\pi q}{t}$, and is therefore itself cyclic.

\sp For the converse implication suppose that the additive group generated by the set
$$
\big\{\log|\phi_\om'(x_\om)|- \b|\om|:\om\in E_p^*\big\} 
$$
is cyclic for some $\b\in\R$. This means that there exists $\g\in(0,+\infty)$ such that 
$$
\log|\phi_\om'(x_\om)|- \b|\om|=2\pi \g l_\om
$$
for all $\om\in E_p^*$ and some $l_\om\in\Z$. There then exists $t\in\R\sms\{0\}$ such that $t\g\in\N$. But then
$$
e^{-t\b|\om|}|\phi_\om'(x_\om)|^t=\exp\((2\pi t\g)l_\om\),
$$
implying that the multiplicative group generated by the set
$$
\big\{e^{-t\b|\om|}|\phi_\om'(x_\om)|^t:\om\in E_p^*\big\} 
$$
is a subgroup of $<e^{2\pi t\g}>$, the cyclic group generated by $e^{2\pi t\g}$, and is therefore itself cyclic. This means that $t\in\R\sms\{0\}$ is not strongly $\cS$-generic, and this finally means that the system $\cS$ is not strongly D-generic. We are done.
\epf

\section{Asymptotic Results for Multipliers; Statements and First Preparations}

In this section we keep the setting of the previous one. In this framework
we can formulate our  main asymptotic  result,  which has the dual virtues of being relatively easy to prove in this setting  and also having many interesting applications, as illustrated in the introduction. In a later section we will also formulate the general result for  $C^2$ multidimensional contractions, although the basic statements will be exactly the same. We can now define  two natural counting functions in the present context
corresponding to ``preimages'' and ``periodic points'' respectively.

\bdfn
We can naturally  order the countable family of the compositions of contractions $\vp\in E_A^*$ in two different ways. Fix $\rho\in E_A^\infty$ arbitrary and set $\xi:=\pi_\cS(\rho)\in J_\cS$. Let 
$$
E_\rho^*:=\{\om\in E_A^*:\om\rho\in E_A^*\},
$$
and for all integers $n\ge 1$ let
$$
E_\rho^n:=\{\om\in E_A^n:\om\rho\in E_A^*\}.
$$
We recall from the previous section the set
$$
E_p^*=\{\om\in E_A^*:A_{\om_{|\om|}\om_1}=1\},
$$
and for all integers $n\ge 1$ we put
$$
E_p^n:=\{\om\in E_A^n:A_{\om_n\om_1}=1\},
$$
i.e., the words $\om$ in $E_A^*$ such that the words $\om^\infty\in E_A^\infty$, the infinite concatenations of $\om$s, are periodic points of the shift map $\sg:E_A^\infty\to E_A^\infty$ with period $n$.
\begin{enumerate}
\item
Firstly, we can associate the weights
$$
\lambda_\rho(\om) := -\log |\phi_\om'(\xi)| > 0, \  \  \om\in E_\rho^*,
$$
and
\sp\item
Secondly, we can use  the weights  
$$
\lambda_p(\om) := -\log |\phi_\om'(x_\om)| > 0, \  \  \om\in E_p^*,
$$
where we recall that $x_\om (= \phi_\om(x_\om))$ is the unique  fixed point for the contraction $\phi_\om:X_{i(\om_1)}\to X_{i(\om_1)}$; we note that  $t(\om)=i(\om_1)$.
\end{enumerate}
We can associate appropriate counting functions to each of these weights, defined by 
$$
\pi_\rho(T) := \left\{\om\in E_\rho^*  \hbox{ : } \lambda_\rho(\om) \leq T \right\} 
\ \hbox{ and } \
\pi_p(T) := \left\{\om\in E_p^* \hbox{ : } \lambda_p(\om) \leq T \right\}, 
$$
respectively, 
and their cardinalities
$$
N_\rho(T):=\#\pi_\rho(T) \ 
\hbox{ and } \  N_p(T):=\#\pi_p(T),
$$ 
respectively, 
for each $T>0$, i.e. the number of words $\om\in E_i^*$ for which the corresponding 
weight $\lambda_i(\om)$  doesn't exceed $T$ for $i = \rho, p$.
\edfn

\fr The functions $\pi_\rho(T)$ and $\pi_p(T)$ are clearly  both monotone increasing in $T$.

\fr We first prove the following  basic result, showing that the rates of growth of these two functions are both equal to the Hausdorff Dimension of the limit set $J_{\mathcal S}$.

\begin{prop}\label{delta} 
If the (finitely irreducible) conformal GDMS $\cS$ is strongly regular, then
$$
\delta_\cS 
= \lim_{T \to +\infty} \frac{1}{T} \log N_\rho(T) 
= \lim_{T \to +\infty} \frac{1}{T} \log N_p(T).
$$
\end{prop}
\begin{proof}
Fix $i\in\{\rho,p\}$. Write $\d:=\d_\cS$. Assume for a contradiction  that
$$
\varlimsup_{T \to +\infty} \frac{1}{T} \log N_i(T)>\d.
$$
There then exists $\e>0$ and an increasing unbounded sequence $T_n\to+\infty$ such that 
$$
 N_i(T_n)\ge e^{(\d+\e)T_n}.
$$
We recall from the definition of a conformal GDMS that $\|\phi_e'\|_\infty \leq \ka\in(0,1)$ for all $e \in E$, and then $\|\phi_{\omega}'\|_\infty \leq \ka^{|\om|}$ for all $\omega \in E_A^*$. Since 
\beq\label{1_2016_08_20}
\lambda_i(\omega)+\log\|\phi_\omega'\|_\infty \ge 0
\eeq
for all $\omega \in E_A^*$.
we conclude that whenever $\om\in \pi_i(T_n)$, i.e. whenever $\lambda_i(\omega) \leq T_n$, then
$$
|\om| 
\leq \frac{T_n}{ |\log\ka|} 
\leq  k_n := \left[\frac{T_n}{ |\log\ka|}\right] + 1,
$$
where $[\cdot]$ denotes the integer part. Therefore, we can also bound
$$
\sum_{j=1}^{k_n}\sum_{\omega \in E^j_A}\|\phi_\omega'\|_\infty^\delta
\ge 
\sum_{\omega \in\pi_i(T_n)}\|\phi_\omega'\|_\infty^\delta  
\geq N_i(T_n) e^{-  \delta T_n} 
\geq e^{\e T_n}.
$$
Hence, there exists $1 \leq j_n \leq k_n$ such that 
$$ 
\sum_{\omega \in E^{j_n}_A}\|\phi_\omega'\|_\infty^\delta 
\geq \frac{1}{k_n}e^{\e T_n}.
$$
In particular, $\lim_{n\to\infty}j_n=+\infty$. Recalling that each strongly regular system is regular and invoking \eqref{1_2016_08_19}, we finally get
$$
\begin{aligned}
0
&=\P(\d)
=\lim_{n\to +\infty}\frac{1}{j_n}\log\sum_{\omega \in E_A^{j_n}}\|\phi_\om'\|_\infty^\d
\geq  \varlimsup_{n\to +\infty}\frac{1}{j_n}\log
\left(\frac{e^{\e T_n}}{k_n}\right) \\
&\geq \varlimsup_{n\to +\infty}\frac{1}{k_n}\log
\left(\frac{e^{\e T_n}}{k_n}\right) 
=\varlimsup_{n\to +\infty}\frac{1}{k_n}(\e T_n-\log k_n) \\
&=\e\varlimsup_{n\to +\infty}\frac{T_n}{k_n} 
 = \e|\log\ka|> 0. 
\end{aligned}
$$
This contradiction shows that
\beq\label{1_2016_01_20}
\varlimsup_{T \to +\infty} \frac{1}{T} \log N_i(T)\le\d.
\eeq
For the lower bound recall that
$$
\chi_\d=-\int_{E_A^\infty}\log|\phi_{\om_1}'(\pi(\sg(\om)))|\,d\mu_\d>0
$$
is the Lyapunov exponent of the measure $\mu_\d$ with respect to the shift map $\sg:E_A^\infty\to E_A^\infty$. Since the system $\cS$ is strongly regular, it follows from Observations~\ref{o1_2016_01_20-plus} and \ref{o5_2016_01_20} that $\chi_\d$ is finite. It then further follows from Theorem~\ref{thm-conformal-invariant} (e) that $\h_{\mu_\d}$ is finite and 
$$
\frac{\h_{\mu_\d}}{\chi_\d}=\d.
$$
Recall that along with \eqref{1_2016_08_20} the Bounded Distortion Property, yields
\beq\label{2_2016_08_20}
0\le \lambda_i(\omega)+\log\|\phi_\omega'\|_\infty \le \log C
\eeq
for all $\omega \in E_A^*$ and some constant $C>1$. Using this and \eqref{515pre} we then get for every $\ep>0$ and all integers $n\ge 1$ large enough that
$$
\begin{aligned}
\big\{\om\in E_A^n:\lam_i(\om)&\le (\chi_{\mu_\d}+\ep)n\big\} = \\
&=  \lt\{\om\in E_A^n:\lam_i(\om)\le \lt(\frac{\h_{\mu_\d}}{\d}+\ep\rt)n\rt\} \\
&\spt\lt\{\om\in E_A^n:-\frac1\d\log\mu_\d([\om])\le \lt(\frac{\h_{\mu_\d}}{\d}+
    \ep\rt)n+\frac{\log C_\d}{\d}-\log C\rt\} \\
&\spt \lt\{\om\in E_A^n:-\frac1\d\log\mu_\d([\om])\le \lt(\frac{\h_{\mu_\d}}{\d}+
    2\ep \rt)n\rt\} \\
&=  \big\{\om\in E_A^n:\log\mu_\d([\om])\ge -\(\h_{\mu_\d}+2\ep\d\)n\big\}.
\end{aligned}
$$
Having this, it follows from Breiman-McMillan-Shannon Theorem that 
$$
\#\big\{\om\in E_A^n:\lam_i(\om)\le (\chi_{\mu_\d}+\ep)n\big\}
\ge \exp\((\h_{\mu_\d}-3\ep\d)n\)
$$
for all integers $n\ge 1$ large enough. Since we also  obviously have 
$$
\pi_i\((\chi_{\mu_\d}+\ep)n\)
\spt \{\om\in E_A^n:\lam_i(\om)\le (\chi_{\mu_\d}+\ep)n\big\},
$$
we therefore get for every $T>0$ large enough,
$$
\begin{aligned}
\log N_i(T)
&=   \log N_i\lt((\chi_{\mu_\d}+\ep)\frac{T}{(\chi_{\mu_\d}+\ep)}\rt)
\ge \log N_i\lt((\chi_{\mu_\d}+\ep)\lt[\frac{T}{(\chi_{\mu_\d}+\ep)}\rt]\rt) \\
&\ge (\h_{\mu_\d}-3\ep\d)\lt[\frac{T}{(\chi_{\mu_\d}+\ep)}\rt].
\end{aligned}
$$
Therefore,
$$
\varliminf_{T \to +\infty} \frac{1}{T} \log N_i(T)
\ge \frac{\h_{\mu_\d}-3\ep\d}{\chi_{\mu_\d}+\ep}.
$$
So, letting $\e\downto 0$ yields
$$
\varliminf_{T \to +\infty} \frac{1}{T} \log N_i(T)
\ge \frac{\h_{\mu_\d}}{\chi_{\mu_\d}}
=\d.
$$
Along with \eqref{1_2016_01_20} this completes the proof.
\end{proof}

\sp\fr In particular, this proposition gives one more characterization  of the value of $\delta$.
paper

\sp One of our main objectives in this monograph is to provide a wide ranging  substantial improvement of Proposition~\ref{delta}. This is the asymptotic formula below, formulated at level of conformal graph directed Markov systems, along with its further strengthenings, extensions, and generalizations, both  for conformal graph directed Markov systems and beyond. Our first main result is the following.

\begin{thm}[Asymptotic Formula]\label{asymp}
If $\cS$ is a strongly regular finitely irreducible D-generic conformal GDMS, then
$$
\lim_{T \to +\infty} \frac{N_\rho(T)}{e^{\d T}} 
= \frac{\psi_\d(\rho)}{\d\chi_{\mu_\d}}
$$
and
$$
\lim_{T \to +\infty} \frac{N_p(T)}{e^{\d T}} 
= \frac{1}{\d\chi_{\mu_\d}}.
$$
\end{thm}

\fr The proof of this theorem will be completed as a special case of Theorem~\ref{dynA} (which is proved in Section~\ref{Ikehara}). 

\begin{rem}
\sp \fr If the generic D-genericity hypothesis fails, 
then we may still have an asymptotic formulae, but of a different type, e.g., 
there exists $N_i(T)\sim C \exp (\delta a[(\log T)/a ])$ as $T \to +\infty$.   
This is illustrated by the example in Remark \ref{generic} with $a= \log 2$.
\end{rem}

\sp\fr 
In preparation for the proof of Theorem~\ref{asymp} we now introduce a version of the main tool that will be used in the sequel. The standard strategy, stemming from number theoretical considerations of distributions of prime numbers, in such results is to use an appropriate complex function defined in terms of all of 
the weights $\lambda_\rho(\om)$ and then to apply a Tauberian theorem to convert properties of the function into the required asymptotic formula of $N_\rho(T)$, i.e. the first formula of Theorem~\ref{asymp}. The asymptotic formula for $N_p(T)$, i.e. the second formula of Theorem~\ref{asymp} will be directly derived from the former, i.e. that of $N_\rho(T)$. The basic complex function in the symbolic context is the following.

\bdfn
Given $s\in \mathbb C$ we define the
\emph{Poincar\'e (formal) series}  by:
$$
\eta_\rho(s) 
:= \sum_{\om\in E_\rho^*}e^{-s\lambda_\rho(\om)}
 = \sum_{n=1}^\infty \sum_{\om\in E_\rho^n} e^{-s\lambda_\rho(\om)}.
$$  
\edfn

\fr In fact we will need a localized version 
of this function, which will be introduced and analyzed  in Section~\ref{Local_Poincare}.

\medskip
\fr 
For the present,  we observe that since 
$$
      \sum_{\om\in E_\rho^n} |e^{-s\lambda_\rho(\om)}|
=     \sum_{\om\in E_\rho^n} e^{-\re(s\lambda_\rho(\om))}
\comp \sum_{\om\in E_\rho^n}\|\phi_\om'\|_\infty^{\re s}
\le   \sum_{\om\in E_A^n}\|\phi_\om'\|_\infty^{\re s}
$$
and since 
$$
\lim_{n\to\infty}\frac1n\log \left(\sum_{\om\in E_\rho^n}\|\phi_\om'\|_\infty^{\re s}\right) =\P(\re s)<0
$$
whenever $\re s>\d_\cS$, we get the following preliminary result.

\bobs\label{o1_2015_12_21}
The Poincar\'e series 
$$
\eta_\rho(s) 
=\sum_{n=1}^\infty \sum_{\om\in E_\rho^n} e^{-s\lambda_\rho(\om)}  
$$
converges absolutely 
uniformly on each set $\{s\in\C: \re s>t\}$, for $t>\d_\cS$.
\eobs

\fr For notational convenience  to follow we introduce the following set
$$
\De_\cS^+:=\{s\in\C: \re s>\d_\cS\}.
$$
As have said, the series $\eta_\rho(s)$ will be our main tool to acquire the asymptotic formula for the cardinalities of the sets $\pi_\rho(T)$, i.e. of the numbers  $N_\rho(T)$. An appropriate knowledge of the behavior of the series $\eta_\rho(s)$ on the imaginary line $Re(s) = \delta_{\mathcal S}$ is required for this end.
Indeed, in fact one needs to know that the function $\eta_\rho(s)$ has a meromorphic extension to some open neighborhoods of $\overline{\De_\cS^+} = \{s\in\C: \re s\geq\d_\cS\}$ with the only pole at $s=\delta_\cS$, that this pole is simple and the corresponding residue is to be calculated. 
This extension of $\eta_\rho(s)$ functions will come from an understanding of the spectral properties of the  associated complex RPF operators.  

\

With very little additional work, we  can actually get slightly finer asymptotic results than those of Theorem~\ref{asymp}. These count words subject to their weights being less than $T$ and, additionally,  their images being located in some part of the limit set. 

\bdfn
Let  $\rho\in E_A^\infty$ and let $\tau\in E_A^*$. Fix any Borel set $B \subset X$. Having $T>0$ we define:
\begin{align*}
\pi_\rho(B, T) 
&:= \left\{\om\in E_\rho^*  \hbox{ : }\phi_{\om}(\pi_\cS(\rho)) \in B 
 \hbox{ and }  
 \lambda_\rho(\om) \leq T \right\} 
\\ \hbox{ and }  \\
\pi_p(B, T) 
&:= \left\{\om\in E_p^*  \hbox{ : } x_\om \in B  \hbox{ and }   \lambda_p(\om) \leq T \right\}.
\end{align*}
We also define
\begin{align*}
\pi_\rho(\tau, T) 
&:= \left\{\om\in E_\rho^*: \lambda_\rho(\tau\om) \leq T \right\} 
\  \  \hbox{ and } \ \
\pi_p(\tau, T) 
:= \left\{\om\in E_p^*:\lambda_p(\tau\om) \leq T \right\}.
\end{align*}
The corresponding cardinalities of these sets are denoted by:
$$
N_\rho(B, T):=\#\pi_\rho(B, T) \  \ {\rm and } \  \ N_p(B,T)=\#\pi_p(B, T),
$$
and 
$$
N_\rho(\tau, T):=\#\pi_\rho(\tau, T) \  \ {\rm and } \  \ N_p(\tau, T)=\#\pi_p(\tau, T),\
$$
i.e. the first pair  count the number of words $\om\in E_i^*$ for which the 
weight $\lambda_i(\om)$ doesn't exceed $T$ and, additionally, the image $\phi_\om(\pi_\cS(\rho))$ is in $B$ if $i=\rho$, or the fixed point $x_\om$, 
of $\phi_\om$, is in $B$ if $i=p$, while the second pair  count the number of words $\om\in E_i^*$ for which the weight $\lambda_i(\tau\om)$ doesn't exceed $T$ 
(for $i=p, \rho$) and an initial block of $\om$ coincides with $\tau$.
\edfn

\fr The following are refinements of the asymptotic results presented in Theorem~\ref{asymp}, whose  proof will be completed in Section~\ref{Ikehara}.

\begin{thm}[Asymptotic Equidistribution Formula for Multipliers I]\label{dynA}
Suppose that $\cS$ is a strongly regular finitely irreducible D-generic conformal GDMS.  Fix $\rho\in E_A^\infty$. If $\tau\in E_A^*$ then,
\beq\label{2_2016_01_30}
\lim_{T \to +\infty} \frac{N_\rho(\tau,T)}{e^{\d T}} 
= \frac{\psi_\d(\rho)}{\d\chi_{\mu_\d}}m_\d([\tau]),
\eeq
and 
\beq\label{2_2016_01_30B}
\lim_{T \to +\infty} \frac{N_p(\tau,T)}{e^{\d T}} 
= \frac{1}{\d\chi_{\mu_\d}}\mu_\d([\tau]).
\eeq
\end{thm}

\sp
\begin{thm}[Asymptotic Equidistribution Formula for Multipliers II]\label{dyn}
Suppose that $\cS$ is a strongly regular finitely irreducible D-generic conformal GDMS.  Fix $\rho\in E_A^\infty$. If $B \subset X$ is a Borel set such that $\^m_\d(\bd B)=0$ (equivalently $\^\mu_\d(\bd B)=0$) then,
\beq\label{3_2016_01_30}
\lim_{T \to +\infty} \frac{N_\rho(B,T)}{e^{\d T}} 
= \frac{\psi_\d(\rho)}{\d\chi_{\mu_\d}}\^m_\d(B)
\eeq
and 
\beq\label{3_2016_01_30E}
\lim_{T \to +\infty} \frac{N_p(B,T)}{e^{\d T}} 
= \frac{1}{\d\chi_{\mu_\d}}\^\mu_\d(B).
\eeq
\end{thm}
\bigskip

\fr After establishing  the results of the next section (\ref{Local_Poincare}), we will first prove in Section~\ref{Ikehara}  formula \eqref{2_2016_01_30}. Then, in the same section,  we will deduce from it formula \eqref{2_2016_01_30B}. Finally, still within Section~\ref{Ikehara}, we will deduce Theorem~\ref{dyn} as a consequence of Theorem~\ref{dynA}. The asymptotic estimates for $N_\rho(B, T)$ given in this theorem, will turn out to have wider applications than the basic asymptotic results in  Theorem~\ref{asymp}.  This will be apparent,  particularly in Section~\ref{contracting_diameters} and Section~\ref{diam_Parabolic} where we 
apply these results to deduce asymptotics of the diameters of circles. 

\begin{rem}\label{Rem_Lalley}
Theorem~\ref{dynA} is formulated for a countable state symbolic system. In fact it could be formulated and proved with no real additional  difficulty for ergodic sums of all summable H\"older continuous potentials rather than merely the functions $\lam_\rho(\om)$. In the particular case of a finite state symbolic system this would recover the corresponding results of Lalley \cite{lalley}.
\end{rem}

\section{Complex Localized Poincar\'e Series $\eta_\rho$}\label{Local_Poincare}

In order to prove the asymptotic statements of Theorem \ref{dynA} we want to 
consider a  localized Poincar\'e series, which in turn generalises the Poincar\'e series introduced in the previous section. Again we denote by $\rho\in E_A^\infty$ our reference point and set $\xi:=\pi_\cS(\rho)\in J_\cS$.

\bdfn
Given $s \in \mathbb C$ we define the following {\it localized (formal) Poincar\'e series}.
Fixing $\tau\in E_A^*$ and denoting $q:=|\tau|$, we formally write
$$
\eta_\rho(\tau,s)
:= \sum_{{\om\in E_\rho^*\atop A_{\tau_q\om_1}=1}}e^{-s\lambda_\rho(\tau\om)}.
$$
\edfn

\fr We formally expand the series $\eta_\rho(\tau,s)$ as follows.
$$
\begin{aligned}
\eta_\rho(\tau, s)
:&= \sum_{{\om\in E_\rho^*\atop A_{\tau_q\om_1}=1}}e^{-s\lambda_\rho(\tau\om)}
 = \sum_{{\om\in E_\rho^*\atop A_{\tau_q\om_1}=1}}|\phi_{\tau\om}'|^s(\pi(\rho)) 
 = \sum_{{\om\in E_\rho^*\atop A_{\tau_q\om_1}=1}} |\phi_\tau'|^s(\pi(\om\rho))
  |\phi_{\om}'|^s(\pi(\rho)) \\
&= \sum_{n=1}^\infty \sum_{{\om\in E_\rho^n\atop A_{\tau_q\om_1}=1}}
   |\phi_\tau'|^s\circ\pi(\om\rho)|\phi_{\om}'|^s(\pi(\rho)) \\
&= \sum_{n=1}^\infty \pf_s^n\(|\phi_\tau'|^s\circ\pi\)(\rho).
\end{aligned}
$$
Defining the operator $\pf_{s,\tau}^{(n)}$ from $\H_\a(A)$ to $\H_a(A)$ by 
$$
\H_\a(A)\ni g\mapsto  \pf_{s,\tau}^{(n)}g 
:=\pf_s^n\(g\cdot(|\phi_\tau'|^s\circ\pi)\)\in\H_\a(A),
$$
we then formally write
$$
\eta_\rho(\tau, s) =  \sum_{n=1}^\infty  \pf_{s,\tau}^{(n)}\1(\rho).
$$
The same argument as that leading to Observation~\ref{o1_2015_12_21} leads to  the following corresponding result.

\bobs\label{o2_2015_12_22}
For every $\tau\in E_\rho^*$ the localized Poincar\'e series $\eta_\rho(\tau,s)$ converges absolutely uniformly on each set 
$$
\{s\in\C: \re s>t\}(\sbt \De_\cS^+),
$$
$t>\d_\cS$, thus defining a holomorphic function on $\De_\cS^+$.
\eobs

\fr Our main result about localized Poincar\'e series, 
which is crucial to us for  obtaining the asymptotic behavior of $N_\rho(\tau,T)$, is the following.

\bthm\label{t2nh18}
Assume that the finitely irreducible strongly regular conformal GDMS $\cS$ is D-generic. If $\tau\in E_A^*$ then 

\sp
\begin{itemize}

\item [(a)] the function $\De_\cS^+\ni s\longmapsto\eta_\rho(\tau,s)\in\C$ has a meromorphic extension to some neighborhood of the vertical line $\re(s)=\d_\cS$,

\sp\item[(b)] this extension has a single  pole $s=\delta_{\cS}$, and 

\sp\item[(c)] the pole $s= \delta = \delta_{\cS}$ is simple and its residue is equal to $\frac{\psi_\d(\rho)}{\chi_{\mu_\d}}m_\d([\tau])$.
\end{itemize}
\ethm

\bpf
By Observation~\ref{o2_2015_12_22} and by the Identity Theorem for meromorphic functions, in order to prove the theorem it suffices to do the following.

\sp
\begin{enumerate}
\item Show that for every $s_0=\d_\cS+it_0\in\Ga_\cS^+$ with $t_0\ne 0$ The function $\eta_\rho(\tau,\cdot)$ has a holomorphic extension to some open neighborhood of $s_0$ in $\C$. 

\sp\item Show that the function $\eta_\rho(\tau,\cdot)$ has a meromorphic extension to some open neighborhood of $\d_\cS$ in $\C$ with a simple pole at $\d_\cS$.

\sp\item Calculate the residue of this extension at the point $s=\d_\cS$ to show that it is equal to $\frac{\psi_\d(\rho)}{\chi_{\mu_\d}}m_\d([\tau])$.
\end{enumerate}

\sp\fr We first deal with item (1). Let $\La\sbt\C$ be the set of all eigenvalues of the operator $\pf_{s_0}:\H_\a(A)\to\H_\a(A)$ whose moduli are equal to $1$. By Proposition~\ref{p1nh15} this set is finite, and, by Lemma~\ref{l1nh18.1}, it consists of only simple eigenvalues. Write
$$
\La=\{\lam_j\}_{j=1}^q,
$$
where $q:=\#\La$. Then, invoking Observation~\ref{o1_2016_01_20},  Observation~\ref{o1_2016_01_21}, and Proposition~\ref{p1nh15} (along with the fact that $\P(\d_\cS)=0$), we see that the Kato--Rellich Perturbation Theorem applies and it produces holomorphic functions 
$$
U\ni s\mapsto\lam_j(s)\in\C, \  \  j=1,2,\ld,q
$$ 
defined on some sufficiently small neighborhood $U\sbt \Ga_\cS^+$ of $s_0$ with the following properties for all $j=1,2,\ld,q$:

\sp \begin{itemize}
\item $\lam_j(s_0)=\lam_j$,

\sp\item $\lam_j(s)$ is a simple isolated eigenvalue of the operator $\pf_s:\H_\a(A)\to\H_\a(A)$
\end{itemize}

\sp\fr Invoking Proposition~\ref{p1nh15} for  the third time, we can further write, perhaps with a smaller neighborhood $U$ of $s_0$, that
$$
\pf_s=\sum_{j=1}^q\lam_j(s)P_{s,j}+\De_s,
$$
where 

\sp \begin{itemize}
\item $P_{s,j}:\H_\a(A)\to\H_\a(A)$ are projections  onto respective $1$-dimensional spaces $\Ker\(\lam_j(s)I-\pf_s\)$,

\sp\item all functions $U\ni s\mapsto \De_s, P_{s,j}$, \ $j=1,2,\ld,q$, are holomorphic,

\sp\item $r(\De_s)\le e^{-\a/2}$ for every $s\in U$, and 

\sp\item $P_{s,i}P_{s,j}=0$ whenever $i\ne j$ and $\De_sP_{s,j}=P_{s,j}\De_s=0$ for all $s\in U$.
\end{itemize}

\sp\fr In consequence
\beq\label{1nh19.1}
\pf_s^n=\sum_{j=1}^q\lam_j^n(s)P_{s,j}+\De_s^n
\eeq
for all integers $n\ge 0$. Shrinking $U$ again if necessary, we will have that
$$
||\De_s^n||_\a\le Ce^{-\frac{\a}3n}
$$
for all integers $n\ge 0$ and some constant $C\in(0,+\infty)$ independent of $n$. Since the system $\cS$ is D-generic, it follows from Proposition~\ref{p1nh12}  that $\lam_j(s)\ne 1$ for all $s\in U$ and all $j=1,2,\ld,q$. Denoting by $S_\infty(s)$ the holomorphic function 
$$
U\ni s\longmapsto \De_\infty(s):=  \sum_{n=1}^\infty \De_s^n(|\phi_\tau'|^s\circ \pi)(\rho)
$$ 
and summing equation \eqref{1nh19.1} over all $n\ge 1$, we obtain
$$
\eta_\rho(\tau,s)
=\sum_{n=1}^\infty\pf_s^n\(|\phi_\tau'|^s\circ\pi\)(\rho)
=\sum_{j=1}^q\lam_j(s)(1-\lam_j(s))^{-1}P_{s,j}\(|\phi_\tau'|^s\circ\pi\)(\rho)+
       \De_\infty(s)
$$
for all $s\in U\cap \{s\in\C:\re(s)>\d_\cS\}$. But (remembering that $\lam_j(s)\ne 1$) since, all the terms of the right-hand side of this equation are holomorphic functions from $U$ to $\C$, the formula
$$
U\ni s\mapsto \sum_{j=1}^q\lam_j(s)(1-\lam_j(s))^{-1}P_{s,j}\(|\phi_\tau'|^s\circ\pi\)+ \De_\infty(s)\in\C
$$
provides the required holomorphic extension of the function $\eta_\rho(\tau,s)$ to a neighborhood of $s_0$. 

\sp Now we shall deal will items (2) and (3). It follows from Theorem~\ref{thm-conformal-invariant} (h) and (i), and the Kato--Rellich Perturbation Theorem that
\beq\label{1nh20}
\pf_s^n=\lam_s^nQ_s+S_s^n, \  \  n\ge0,
\eeq
for all $s\in U\sbt\Ga_\cS^+$, a sufficiently small neighborhood of $\d$, where

\sp
\begin{enumerate}
\item [(4)] $\lam_s$ is a simple isolated eigenvalues of $\pf_s:\H_\a(A)\to\H_\a(A)$ and the function $U\ni s\mapsto \lam_s\in\C$ is holomorphic,

\sp\item [(5)] $Q_s:\H_\a(A)\to\H_\a(A)$ is a projector onto the $1$-dimensional eigenspace of $\lam_s$, and the map $U\ni s\mapsto Q_s\in L(\H_\a(A))$ is holomorphic,

\sp\item [(6)] $\exists_{\ka\in (0,1)}\, \exists_{C>0}\, \forall_{s\in U}\, \forall_{n\ge 0}$
$$
\|S_s^n\|_\a\le C\ka^n,
$$
and the map $U\ni s\mapsto S_s\in L(\H_\a(A))$ is holomorphic, and 

\sp\item [(7)] All three operators $\pf_s,\, Q_s$, and $S_s$ mutually commute and  $Q_sS_s=0$.
\end{enumerate}

\sp\fr Let us write
$$
H_{\tau,s}:=Q_s\(|\phi_\tau'|^s\circ\pi\).
$$
It follows from (5) that the function $U\ni s\mapsto H_{\tau,s}\in \H_\a(A)$ is holomorphic, whence the function valued map  $U\ni s\mapsto H_s(\rho)\in\C$ is holomorphic too. It follows from (6) that the series 
$$
S_\infty(s):=\sum_{n=1}^\infty S_s^n
$$ 
converges absolutely uniformly to a holomorphic function, whence the function $U\ni s\mapsto \Sg_\infty(s)\in \H_\a(A)$ is holomorphic too. Since, by Theorem~\ref{t3_2016_01_12}, the function $s\mapsto\lam_s$ is not constant on any neighborhood of $\d$, it follows from (4) that shrinking $U$ if necessary, we will have that
$$
\lam_s\ne 1
$$
for all $s\in U\sms \{\d\}$. It follows from Theorem~\ref{t3_2016_01_12}, the definition of $\d$, and Proposition~\ref{p1nh15} (1) that
$$
|\lam_s|<1
$$
for all $s\in U\cap\{s\in\C:\re(s)>\d_\cS\}$. It therefore follows from \eqref{1nh20} that 
$$
\eta_\rho(s)=\lam_s(1-\lam_s)^{-1}H_{\tau,s}(\rho)+S_\infty(s)
$$
for all $s\in U\cap \{s\in\C:\re(s)>\d_\cS\}$, and consequently, the map
\beq\label{1nh21}
U\ni s\mapsto \lam_s(1-\lam_s)^{-1}H_{\tau,s}(\rho)+S_\infty(s)
\eeq
is a meromorphic extension of $\eta_\rho(\tau,\cdot)$ to $U$. We keep the same symbol $\eta_\rho(\tau,s)$ for this extension. Now, using Theorem~\ref{t1_2016_01_29}, we get
$$
\begin{aligned}
\lim_{s\downto\d}\frac{s-\d}{1-\lam_s}
&=-\lt(\lim_{s\downto\d}\frac{\lam_s-1}{s-\d}\rt)^{-1}
=-\lt(\lim_{s\downto\d}\frac{\lam_s-\lam_\d}{s-\d}\rt)^{-1}
=-\(\lam_\d'\)^{-1} \\
&=-\lt(\frac{d}{ds}\Big|_{s=\d}e^{\P(s)}\rt)^{-1}
=-\(\P'(\d)e^{\P(\d)}\)^{-1}
=-(\P'(\d))^{-1} \\
&=\frac1{\chi_{\mu_\d}}.
\end{aligned}
$$
Since $\lam_\d=1$ and 
$$
H_{\d,\tau}(\rho)
=Q_\d\(|\phi_\tau'|^\d\circ\pi\)(\rho)
=\lt(\int_{E_A^\infty}|\phi_\tau'|^\d\circ\pi\,dm_\d\rt)\psi_\d(\rho)
=\psi_\d(\rho)m_\d([\tau]),
$$
we therefore conclude that
$$
\res_\d\(\eta_\rho(\tau,\cdot)\)=\frac{\psi_\d(\rho)}{\chi_{\mu_\d}}m_\d([\tau]).
$$
The proof is thus complete.
\epf

\fr We can take $\tau$ to be the neutral (empty) word and deduce the corresponding results for the original Poincar\'e series
\begin{cor}\label{Poincare-cor}
Assume that the finitely irreducible strongly regular conformal GDMS $\cS$ is D-generic. Then 

\sp
\begin{itemize}

\item [(a)] the function $\eta_\rho(s)$ has a meromorphic extension to some neighborhood of the vertical line $\re(s)=\d_\cS$,

\sp\item[(b)] this extension has a single  pole $s=\delta_{\cS}$, and 

\sp\item[(c)] the pole $s=\delta - \delta_{\cS}$ is simple and its residue is equal to    $\frac{\psi_\d(\rho)}{\chi_{\mu_\d}}m_\d([\tau])$.
\end{itemize}
\end{cor}

\section{Asymptotic Results for Multipliers; Concluding of Proofs}\label{Ikehara}

We are now in position to complete the proof of Theorem \ref{dynA} and then, as its consequence, of Theorem \ref{dyn}. We aim to apply the Ikehara-Wiener Tauberian Theorem~\cite{wiener}, which is a familiar ingredient in the classical analytic
 proof of the Prime Number Theorem in Number Theory.

\bthm[Ikehara-Wiener Tauberian Theorem, \cite{wiener}]\label{tTauberian}
Let  $M$ and $\th$ be positive real numbers. Assume that $\alpha: [M,+\infty) \to (0,+\infty)$ is monotone increasing and continuous from the left, and also that there exists a (real) number $D > 0$ such that the function
$$
s\longmapsto\int_M^{+\infty} x^{-s} d\alpha(x) - \frac{D}{s-\th}\in\C
$$
is analytic in a neighborhood of $\re(s) \geq \th$. Then 
$$
\lim_{x\to+\infty}\frac{\alpha(x)}{x^\th}=\frac{D}{\th}.
$$
\ethm

\fr We can now apply this general result in the present setting to prove the asymptotic equidistribution results. We begin with the proof of formula \eqref{2_2016_01_30} in Theorem~\ref{dynA}.

\begin{proof}[Proof of formula \eqref{2_2016_01_30} in Theorem~\ref{dynA}] 
Let $\tau\in E_A^*$ be an arbitrary. We define the function $M_\rho(\tau,\cdot):[1,+\infty)\to\N_0$ by the formula
$$
M_\rho(\tau,T):=N_\rho(\tau,\log T)=\{\tau\om\in E_\rho^*:|\vp_{\tau\om}'(\xi)|^{-1}\le T\}.
$$
We then have for every $s>\d$ that
$$
\eta_\rho(\tau,s)=\int_1^\infty T^{-s} dM_\rho(\tau,T).
$$
Now Theorem~\ref{t2nh18} tells us that Theorem~\ref{tTauberian} applies with the function $\a$ being equal to $M_\rho(\tau,\cdot)$ and with $\th:=\d_\cS$, to give 
$$
\lim_{T\to+\infty}\frac{M_\rho(\tau,T)}{T^\d}
=\frac{\psi_\d(\rho)}{\d\chi_{\mu_\d}}m_\d([\tau]).
$$
Consequently
\beq\label{1_2016_01_30}
\lim_{T\to+\infty}\frac{N_\rho(\tau,T)}{e^{\d T}}
=\lim_{T\to+\infty}\frac{M_\rho(\tau,e^T)}{e^{\d T}}
=\frac{\psi_\d(\rho)}{\d\chi_{\mu_\d}}m_\d([\tau]).
\eeq
This means that \eqref{2_2016_01_30} is proved.
\end{proof}

Now we move onto the proof of \eqref{2_2016_01_30B}. However the first step to do this is of quite general character and will be also used in Section~\ref{contracting_diameters}. We therefore present it as a separate independent procedure. Fix an integer $q\ge 0$. Let $H\sbt E_A^\infty$ be a set representable as a (disjoint) union of cylinders of length $q$. Let 
$$
\cR_{q,H,\rho}(T):=\big\{\om\in \pi_\rho(T):|\om|>q \  {\rm and } \  
\om|_{|\om|-q+1}^{|\om|}\in H\big\}
$$
and the corresponding counting numbers
$$
R_{q,H,\rho}(T):=\#\cR_{q,H,\rho}(T).
$$
We shall prove the following. 

\begin{lem}\label{l1_2016_11_07} 
If $q\ge 0$ is an integer and $H\sbt E_A^\infty$ is a (disjoint) union of cylinders of length $q$, then the limit below exists and 
\beq\label{1da11.3}
\lim_{T\to\infty}\frac{R_{q,H,\rho}(T)}{e^{\d T}}
\le K^{2\d}(\d\chi_{\mu_\d})^{-1}m_\d(H).
\eeq
\end{lem}

\bpf 
As in the proof of formula \eqref{2_2016_01_30} in Theorem~\ref{dynA}, the Poincar\'e series corresponding to the counting scheme $\#\cR_{q,H,\rho}(T)$ is the function $\hat\eta_{H,\rho}(s)$, where for any $\g\in E_A^\infty$,
$$
\begin{aligned}
\hat\eta_{H,\g}(s)
&:=\sum_{\om \in E_\g^*,\,  |\om|\ge q+1 \atop \om|_{|\om|-q+1}^{|\om|}\in H}   |\phi_\om'(\pi_\cS(\g))|^s
 =\sum_{n=q+1}^\infty \sum_{\om \in E_\g^n \atop \om|_{n-q+1}^n\in H}   |\phi_\om'(\pi_\cS(\g))|^s \\
&=\sum_{n=q+1}^\infty \sum_{\om \in E_\g^n \atop \sg^{n-q}(\om\g)\in H}   |\phi_\om'(\pi_\cS(\g))|^s 
=\sum_{n=q+1}^\infty \sum_{\om \in E_\g^n \atop \om\g\in \sg^{-(n-q)}(H}   |\phi_\om'(\pi_\cS(\g))|^s \\
&=\sum_{n=q+1}^\infty \sum_{\om \in E_\g^n}\1_H\circ\sg^{n-q}(\om\g)\cdot|\phi_\om'(\pi_\cS(\g))|^s \\
&=\sum_{n=q+1}^\infty\pf_s^n\(\1_H\circ\sg^{n-q}\)(\g)
 =\sum_{n=q+1}^\infty\pf_s^q\(\pf_s^{n-q}\(\1_H\circ\sg^{n-q}\)\) 
      (\g) \\
&=\sum_{n=q+1}^\infty\pf_s^q\(\1_{[F_{q,\rho}^c]}\pf_s^{n-q}\1\)(\g)
 =\pf_s^q\lt(\1_H\sum_{n=q+1}^\infty\pf_s^{n-q}\1\rt)(\g).
\end{aligned}
$$
Now, the same reasoning as in the proof of Theorem~\ref{t2nh18} shows that the function 
$$
s\longmapsto \eta_q(s):=\sum_{n=q+1}^\infty\pf_s^{n-q}\1(\g)
$$
has a meromorphic extension, denoted by the same symbol $\eta_q(s)$, to some neighborhood, call it $G$, of the vertical line $\re(s)=\d_\cS$ with only pole at $s=\d_\cS$. This is again a simple pole with residue equal to $\chi_{\mu_\d}\psi_\d(\g)$. Since the operators $\pf_s^q$ are locally uniformly bounded at all points of $G$, the function
$$
s\longmapsto\pf_s^q\lt(\1_H\sum_{n=q+1}^\infty\pf_s^{n-q}\1\rt)(\g)
$$
has holomorphic extension, which we will still 
call $\hat\eta_{H,\g}(s)$, to $G\sms\{\d\}$. In addition
$$
\begin{aligned}
\lim_{s\to\d}(\delta-s)\hat\eta_{H,\g}(s)
&=\pf_\d^q\lt(\1_H\lim_{s\to\d}(\delta-s)\eta_q(s)\rt)(g)
=\pf_\d^q\lt(\1_H\chi_{\mu_\d}^{-1}\psi_\d\rt)(\g) \\
&=\chi_{\mu_\d}^{-1}\pf_\d^q\lt(\1_H\psi_\d\rt)(\g)
 \le \chi_{\mu_\d}^{-1}\|\psi_\d\|_\infty\pf_\d^q\lt(\1_H\rt)(\g) \\
&\le K^\d\chi_{\mu_\d}^{-1}\pf_\d^q\lt(\1_H\rt)(\g) \\
&\le K^{2\d}\chi_{\mu_\d}^{-1}m_\d(H).
\end{aligned}
$$
Therefore, we can apply the Ikehara-Wiener Tauberian Theorem (Theorem~\ref{tTauberian}) in exactly the same way as in the proof of \eqref{2_2016_01_30}, to conclude that
$$
\lim_{T\to\infty}\frac{R_{q,H,\rho}(T)}{e^{\d T}}
=\frac{\res_\d(\hat\eta_{q,H,\rho})}{\d}
\le K^{2\d}(\d\chi_{\mu_\d})^{-1}m_\d(H).
$$
The proof is complete.
\epf

\bpf[Proof of formula \eqref{2_2016_01_30B} in Theorem~\ref{dynA}] For every $\g\in E_A^*$ fix exactly one $\g^+\in E_A^\infty$ such that
$$
\g\g^+\in E_A^\infty.
$$
Observe that for every integer $q\ge 1$, every $\g\in E_A^q$, and every $\om\in E_A^*$ such that $\g\om\in E_p^*$, we have
\beq\label{1pp3}
K_q^{-1}|\phi_{\g\om}'(\pi(\g\g^+))|
\le |\phi_{\g\om}'(x_{\g\om})|
\le K_q|\phi_{\g\om}'(\pi(\g\g^+))|.
\eeq
It then follows from \eqref{1pp3} that 
\beq\label{2pp3}
\pi_p(\g,T)\sbt\pi_{\g\g^+}(\g,T+\log K_q)
\eeq
and 
\beq\label{3pp3}
\pi_{\g\g^+}(\g,T)\sbt \pi_p(\g,T+\log K_q).
\eeq
Let
$$
k:=|\tau|.
$$
Using \eqref{3pp3} and applying formula \eqref{2_2016_01_30} of Theorem~\ref{dyn}, we obtain that
$$
\begin{aligned}
\varliminf_{T\to\infty}\frac{N_p(\tau,T)}{e^{\d T}}
&\ge \varliminf_{T\to\infty}\sum_{\g\in E_A^q\atop \tau\g\in E_A^{q+k}}
\frac{N_{\tau\g(\tau\g)^+}\(\tau\g,T-\log K_{q+k}\)}{\exp\(\d(T-\log K_{q+k})\)}
 K_{q+k}^{-\d} \\
&\ge K_{q+k}^{-\d}\sum_{\g\in E_A^q\atop \tau\g\in E_A^{q+k}}
\varliminf_{T\to\infty}
\frac{N_{\tau\g(\tau\g)^+}\(\tau\g,T-\log K_{q+k}\)}{\exp\(\d(T-\log K_{q+k})\)}\\ 
&=K_{q+k}^{-\d}\frac{1}{\d\chi_\d}\sum_{\g\in E_A^q\atop \tau\g\in E_A^{q+k}}
  \psi_\d(\tau\g(\tau\g)^+)m_\d([\tau\g]) \\
&\ge K_{q+k}^{-2\d}\frac{1}{\d\chi_\d}\sum_{\g\in E_A^q\atop \tau\g\in E_A^{q+k}} \mu_\d([\tau\g]) \\
&=K_{q+k}^{-2\d}\frac{1}{\d\chi_\d}\mu_\d([\tau]).
\end{aligned}
$$
Therefore, taking the limit with $q\to\infty$, we obtain
\beq\label{1pp4}
\varliminf_{T\to\infty}\frac{N_p(\tau,T)}{e^{\d T}}
\ge \frac{1}{\d\chi_\d}\mu_\d([\tau]).
\eeq
Passing to the proof of the upper bound of the  limit supremum, we split $E_A^q$,
in a way that will be specified later,
into two disjoint sets $F_q$ and its complement $F_q^c:=E_A^q\sms F_q$ (each of which naturally consists of words of length $q$) with $F_q$ being finite. In particular,
$$
E_A^q=F_q\cup F_q^c.
$$
So far we have not imposed any additional hypotheses on the sets $F_q$ and $F_q^c$. This  will be done later in the course of the proof. We set 
$$
\cR_{q,\rho}(T):=\cR_{q,F_q^c}(T)
$$
and
$$
R_{q,\rho}(T):=\#\cR_{q,\rho}(T), 
$$
and note that because of \eqref{2pp3}, we have 
$$
\begin{aligned}
\pi_p(\tau,T)
&=\bu_{\g\in F_q\atop\g\tau\in E_A^*}
\pi_{\tau\g(\tau\g)^+}\(\tau\g,T+\log K_{q+k}\)\cup 
\bu_{\g\in F_q^c\atop\g\tau\in E_A^*}
\pi_{\tau\g(\tau\g)^+}\(\tau\g,T+\log K_{q+k}\)\\
&\sbt \bu_{\g\in F_q\atop\g\tau\in E_A^*}
\pi_{\tau\g(\tau\g)^+}\(\tau\g,T+\log K_{q+k}\)\cup 
\bu_{\g\in F_q^c}\pi_{\tau\tau^+}\(\g,T+\log K_{q+k}+\log K\)\\
&= \bu_{\g\in F_q\atop\g\tau\in E_A^*}
\pi_{\tau\g(\tau\g)^+}\(\tau\g,T+\log K_{q+k}\)\cup 
\cR_{q,\tau\tau^+}(T+\log K_{q+k}+\log K).
\end{aligned}
$$
Therefore, using finiteness of the set $F_q$, Theorem~\ref{dyn}, and \eqref{1da11.3}, we further obtain
$$
\begin{aligned}
\varlimsup_{T\to\infty}\frac{N_p(\tau,T)}{e^{\d T}}
&\le \sum_{\g\in F_q\atop\g\tau\in E_A^*}
\frac{N_{\tau\g(\tau\g)^+}\(\tau\g,T+\log K_{q+k}\)}{\exp\(\d(T+\log K_{q+k})\)}K_{q+k}^\d + \varlimsup_{T\to\infty}
\frac{R_{q,\tau\tau^+}(T+\log K_{q+k}+\log K)}{e^{\d T}}\\
&\le K_{q+k}^\d\frac{1}{\d\chi_\d}\sum_{\g\in F_q\atop\g\tau\in E_A^*}
 \psi_\d(\tau\g(\tau\g)^+)m_\d([\tau\g])+K^{3\d}K_{q+k}^\d\frac{1}{\d\chi_\d}m_\d\([F_q^c]\) \\
&\le K_{q+k}^{2\d}\frac{1}{\d\chi_\d}\sum_{\g\in E_A^q\atop \tau\g\in E_A^{q+k}} \mu_\d([\tau\g])+K^{3\d}K_{q+k}^\d\frac{1}{\d\chi_\d}m_\d\([F_q^c]\) \\
&\le K_{q+k}^{2\d}\frac{1}{\d\chi_\d}\mu_\d([\tau])+ K^{3\d}K_{q+k}^\d\frac{1}{\d\chi_\d}m_\d\([F_q^c]\).
\end{aligned}
$$
Hence, taking finite sets $F_{q,\rho}$ with $m_\d\([F_{q,\rho}]\)$  converging to one, so that $m_\d\([F_{q,\rho}^c]\)$ converges to zero, we obtain
$$
\varlimsup_{T\to\infty}\frac{N_p(\tau,T)}{e^{\d T}}
\le K_{q+k}^{2\d}\frac{1}{\d\chi_\d}\mu_\d([\tau]).
$$
Therefore, taking the limit with $q\to\infty$, we obtain
$$
\varlimsup_{T\to\infty}\frac{N_p(\tau,T)}{e^{\d T}}
\le \frac{1}{\d\chi_\d}([\tau]).
$$
Along with \eqref{1pp4} this yields
\beq\label{1_2017_03_10}
\lim_{T\to\infty}\frac{N_p(\tau,T)}{e^{\d T}}
=\frac{1}{\d\chi_\d}\mu_\d([\tau]).
\eeq
The proof of formula \eqref{2_2016_01_30B} in Theorem~\ref{dynA} is thus complete. This simultaneously finishes the proof of all of Theorem~\ref{dynA}
\epf

\bpf[Proof of Theorem~\ref{dyn}]

\sp The same proof, as a consequence of Theorem~\ref{dynA} goes through  for $i=\rho$ and $i=p$. We therefore denote
$$
C_i:= 
\begin{cases}
\frac{1}{\d\chi_\d}\psi_\d(\rho)
\ \  \  &{\rm if } \  \  \  i=\rho \\
\frac{1}{\d\chi_\d}
\  \  \  &{\rm if } \  \  \  i=p,
\end{cases}
$$

$$
\nu_i:=
\begin{cases}
m_\d
\ \  \  &{\rm if } \  \  \  i=\rho \\
\mu_d
\ \  \  &{\rm if } \  \  \  i=p
\end{cases}
\  \  \  {\rm and} \  \  \
\^\nu_i:=
\begin{cases}
\^m_\d
\  \  \  &{\rm if } \  \  \  i=\rho \\
\^\mu_\d
\  \  \  &{\rm if } \  \  \  i=p.
\end{cases}
$$
We shall first prove both formulae \eqref{3_2016_01_30} and \eqref{3_2016_01_30E} for all sets $B$ that are open. To emphasize this, let  us denote an arbitrary open subset of $X$ by $V$. We assume that $\^\nu_i(\bd V)=0$. Then  for every $s\in(0,1)$ there exists a finite set $\Ga_s(V)$ consisting of mutually incomparable elements of $E_A^*$ such that
$$
\bu_{\tau\in\Ga_s(V)}\phi_\tau\(X_{t(\tau)}\)\sbt V 
\  \  \  \and  \  \  \
\nu_i\lt(\bu_{\tau\in\Ga_s(V)}[\tau]\rt) 
=\^\nu_i\lt(\bu_{\tau\in\Ga_s(V)}\phi_\tau\(X_{t(\tau)}\)\rt)
\ge s \^\nu_i(V)
$$
where the ``$=$'' sign in this formula is due to \eqref{4_2016_12_14}. So, for both $i=\rho, p$, using \eqref{1_2016_01_30}, we get that
$$
\begin{aligned}
\varliminf_{T\to+\infty}\frac{N_i(V,T)}{e^{\d T}}
&\ge \sum_{\tau\in\Ga_\ka(V)}\varliminf_{T\to+\infty}\frac{N_i(\tau,T)}{e^{\d T}}
 = \sum_{\tau\in\Ga_\ka(V)}C_i\nu_i([\tau])\\
&=C_i\nu_i\lt(\bu_{\tau\in\Ga_\ka(V)}[\tau]\rt) \\
&\ge sC_i\^\nu_i(V).
\end{aligned}
$$
Letting $s\upto 1$, we thus obtain
\beq\label{1nh23}
\varliminf_{T\to+\infty}\frac{N_i(V,T)}{e^{\d T}}
\ge C_i\^\nu_i(V).
\eeq
Therefore, we also have
\beq\label{4nh23}
\varliminf_{T\to+\infty}\frac{N_i(\ov V^c,T)}{e^{\d T}}
\ge C_i\^\nu_i(\ov V^c).
\eeq
But since $\nu_i(\bd V)=0$, we have $\nu_i(V)+\nu_i(\ov V^c)=1$, whence 
\beq\label{2nh23}
\varliminf_{T\to+\infty}\frac{N_i(\ov V^c,T)}{e^{\d T}}
\ge C_i(1-\^\nu_i(V)).
\eeq
Therefore, using (\ref{1_2016_01_30}) and \eqref{1_2017_03_10}, both with $\tau$ replaced by $E_A^{\mathbb N}$, we get
\beq\label{3nh23}
\begin{aligned}
C_i 
&=\lim_{T\to+\infty}\frac{N_i(T)}{e^{\d T}}
\ge \varlimsup_{T\to+\infty}\frac{N_i(V,T)+N_i(\ov V^c,T)}{e^{\d T}} \\
&\ge \varlimsup_{T\to+\infty}\frac{N_i(V,T)}{e^{\d T}}  +
    \varliminf_{T\to+\infty}\frac{N_i(\ov V^c,T)}{e^{\d T}} \\
&\ge \varlimsup_{T\to+\infty}\frac{N_i(V,T)} {e^{\d T}} +
   C_i\(1- \^\nu_i(V)\).
   \end{aligned}
\eeq
Thus,
$$
\varlimsup_{T\to+\infty}\frac{N_i(V,T)}{e^{\d T}}
\le C_i\^\nu_i(V).
$$
Along with \eqref{1nh23} this implies
\beq\label{1nh24}
\lim_{T\to+\infty}\frac{N_i(V,T)}{e^{\d T}}
= C_i\^\nu_i(V).
\eeq
Finally, let $B$ be an arbitrary Borel subset of $X$ such that $\^\nu_i(\bd B)=0$. Then $\ov B=B\cup\bd B$ and
$$
\^\nu_i(\ov B)=\^\nu_i(B).
$$
Since the measure 
$\nu_i$
is outer regular, given $\e>0$ there exists an open set $G\sbt X$ such that $B\sbt G$ and 
\beq\label{2nh24}
\^\nu_i(G)\le \^\nu_i(B)+\e.
\eeq
Now, for every $x\in \ov B$ there exists an open set $V_x\sbt G$, in fact an open ball centered at $x$, such that $x\in V_x$ and 
$$
\^\nu_i(\bd V_x)=0.
$$
In particular, $\{V_x\}_{x\in\ov B}$ is a open cover of $\ov B$. Since $\ov B$ is compact, there thus exists a finite set $F\sbt \ov B$ such that
$$
\ov B\sbt V:=\bu_{x\in F}V_x\sbt G.
$$
Since $F$ is finite, $\bd V\sbt \bu_{x\in F}\bd V_x$, whence $\nu_i(\bd V)=0$. Therefore, \eqref{1nh24} applies to $V$ to give
$$
\begin{aligned}
\varlimsup_{T\to+\infty}\frac{N_i(B,T)}{e^{\d T}}
&\le \varlimsup_{T\to+\infty}\frac{N_i(\ov B,T)}{e^{\d T}}
\le \lim_{T\to+\infty}\frac{N_i(V,T)}{e^{\d T}}
=C_i\^\nu_i(V) \\
&\le C_i\^\nu_i(G) \\
&\le C_i(\^\nu_i(B)+\e).
\end{aligned}
$$
Letting $\e\downto 0$, we therefore get
\beq\label{1nh25}
\varlimsup_{T\to+\infty}\frac{N_i(B,T)}{e^{\d T}}
\le C_i\^\nu_i(B).
\eeq
Now, we can finish the argument in the same way as in the case of open sets. Since $\bd B^c=\bd B$, we have $m_\d(\bd B^c)=0$. In particular, \eqref{1nh25} also yields
$$
\varlimsup_{T\to+\infty}\frac{N_i(B^c,T)}{e^{\d T}}
\le C_i\^\nu_i(B^c)
=C_i(1-\^\nu_i(B)).
$$
Therefore, using Theorem~\ref{asymp} we can write
$$
\begin{aligned}
C_i
&=\lim_{T\to+\infty}\frac{N_i(T)}{e^{\d T}}
 =  \lim_{T\to+\infty}\frac{N_i(B,T)+N_i(B^c,T)}{e^{\d T}} \\
&\le \varliminf_{T\to+\infty}\frac{N_i(B,T)}{e^{\d T}}  +
    \varlimsup_{T\to+\infty}\frac{N_i(B^c,T)}{e^{\d T}} \\
&\le \varliminf_{T\to+\infty}\frac{N_i(B,T)}{e^{\d T}} +
   C_i\(1-\^\nu_i(B)\). 
\end{aligned}
$$
Thus,
$$
\varliminf_{T\to+\infty}\frac{N_i(B,T)}{e^{\d T}}
\ge C_i\^\nu_i(B).
$$
Along with \eqref{1nh25} this gives
$$
\lim_{T\to+\infty}\frac{N_i(B,T)}{e^{\d T}}
=C_i\^\nu_i(B),
$$
and the proof of the theorem is complete.
\end{proof}

\section{Asymptotic Results for Diameters}\label{contracting_diameters} 

In this section we obtain asymptotic counting properties corresponding to the functions 
$$
-\log\diam\(\phi_\om(X_{t(\om)}\), \  \  \om\in E_A^*.
$$
These are relatively simple consequences of Theorem~\ref{dyn},  but not quite so simple as one would expect. The subtle difficulty is due to the fact that the functions $N_i(B,T)$, $i=\rho, p$ are very sensitive to additive changes. In fact it follows from Theorem~\ref{dyn} that for every $u>0$,
$$
\lim_{T\to\infty}\frac{N_i(B,T+u)}{N_i(B,T)}=e^{\d u}>0.
$$
In fact we will do something more general, namely for every $v\in V$ we fix an arbitrary set $Y_v\sbt X_v$, having at least two points, and we look at asymptotic counting properties corresponding to the functions
$$
-\log\diam\(\phi_\om(Y_{t(\om)}\)), \  \  \om\in E_A^*.
$$
Such a generalization is interesting in its own right, but will turn out to be particularly useful when dealing with asymptotic counting properties for diameters in the context of parabolic GDMSs, see Section~\ref{diam_Parabolic}.

So, again $\cS$ is a finitely irreducible conformal GDMS, we fix $\rho\in E_A^\infty$ and put $\xi=\pi_\cS(\rho)$. We denote 
$$
\De(\om):=-\log\diam\(\phi_\om(Y_{t(\om)}\)), \  \  \om\in E_A^*,
$$
with the natural convention that for $\om=\ep$, being the empty (neutral) word:
$$
\De(\ep)=-\log\diam(Y_{i(\rho)}),
$$
and further, for any $T>0$,
$$
\cD^{\rho}_{Y_{i(\rho)}}(B,T)
:=\cD^{\rho}(B,T)
:=\{\om\in E_\rho^*:\De(\om)\le T \ {\rm and } \  \phi_\om(\xi)\in B\},
$$
$$
D^{\rho}_{Y_{i(\rho)}}(B,T)
=D^{\rho}_{Y_{i(\rho)}}(B,T)
:=\#\cD^{\rho}_{Y_{i(\rho)}}(B,T).
$$
The main result of this section is the following.

\bthm\label{t1da7}
Suppose that $\cS$ is a strongly regular finitely irreducible conformal $D$-generic GDMS. Fix $\rho\in E_A^\infty$ and $Y\sbt X_{i(\rho)}$ having at least two points. If $B \subset X$ is a Borel set such that $\^m_\d(\bd B)=0$ (equivalently $\^\mu_\d(\bd B)=0$) then,
\beq\label{3_2016_01_30C}
\lim_{T \to +\infty} \frac{D^{\rho}_Y(B,T)}{e^{\d T}} 
=C_\rho(Y)\^m_\d(B),
\eeq
where $C_\rho(Y)\in (0,+\infty)$ is a constant depending only on the system $\cS$, the word $\rho$ (but see Remark~\ref{r1_2017_03_20}), and the set $Y$. In addition
\beq\label{1da7.1}
K^{-2\d}(\d\chi_\d)^{-1}\diam^\d(Y)
\le C_\rho(Y)
\le K^{2\d}(\d\chi_\d)^{-1}\diam^\d(Y).
\eeq
\ethm

\sp We first shall prove the following auxiliary result. It is trivial in the case of finite alphabet $E$ but requires an argument in the infinite case.

\blem\label{l1da8}
With the hypotheses of Theorem~\ref{t1da7}, for every integer $q\ge 1$ let 
$$
\pi_i^{(q)}(B,T):=\pi_i(B,T)\cap E_A^q, \  \  i=\rho, p,
$$
and 
$$
N_i^{(q)}(B,T):=\#\pi_i^{(q)}(B,T).
$$
Then
$$
\lim_{T\to\infty}\frac{N_i^{(q)}(B,T)}{e^{\d T}}=0.
$$
\elem

\bpf
Since $N_i^{(q)}(B,T)\le N_i^{(q)}(T):=N_i^{(q)}(X,T)$, it suffices to prove that
$$
\lim_{T\to\infty}\frac{N_i^{(q)}(T)}{e^{\d T}}=0.
$$
By considering the iterate $\cS^q$ of $\cS$ it is further evident that it suffices to show that 
$$
\lim_{T\to\infty}\frac{N_i^{(1)}(T)}{e^{\d T}}=0.
$$
To see this consider the Poincar\'e series
$$
s\longmapsto \eta_\rho^{(1)}(s):=\pf_s\1(\rho),
$$
notice that it is holomorphic throughout $\{s\in\C:\re(s)>\g_\cS\} \spt \ov{\De_\cS^+}$, and conclude the proof with the help of the Ikehara-Wiener Tauberian Theorem (Theorem~\ref{tTauberian}), in the same way as in the proof of Theorem~\ref{asymp}.
\epf

\fr Denote also 
$$
\cD^{(\rho,q)}(B,T)
:=\cD^\rho(B,T)\cap E_\rho^q
=\cD^\rho(B,T)\cap E_\rho^A.
$$
By (BDP)
$$
N_i^{(\rho,q)}(B,T-\log K)
\le D^{(\rho,q)}(B,T)
\le N_i^{(\rho,q)}(B,T+\log K).
$$
Therefore, as an immediate consequence of Lemma~\ref{l1da8}, we get the following.

\bcor\label{c1da8.1}
With the hypotheses of Theorem~\ref{t1da7}, for every integer $q\ge 1$, we have
$$
\lim_{T\to\infty}\frac{D^{(\rho,q)}(B,T)}{e^{\d T}}=0.
$$
\ecor

Now we can turn  to the actual proof of Theorem~\ref{t1da7}.

\bpf[Proof of Theorem~\ref{t1da7}]
Fix an integer $q\ge 0$ and define:
$$
K_q:=\sup\lt\{\frac{|\phi_\om'(y)|}{|\phi_\om'(x)|}:\tau\in E_A^q,\, x, y\in \Conv\(\phi_\tau(X_{t(\tau)})),\, \om\in E_\tau^*\rt\}\ge 1,
$$
where $\Conv(F)$ is the convex hull of a set $F\sbt \R^d$. In particular $K_0=K$, the distortion constant of the system $\cS$. (BDP) yields \beq\label{4da9}
\lim_{q\to\infty}K_q=1.
\eeq
(BDP) again, along with the Mean Value Theorem, imply that for all $\tau\in E_\rho^*$ and all $\om\in E_\tau^*$, we have that
$$
\diam\(\phi_{\om\tau}(Y)\)
=\diam\(\phi_\om\(\phi_\tau(Y)\)\)
\le K_q|\phi_\om'(\phi_\tau(\xi))|\diam(\phi_\tau(Y))
$$
and 
$$
\diam\(\phi_{\om\tau}(Y)\)
\ge K_q^{-1}|\phi_\om'(\phi_\tau(\xi))|\diam(\phi_\tau(Y)).
$$
Equivalently
\beq\label{1da9}
\lam_{\tau\rho}(\om)+\De(\tau)-\log K_q
\le \De(\om\tau)
\le \lam_{\tau\rho}(\om)+\De(\tau)+\log K_q.
\eeq
Denote
$$
\cD_\tau^\rho(B,T):=\{\om\in E_\tau^*:\om\tau\in \cD^\rho(B,T)\}
$$
and 
$$
D_\tau^\rho(B,T):=\#\cD_\tau^\rho(B,T).
$$
Formula \eqref{1da9} then gives
\beq\label{2da9}
\pi_{\tau\rho}(B,T)\sbt \cD_\tau^\rho(B,T+\De(\tau)+\log K_q)
\eeq
and 
\beq\label{3da9}
\cD_\tau^\rho(B,T)\sbt \pi_{\tau\rho}(B,T-\De(\tau)+\log K_q).
\eeq
The former equation is equivalent to
$$
\cD_\tau^\rho(B,T)\spt \pi_{\tau\rho}(B,T-\De(\tau)-\log K_q).
$$
This formula and \eqref{3da9} yield
\beq\label{2da10}
N_{\tau\rho}\(B,T-\De(\tau)-\log K_q\)
\le {D}_\tau^\rho(B,T)
\le N_{\tau\rho}\(B,T-\De(\tau)+\log K_q\).
\eeq
since
\beq\label{1da10}
\cD^\rho(B,T)
=\bu_{\tau\in E_\rho^q}\cD_\tau^\rho(B,T)\tau\cup\bu_{j=0}^q\cD^{(\rho,j)}(B,T)
\eeq
and since all the terms in this union are mutually disjoint, formula \eqref{1da10} yields
$$
D^\rho(B,T)\ge \sum_{\tau\in E_\rho^q}D_\tau^\rho(B,T).
$$
By inserting it into  formula \eqref{2da10}, we get
$$
D^\rho(B,T)\ge \sum_{\tau\in E_\rho^q}N_{\tau\rho}\(B,T-\De(\tau)-\log K_q\).
$$
Therefore,
$$
\begin{aligned}
\frac{D^\rho(B,T)}{e^\d T}
&\ge  \sum_{\tau\in E_\rho^q}\frac{N_{\tau\rho}\(B,T-\De(\tau)-\log K_q\)}
      {\exp\(\d(T-\De(\tau)-\log K_q)\)}\cdot \frac{\exp\(\d(T-\De(\tau)-\log
        K_q)\)}{e^\d T} \\
&=  \sum_{\tau\in E_\rho^q}\frac{N_{\tau\rho}\(B,T-\De(\tau)-\log K_q\)}
      {\exp\(\d(T-\De(\tau)-\log K_q)\)}K_q^{-\d}e^{-\d\De(\tau)}   \\
&=  K_q^{-\d}\sum_{\tau\in E_\rho^q}\frac{N_{\tau\rho}\(B,T-\De(\tau)-\log K_q\)}
      {\exp\(\d(T-\De(\tau)-\log K_q)\)}e^{-\d\De(\tau)} .           
\end{aligned}
$$
Hence, applying Theorem~\ref{dyn}, we get
\beq\label{1da11}
\begin{aligned}
\varliminf_{T\to\infty}\frac{D^\rho(B,T)}{e^\d T}
&\ge  K_q^{-\d}\sum_{\tau\in E_\rho^q}e^{-\d\De(\tau)}\varliminf_{T\to\infty}
      \frac{N_{\tau\rho}\(B,T-\De(\tau)-\log K_q\)}{\exp\(\d(T-\De(\tau)-\log K_q)\)}\\  
&\ge  K_q^{-\d}\sum_{\tau\in E_\rho^q}e^{-\d\De(\tau)}
      (\chi_\d\d)^{-1}\psi_\d(\tau\rho)m_\d(B) \\
&=(\chi_\d\d)^{-1}m_\d(B)K_q^{-\d}\sum_{\tau\in E_\rho^q}e^{-\d\De(\tau)}
     \psi_\d(\tau\rho).
\end{aligned}
\eeq
This is a good enough lower bound for us but getting a sufficiently good upper bound is more subtle. As in the proof of formula \eqref{2_2016_01_30B} in Theorem~\ref{dynA}, we split $E_A^q$, at the moment arbitrarily, into two disjoint sets $F_q$ and its complement $F_q^c:=E_A^q\sms F_q$ (each of which naturally consists of words of length $q$) with $F_q$ being finite. In particular,
$$
E_A^q=F_q\cup F_q^c.
$$
So far we do not require anything more from the sets $F_q$ and $F_q^c$. 
We will make specific choices later in the course of the proof. We are now primarily interested in the sets
$$
\cR_{q,\rho}(T)
:= \cR_{q,[F_q^c],\rho}(T)
=\big\{\om\in \pi_\rho(T):|\om|>q \  {\rm and } \  
\om|_{|\om|-q+1}^{|\om|}\in F_{q,\rho}^c\big\}
$$
and the corresponding counting numbers
$$
R_{q,\rho}(T):=\#\cR_{q,\rho}(T).
$$
We are interested in estimating from above, the upper limit
$$
\varlimsup_{T\to\infty}\frac{D^\rho(B,T)}{e^\d T}.
$$
First of all, Lemma~\ref{l1_2016_11_07} yields  
\beq\label{1da11.3-again}
\lim_{T\to\infty}\frac{R_{q,\rho}(T)}{e^{\d T}}
\le K^{2\d}\d^{-1}\chi_{\mu_\d}m_\d\([F_q^c]\).
\eeq
Denote now 
$$
\cR_{q,\rho}^*(T):=\big\{\om\in \cD^\rho(T):|\om|>q \  {\rm and } \  
\om|_{|\om|-q+1}^{|\om|}\in F_q^c\big\}
$$
and the corresponding counting numbers
$$
R_{q,\rho}^*(T):=\#\cR_{q,\rho}^*(T).
$$
It follows from \eqref{1da9}, applied with $\tau$ being empty (neutral) word, that 
$$ 
\cR_{q,\rho}^*(T)\sbt \cR_{q,\rho}\(T+\log\De(\ep)+\log K\).
$$
Along with \eqref{1da11.3} this yields
$$
\varlimsup_{T\to\infty}\frac{R_{q,\rho}^*(T)}{e^{\d T}}
\le K^{3\d}\d^{-1}\chi_{\mu_\d}m_\d\De(\ep)\([F_q^c]\).
$$
Now we write
$$
\bu_{\tau\in E_\rho^q}\cD_\tau^\rho(B,T)\tau
=\bu_{\tau\in F_q\cap E_\rho^q}\cD_\tau^\rho(B,T)\tau
\cup\cR_{q,\rho}^*(T).
$$
Together  with \eqref{1da10} and \eqref{2da10} this yields
$$
\begin{aligned}
D^\rho(B,T)
&\le\sum_{\tau\in F_{q,\rho}\cap E_\rho^q}D_\tau^\rho(B,T)\tau
     +R_{q,\rho}^*(T)+\sum_{j=0}^qD^{(\rho,j)}(B,T) \\
&\le \sum_{\tau\in F_{q,\rho}\cap E_\rho^q} 
     N_{\tau\rho}(B,T-\De(\tau)+\log K_q)+
     R_{q,\rho}^*(T)+\sum_{j=0}^qD^{(\rho,j)}(B,T).
\end{aligned}
$$
Hence, invoking also Corollary`\ref{c1da8.1} and finiteness of the set $F_{q,\rho}$, we get
\beq\label{2da11}
\begin{aligned}
\varlimsup_{T\to\infty}\frac{D^\rho(B,T)}{e^{\d T}}
&\le K_q^\d \sum_{\tau\in F_{q,\rho}\cap E_\rho^q}e^{-\De(\tau)}
  \varlimsup_{T\to\infty}\frac{N_{\tau\rho}(B,T-\De(\tau)+\log
   K_q)}{\exp\(\d(T-\De(\tau)+\log K_q)\)} 
   +\varlimsup_{T\to\infty}\frac{R_{q,\rho}^*(T)}{e^{\d T}}\\
&\le (\chi_\d\d)^{-1}m_\d(B)K_q^{\d}\sum_{\tau\in E_\rho^q}
    e^{-\d\De(\tau)}\psi_\d(\tau\rho)+
    K^{3\d}(\d\chi_\d)^{-1}\De(\ep)m_\d\([F_q]\).
\end{aligned}
\eeq
Hence, taking finite sets $F_{q,\rho}$ with $m_\d\([F_{q,\rho}]\)$  converging to one, with  $m_\d\([F_{q,\rho}^c]\)$ converging to zero, we obtain
\beq\label{1_2016_10_06}
\varlimsup_{T\to\infty}\frac{D^\rho(B,T)}{e^{\d T}}
\le K_q^{\d}(\chi_\d\d)^{-1}m_\d(B)\sum_{\tau\in E_\rho^q}
    e^{-\d\De(\tau)}\psi_\d(\tau\rho).
\eeq
Since
$$
\psi_\d(\rho)
=\pf_\d^q\psi_\d(\rho)
\lek \sum_{\tau\in E_\rho^q}e^{-\d\De(\tau)}\psi_\d(\tau\rho)
\lek \pf_\d^q\psi_\d(\rho)
=\psi_\d(\rho),
$$
we conclude from \eqref{1da11} and \eqref{1_2016_10_06} that both $\varliminf_{T\to\infty}\frac{D^\rho(B,T)}{e^{\d T}}$ and $\varlimsup_{T\to\infty}\frac{D^\rho(B,T)}{e^{\d T}}$ are finite and positive numbers. Furthermore, we conclude from these same two formulae  that for every $q\ge 1$,
$$
1
\le\frac{\varlimsup_{T\to\infty}\frac{D^\rho(B,T)}{e^{\d T}}}
{\varliminf_{T\to\infty}\frac{D^\rho(B,T)}{e^{\d T}}}
\le K_q^{2\d}.
$$
Formula \eqref{4da9} then yields that  the limit $\lim_{T\to\infty}\frac{D^\rho(B,T)}{e^{\d T}}$  exists and is finite and positive. Invoking \eqref{1da11} and \eqref{1_2016_10_06} again along with \eqref{4da9}, we thus deduce the limit
$$
\lim_{q\to\infty}\sum_{\tau\in E_\rho^q}
    e^{-\d\De(\tau)}\psi_\d(\tau\rho)
$$
also exists, is finite and positive. Denoting this limit 
 by $C_\cS'$, we thus conclude that
$$
\lim_{T\to\infty}\frac{D^\rho(B,T)}{e^{\d T}}
=\frac1{\d\chi_\d}C_\cS'm_\d(B),
$$
and so, in order to complete the proof of Theorem~\ref{t1da7}, we only need to estimate $C_\cS'$. Indeed,
$$
\begin{aligned}
\sum_{\tau\in E_\rho^q}e^{-\d\De(\tau)}\psi_\d(\tau\rho)
&=\sum_{\tau\in E_\rho^q}\diam^\d\(\phi_\tau(Y)\)
      \psi_\d(\tau\rho)
\le \sum_{\tau\in E_\rho^q}\|\phi_\tau'\|_\infty^\d\diam^\d(Y)
    \psi_\d(\tau\rho) \\
&\le K^\d\diam^\d(Y)\sum_{\tau\in E_\rho^q}| 
     \phi_\tau'(\pi_\cS(\rho))|^\d\psi_\d(\tau\rho) \\
&= K^\d\psi_\d(\rho)\diam^\d(Y) \\
&\le K^{2\d}\diam^\d(Y),
\end{aligned}
$$
and similarly,
$$
\sum_{\tau\in E_\rho^q}e^{-\d\De(\tau)}\psi_\d(\tau\rho)
\ge  K^{-2\d}\diam^\d(Y).
$$
The proof is complete.
\epf

We can now consider a slightly different approach to counting 
diameters.
Given a set $B\sbt X$, we define:
$$
\mathcal E_Y^\rho(B,T) := \{ \omega \in E_\rho^* \hbox{ : } \Delta(\omega) 
\leq T \ \hbox{ and } \  \phi_\omega (Y) \cap B \neq \emptyset\}
$$
and 
$$
E_Y^\rho(B,T) := \#\mathcal E_Y^\rho(B,T).
$$

\bthm\label{t1ma1}
Suppose that $\cS$ is a strongly regular finitely irreducible conformal $D$-generic GDMS. Fix $\rho\in E_A^\infty$ and $Y\sbt X_{i(\rho)}$ having at least two points and such that $\pi_\cS(\rho)\in Y$. If $B \subset X$ is a Borel set such that $\^m_\d(\bd B)=0$ (equivalently $\^\mu_\d(\bd B)=0$) then,
\beq\label{3_2016_01_30C}
\lim_{T \to +\infty} \frac{E^{\rho}_Y(B,T)}{e^{\d T}} 
=C_\rho(Y)\^m_\d(B),
\eeq
where $C_\rho(Y)\in (0,+\infty)$ is a constant, in fact the one produced in Theorem~\ref{t1da7}, depending only on the system $\cS$, the word $\rho$ (but see Remark~\ref{r1_2017_03_20}), and the set $Y$. In addition
\beq\label{1da7.1_D}
K^{-2\d}(\d\chi_\d)^{-1}\diam^\d(Y)
\le C_\rho(Y)
\le K^{2\d}(\d\chi_\d)^{-1}\diam^\d(Y).
\eeq
\ethm

\begin{proof}
Since $\pi_\cS(\rho)\in Y$ we have that 
$$
D_Y^\rho(B,T) \leq E_Y^\rho(B,T).
$$
It therefore follows from Theorem~\ref{t1da7} that
\beq\label{1ma1}
\liminf_{T \to +\infty} \frac{E_Y^\rho(B,T)}{e^{\delta T}} \geq C_\rho(Y)  \widetilde m_\d(B).
\eeq
Since $\mathcal E_Y^\rho (T) = \mathcal E_Y^\rho (X, T) = \mathcal D_Y^\rho (T) $, Theorem~\ref{t1da7}, also yields
\beq\label{2ma1}
\lim_{T \to +\infty} \frac{E_Y^\rho(T)}{e^{\delta T}} = C_S(Y).
\eeq
Now fix $(\epsilon_n)_{n=1}^\infty$, a sequence of positive numbers converging to zero  such that for all $n\geq 1$
$$
\widetilde m_\delta (\partial B(B, \epsilon_n))=0.
$$
Then $\widetilde m_\delta (\partial B^c(B, \epsilon_n))=0$ and $\phi_\om(Y)$ intersects at most one of the sets $B$ or $B^c(B; \epsilon_k) \cap B^c$ if $\Delta(\om) \geq \log(1/\epsilon_n)$.  Hence applying formula (\ref{1ma1})
to the sets $B^c(B, \epsilon_n) \cap B^c$ and using (\ref{2ma1}) we get for every $n \geq 1$ that 
$$
\begin{aligned}
C_\rho(Y) & \geq\limsup_{T \to +\infty} \frac{E^{\rho}_Y(B,T) + E_Y^{\rho}(B^c(B, \epsilon_n), T)}{e^{\delta T}} \cr
&\geq \limsup_{T \to +\infty} \frac{E^{\rho}_Y(B,T)}{e^{\delta T}}
+ 
\liminf_{T \to +\infty} \frac{E_Y^{\rho}(B^c(B, \epsilon_n), T)}{e^{\delta T}}\cr
&\geq \limsup_{T \to +\infty} \frac{E^{\rho}_Y(B,T)}{e^{\delta T}}
+ 
C_{\rho}(Y)\widetilde m_\delta (B^c(B, \epsilon_n)).\cr
\end{aligned}
$$
But 
$\lim_{n \to +\infty} \widetilde m_\delta (B^c(B, \epsilon_n)) = \widetilde m_\delta(B^c) = 1 - m_\delta(B)$,
(remembering that $\widetilde m_\delta(\partial B)= 0$), and therefore 
$$
C_\rho(Y) \geq \limsup_{T \to +\infty} \frac{E_Y^{\rho}(B,T)}{e^{\delta T}} + C_{\rho}(Y) (1- m_\delta(B)).
$$
Hence 
$$
\limsup_{T \to +\infty}   \frac{E_Y^\rho(B,T)}{e^{\delta T}} \leq C_\rho(Y)m_\delta(B).
$$
Along with (\ref{1ma1}) this finishes the proof of the first part of the theorem. The second part, i.e. \eqref{1da7.1_D}, is just formula \eqref{1da7.1}.
\end{proof}

\brem\label{r1_2017_03_20}
Since the left-hand side of \eqref{3_2016_01_30C} depends only on $\rho_1$, i.e. the first coordinate of $\rho$, we obtain that the constant $C_Y(\rho)$ of Theorem~\ref{t1ma1} and Theorem~\ref{t1da7}, also depends in fact only on $\rho_1$.
We could have provided a direct argument of  this already when proving Theorem~\ref{t1da7} and  this would not affect the proof of Theorem~\ref{t1ma1}. However,  our approach seems to be most economical.
\erem

\sp We say that a graph directed Markov system $\cS$ has the property (A) if for every vertex $v\in V$ there exists $a_v\in E$ such that
$$
i(a_v)=v
$$
and 
$$
A_{ea_v}=1
$$
whenever $t(e)=v$. As an immediate consequence of Theorem~\ref{t1da7},  Theorem~\ref{t1ma1} and Remark~\ref{r1_2017_03_20}, we get the following.

\bthm\label{t1da12.1}
Suppose that $\cS$ is a strongly regular finitely irreducible $D$-generic conformal GDMS with property (A). For any $v\in V$ let $Y_v\sbt X_v$ having at least two points  fixed. If $B \subset X$ is a Borel set such that $\^m_\d(\bd B)=0$ (equivalently $\^\mu_\d(\bd B)=0$) and $\rho\in E_A^\infty$ is with $\rho_1=a_v$, then,
\beq\label{3_2016_01_30N}
\lim_{T \to +\infty} \frac{D^{\rho}_Y(B,T)}{e^{\d T}} 
=\lim_{T \to +\infty} \frac{E^{\rho}_Y(B,T)}{e^{\d T}}
=C_v(Y_v)\^m_\d(B),
\eeq
where $C_v(Y_v)\in (0,+\infty)$
is a constant depending only on the vertex $v\in V$ and the set $Y_v$. In particular, this holds for $Y_v:=X_v$, $v\in V$.
\ethm

\fr Recall, see \cite{CTU} for example, that a GDMS $\cS$ is maximal if $A_{ab}=1$ whenever $t(a)=i(b)$. Since every iterated function system is maximal and finitely irreducible and  since each maximal GDMS has property (A), as an immediate consequence of Theorem~\ref{t1da12.1}, and Remark~\ref{r1_2017_03_20} (improved to claim that now $C_\rho(Y)$ depends only on $i(\rho_1)$ and $Y$) we get the following two corollaries.

\bcor\label{t1da12.1G}
Suppose that $\cS$ is a strongly regular finitely irreducible $D$-generic maximal conformal GDMS. For any $v\in V$ let $Y_v\sbt X_v$ having at least two points be fixed. If $B \subset X$ is a Borel set such that $\^m_\d(\bd B)=0$ (equivalently $\^\mu_\d(\bd B)=0$) and $\rho\in E_A^\infty$ is with $i(\rho_1)=v$, then,
\beq\label{3_2016_01_30P}
\lim_{T \to +\infty} \frac{D^{\rho}_Y(B,T)}{e^{\d T}} 
=\lim_{T \to +\infty} \frac{E^{\rho}_Y(B,T)}{e^{\d T}}
=C_v(Y_v)\^m_\d(B),
\eeq
where $C_v(Y_v)\in (0,+\infty)$
is a constant depending only on the vertex $v\in V$ and the set $Y_v$. In particular, this holds for $Y_v:=X_v$, $v\in V$.
\ecor

\bcor\label{c1da12.1H}
Suppose that $\cS$ is a strongly regular $D$-generic conformal IFS acting on a phase space $X$. Fix $Y\sbt X$ having at least two points. If $B \subset X$ is a Borel set such that $\^m_\d(\bd B)=0$ (equivalently $\^\mu_\d(\bd B)=0$) and $\rho\in E_A^\infty$, then,
\beq\label{3_2016_01_30Q}
\lim_{T \to +\infty} \frac{D^{\rho}_Y(B,T)}{e^{\d T}} 
=\lim_{T \to +\infty} \frac{E^{\rho}_Y(B,T)}{e^{\d T}}
=C(Y)\^m_\d(B),
\eeq
where $C(Y)\in (0,+\infty)$
is a constant depending only on the set $Y$. In particular, this holds for $Y:=X$.
\ecor

\

\part{{\Large Parabolic Conformal Graph Directed Markov Systems}}

\sp\section{Parabolic GDMS; Preliminaries}\label{section:parabolic}
We will  want to apply the previous  results (Theorem~\ref{dyn}) to prove counting theorems for a variety of examples. In particular, these  results can then  be applied to prove the geometric counting problems for Apollonian packings and related topics. In order to do this, that is in order to be in position to apply Theorem \ref{dyn}, we formulate these geometric counting problems in the  framework of conformal parabolic iterated function systems, and more generally of parabolic graph directed Markov systems. Therefore, we first prove appropriate counting results, i.e. Theorem~\ref{t2pc6}, for parabolic systems, which is both an analogue of Theorem~\ref{dyn} in this setting and its (Theorem~\ref{t2pc6}) quite involved, corollary.

In present section, following  \cite{MU_Parabolic_1} and \cite{MU_GDMS}, we describe  the suitable parabolic setting, canonically associated to it an ordinary (uniformly contracting) conformal graph directed Markov system (a kind of inducing), and we prove Theorem~\ref{t1pc3}, which is a somewhat surprising and  remarkable result about parabolic systems. 

In Section~\ref{Parabolic_Counting}, we obtain actual counting results for parabolic systems and in Section~
\ref{tangent Schottky}
we apply them in geometric contexts such as Apollonian packings and the like. In the whole of Section~
\ref{CPDS}
we apply our general theorems, i.e. Theorem~\ref{dyn} and Theorem~\ref{t2pc6}, to other counting problems naturally arising in the realm of Kleinian groups and one-dimensional systems.

\sp\fr As in Section~\ref{Attracting_GDMS_Prel} we assume that we  are given a directed multigraph $(V,E,i,t)$ ($E$ countable, $V$ finite), an incidence matrix $A:E\times E\to \{0,1\}$, and two functions $i,t:E\to V$ such that $A_{ab} = 1$ implies $t(b) = i(a)$. Also, we have nonempty compact metric spaces $\{X_v\}_{v\in V}$. Suppose further that we have  a collection of
conformal maps $\phi_e:X_{t(e)}\to X_{i(e)}$, $e\in E$, satisfying the following conditions (which are more general than in Section~\ref{Attracting_GDMS_Prel} in that we don't necessarily assume 
 that the maps are uniform contractions).

\sp\begin{itemize}
\item[(1)](Open Set Condition)
 $\phi_a(\Int(X))\cap \phi_b(\Int(X))=\es$ for all $a, b\in E$ with $a\ne b$.

\sp\item[(2)] $|\phi_e'(x)|<1$ everywhere except for finitely many
pairs $(e,x_e)$, $e\in E$, for which $x_i$ is the unique fixed point
of $\phi_e$ and $|\phi_e'(x_e)|
=1$. Such pairs and indices $i$ will be called parabolic and the set of
parabolic indices will be denoted by $\Om$. All other indices will be 
called hyperbolic. We assume that $A_{ee}=1$ for all $e\in\Om$.

\sp\item[(3)]  $\forall n\ge 1 \  \forall \om = (\om_1\om_2...\om_n)\in E_A^n$
if $\om_n$ is a hyperbolic
index or $\om_{n-1}\ne \om_n$, then $\phi_{\om}$ extends conformally to
an open connected set $W_{t(\om_n)}\sbt\R^d$ and maps $W_{t(\om_n)}$ into $W_{i(\om_n)}$.

\sp\item[(4)] If $e\in E$ is a parabolic index, then 
$$
\bi_{n\ge 0}\phi_{e^n}(X)
=\{x_e\}
$$ 
and the diameters of the sets $\phi_{e^n}(X)$ converge
to 0.

\sp\item[(5)] (Bounded Distortion Property) $\exists K\ge 1 \  \forall
n\ge 1
 \  \forall \om\in E_A^n  \  \forall x,y\in W_{t(\om_n)}$, if $\om_n$ 
is a hyperbolic index or $\om_{n-1}\ne \om_n$, then
$$
{|\phi_\om'(y)| \over |\phi_\om'(x)| } \le K.
$$
\item[(6)] $\exists \ka<1 \  \forall n\ge 1  \  \forall \om\in E_A^n$ if 
$\om_n$ is a hyperbolic index or $\om_{n-1}\ne \om_n$, then
$\|\phi_\om'\|\le \ka$.

\sp\item[(7)] (Cone Condition)  There exist $\a,l>0$ such that for
 every $x\in\bd X \sbt\R^d$ there exists an open cone 
$\Con(x,\a,l)\sbt \Int(X)$ with vertex $x$, central
angle of Lebesgue measure $\a$, and altitude $l$.

\sp\item[(8)] There exists a constant $L\ge 1$ such that
\[
\bigg|\frac{|\phi_e'(y)|}{|\phi_e'(x)|}-1 \bigg| \le L\|y-x\|^\alpha,
\]
for every $e\in E$ and every pair of points $x,y\in V$.
\end{itemize}

\sp\fr We call such a system of maps 
$$
\cS=\{\phi_e:e\in E\}
$$ 
a {\it  subparabolic conformal graph directed Markov system. }

Let us note that conditions (1), (3), (5)--(7) are modeled on similar  conditions which were used to examine hyperbolic conformal systems.

\bdfn\label{Def_Parabolic}
If $\Om\ne\es$, we call the system $\cS=\{\phi_i:i\in E\}$ {\it parabolic}. 
\edfn

As stated in (2) the elements of the set $E\sms \Om$ are called
hyperbolic.
We extend this name to all the words appearing in (5) and (6). It follows
from (3) that for every hyperbolic word $\om$,
$$
\phi_\om(W_{t(\om)})\sbt W_{t(\om)}.
$$
Note that our conditions ensure that $\f_e'(x) \neq 0$ 
for all $e\in E$ and all $x \in X_{t(i)}$. It was proved (though only for IFSs although the case of GDMSs can be treated completely similarly) in \cite{MU_Parabolic_1} (comp. \cite{MU_GDMS}) that
\beq\label{1_2016_03_15}
\lim_{n\to\infty}\sup_{\om\in E_A^n}\big\{\diam(\phi_\om(X_{t(\om)}))\big\}=0.
\eeq
As its immediate consequence, we record the following.

\bcor\lab{p1c2.3} 
The map $\pi=\pi_\cS:E_A^\infty\to X:=\du_{v\in V}X_v$, 
$$
\{\pi(\om)\}:=\bi_{n\ge 0}\phi_{\om|_n}(X),
$$
is well defined, i.e. this intersection is always a singleton, and the map $\pi$ is uniformly continuous.
\ecor

\fr As for hyperbolic (attracting) systems the limit set $J = J_\cS$ of the system $\cS = \{\f_e\}_{e\in e}$ is defined to be
$$
J_\cS:=\pi(E_A^\infty)
$$
and it enjoys the following self-reproducing property:
$$
J = \bu_{e\in E} \f_e(J).
$$
We now, still  following \cite{MU_Parabolic_1} and \cite{MU_GDMS}, want to associate to the parabolic system $\cS$ a canonical hyperbolic system $\cS^*$.  We will then be able to apply the ideas from the previous section 
to $\cS^*$. The set of edges is defined as follows:
$$
E_*:= \big\{i^nj: n\ge 1, \  i\in\Om, \ i\ne j\in E, \ A_{ij}=1\big\} \cup 
(E\sms \Om)\sbt E_A^*.
$$ 
We set
$$
V_*=t(E_*)\cup i(E_*)
$$
and keep the functions $t$ and $i$ on $E_*$ as the restrictions of $t$ and $i$ from $E_A^*$. The incidence matrix $A^*:E_*\times E_*\to\{0,1\}$ is defined in the natural (and the only reasonable) way by declaring that $A^*_{ab}=1$ if and only if $ab\in E_A^*$. Finally 
$$
\cS^*=\{\phi_e:X_{t(e)}\to X_{t(e)}|\, e\in E^*\}.
$$
It immediately follows from our assumptions (see \cite{MU_Parabolic_1} and \cite{MU_GDMS} for more details) that the following is true.

\bthm\lab{p1t5.2} 
The system $S^*$ is a hyperbolic (contracting) conformal GDMS and the limit sets $J_\cS$ and $J_{\cS^*}$ differ only by a countable set. If the system $\cS$ is finitely irreducible, then so is the system $\cS^*$.
\ethm

\fr 
The price we pay by replacing  the non-uniform ``contractions'' in $\cS$ with the
uniform contractions in $\cS^*$ is that even if the alphabet $E$ is finite, the alphabet $E^*$ of $\cS^*$ is always infinite. In particular, already at the first level (just the maps $\phi_\om$, $\om\in E^*$,), we get more scaling factors to deal with. In order to understand them, we will need the following quantitative result, whose complete proof can be found in \cite{SzUZ}. 

\bprop\lab{p1c5.13} 
Let $\cS$ be a conformal parabolic GDMS. Then there exists a constant $C\in(0,+\infty)$ and for every $i\in\Om$ there exists some constant
$p_i\in(0,+\infty) $ such that for all $n\ge 1$ and for all $z\in X_i:=
\bu_{j\in I\sms\{i\}}\phi_j(X)$,
$$
C^{-1}n^{-{p_i+1\over p_i}}\le |\phi_{i^n}'(z)|
\le Cn^{-{p_i+1\over p_i}}.
$$
Furthermore, if  $d=2$ then all constants $\b_i$ are integers $\ge 1$ and if $d\ge 3$ then all constants $\b_i$ are equal to $1$. 
\eprop

\sp\fr Let us also introduce the following auxiliary system:
$$
\cS^-:=\{\phi_e:e\in E\sms\Om\}.
$$
As an immediate consequence of Proposition~\ref{p1c5.13}  we get the following.

\bprop\label{p1_2016_03_16}
If $\cS$ is a finitely irreducible conformal parabolic GDMS, then
$$
\Ga_{\cS^*}=\Ga_{\cS^-}\sms\lt(-\infty,\frac{p_\cS}{p_\cS+1}\rt] 
\  \  \and  \  \ \,
\g_{\cS^*}=\max\lt\{\g_{\cS^*},\frac{p_\cS}{p_\cS+1}\rt\},
$$
where 
$$
p_\cS:=\max\{p_i:i\in\Om\}.
$$
In particular if the alphabet $E$ is finite, then
$$
\Ga_{\cS^*}=\lt(\frac{p_\cS}{p_\cS+1},+\infty\rt),
 \  \
\g_{\cS^*}=\frac{p_\cS}{p_\cS+1},
$$
and the system $\cS^*$ is hereditarily (co-finitely) regular.
\eprop

We set
$$
\d_\cS:=\d_\cS^*,
$$
$$
m_{\d_\cS}:=m_{\d_{\cS^*}}^*
\  \  \  {\rm and }\  \  \
\^m_{\d_\cS}:=\^m_{\d_{\cS^*}}^*.
$$
Given $\rho\in E_{A^*}^\infty$, let
$$
\Om_\rho:=\{a\in \Om:A_{a\rho_1}=1\}.%
$$
of course $\Om_\rho$, regarded   as a function 
of $\rho$,  depends only on $\rho_1$. We will need the following facts proved in \cite{MU_Parabolic_1}, comp. \cite{MU_GDMS}.

\bthm\label{t1_2017_02_18}
If $\cS$ is a finite alphabet irreducible conformal parabolic GDMS, then

\sp\begin{enumerate}
\item $\d_\cS=\HD(J_\cS)$,

\sp\item 
The measure $\^m_{\d_\cS}$ is $\d$--conformal for the original system $\cS$ in the sense that
$$
\^m_{\d_\cS}(\phi_\om(F))=\int_F|\phi_\om'|^{\d_\cS}\,d\^m_{\d_\cS}
$$
for every $\om\in E_A$ and every Borel set $F\sbt X_{t(\om)}$, and
$$
\^m_{\d_\cS}\(\phi_\a(X_{t(\a)})\cap \phi_\b(X_{t(\b)})\)=0
$$
whenever $\a,\b\in E_A^*$ and are incomparable. 

\sp\item There exists a,  unique up to multiplicative constant, $\sg$--finite shift--invariant measure $\mu_{\d_\cS}$ on $E_A^\infty$, absolutely continuous with respect to $m_{\d_\cS}$. The measure $\mu_{\d_\cS}$ is equivalent to $m_{\d_\cS}$ and 

\sp \begin{itemize} 
\item[(a)] The Radon--Nikodym derivative of $\mu_{\d_\cS}$ with respect to $m_\d$ is given by the following formula:
$$
\psi_{\d_\cS}(\rho):=\frac{d\mu_{\d_\cS}}{d m_{\d_\cS}}(\rho)=\psi_{\d_\cS}^*(\rho)+\sum_{a\in\Om_\rho}\sum_{k=1}^\infty
|\phi_{a^k}'(\pi(\rho))|^\d\psi_{\d_\cS}^*(a^k\rho).
$$
\item[(b)] The measure $\mu_{\d_\cS}$ (and $\^\mu_{\d_\cS}:=\mu_{\d_\cS}\circ\pi_\cS^{-1})$ is finite (we then always treat it as normalized so that it is a probability measure) if and only if 
$$
{\d_\cS}>\frac{2p_\cS}{p_\cS+1}.
$$
More precisely, the following conditions are equivalent:

\sp\begin{itemize}
\item[(b1)] $\delta_\cS > \frac{2p_a}{p_a+1}$,

\sp\item[(b2)] There exists an integer $l\geq 1$ such that $\mu_{\d_\cS}([a^l]) < +\infty$, and 

\sp\item[(b3)] For every integer $l \geq 1$, $\mu_{\d_\cS}([a^l]) < +\infty$.
\end{itemize}
\end{itemize}

\sp\item Furthermore, we have that
$$
\chi_{\d_\cS}:=-\int_{E_A^\infty}\log\big|\phi_{\om_1}'(\pi_\cS(\om)\big|d\mu_\d=\chi_{\d_\cS}^*\in (0,+\infty)
$$
and, as for attracting GDMSs, we call $\chi_{\d_\cS}$ the Lyapunov exponent of the system $\cS$ with respect to measure $\mu_{\d_\cS}$.
\end{enumerate}  
\ethm

\sp\fr For future use we denote
$$
\Om_\infty=\Om_\infty(\cS):=\lt\{a\in\Om:\frac{2p_a}{p_a+1}\ge \delta_\cS\rt\}.
$$

A crucial feature of the hyperbolic systems arising from  parabolic systems is that they are automatically $D$-hyperbolic. We have already seen that this is not necessarily true for hyperbolic systems.

\bthm\label{t1pc3}
If $\cS$ is a finitely irreducible conformal parabolic GDMS with finite alphabet, then $\cS^*$, the associated contracting (hyperbolic) GDMS, is D-generic.
\ethm

\bpf
Assume for a contradiction that $\cS^*$ is not D-generic. According to Proposition~\ref{p1nh13} this means that the additive group generated by the set 
$$
\big\{-\log|\phi_\om'(x_\om)|:\om\in E_{*A^*}^*\big\}\sbt\R
$$
is cyclic. Denote its generator by $M>0$. Fix $b\in \Om$ and then take $\a\in E_A^*$ such that $\a_1\ne b$ and $\a b^2\a_1\in E_A^*$. Note that then $\a b^2\a_1\in E_{*A^*}^*$ and moreover $\a b^n\a_1\in E_{*A^*}^*$ for all integers $n\ge 2$. For every integer $n\ge 2$ denote by $x_n\in J_{\cS^*}$ the only fixed point of the map $\phi_{\a b^n\a_1}:X_{t(\a_1)}\to X_{t(\a_1)}$. We know from the above that for every $n\ge 2$ there exists an integer $k_n\ge 1$ such that 
\beq\label{1pc3}
-\log\big|\phi_{\a b^n\a_1}'(x_n)\big|=Mk_n.
\eeq
By Proposition~\ref{p1c5.13} we have that
\beq\label{2pc3}
\big|\phi_{\a b^n\a_1}'(x_n)\big|
=|\phi_{\a_1}'(x_n)|\cdot\big|\phi_{b^n}'(\phi_{\a_1}(x_n))\big|\cdot
       \big|\phi_\a'(\phi_{b^n\a_1}(x_n))\big|
       =C_nn^{-\frac{p_b+1}{p_b}}
\eeq
with some $C_n\in (C^{-1},C)$, where $C$ is the constant coming from Proposition~\ref{p1c5.13}. Combining this with \eqref{1pc3} yields
$$
k_n=-\frac1M\log C_n+\frac{p_b+1}{Mp_b}\log n.
$$
On the other hand 
$$
\lim_{n\to\infty}x_n
=\lim_{n\to\infty}\phi_{\a b^n\a_1}(x_n)
=\phi_\a\(\lim_{n\to\infty}\phi_b^n(\phi_{\a_1}(x_n))\)
=\phi_\a(x_b)
$$
and
$$
\lim_{n\to\infty}\phi_{b^n\a_1}(x_n)=x_b.
$$
Keeping in mind that $\phi_b(x_b)=x_b$ and $|\phi_b'(x_b)|=1$ and using the Bounded Distortion Property, we therefore get
$$
\begin{aligned}
\lim_{n\to\infty}\frac{\big|\phi_{\a b^{n+1}\a_1}'(x_{n+1})\big|}
           {\big|\phi_{\a b^n\a_1}'(x_n)\big|}
&=\lim_{n\to\infty}\frac{\big|\phi_{\a b^{n+1}\a_1}'(\phi_\a(x_b))\big|}
          {\big|\phi_{\a b^n\a_1}'(\phi_\a(x_b))\big|}    \\
&=\lim_{n\to\infty}\frac{\big|\phi_\a'\(\phi_b^{n+1}(\phi_{\a_1\a}(x_b))\)\big|
   \cdot\big|\phi_b'\(\phi_b^n(\phi_{\a_1\a}(x_b))\)\big|}  
   {\big|\phi_\a'\(\phi_b^n(\phi_{\a_1\a}(x_b))\)\big|}\\
&=\lim_{n\to\infty}\frac{|\phi_\a'(x_b)|\cdot|\phi_b'(x_b)|}{|\phi_\a'(x_b)|}
=|\phi_b'(x_b)|
=1.
\end{aligned}
$$
Equivalently:
$$
\lim_{n\to\infty}\(-\log|\phi_{\a b^{n+1}\a_1}'(x_{n+1})|-\
   (-\log|\phi_{\a b^n\a_1}'(x_n)|\)
   =0.
$$
Using  \eqref{1pc3} this gives that 
$
\lim_{n\to\infty}(k_{n+1}-k_n)=0.
$
Since all $k_n$, $n\ge 1$, are integers, this implies that the sequence $(k_n)_{n=1}^\infty$ is eventually constant. However, it  follows from \eqref{1pc3} that $\lim_{n\to\infty}k_n=+\infty$, and the contradiction we obtain finishes the proof.
\epf

\brem\label{r2_2017_02_17}
We could generalize slightly the concepts of subparabolic and parabolic systems by requiring in item (2) of their definition that not merely some elements $\phi_e$, $e\in E$, have parabolic fixed points but some finitely many elements $\phi_\om$, $\om\in E_A^*$, have such points. In other words it would suffice to assume that some iterate of the system $\cS$ in the sense of Remark~\ref{r1_2017_04_01} is parabolic. Indeed, this would not really affect any considerations of this and any next section involving parabolic GDMSs, and such generalization will turn out to be needed in Subsection~\ref{1dimparabolicexamples} for the Farey map,  Subsections~\ref {tangent Schottky} and \ref{Apollonian Circle Packings} when we deal respectively with Schottky groups with tangencies and Apollonian circle packings. 
\erem

\section{Poincar\'e's Series for $\cS^*$, the Associated Countable Alphabet Attracting GDMS}\label{Parabolic_Counting_1}

In this section we again let  $\cS$ be a finitely irreducible conformal parabolic GDMS. Our goal is to describe the 
Poincar\'e series and the associated  asymptotic (equidistribution) results for the system $\cS$.  
This is achieved by
 means of the transfer operator associated to the associated hyperbolic system $\cS^*$.

We begin by formulating the required notation. Fix first $\rho\in E_{A^*}^\infty$ arbitrary. Denote $\xi:=\pi_{\cS^*}(\rho)$. Treating $\rho$ in an obvious way as an element of $E_A^\infty$, we can also write $\xi=\pi_\cS(\rho)$.
Fix next an arbitrary $\tau\in E_{*A^*}^*$. 

Let $\eta_i^*(\tau,s)$, $i=\rho, p$, be the corresponding Poincar\'e series for the contracting system $\cS^*$, and we continue to use 
$$
\eta_i(\tau,s), \  \  i=\rho, p, 
$$
to denote the Poincar\'e  series for the original (now parabolic) system $\cS$. 
This allows to deduce the 
 analytical properties of $\eta_i$ from those for the $\eta_i^*$, to which we can apply the results  already established in Proposition~\ref{t2nh18}. 
 
 We show that the Poincar\'e series $\eta_i(\tau,s)$ for the parabolic system $\mathcal S$ can be expressed in terms of the
 Poincar\'e series for $\eta_i^*(\tau,s)$ for the hyperbolic system $\mathcal S^*$. 
 In particular, we can deduce properties for $\eta_i^*(\tau,s)$ which are the analogue of those for $\eta_i(\tau,s)$, already established in Proposition 6.3.
We can formally  write
\beq\label{1pc5}
\begin{aligned}
\eta_\rho(\tau,s)
&=\sum_{\om\in E_\rho^*:\tau\om\in E_A^*}|\phi_{\tau\om}'(\pi(\rho))|^s \cr
&=\sum_{\om\in E_{*,\rho}^*\atop\tau\om\in E_{*A^*}^*}| \phi_{\tau\om}'(\pi(\rho))|^s+
\sum_{a\in\Om_\rho}\sum_{k=1}^\infty
\sum_{\om\in E_{*A^*}^*\atop\tau\om a\in E_A^*}|\phi_{\tau\om a^k}'(\pi(\rho))|^s\\
&=\sum_{\om\in E_{*\rho}^*\atop\tau\om\in E_{*A^*}^*}| \phi_{\tau\om}'(\pi(\rho))|^s+
\sum_{a\in\Om_\rho}\sum_{k=1}^\infty
\sum_{\om\in E_{*A^*}^*\atop\tau\om a\in E_A^*}|\phi_{\tau\om}'(\pi(a^k\rho))|^s
|\phi_{a^k}'(\pi(\rho))|^s\\
&=\sum_{\om\in E_{*\rho}^*\atop\tau\om\in E_{*A^*}^*}| \phi_{\tau\om}'(\pi(\rho))|^s+
\sum_{a\in\Om_\rho}\sum_{k=1}^\infty
|\phi_{a^k}'(\pi(\rho))|^s
\sum_{\om\in E_{*A^*}^*\atop\tau\om a\in E_A^*}|\phi_{\tau\om}'(\pi(a^k\rho))|^s \\
&=\eta_\rho^*(\tau,s)+
\sum_{a\in\Om_\rho}\sum_{k=1}^\infty
|\phi_{a^k}'(\pi(\rho))|^s 
\eta_{a^k\rho}^*(\tau,s).
\end{aligned}
\eeq
Since by  Theorem~\ref{t1pc3} we have that $\cS^*$ is $D$-generic 
 it follows from the proof of Theorem~\ref{t2nh18} that for every $s_0={\d_\cS}+it_0\in\Ga_\cS^+$ with $t_0\ne 0$ all functions $\eta_{a^k\rho}^*(\tau,\cdot)$ have holomorphic extensions on a common neighborhood, denoted by $U$, of $s_0\in \Ga_{\cS^*}^+$ of the form
$$
U\ni s\longmapsto\sum_{j=1}^q\lam_j^*(s)(1-\lam_j^*(s))^{-1}P_{s,j}^*(|\phi_\tau'|^s\circ\pi)(a^k\rho)+S_\infty^*(s)\in\C,
$$
where all the symbols ``$*$'' indicate that the appropriate objects pertain to the system $\cS^*$. Since
$$
|P_{s,j}^*(|\phi_\tau'|^s\circ\pi^*)(a^k\rho)|
\le\|P_{s,j}^*(|\phi_\tau'|^s\circ\pi^*)\|_\infty
\le\|P_{s,j}^*(|\phi_\tau'|^s\circ\pi^*)\|_\a<+\infty,
$$
it follows that all the functions $\eta_{a^k\rho}^*(\tau,\cdot)$ are uniformly bounded on $U$. Since also ${\d_\cS}>\frac{p_a}{p_a+1}$ and since
\beq\label{2pc5}
\big||\phi_{a^k}'(\pi(\rho))|^s\big|
\le |\phi_{a^k}'(\pi(\rho))|^{\d_\cS}
\lek (k+1)^{-\frac{p_a+1}{p_a}{\d_\cS}},
\eeq
we eventually conclude that the series in \eqref{1pc5} converges absolutely uniformly on $U$, thus representing a holomorphic function. We are therefore left to consider the case of $s_0={\d_\cS}$. By virtue of \eqref{1nh21} we then have for every $k\ge 0$ that
$$
\eta_{a^k\rho}^*(\tau,s)
=\lam_s^*(1-\lam_s^*)^{-1}H_{\tau,s}^*(a^k\rho)+
\Sg_\infty(s).
$$
Substituting this into \eqref{1pc5}, we therefore get
$$
\eta_\rho(\tau,s)
=\eta_\rho^*(\tau,s)
+\lam_s^*(1-\lam_s^*)^{-1}\sum_{a\in\Om_\rho}\sum_{k=1}^\infty
|\phi_{a^k}'(\pi(\rho))|^sH_{\tau,s}^*(a^k\rho)+
\sum_{a\in\Om_\rho}\sum_{k=1}^\infty
|\phi_{a^k}'(\pi(\rho))|^s\Sg_\infty(s),
$$
and by \eqref{2pc5} both series involved in the above formula converge absolutely uniformly on $U$. Looking up now at the calculations from the end of the proof of
Theorem~\ref{t2nh18} and invoking Theorem~\ref{t1_2017_02_18} (3) and (4), we conclude that the function $U\ni s\mapsto\eta_\rho(\tau,s)$ is meromorphic with a simple pole at $s={\d_\cS}$ whose residue is equal to 
$$
\begin{aligned}
\frac{\psi_{\d_\cS}^*(\rho)}{\chi_{\d_\cS}^*}m_{\d_\cS}^*([\tau])&+ \sum_{a\in\Om_\rho}\sum_{k=1}^\infty
|\phi_{a^k}'(\pi(\rho))|^{\d_\cS}\psi_\d^*(a^k\rho) m_{\d_\cS}^*([\tau])= \\
&=\frac1{\chi_{\d_\cS}}\Big(\psi_{\d_\cS}^*(\rho)+\sum_{a\in\Om_\rho}\sum_{k=1}^\infty
|\phi_{a^k}'(\pi(\rho))|^\d\psi_{\d_\cS}^*(a^k\rho)\Big) m_{\d_\cS}([\tau]) \\
&=\frac{\psi_{\d_\cS}(\rho)}{\chi_{\d_\cS}}m_{\d_\cS}([\tau]).
\end{aligned}
$$
We have thus proved the following.

\bthm\label{t1pc6}
If $\cS$ is a finite irreducible parabolic conformal GDMS, $\rho\in E_{A^*}^\infty$, and $\tau\in E_{*A^*}^*$, then 

\sp
\begin{itemize}

\item [(a)] The function $\De_\cS^+\ni s\longmapsto\eta_\rho(\tau,s)\in\C$ has a meromorphic extension to some neighbourhood of the vertical line $\re(s)=\d_\cS$,

\sp\item[(b)] This extension has a single  pole $s=\delta_{\cS}$, and 

\sp\item[(c)] The pole $s=\delta_{\cS}$ is simple and its residue is equal to $\frac{\psi_{\d_\cS}(\rho)}{\chi_{\d_\cS}}m_{\d_\cS}([\tau])$.
\end{itemize}
\ethm

\section{Asymptotic Results for Multipliers}\label{Parabolic_Counting}

Now that we have established Theorem~\ref{t1pc6}, we are ready to prove the following theorem which,  along with its two corollaries  below,  constitutes the main results of this section.

\bthm[Asymptotic Equidistribution of Multipliers for Parabolic Systems I]\label{t2pc6} 
Suppose that $\cS$ is a finite irreducible parabolic conformal GDMS. Fix $\rho\in E_A^\infty$. If $\tau\in E_A^*$ then,
\beq\label{2_2016_03_25}
\lim_{T \to +\infty} \frac{N_\rho(\tau,T)}{e^{{\d_\cS} T}} 
= \frac{\psi_{\d_\cS}(\rho)}{{\d_\cS}\chi_{\mu_\d}}m_{\d_\cS}([\tau]),
\eeq
and 
\beq\label{2_2016_03_25_P}
\lim_{T \to +\infty} \frac{N_p(\tau,T)}{e^{{\d_\cS} T}} 
= \frac{1}{{\d_\cS}\chi_{\mu_{\d_\cS}}}\mu_{\d_\cS}([\tau]).
\eeq
\end{thm}

\begin{proof}
We first prove formula \eqref{2_2016_03_25}. If $\rho \in E_{*A^*}^\infty$
and $\tau \in E^*_{*A^*}$, this formula follows 
from Theorem \ref{t1pc6} in exactly the same way as 
formula (\ref{2_2016_01_30B}) in Theorem~\ref{dynA}
follows from Theorem~\ref{t2nh18}.

Now keep 
$\tau \in E^*_{*A^*}$ and let $\rho\in E_A^\infty$ be arbitrary.  
Then for every $q \geq 1$ large enough there exists $\rho_q \in E^\infty_{*A^*}$
such that 
$$
\rho|_q = \rho_q|_q.
$$
Since 
$\lim_{q \to \infty} d(\rho, \rho_q) = 0$, the Bounded Distortion Property (BDP) for the attracting system $\cS^*$ yields a function 
$q \mapsto \widehat K_q \in [1, +\infty)$ such that 
\begin{equation}\label{1ma5.1}
\lim_{q \to \infty} \widehat K_q = 1
\end{equation}
and
$$
\widehat K_q^{-1}
\leq 
\frac{|\phi_\tau'(\pi_{\mathcal S}(\rho))|}{|\phi_\tau'(\pi_{\mathcal S}(\rho_q))|}
\leq \widehat K_q 
$$
for all $q \geq 1$ large enough as indicated above.  Hence 
$$
N_{\rho_q}(\tau, T - \log \widehat K_q) \leq N_\rho(\tau , T )
\leq N_{\rho_q}(\tau, T + \log \widehat K_q).
$$
Therefore, dividing by $e^{{\d_\cS} T}$ we get that
$$
\frac{N_{\rho_q}(\tau, T - \log \widehat K_q)}{\exp({\d_\cS}(T - \log \widetilde K_q))} \widehat K_q^{-{\d_\cS}}
\leq \frac{N_\rho(\tau,T)}{e^{{\d_\cS} T}}
\leq \frac{N_{\rho_q}(\tau, T + \log \widehat K_q)}{\exp({\d_\cS}(T + \log \widetilde K_q))} \widehat K_q^{{\d_\cS}}.
$$
Since $\rho_q \in E^\infty_{*A^*}$ and $\tau \in E^*_{*A^*}$
we thus obtain 
$$
\widehat K_q^{-{\d_\cS}} \frac{\psi_{\d_\cS}(\rho)}{{\d_\cS} \chi_{\d_\cS}}m_\delta([\tau])
\leq \liminf_{T \to +\infty} \frac{N_\rho(\tau,T)}{e^{{\d_\cS} T}}
\leq \limsup_{T \to +\infty} \frac{N_\rho(\tau,T)}{e^{{\d_\cS} T}}
\leq \widehat K_q^{{\d_\cS}} \frac{\psi_{\d_\cS}(\rho)}{{\d_\cS} \chi_{\d_\cS}}m_{\d_\cS}([\tau]).
$$

Invoking  (\ref{1ma5.1}) we now conclude that 
\begin{equation}
\lim_{T \to +\infty} \frac{N_\rho(\tau,T)}{e^{{\d_\cS} T}}
= \frac{\psi_{\d_\cS}(\rho)}{{\d_\cS} \chi_{\d_\cS}}m_{\d_\cS}([\tau]).
\label{2ma5.1}
\end{equation}
Working in full generality, we now assume that $\rho \in E_A^\infty$
and $\tau \in E_A^*$. Then there exists $\mathcal F_\tau$, a countable collection of mutually incomparable elements of $E^*_{*A^*}$, each of which extends $\tau$, such that 
$$
m_{\d_\cS}\lt([\tau]\sms \bu_{\omega \in \mathcal F_\tau}[\om]\rt)=0.
$$
Noting that then the family $\{[\omega] \hbox{ : } \omega \in \mathcal F_\tau\}$ consists of mutually disjoint sets, we thus get that from (\ref{2ma5.1}) that
$$
\begin{aligned}
\liminf_{T \to +\infty}
\frac{N_\rho(\tau,T)}{e^{{\d_\cS} T}}
&\geq
\liminf_{T \to +\infty}
\frac{\sum_{\omega \in \mathcal F_\tau} N_\rho(\omega,T)}{e^{{\d_\cS} T}}
\geq
\sum_{\omega \in \mathcal F_\tau}
\liminf_{T \to +\infty}
\frac{ N_\rho(\omega,T)}{e^{{\d_\cS} T}}\cr
&=\sum_{\omega \in \mathcal F_\tau}
\frac{\psi_{\d_\cS}(\rho)}{{\d_\cS} \chi_{\mu_{\d_\cS}}} m_{\d_\cS}([\omega])
= \frac{\psi_{\d_\cS}(\rho)}{{\d_\cS} \chi_{\mu_{\d_\cS}}} m_{\d_\cS}([\tau]).
\end{aligned}
$$
Having this (and already knowing that the neutral 
word $\emptyset$ belongs to $E^*_{*A^*}$) then  (\ref{2ma5.1}) gives that 
$$
\lim_{T \to +\infty} \frac{N_\rho(T)}{e^{{\d_\cS} T}}
= \frac{\psi_{\d_\cS}(\rho)}{{\d_\cS} \chi_{\d_\cS}} m_{\d_\cS}([\emptyset])
=
\frac{\psi_{\d_\cS}(\rho)}{{\d_\cS} \chi_{\d_\cS}} m_{\d_\cS}(E^\infty_{*A^*})
= \frac{\psi_{\d_\cS}(\rho)}{{\d_\cS} \chi_{\d_\cS}},
$$
we deduce that 
$$
\lim_{T\to +\infty} \frac{N_\rho(\tau, T)}{e^{{\d_\cS} T}}
= \frac{\psi_{\d_\cS}(\rho)}{e^{{\d_\cS} T}} m_{\d_\cS}([\tau])
$$
in the say way (although it is now in fact simpler) as formula
(\ref{1nh24}) is deduced from (\ref{1nh23})
and (\ref{1_2016_01_30}),  the latter applied with $\tau=\emptyset$
(i.e., the empty word).  The proof of formula (\ref{2_2016_03_25}) is then complete.

Now we prove formula (\ref{2_2016_03_25_P}). First assume that 
$\tau$ is not a power of an element from $\Omega$.  This means that either 
$$
\tau = a^j \beta
$$
where $a \in \Omega$, $j \geq 1$, and $\beta_1 \neq a$ or 
$$
\tau= \beta
$$
where $\beta_1 \not\in \Omega$.  In either case, 
$$
\tau = a^j \beta,
$$
with $j \geq 0$.  As in the  proof of formula (\ref{2_2016_01_30B}) in Theorem \ref{dynA}, for every $\gamma \in E_A^*$ fix $\gamma^+ \in E_A^\infty$ (which in fact can be selected to depend only on $\gamma_{|\gamma|}$) such that 
$$
\gamma \gamma^+ \in E_A^\infty.
$$
Fix $q \geq 1$ and $\gamma \in E_A^q$ arbitrarily.  Consider an arbitrary element $\omega b^k \in E_A^*$, $\omega \in E_{*,A^*}^*$, $b \in \Omega$ such that $a^j \beta \gamma \omega b^k \in E_p^*$.  Consider two cases:

\medskip
\noindent
{\it Case $1^0$}.   Assume $b \neq a$ if $j \geq 1$.  Then
$$
\big|\phi'_{a^j\beta \gamma\omega b^k}(x_{a^j\beta \gamma\omega b^k})\big|
= 
\big|\phi'_{a^j\beta \gamma\omega}\(\pi_{\mathcal S}((b^ka^j\beta \gamma\omega)^\infty)\)\big|\cdot
\big|\phi'_{b^k}\(\pi_{\mathcal S}((a^j\beta \gamma\omega b^k)^\infty)\)\big|
$$
and 
$$
\big|\phi'_{a^j\beta\omega b^k}\(\pi_{\mathcal S}(a^j \beta \gamma \gamma^+)\)\big|
= 
\big|\phi'_{a^j\beta \gamma\omega}(\pi_{\mathcal S}(b^ka^j\beta \gamma\gamma^+)\)\big|\cdot
\big|\phi'_{b^k}\(\pi_{\mathcal S}(a^j\beta \gamma\gamma^+)\)\big|.
$$
Since $\omega \in E^*_{*A^*}$ and since either $b \neq a$ if $j \geq 1$ or $\beta_1 \not\in \Omega$ if $j = 0$, by the (BDP) we get that 
$$
{\^K}_q^{-1}
\leq
\frac{\big|\phi'_{a^j\beta \gamma\omega}(\pi_{\mathcal S}\((b^k a^j \beta \gamma \omega)^\infty)\)\big|}{
\big|\phi'_{a^j\beta \gamma\omega}\(\pi_{\mathcal S}(b^k a^j \beta \gamma \gamma^+)\)\big|
}
\leq {\^ K}_q
$$
and 
$$
{\^ K}_q^{-1}
\leq
\frac{\big|\phi'_{b^k}\(\pi_{\mathcal S}((a^j \beta \gamma \omega b^k)^\infty)\)\big|}{
\big|\phi'_{b^k}\(\pi_{\mathcal S}(( a^j \beta \gamma \gamma^+)\)\big|
}
\leq {\^ K}_q
$$
with some ``distortion" function $q \mapsto {\^ K}_q \in [1, +\infty)$ such that
$\lim_{q\to \infty} {\widehat K}_q  = 1$.  Consequently, 
\beq\label{1ma7}
{\^K}_q^{-2}
\leq
\frac{\big|\phi'_{a^j\beta \gamma\omega  b^k}(x_{a^j \beta \gamma \omega b^k})\big|}{
\big|\phi'_{a^j\beta \gamma\omega b^k}(\pi_{\mathcal S}((a^j \beta \gamma \gamma^+))\big|
}
\leq {\^ K}_q^2.
\eeq
\medskip

\noindent
{\it Case $2^0$}.   Assume $j \geq 1$ and $b =  a$.  Then
\begin{equation}\label{3ma7}
\big|\phi'_{a^j\beta \gamma\omega b^k}(x_{a^j\beta \gamma\omega b^k})\big|
= 
\big|\phi'_{a^j\beta \gamma\omega}\(\pi_{\mathcal S}((a^{j+k}\beta \gamma\omega)^\infty)\)\big|\cdot
\big|\phi'_{a^k}(\pi_{\mathcal S}\((a^j\beta \gamma\omega b^k)^\infty)\)\big|
\end{equation}

\fr and
\begin{equation}
\big|\phi'_{a^j\beta \gamma\omega b^k}\(\pi_{\mathcal S}(a^j \beta \gamma \gamma^+)\)\big|
= 
\big|\phi'_{a^j\beta \gamma\omega}\(\pi_{\mathcal S}(a^{j+k}\beta \gamma\gamma^+)\) \big|. \big|\phi'_{a^k}\(\pi_{\mathcal S}(a^{j}\beta \gamma\gamma^+)\)\big|.
\label{4ma7}
\end{equation}
Again by (BDP) we have that 
\begin{equation}
{\^K}_q^{-1}
\leq
\frac{\big|\phi'_{a^j\beta \gamma\omega }\(\pi_{\mathcal S}((a^{j+k} \beta \gamma \omega)^\infty))\)\big|}{
\big|\phi'_{a^j\beta \gamma\omega}\(\pi_{\mathcal S}(a^{j+k} \beta \gamma \gamma^+)\)\big|}
\leq {\^ K}_q.
\label{2ma7}
\end{equation}
By the Chain Rule
\begin{equation}\label{2ma8}
\big|\phi'_{a^k}(\pi_{\mathcal S}(a^{j} \beta \gamma\ka))\big|
=\big|\phi'_{a^{k+j}}(\pi_{\mathcal S}( \beta \gamma\ka))\big|\cdot
\big|\phi'_{a^{j}}\(\pi_{\mathcal S}( \beta \gamma\ka))\big|^{-1}
\end{equation}
for every $\ka\in E_A^*$ such that $\gamma\ka \in E_A^*$. Since $\beta_1 \neq a$
we have that 
$$
{\^K}_q^{-1}
\leq
\frac{\big|\phi'_{a^{j+k}}\(\pi_{\mathcal S}(( \beta \gamma \omega a^k)^\infty)\)\big|}{\big|\phi'_{a^{j+k}}\(\pi_{\mathcal S}(\beta\gamma \gamma^+)\)
\big|}
\leq {\^K}_q
$$
and 
$$
{\^K}_q^{-1}
\leq
\frac{\big|\phi'_{a^{j}}(\pi_{\mathcal S}\(( \beta \gamma \omega a^k)^\infty)\)\big|}{\big|\phi'_{a^{j}}\(\pi_{\mathcal S}(\beta\gamma\gamma^+)\) \big|}
\leq {\^K}_q.
$$
Hence, invoking (\ref{2ma8}) we get that 
$$
{\^K}_q^{-2}
\leq
\frac{\big|\phi'_{a^{k}}\(\pi_{\mathcal S}(( a^j\beta \gamma \omega a^k)^\infty)\) \big|}{\big|\phi'_{a^{k}}(\pi_{\mathcal S}(a^j\beta\gamma\gamma^+)\)\big|}
\leq {\^K}_q^2.
$$
Along with (\ref{2ma7}), (\ref{3ma7}) and (\ref{4ma7}) this yields
\beq\label{1ma8}
{\^K}_q^{-3}
\leq
\frac{\big|\phi'_{a^{j}\beta \gamma \omega b^k}(x_{a^j\beta\gamma \omega b^k})\big|}{\big|\phi'_{a^{j}\beta \gamma \omega b^k}\(\pi_{\mathcal S}(a^j\beta  \gamma \gamma^+)\)\big|
}\leq {\^K}_q^3.
\end{equation}

Now it follows from (\ref{1ma7}) and (\ref{1ma8})
that 
\begin{equation}\label{3ma8}
\pi_{a^j \beta \gamma \gamma^+}(a^j \beta \gamma, T-3 \log \^K_q)
\sbt\pi_p(\alpha^j \beta \gamma, T)
\sbt
\pi_{a^j \beta \gamma \gamma^+}(a^j \beta \gamma, T+3 \log \^K_q).
\end{equation}
Therefore, applying (\ref{2_2016_03_25}) we get that 
\begin{equation}
\begin{aligned}
\liminf_{T \to +\infty} \frac{N_p(a^j \beta, T)}{e^{{\d_\cS} T}}
&\geq \liminf_{T \to +\infty} \sum_{\gamma \in E_A^q\atop a^j\beta \gamma \in E_A^*}
\frac{N_{\alpha^j \beta \gamma \gamma^+}(a^j \beta \gamma, T- 3\log {\^K}_q)}{e^{{\d_\cS} T}}\cr
&\geq  \sum_{\gamma \in E_A^q \atop a^j\beta \gamma \in E_A^*}
 \liminf_{T \to +\infty} 
\frac{N_{\alpha^j \beta \gamma \gamma^+}(a^j \beta \gamma, T- 3\log {\^K}_q)}{e^{{\d_\cS} T}} \cr
&=
\sum_{\gamma \in E_A^q \atop a^j\beta \gamma \in E_A^*}
 \liminf_{T \to +\infty} 
\frac{
N_{\alpha^j \beta \gamma \gamma^+}(a^j \beta \gamma, T- 3\log {\^K}_q)}
{e^{{\d_\cS}(T - 3 \log {\^K}_q )}} \^K_q^{-3{\d_\cS}}\cr
&= {\^K}_q^{- 3{\d_\cS}}
\sum_{\gamma \in E_A^\infty\atop a^j\beta \gamma \in E_A^*}
\frac{\psi_{\d_\cS}(a^j \beta\gamma\gamma^+)}{{\d_\cS} \chi_{\d_\cS}}
m_{\d_\cS}([a^j\beta\gamma]) \cr
&\geq K_q^{-1}(a^j\beta)\frac{1}{{\d_\cS} \chi_{\d_\cS}} \mu_{\d_\cS}([a^j\beta]),
\label{1ma9}
\end{aligned}
\end{equation}
with some function $q \mapsto K_q(a^j\beta) \in [1, +\infty)$ for which 
$\lim_{q \to +\infty} K_q(a^j\beta) =1 $ and which exists because $a^j\beta$ is not a power of an element from $\Omega$.  Taking the limit  in (\ref{1ma9}) as $q \to +\infty$, we thus get that 
\beq\label{2ma9}
\liminf_{T \to +\infty} \frac{N_p(a^j \beta, T)}{e^{{\d_\cS} T}} \geq 
\frac{1}{{\d_\cS} \chi_\delta} \mu_{\d_\cS}([a^j\beta]).
\eeq
In the general case, i.e., making no assumptions on  $\tau \in E_A^*$ we proceed in the same way
as in the proof of formula (\ref{2_2016_03_25}). We can  fix $\mathcal F_\tau$, a countable collection of mutually incomparable words extending $\tau$, not being powers (concatenations) of elements from $\Omega$, and such that 
$$
\mu_{\d_\cS}\lt([\tau] \sms \bu_{\omega \in \mathcal F_\tau}\rt)=0.
$$
Noting that then the family $\{[\omega] \hbox{ : } \omega \in \mathcal F_\tau\}$
consists of mutually disjoint sets, we thus get that from (\ref{2ma9}) that 
\begin{equation}
\begin{aligned}
\liminf_{T \to +\infty} \frac{N_p(\tau, T)}{e^{{\d_\cS} T}} &\geq 
\liminf_{T \to +\infty} \frac{\sum_{\omega \in \mathcal F_\tau} N_p(\omega,T)}{e^{{\d_\cS} T}}
\geq 
\sum_{\omega \in \mathcal F_\tau} 
 \liminf_{T \to +\infty}  \frac{N_p(\omega,T)}{e^{{\d_\cS} T}}\cr
 &= \frac{1}{{\d_\cS} \chi_{\d_\cS}} 
 \sum_{\omega \in \mathcal F_\tau} \mu_{\d_\cS}([\omega]) 
 = \frac{1}{{\d_\cS} \chi_{\d_\cS}} \mu_{\d_\cS}([\tau]).
 \end{aligned} 
 \label{3ma9}
\end{equation}
For the upper bound we again deal first with words $a^j \beta$, i.e., the same as those leading to 
(\ref{2ma9}).  Since the alphabet $E$ is finite it follows from the left hand side 
of (\ref{3ma8}) and from (\ref{2_2016_03_25}) that

\begin{equation}\label{1ma10}
\begin{aligned}
\limsup_{T \to +\infty}
 \frac{
 N_p(a^j \beta, T)
 }{
 e^{{\d_\cS} T}
 }
&\leq \limsup_{T \to +\infty} \sum_{\gamma \in E_A^\infty\atop a^j\beta \gamma \in E_A^*}
\frac{
N_{\alpha^j \beta \gamma \gamma^+}(a^j \beta \gamma, T+ 3\log {\^K}_q)
}{
e^{{\d_\cS} T}
}\cr
& \leq
\sum_{\gamma \in E_A^\infty\atop a^j\beta \gamma \in E_A^*}
 \limsup_{T \to +\infty} 
\frac{
N_{\alpha^j \beta \gamma \gamma^+}(a^j \beta \gamma, T+ 3\log {\^K}_q)
}{
e^{{\d_\cS} T}
} \cr
&= 
\sum_{\gamma \in E_A^\infty \atop a^j\beta \gamma \in E_A^*}
 \limsup_{T \to +\infty} 
\frac{
N_{\alpha^j \beta \gamma \gamma^+}(a^j \beta \gamma, T+ 3\log {\^K}_q)
}{
e^{{\d_\cS} (T+3\log {\^K}_q) } }
{\^K}_q^{3{\d_\cS}}
 \cr
&= 
{\^K}_q^{ 3{\d_\cS}}
\sum_{\gamma \in E_A^\infty\atop a^j\beta \gamma \in E_A^*}
\frac{\psi_{\d_\cS}(a^j \beta\gamma\gamma^+)}{{\d_\cS} \chi_{\d_\cS}}
m_{\d_\cS}([a^j\beta\gamma]) \cr
&\leq K_q(a^j\beta) \frac{1}{{\d_\cS} \chi_{\d_\cS}} \mu_{\d_\cS}([a^j\beta]).
\end{aligned}
\end{equation}
Taking the limit $q \to +\infty$ in (\ref{1ma10}) we thus get that 
$$
\liminf_{T \to +\infty} \frac{N_p(a^j \beta, T)}{e^{{\d_\cS} T}}
\leq \frac{1}{{\d_\cS} \chi_{\d_\cS}} \mu_{\d_\cS}([a^j\beta]).
$$ 
Along with (\ref{2ma9}) this gives  
\begin{equation}
\lim_{T \to +\infty} \frac{N_p(a^j \beta, T)}{e^{{\d_\cS} T}}
=  \frac{1}{{\d_\cS} \chi_{\d_\cS}} \mu_{\d_\cS}([a^j\beta]).
\label{2ma10}
\end{equation}
Passing to the upper bound in the general case, we only need to deal with powers of parabolic elements. Because of (\ref{3ma9}) and  Theorem~\ref{t1_2017_02_18}  
 (b1)--(b3), formula (\ref{2_2016_03_25_P})
holds for all words $\tau = a^l$, $l \geq 1$, where $a\in \Omega$ is such that
${\d_\cS} \leq \frac{2p_a}{p_a +1}$. In what follows, we can thus assume that  
$$
{\d_\cS} > \frac{2p_a}{p_a+1}.
$$
Then for every integer $j \geq -1$, we have 
\begin{equation}
[a^{j+1}]\sms[a^{j+2}] = \bu\big\{[a^{j+1}e]:e \in E\sms\{a\} \ {\rm and } A_{ae}=1\big\}.
\label{3ma10}
\end{equation}
Since the set $E \sms \{a\}$ is finite it thus follows from (\ref{2ma10}) that 
 \begin{equation}
 \frac{N_p([a^{j+1}] \sms[a^{j+2}], T)}{e^{{\d_\cS} T}}
 = \frac{1}{{\d_\cS} \chi_{\d_\cS}} \mu_{\d_\cS}([a^{j+1}] \sms[a^{j+2}])
 =  \frac{1}{{\d_\cS} \chi_{\d_\cS}} (\mu_{\d_\cS}([a^{j+1}]) - \mu_{\d_\cS}([a^{j+2}])).
 \label{4ma10}
 \end{equation}
Now if $\omega \in [a^{j+1}]\sms[a^{j+2}]$ then 
$\omega = a^j(a\ka a^l)$ with $\ka_1, \ka_{|\ka|} \in E\sms \{a\}$, $A_{a\ka_1}=1$, $A_{\ka_{|\ka|}a}=1$, and $l \geq 0$. Then
$$
e^{-T}
\leq 
\big|\phi'_{a^j (a\ka a^l)}(x_{a^j(a\ka a^l)})\big|
=\big|\phi'_{a\ka a^l}(x_{a^j(a\ka a^l)})\big|\cdot
\big|\phi'_{a^j}(x_{a\ka a^{j+l}})\big|
\asymp(j+2)^{-(p_a+1)/p_a} \big|\phi'_{a\ka a^l}(x_{a^j(a\ka a^l)})\big|.
$$
Denoting by $Q\ge 1$ the multiplicative constant corresponding to the ``$\asymp$'' sign above, we thus get 
\begin{equation}
\big|\phi'_{a\ka a^l}(x_{a^j(a\ka a^l)})\big| 
\geq Q^{-1}(j+2)^{-\frac{p_a+1}{p_a}}e^{-T}.
\label{1ma11}
\end{equation}
Now fix a word $\beta \in E_A^\infty$ with $\beta_1 = a$ and $\beta_2 \neq a$.
Then 
$$
\big|\phi'_{a\ka a^l}(x_{a^j(a\ka a^l)})\big|
=
\big|\phi'_{a^l}(x_{a^j(a\ka a^l)})\big|\cdot
|\phi'_{a\ka}(x_{a^{j+l}\ka})|
\asymp
(j+l+2)^{-\frac{p_a+1}{p_a}}\cdot j^{\frac{p_a+1}{p_a}} \big|\phi'_{a\ka} \(\pi_{\mathcal S}(\beta)\)\big|.
$$
It therefore follows from (\ref{1ma11}) that 
$$
\big|\phi'_{a x }(\pi_{\mathcal S}(\beta))\big|
\geq Q^{-2} (j+l+2)^{\frac{p_a+1}{p_a}}e^{-T}.
$$
Equivalently,
$$
-\log \big|\phi'_{a\ka}(\pi_{\mathcal S}(\beta))\big|
\leq 2\log Q -  \frac{(p_a+1)}{p_a}\log (j+l+2) +T.
$$ 
Hence 
$$
a\ka \in \pi_\beta\lt([a\ka_1], 2\log Q -  \frac{(p_a+1)}{p_a}\log (j+l+2) +T\rt).
$$
Therefore,
\beq\label{1_2017_03_17}
N_p([\alpha^{j+1}]\sms [a^{j+2}])
\leq 
\sum_{b \in E \sms \{a\}\atop A_{ab}=1} \sum_{l=0}^\infty N_\beta\left([ab], 2\log Q - \frac{p_a+1}{p_a}
\log(j+l+2) + T\right).
\eeq
By formula (\ref{2_2016_03_25}),  and since the alphabet $E$ is finite, there exists $T_1 > 0$ such that
\begin{equation}
e^{-{\d_\cS} S} N_\beta([ab], S) \leq \frac{\psi_{\d_\cS}(\beta)}{{\d_\cS} \chi_{\d_\cS}}
m_{\d_\cS}([ab])
\leq \frac{\psi_{\d_\cS}(\beta)}{{\d_\cS} \chi_{\d_\cS}}
\label{1ma12}
\end{equation}
for every $b \in E\sms\{a\}$ with $A_{ab}=1$ and every $S \geq T_1$.  Now
$$
2\log Q - \frac{p_a+1}{p_a} \log (j+l+2) + T \geq T_1
$$
if and only if 
\begin{equation}
j+l+2 \leq s_T:= Q^{\frac{2p_a}{p_a+1}}\exp\lt(\frac{p_a}{p_a+1}(T-T_1)\rt).
\label{2ma12}
\end{equation}
In addition, if
\begin{equation}
2\log Q - \frac{p_a+1}{p_a} \log (j+l+2) + T \le -1 
\label{3ma12}
\end{equation}
then 
\begin{equation}
 N_\beta\lt([ab],2\log Q - \frac{p_a+1}{p_a} \log (j+l+2) + T \rt) = 0.
 \label{4ma12}
\end{equation}
Formula (\ref{3ma12}) just means that 
\begin{equation}
j+l+2 \geq u_T:= e Q^{\frac{2p_a}{p_a+1}}e^{\frac{p_a}{p_a+1}T}.
 \label{5ma12}
\end{equation}
Therefore, returning to  formula \eqref{1_2017_03_17}, for every $q \geq 1$ we get that 
\begin{equation}
\begin{aligned}
\sum_{j=q+1}^\infty &e^{-{\d_\cS} T} N_p([\alpha^{j+1}] \sms [a^{j+2}])
\le \cr &\!\!\!\!\leq\sum_{b \in E\sms\{a\}\atop A_{ab}=1}
\sum_{j=q}^\infty \sum_{l:j+l+2 \leq s_T}\!\!
\frac{ N_\beta\lt([ab],2\log Q - \frac{p_a+1}{p_a} \log (j+l+2) + T \rt)}{
\exp\lt({\d_\cS}( 2\log Q - \frac{p_a+1}{p_a} \log (j+l+2)+T)\rt)}Q^{2{\d_\cS}}
 (j+l+2)^{-\frac{p_a+1}{p_a}{\d_\cS}}+\cr
& \  \  \  \  \  \  \  \  \  \  \  \  \  \  \  \  \  \ +
\sum_{s_T + 1 \leq j+l+2 \leq u_T} e^{-{\d_\cS} T} N_\beta([ab], T_1)\cr
&\!\!\!\! \leq Q^{2{\d_\cS}} \#E \frac{\psi_{\d_\cS}(\beta)}{{\d_\cS} \chi_{\d_\cS}}
\sum_{j=q}^\infty \sum_{k=j}^\infty k^{-\frac{p_a+1}{p_a}{\d_\cS}} + N_a e^{-{\d_\cS} T}
u_T^2 \cr
&\!\!\!\!\leq \widehat Q_1 
\sum_{j=q}^\infty j^{1-\frac{p_a+1}{p_a}{\d_\cS}}+\widehat Q_2\exp\lt(\lt(\frac{2p_a}{p_a+1}-{\d_\cS}\rt)T\rt)\cr
&\!\!\!\! \leq \widehat Q_3 
q^{2-\frac{p_a+1}{p_a}{\d_\cS}}+\widehat Q_2\exp\lt(\lt(\frac{2p_a}{p_a+1}-{\d_\cS}\rt)T\rt),
\label{1ma13}
\end{aligned}
\end{equation}
where 
$N_a:= \max\big\{N_\beta ([ab], T_1) \hbox{ : } b \in E \backslash \{a\}, \,  A_{ab}=1\big\}$, $\widehat Q_1,\widehat Q_2, \widehat Q_3 \geq 1$ are universal constants, and the last inequality holds because ${\d_\cS}> \frac{2 p_a}{p_a + 1}$. 

Applying (\ref{4ma10}) and (\ref{1ma13}) we obtain for all integers $q \geq k+2$ the following estimate
$$
\begin{aligned}
&\varlimsup_{T \to +\infty}
\left|
\frac{N_p([a^k],T)}{e^{{\d_\cS} T}} - \frac{1}{{\d_\cS} \chi_{\d_\cS}} \mu_{\d_\cS}([a^k])
\right|=\cr
& \ \ =\varlimsup_{T \to +\infty}
\left|
\sum_{j=k-1}^q
\frac{N_p([a^{j+1}]\sms[a^{j+2}],T) }{e^{{\d_\cS} T}}\right.\
+ \left.\sum_{j=q+1}^\infty
\frac{N_p([a^{j+1}]\sms[a^{j+2}],T) }{e^{{\d_\cS} T}} 
 - \frac{1}{{\d_\cS} \chi_{\d_\cS}} \mu_{\d_\cS}([a^k])
\right|\cr
&\ \ \leq \varlimsup_{T \to +\infty}
\left|
\sum_{j=k-1}^q
\frac{N_p([a^{j+1}]\sms[a^{j+2}],T) }{e^{{\d_\cS} T}}
- \frac{1}{{\d_\cS} \chi_{\d_\cS}} \mu_{\d_\cS}([a^k])
\right|+
\varlimsup_{T \to +\infty} \left|\sum_{j=q+1}^\infty
\frac{N_p([a^{j+1}]\sms[a^{j+2}],T) }{e^{{\d_\cS} T}} 
\right|\cr
&\ \  \leq \frac{1}{{\d_\cS} \chi_{\d_\cS}}
\left|
\sum_{j=k-1}^q
 \mu_{\d_\cS}([a^{j+1}]\sms [a^{j+2}])   - \mu([a^k])\right|
 + \varlimsup_{T \to +\infty} \widehat Q_2 \sum_{a \in \Omega}
 \exp \left(\lt(\frac{2p_a}{p_a+1} - {\d_\cS} \rt) T \right)\cr
&\  \ =\frac{1}{{\d_\cS} \chi_{\d_\cS}}\big|\mu_{\d_\cS}([a^k]\sms [a^{q+2}]) - \mu([a^k])\big|\cr 
&\  \ =  \frac{1}{{\d_\cS} \chi_{\d_\cS}} \mu_{\d_\cS}([a^{q+2}]).
\end{aligned}
$$
But since ${\d_\cS} > \frac{2p_a}{p_a+1}$ we have that $\lim_{q \to \infty} \mu_{\d_\cS}([a^{q+2}]) = 0$ and therefore 
$$
\varlimsup_{T \to +\infty} 
 \left|\frac{N_p([a^k],T)}{e^{{\d_\cS} T}} - \frac{1}{{\d_\cS} \chi_{\d_\cS}}\mu_{\d_\cS}([a^k])\right|=0.
$$
This just means that
$$
\lim_{T\to \infty} \frac{N_p([a^k],T)}{e^{{\d_\cS} T}} 
= \frac{1}{{\d_\cS} \chi_{\d_\cS}}\mu_{\d_\cS}([a^k]).
$$
The proof of our theorem is thus complete.
\end{proof} 

\sp The proof of the following theorem, based on Theorem~\ref{t2pc6}, is exactly the same as the proof of Theorem~\ref{dyn} based on Theorem~\ref{dynA}.

\begin{thm} 
[Asymptotic Equidistribution of Multipliers for Parabolic 
Systems II]\label{t2pc6_B}
Suppose that $\cS$ is a finite irreducible parabolic conformal GDMS.  Fix $\rho\in E_A^\infty$. If $B \subset X$ is a Borel set such that $\^m_{\d_\cS}(\bd B)=0$ (equivalently $\^\mu_{\d_\cS}(\bd B)=0$) then,
\beq\label{3_2016_03_25}
\lim_{T \to +\infty} \frac{N_\rho(B,T)}{e^{{\d_\cS} T}} 
= \frac{\psi_{\d_\cS}(\rho)}{{\d_\cS}\chi_{\mu_{\d_\cS}}}\^m_{\d_\cS}(B)
\eeq
and 
\beq\label{3_2016_03_25_B}
\lim_{T \to +\infty} \frac{N_p(B,T)}{e^{{\d_\cS} T}} 
= \frac{1}{{\d_\cS}\chi_{\mu_{\d_\cS}}}\^\mu_{\d_\cS}(B).
\eeq
\end{thm}

\sp\fr We have as an immediate corollary the following:

\bthm[Asymptotic Equidistribution of Multipliers for Parabolic Systems]
\label{t2pc6-again_Q}
Suppose that $\cS$ is a finite irreducible parabolic conformal GDMS. Fix $\rho\in E_A^\infty$. Then 
\beq\label{3_2016_03_25G}
\lim_{T \to +\infty} \frac{N_\rho(T)}{e^{{\d_\cS} T}} 
= \frac{\psi_{\d_\cS}(\rho)}{{\d_\cS}\chi_{\mu_{\d_\cS}}}
\eeq
and 
\beq\label{3_2016_03_25_H}
\lim_{T \to +\infty} \frac{N_p(T)}{e^{{\d_\cS} T}} 
= \frac{1}{{\d_\cS}\chi_{\mu_{\d_\cS}}}\^\mu_{\d_\cS}(J_\cS).
\eeq
\ethm


\section{Asymptotic Results for Diameters}\label{diam_Parabolic}

We now want to use the asymptotic results established 
in the previous section to show the asymptotic formulae for diameters of images of a set.

In this section we assume that $\cS$ is a finite irreducible conformal parabolic GDMS. Our task here is, for parabolic systems, the same as it was in Section~\ref{contracting_diameters} for attracting systems, i.e. to obtain asymptotic counting properties corresponding to the function $-\log\diam(\phi_\om(Y))$, $\om\in E_A^*$. The 
notation here is similar to that  in Section~\ref{contracting_diameters} but is  slightly enhanced. 
Our strategy now is to use the  full generality of Theorem~\ref{t1da7} and to deduce from it the main result of the current section, which is the following.

\bthm[Asymptotic Equidistribution Formula of Diameters for Parabolic Systems, I]\label{t1dp13}
Suppose that $\cS$ is a finite irreducible parabolic conformal GDMS. Fix $\rho\in E_A^\infty$ and $Y\sbt X_{i(\rho)}$ having at least two points. If $B \subset X$ is a Borel set such that $m_{\d_\cS}(\bd B)=0$ (equivalently $\mu_{\d_\cS}(\bd B)=0$) then,
\beq\label{1dp13}
\lim_{T \to +\infty} \frac{D^{\rho}_Y(B,T)}{e^{{\d_\cS} T}} 
=C_\rho(Y)m_{\d_\cS}(B),
\eeq
where $C_\rho(Y)\in (0,+\infty]$ is a constant depending only on the system $\cS$, the word $\rho$ (but see Remark~\ref{r1_2017_03_20B}), and the set $Y$. In addition $C_\rho(Y)$ is finite if 
and only if either

\begin{enumerate}
\item 
$$
\ov Y\cap \Om_\infty=(\ov Y\cap \Om_\infty\cap\Om_\rho)=\es
$$
or 
\item 
$$
{\d_\cS}>\max\big\{p(a):a\in\Om_\rho \  \  {\rm and} \  \  x_a\in \ov Y \big\}.
$$
\end{enumerate}
Then the function $[\rho_1]\ni \om\longmapsto C_\om(Y)$ is uniformly separated away from zero and bounded above.
\ethm

\bpf
Recall that
$$
\Om_\rho=\{a\in\Om:A_{a\rho_1}=1\}.
$$
We know that
$$
E_\rho^*=E_{*\rho}^*\cup \bu_{a\in \Om_\rho}\bu_{k=1}^\infty E_{*a}^*a^k
$$
and this union consists of mutually incomparable terms. Therefore,
$$
\cD_Y^\rho(B,T)
=\cD_{Y,\cS^*}^\rho(B,T)\cup \bu_{a\in \Om_\rho}\bu_{k=1}^\infty\cD_{\phi_{a^k}(Y),\cS^*}^{a^k\rho}(B,T),
$$
and this union consists of mutually disjoint terms. Therefore,
\beq\label{1_2017_03_13}
\frac{D_Y^\rho(B,T)}{e^{{\d_\cS} T}}
\ge \frac{D_{Y,\cS^*}^\rho(B,T)}{e^{{\d_\cS} T}}+\sum_{a\in \Om_\rho}
     \sum_{k=1}^\infty\frac{D_{\phi_{a^k}(Y),\cS^*}^{a^k\rho}(B,T)}{e^{{\d_\cS} T}},
\eeq
and for every $q\ge 1$:
\beq\label{1dp14}
\frac{D_Y^\rho(B,T)}{e^{{\d_\cS} T}}
\le \frac{D_{Y,\cS^*}^\rho(B,T)}{e^{{\d_\cS} T}}+\sum_{a\in \Om_\rho}
     \sum_{k=1}^q\frac{D_{\phi_{a^k}(Y),\cS^*}^{a^k\rho}(B,T)}{e^{{\d_\cS} T}}
     +\sum_{a\in \Om_\rho}
     \sum_{k=q+1}^\infty\frac{D_{\phi_{a^k}(Y),\cS^*}^{a^k\rho}(B,T)}{e^{{\d_\cS} T}}.
\eeq
Assume first that $\rho\in E_{*A^*}^\N$. Then, $a^k\rho\in E_{*A^*}^\N$ for every $a\in \Om_\rho$ and for all integers $k\ge 0$, whence we can invoke Theorem~\ref{t1da7} and \ref{1_2017_03_13}, to conclude that
\beq\label{2dp14}
\begin{aligned}
\varliminf_{T\to\infty}\frac{D_Y^\rho(B,T)}{e^{{\d_\cS}T}}
&\ge \varliminf_{T\to\infty}\frac{D_{Y,\cS^*}^\rho(B,T)}{e^{{\d_\cS} T}}+\sum_{a\in \Om_\rho}
     \sum_{k=1}^\infty\varliminf_{T\to\infty}
        \frac{D_{\phi_{a^k}(Y),\cS^*}^{a^k\rho}(B,T)}{e^{{\d_\cS} T}} \\
&=\Big(C_\rho^{\cS^*}(Y)+\sum_{a\in \Om_\rho}\sum_{k=1}^\infty 
         C_{a^k\rho}^{\cS^*}(\phi_{a^k}(Y))\Big)m_{\d_\cS}^*(B)\\
&=\Big(C_\rho^{\cS^*}(Y)+\sum_{a\in \Om_\rho}\sum_{k=1}^\infty 
         C_{a^k\rho}^{\cS^*}(\phi_{a^k}(Y))\Big)m_{\d_\cS}(B).
\end{aligned}         
\eeq
Since for every $a\in \Om_\rho$ and for all integers $k\ge 0$
$$
\diam\(\phi_{a^k}(Y)\)\comp 
\begin{cases}
(k+1)^{-\frac1{p_a}} &{\rm if } \  \  \  \ov Y\cap \Om_\infty\cap\Om_\rho\ne\es, \\
(k+1)^{-\frac{p_a+1}{p_a}} &{\rm if } \  \  \  \ov Y\cap \Om_\infty\cap\Om_\rho=\es,
\end{cases}
$$
formula \eqref{2dp14} along with \eqref{1da7.1},  complete the proof of Theorem~\ref{t1dp13} if neither (1) nor (2) hold.
So, for the rest of the proof of the present case of $\rho\in E_{*A^*}^\N$, we assume that at least one of  (1) or (2) holds. Then
\beq\label{1_2017_03_13}
C_\rho^{\cS^*}(Y)+\sum_{a\in \Om_\rho}\sum_{k=1}^\infty 
         C_{a^k\rho}^{\cS^*}(\phi_{a^k}(Y))<+\infty,
\eeq 
and in addition, this number is bounded away from zero and bounded above independently of $\rho\in E_{*A}^\N$ because of \eqref{1da7.1}.

Now fix $a\in \Om_\rho$. If $\om\in\cD_{\phi_{a^k}(Y),\cS^*}^{a^k\rho}(B,T)$, then 
$$
\diam\(\phi_\om(\phi_{a^k}(Y))\)\ge e^{-T},
$$
and, as
$$
\diam\(\phi_\om(\phi_{a^k}(Y))\)
\le \|\phi_\om'\|_\infty\diam\(\phi_{a^k}(Y)\)
\le Q_1\diam(\phi_\om(X_{i(\om)}))\diam\(\phi_{a^k}(Y)\)
$$
with some constant $Q_1>0$, we thus conclude that
$$
\diam(\phi_\om(X_{i(\om)}))
\ge Q_1^{-1}e^{-T}\diam^{-1}\(\phi_{a^k}(Y)\).
$$
Equivalently,
$$
\De(\om)\le \log Q_1+\log\diam\(\phi_{a^k}(Y)\)+T.
$$
Thus
$$
\om\in \cD_{X_{i(a)},\cS^*}^{a\rho}\(\log Q_1+\log\diam\(\phi_{a^k}(Y)\)+T\).
$$
In conclusion,
\beq\label{1dp15}
\cD_{\phi_{a^k}(Y),\cS^*}^{a^k\rho}(B,T)
\sbt \cD_{X_{i(a)},\cS^*}^{a\rho}\(\log Q_1+\log\diam\(\phi_{a^k}(Y)\)+T\).
\eeq
By virtue of Theorem~\ref{t1da7} there exists $T_1>0$ such that 
\beq\label{2dp15}
\frac{D_{X_{i(a)},\cS^*}^{a\rho}(B,S)}{e^{{\d_\cS} S}}
\le C_{a\rho}^{\cS^*}(X_{i(a)})+1
\eeq
for all $S\ge T_1$. Now, let $k_2(T)$ be the least integer such that
$$
\log Q_2+\log\diam\(\phi_{a^k}(Y)\)+T<0.
$$
Then
\beq\label{3dp15}
\cD_{X_{i(a)},\cS^*}^{a\rho}\(\log Q_2+\log\diam\(\phi_{a^k}(Y)\)+T\)=\es
\eeq
for all $k\ge k_2(T)$ and 
$$
k_2(T)\le 
\begin{cases}
Q_2^{p_a}e^{p_a(T-T_1)} \  \  &{\rm if } \  \  (2) \  {\rm holds} \\
Q_2^{\frac{p_a}{p_a+1}}e^{\frac{p_a}{p_a+1}(T-T_1)} \  \  &{\rm if } \  \  (1) \  {\rm holds} \\
\end{cases}
$$
with some constant $Q_2\in (0,+\infty)$, which in general depends on $Y$ if (1) holds. Furthermore, let $k_1(T)$ be least integer such that
$$
\log Q_2+\log\diam\(\phi_{a^k}(Y)\)+T<T_1.
$$
Then, on the one hand, 
$$
\log Q_2+\log\diam\(\phi_{a^k}(Y)\)+T<T_1, 
$$
for all $k\ge k_1(T)$ and (so) it follows from \eqref{1dp15} that
$$
\cD_{\phi_{a^k}(Y),\cS^*}^{a^k\rho}(B,T)
\sbt \cD_{X_{i(a)},\cS^*}^{a\rho}(T_1).
$$
On the other hand, 
$$
\log Q_2+\log\diam\(\phi_{a^k}(Y)\)+T\ge T_1
$$
for all $0\le k\le k_1(T)$. All of 
this,  together  with \eqref{1dp15}--\eqref{3dp15},  yield
\beq\label{4dp15}
\begin{aligned}
\sum_{k=q+1}^\infty &\frac{D_{\phi_{a^k}(Y),\cS^*}^{a^k\rho}
      (B,T)}{e^{{\d_\cS} T}}=\\
&=\sum_{k=q+1}^{k_1(T)} 
     \frac{D_{\phi_{a^k}(Y),\cS^*}^{a^k\rho}(B,T)}{e^{{\d_\cS} T}}
   +\sum_{k=k_1(T)+1}^{k_2(T)}
     \frac{D_{\phi_{a^k}(Y),\cS^*}^{a^k\rho}(B,T)}{e^{{\d_\cS} T}}\\
&\le \sum_{k=q+1}^{[Q_2e^{p_a(T-T_1)}]}
    \frac{D_{X_{i(a)},\cS^*}^{a\rho}\(\log Q_2+\log\diam\(\phi_{a^k}(Y)\)+T\)}{\exp\({\d_\cS}\(\log Q_2+\log\diam\(\phi_{a^k}(Y)\)+T\)\)}Q_2^{\d_\cS}\diam^{\d_\cS}\(\phi_{a^k}(Y)\) + \\
& \  \  \  \  \  \  \  \  \  \  \  \  \  \  \  \ + \sum_{k=k_1(T)+1}^{k_2(T)}
  \!\!\!\!\frac{D_{X_{i(a)},\cS^*}^{a\rho}(T_1)}{e^{{\d_\cS} T_1}}e^{{\d_\cS}(T_1-T)} \\
&\le Q_2^{\d_\cS} \sum_{k=q+1}^{k_1(T)}
\(C_{a\rho}^{\cS^*}(X_{i(a)})+1\)\diam^{\d_\cS}\(\phi_{a^k}(Y)\)
 +\(C_{a\rho}^{\cS^*}(X_{i(a)})+1\)e^{{\d_\cS}(T_1-T)}k_2(T) \\
&\le Q_2^{\d_\cS} \sum_{k=q+1}^\infty
\(C_{a\rho}^{\cS^*}(X_{i(a)})+1\)\diam^{\d_\cS}\(\phi_{a^k}(Y)\)
 +\(C_{a\rho}^{\cS^*}(X_{i(a)})+1\)e^{{\d_\cS}(T_1-T)}k_2(T).
\end{aligned}         
\eeq
Denote by $\Sg_1(q,T)$ the maximum over all $a\in\Om_\rho$ of the first term in the last line of the above formula and by $\Sg_2(T)$ the second term. Because we are assuming (1) or (2), we have that in either case
\beq\label{3_2017_03_13}
\lim_{q\to\infty}\Sg_1(q,T)=0 \  \  {\rm and}  \  \ \lim_{T\to\infty}\Sg_2(T)=0 .
\eeq
Keeping $q\ge 1$ fixed, inserting \eqref{4dp15} to \eqref{1dp14}, and applying Theorem~\ref{t1da7}, we obtain
$$
\begin{aligned}
\varlimsup_{T\to\infty}&\frac{D_Y^\rho(B,T)}{e^{{\d_\cS} T}}\le \\
&\le \varlimsup_{T\to\infty}\frac{D_{Y,\cS^*}^\rho(B,T)}{e^{{\d_\cS} T}}+\sum_{a\in \Om_\rho}\sum_{k=1}^q\varlimsup_{T\to\infty}
  \frac{D_{\phi_{a^k}(Y),\cS^*}^{a^k\rho}(B,T)}{e^{{\d_\cS} T}}
 +\#\Om_\rho\(\Sg_1(q,T)+\Sg_2(T)\) \\
&\le \lt(C_\rho^{\cS^*}(Y)+\sum_{a\in \Om_\rho}
     \sum_{k=1}^qC_{a^k\rho}^{\cS^*}(\phi_{a^k}(Y))\rt) m_{\d_\cS}(B)
  +\#\Om\(\Sg_1(q,T)+\Sg_2(T)\).
\end{aligned} 
$$
Therefore, invoking \eqref{3_2017_03_13}, we obtain by letting $q\to\infty$, that
$$
\varlimsup_{T\to\infty}\frac{D_Y^\rho(B,T)}{e^{{\d_\cS} T}}
\le \lt(C_\rho^{\cS^*}(Y)+\sum_{a\in \Om_\rho}
 \sum_{k=1}^\infty C_{a^k\rho}^{\cS^*}(\phi_{a^k}(Y))\rt) m_{\d_\cS}(B).
$$
Along with \eqref{2dp14} this shows that formula \eqref{1dp13} holds. The number
$$
C_\rho^{\cS^*}(Y)+\sum_{a\in \Om_\rho}
 \sum_{k=1}^\infty C_{a^k\rho}^{\cS^*}(\phi_{a^k}(Y))
$$
is finite because of \eqref{1_2017_03_13}. Invoking also the sentence following this formula, we conclude the proof in the case of words $\rho\in E_{*A^*}^\N$.

\sp 
Now, we pass to the general case, i.e., all we assume is that $\rho \in E_A^{\mathbb N}$. For every $k \geq 1$ choose $\rho^{(k)}\in E_{*A^*}^\N$ such that $$
\rho^{(k)}|_k= \rho|_k.
$$
We already know that there exists a constant $M \geq 1$ such that 
$$
M^{-1}\leq C_Y(\rho^{(k)}) \leq M
$$
for all integers $k\ge 1$. So, passing to a subsequence, we may assume without loss of generality that the limit
$$
\lim_{k \to +\infty} C_Y(\rho^{(k)})
$$
exists and belongs to the interval $[M^{-1},M]$. We denote this limit by $C_Y(\rho)$.

Assume first that $B \subseteq X$ is an open set.   
In order to emphasize the openness of the set $B$ and in order to clearly separate the present setup from the next one, we now denote $B$ by $V$.
Fixing $\e> 0$ there then exist $F$, a compact subset of $V$ and a number $r(\e) > 0$ such that 
\begin{equation}
\label{1ma14}
m_{\d_\cS}(V \backslash F) < \e \  \  \hbox{ and } \ \ 
m_{\d_\cS}(B(V, r(\e)) \backslash V) < \e
\end{equation}
and 
\begin{equation}
\label{1ma14.1}
m_{\d_\cS}(\partial F)= 0  \  \  \hbox{ and } \  \
m_{\d_\cS}(\partial B(V, r(\e)))= 0,
\end{equation}
where in writing the latter of these four requirements we used the fact that $m_{\d_\cS}(\bd V)=0$. Hence there exists $k\ge 1$ so large that for every $\om\in E_\rho$ (simultaneously meaning that $\om\in E_{\rho_k}$, we have that
$$
\phi_{\omega}\(\pi_{\mathcal S}(\rho^{(k)})\) \in F_\e\  \
\implies \  \  \phi_\omega(\pi_{\mathcal S}(\rho) \in V)
$$
and 
$$
\phi_{\omega}\(\pi_{\mathcal S}(\rho)\) \in V \  \
\implies \  \  \phi_\omega\(\pi_{\mathcal S}(\rho^{(k)})\) \in B(V, r(\e))).
$$
Therefore, for every $T > 0$,
$$
\mathcal D_Y^{\rho^{(k)}} (F_\e, T)  
\subseteq
\mathcal D_Y^{\rho} (V, T)
\subseteq
\mathcal D_Y^{\rho^{(k)}}(B(V, r(\e)), T)
$$
so,
$$
 D_Y^{\rho^{(k)}} (F_\e, T)  
\le
 D_Y^{\rho} (V, T)
\le
 D_Y^{\rho^{(k)}}(B(V, r(\epsilon)), T).
$$
Hence, applying the already proven assertion for words in $E_{*A^*}^\infty$
one gets 
$$
\begin{aligned}
C_{\rho^{(k)}}(Y) m_{\d_\cS}(F_\epsilon)
& =
\lim_{T \to +\infty}
\frac{D_Y^{(\rho^{(k)})}  (F_\epsilon, T)}{e^{{\d_\cS} T}}
\leq
\liminf_{T \to +\infty}
\frac{D_Y^{(\rho)}  (V, T)}{e^{{\d_\cS} T}}
\leq
\limsup_{T \to +\infty}
\frac{D_Y^{(\rho)}  (V, T)}{e^{{\d_\cS} T}} \cr
& \leq \lim_{T \to +\infty}
\frac{D_Y^{\rho^{(k)}} (B(V, r(\epsilon), T))}{e^{{\d_\cS} T}} 
= C_{\rho^{(k)}}(Y) m_{\d_\cS}(B(V, r(\epsilon))).
\end{aligned}
$$
So, letting $k \to +\infty$ and invoking \eqref{1ma14.1} we obtain that 
$$
C_\rho(Y)m_{\d_\cS}(F_\epsilon)
\leq
\liminf_{T \to +\infty}
\frac{D_Y^\rho (V, T)}{e^{{\d_\cS} T}}
\leq \limsup_{T \to +\infty}
\frac{D_Y^\rho (V, T)}{e^{{\d_\cS} T}}
\leq C_\rho(Y)
m_{\d_\cS}(B(V, r(\epsilon))).
$$
Hence, letting $\epsilon \to 0$ and invoking \ref{1ma14} we get that
$$
C_\rho(Y)m_{\d_\cS}(V)
\leq
\liminf_{T \to +\infty}
\frac{D_Y^\rho (V, T)}{e^{{\d_\cS} T}}
\leq \limsup_{T \to +\infty}
\frac{D_Y^\rho (V, T)}{e^{{\d_\cS}T}}
\leq C_\rho(Y)
m_{\d_\cS}(V),
$$
and the theorem is fully proved for all open sets $B$.
Having shown this, the general case can be taken care of in exactly the same way as the part of the proof of Theorem~\ref{dyn}, starting right after formula
(\ref{1nh24}).  This completes the proof.
\epf

As in Section~\ref{contracting_diameters} we can now, 
in the context of asymptotics of diameters, present an asymptotic formula.
 The appropriate definitions in the parabolic setting are the same as in the attracting one, and we briefly recall them now. Given a set $B\sbt X$, we define:
$$
\mathcal E_Y^\rho(B,T) := \{ \omega \in E_\rho^* \hbox{ : } \Delta(\omega) 
\leq T \ \hbox{ and } \  \phi_\omega (Y) \cap B \neq \emptyset\}
$$
and 
$$
E_Y^\rho(B,T) := \#\mathcal E_Y^\rho(B,T).
$$

Having established 
Theorem~\ref{t1dp13}
in a similar way to the way that 
Theorem~\ref{t1ma1} was based on Theorem~\ref{t1da7}, gives us the following.

\bthm[Asymptotic Equidistribution Formula of Diameters for Parabolic Systems, II]\label{t1dp13B}
Suppose that $\cS$ is a finite irreducible parabolic conformal GDMS. Fix $\rho\in E_A^\infty$ and $Y\sbt X_{i(\rho)}$ having at least two points and such that $\pi_\cS(\rho)\in Y$. If $B \subset X$ is a Borel set such that $m_{\d_\cS}(\bd B)=0$ (equivalently $\mu_{\d_\cS}(\bd B)=0$) then,
\beq\label{1dp13B}
\lim_{T \to +\infty} \frac{E^{\rho}_Y(B,T)}{e^{{\d_\cS} T}} 
=C_\rho(Y)m_{\d_\cS}(B),
\eeq
where $C_\rho(Y)\in (0,+\infty)$ is a constant (the same as that of Theorem~\ref{t1dp13}) depending only on the system $\cS$, the word $\rho$ (but see Remark~\ref{r1_2017_03_20B}), and the set $Y$. In addition $C_\rho(Y)$ is finite if and only if 
either

\sp\begin{enumerate}
\item 
$$
\ov Y\cap \Om_\infty=(\ov Y\cap \Om_\infty\cap\Om_\rho)=\es
$$
or 
\item 
$$
{\d_\cS}>\max\big\{p(a):a\in\Om_\rho \  \  {\rm and} \  \  x_a\in \ov Y \big\}.
$$
\end{enumerate}
Then the function $[\rho_1]\ni \om\longmapsto C_\om(Y)$ is uniformly bounded  away from zero and bounded above.
\ethm

\brem\label{r1_2017_03_20B}
We now can 
essentially 
repeat Remark~\ref{r1_2017_03_20}  
verbatim with the  only change being the replacement of Theorem~\ref{t1ma1} and Theorem~\ref{t1da7}, respectively,  by Theorem~\ref{t1dp13B} and Theorem~\ref{t1dp13}. For the sake of completeness, convenience of the reader, and ease of referencing we summarise:

Since the left-hand side of \eqref{1dp13B} depends only on $\rho_1$, i.e. the first coordinate of $\rho$, we obtain that the constant $C_Y(\rho)$ of Theorem~\ref{t1dp13B} and Theorem~\ref{t1dp13}, depends in fact only on $\rho_1$.
Again, we could have provided a direct argument for this already when proving Theorem ~\ref{t1dp13} and  this would not affect the proof of Theorem ~\ref{t1dp13B}. Thus our approach  seems most economical.
\erem 

The  last three results of this section are derived from the, already established, results, in the same way as the last three results of Section~\ref{contracting_diameters} were derived from the earlier results of that section.
 
\bthm\label{t1da12.1T}
Suppose that $\cS$ is a finite irreducible parabolic conformal GDMS with property (A). For any $v\in V$ let $Y_v\sbt X_v$ having at least two points. If $B \subset X$ is a Borel set such that $\^m_{\d_\cS}(\bd B)=0$ (equivalently $\^\mu_{\d_\cS}(\bd B)=0$) and $\rho\in E_A^\infty$ is with $\rho_1=a_v$, then,
\beq\label{3_2016_01_30T}
\lim_{T \to +\infty} \frac{D^{\rho}_Y(B,T)}{e^{{\d_\cS} T}} 
=\lim_{T \to +\infty} \frac{E^{\rho}_Y(B,T)}{e^{{\d_\cS} T}}
=C_v(Y_v)\^m_{\d_\cS}(B),
\eeq
where $C_v(Y_v)\in (0,+\infty]$
is a constant depending only on the vertex $v\in V$ and the set $Y_v$. In particular, this holds for $Y_v:=X_v$, $v\in V$. In addition $C_v(Y)$ is finite if and only if 
either

\sp\begin{enumerate}
\item 
$$
\ov Y\cap \Om_\infty=(\ov Y\cap \Om_\infty\cap\Om_{a_v})=\es
$$
or 
\item 
$$
{\d_\cS}>\max\big\{p(a):a\in\Om_{a_v} \  \  {\rm and} \  \  x_a\in \ov Y \big\}.
$$
\end{enumerate}
\ethm

\bcor\label{t1da12.1I}
Suppose that $\cS$ is a finite irreducible maximal parabolic conformal GDMS. For any $v\in V$ let $Y_v\sbt X_v$ having at least two points be fixed. If $B \subset X$ is a Borel set such that $\^m_{\d_\cS}(\bd B)=0$ (equivalently $\^\mu_{\d_\cS}(\bd B)=0$) and $\rho\in E_A^\infty$ is with $i(\rho_1)=v$, then,
\beq\label{3_2016_01_30R}
\lim_{T \to +\infty} \frac{D^{\rho}_Y(B,T)}{e^{{\d_\cS} T}} 
=\lim_{T \to +\infty} \frac{E^{\rho}_Y(B,T)}{e^{{\d_\cS} T}}
=C_v(Y_v)\^m_{\d_\cS}(B),
\eeq
where $C_v(Y_v)\in (0,+\infty)$
is a constant depending only on the vertex $v\in V$ and the set $Y_v$. In particular, this holds for $Y_v:=X_v$, $v\in V$. In addition $C_v(Y)$ is finite if and only if 
either

\sp\begin{enumerate}
\item 
$$
\ov Y\cap \Om_\infty=(\ov Y\cap \Om_\infty\cap\Om_v)=\es
$$
or 
\item 
$$
{\d_\cS}>\max\big\{p(a):a\in\Om_v \  \  {\rm and} \  \  x_a\in \ov Y \big\}.
$$
\end{enumerate}
\ecor

\bcor\label{c1da12.1J}
Suppose that $\cS$ is a finite conformal parabolic IFS acting on a phase space $X$. Fix $Y\sbt X$ having at least two points. If $B \subset X$ is a Borel set such that $\^m_{\d_\cS}(\bd B)=0$ (equivalently $\^\mu_{\d_\cS}(\bd B)=0$) and $\rho\in E_A^\infty$, then,
\beq\label{3_2016_01_30S}
\lim_{T \to +\infty} \frac{D^{\rho}_Y(B,T)}{e^{{\d_\cS} T}} 
=\lim_{T \to +\infty} \frac{E^{\rho}_Y(B,T)}{e^{{\d_\cS} T}}
=C(Y)\^m_{\d_\cS}(B),
\eeq
where $C(Y)\in (0,+\infty)$
is a constant depending only on the set $Y$. In particular, this holds for $Y:=X$. In addition $C(Y)$ is finite if and only if 
either

\sp\begin{enumerate}
\item 
$$
\ov Y\cap \Om_\infty=\es
$$
or 
\item 
$$
{\d_\cS}>\max\big\{p(a):a\in\Om \  \  {\rm and} \  \  x_a\in \ov Y \big\}.
$$
\end{enumerate}
\ecor

\part{{\Large Central Limit Theorems}}\label{Section-CLT}
We now consider  the distribution of weights and the Central Limit Theorems.   
In this section we will formulate the results in full generality and 
give the applications in subsequent sections.
 
Let us consider a conformal, either attracting or parabolic, GDMS. As we did in previous sections,
we can associate to finite words $\omega \in E_A^*$ both the  weights $\lambda_i(\omega)$ ($i=p, \rho$) and the word length 
$|\omega|$.  We would like to understand how  these quantities are related for typical orbits, which leads naturally to the study of Central Limit Theorems.
The most familiar and natural formulation of Central Limit Theorems (CLT) is with respect to invariant measures. However, in the present context it is equally natural  to give versions for preimages and periodic points.

\section{Central Limit Theorems for Multipliers and Diameters: Attracting GDMSs with Invariant Measure $\mu_{\d_\cS}$}\label{subsection:CLT-hyperbolic}

As an immediate consequence of Theorem 2.5.4, Lemma 2.5.6, Lemma 4.8.8 from 
\cite{MU_GDMS}, and Remark \ref{r2_2017_02_17} from our present monograph, we get the following version of the Central Limit Theorem for attracting systems and Gibbs/equilibrium states.

\begin{thm}\label{t1ms1}  
If $\cS$ is a strongly regular finitely irreducible
D--generic conformal GDMS \footnote{In fact $\mu_{\d_\cS}$ below can be replaced by the (unique) Gibbs/equilibrium state of any H\"older continuous summable potential $f:E_A^\infty\to\R$.},
then there exists $\sigma^2 > 0$ (in fact $\sg^2=\P''(0) \neq  0$  because of Remark~\ref{r2_2017_02_17} and since the system $\cS$ is D--generic) such that if $G \subset \mathbb R$
is a Lebesgue measurable set with $\hbox{{\rm Leb}}(\partial G) = 0$, then 
$$
\lim_{n \to +\infty}
\mu_{\d_\cS}\left(
\left\{
\omega \in E_A^\infty \hbox{ : } \frac{-\log \big|\phi_{\omega|_n}' (\pi_{\cS}(\sigma^n(\omega)) )\big| - \chi_{\mu_{\d_\cS}} n}{\sqrt{n}}
\in G
\right\}
\right)
= \frac{1}{\sqrt{2\pi}\sigma} \int_G e^{-\frac{t^2}{2\sigma^2}} \,dt.
$$  
In particular, for any $\alpha <\beta$
$$
\lim_{n \to +\infty}
\mu_{\d_\cS}\left(
\left\{
\omega \in E_A^\infty \hbox{ : } \alpha \leq \frac{-\log\big|\phi_{\omega|_n}' (\pi_{\cS}(\sigma^n(\omega)))\big| - \chi_{\mu_{\d_\cS}} n}{\sqrt{n}}
\leq \beta
\right\}
\right)
= \frac{1}{\sqrt{2\pi}\sigma} \int_\a^\b e^{-\frac{t^2}{2\sigma^2}}\, dt.
$$  
\end{thm}

Since by the Bounded Distortion Property (BDP) of the definition of attracting GDMSs, the numbers 
$$
\big|\log \hbox{\rm diam}\(\phi|_{\omega|_n}(Y_{t(\omega)})\)
-
\log |\phi'_{\omega|_n} (\pi_{\cS}(\sigma^n(\omega))|\big|
$$
are uniformly bounded above and since $\lim_{n \to +\infty} \sqrt{n} = +\infty$ we immediately obtain a version of Theorem~\ref{t1ms1} with 
$-\log \big|\phi_{\omega|_n}'(\pi_{\cS}(\sigma^n(\omega)))\big|$
replaced by 
$-\log \hbox{\rm diam}\(\phi|_{\omega|_n}(Y_{t(\omega)})\)$.
This gives the following.

\begin{thm}\label{t1ms1-again}
Suppose that $\cS$ is a strongly regular finitely irreducible D--generic conformal GDMS\footnote{In fact $\mu_{\d_\cS}$ below can be replaced by the (unique) Gibbs/equilibrium state of any H\"older continuous summable potential $f:E_A^\infty\to\R$.}. Let $\sigma^2:=\P''(0)(\neq  0)$. For every $v \in V$ let $Y_v \subset X_v$ be a set with at least two points. If $G \subset \mathbb R$ is a Lebesgue measurable set with $\hbox{\rm Leb}(\partial G) = 0$, then 
$$
\lim_{n \to +\infty}
\mu_{\d_\cS}\left(
\left\{
\omega \in E_A^\infty \hbox{ : }
 \frac{
 -\log \hbox{\rm diam}(\phi_{\omega|_n}(Y_{t(\omega_n)}))  - \chi_{\mu_{\d_\cS}} n
 }{\sqrt{n}}
\in G
\right\}
\right)
= \frac{1}{\sqrt{2\pi}\sigma} \int_G e^{-\frac{t^2}{2\sigma^2}}\, dt.
$$  
In particular, for any $\alpha < \beta$
$$
\lim_{n \to +\infty}
\mu_{\d_\cS}\left(
\left\{
\omega \in E_A^\infty \hbox{ : } \alpha \leq 
 \frac{-\log \hbox{\rm diam}(\phi_{\omega|_n}(Y_{t(\omega_n)}))  - \chi_{\mu_{\d_\cS}} n}{\sqrt{n}}
\leq \beta
\right\}
\right)
= \frac{1}{\sqrt{2\pi}\sigma} \int_\alpha^\beta e^{-\frac{t^2}{2\sigma^2}}\,dt.
$$  
\end{thm}

\fr Also, as an immediate consequence of the appropriate results from \cite{MU_GDMS} and Remark \ref{r2_2017_02_17} from our present monograph, we get the following Law of Iterated Logarithm.

\begin{thm}\label{t1ms1-2again}
Suppose that there $\cS$ is a strongly regular finitely irreducible D--generic conformal GDMS\footnote{In fact $\mu_{\d_\cS}$ below can be replaced by the (unique) Gibbs/equilibrium state of any H\"older continuous summable potential $f:E_A^\infty\to\R$.}. Let $\sigma^2:=\P''(0)>0$. For every $v \in V$ let $Y_v \subset X_v$ be a set with at least two points. 
Then for $\mu_{\d_\cS}$--a.e. $\omega \in E_A^\infty$, we have that
$$
\limsup_{n \to +\infty}
\frac{-\log\big|(\phi_{\omega|_n}'(\pi_{\cS}(\sigma^n(\omega)))\big|  - \chi_{\mu_{\d_\cS}} n}{\sqrt{n\log\log n}}
= \sqrt{2\pi} \sigma
$$
and 
$$
\limsup_{n \to +\infty}\frac{-\log \hbox{\rm diam}\(\phi_{\omega|_n}(Y_{t(\omega_n)})\)  - \chi_{\mu_{\d_\cS}}n}{\sqrt{n\log\log n}}
= \sqrt{2\pi} \sigma.
$$
\end{thm}

\

\begin{rem}\label{r4_2017_02_17}
It is  possible to  reverse the roles of  the word length and the weights.
More precisely, given $\omega \in E_A$ and $t\ge 0$ we can define $n = n(t,\omega)$ to be the only integer for which
$$
\lambda(\omega|_n) \leq t < \lambda(\omega|_{n+1}).
$$
Ergodicity of measure $\mu_{\d_\cS}$ and Birkhoff's Ergodic Theorem then yield
$$
\lim_{t \to + \infty}\frac{t}{n(t,\omega)}=\chi_{\mu_{\d_\cS}}
$$
for $\mu_{\d_\cS}$--a.e. $\omega \in E_A^\infty$. We claim that there exists $\sigma_0^2 > 0$ such that for any $\a<\b$
$$
\lim_{t \to +\infty}\mu_{\d_\cS}\left(\left\{\omega \in E_A^\infty \hbox{ : } 
\alpha\le \frac{\lambda(\omega|_{n(t,\omega)}) - \chi_{\mu_{\d_\cS}}t }{\sqrt{t}} \le \beta
\right\}\right)
=\frac{1}{\sqrt{2\pi} \sigma}\int_\a^\b e^{- u^2/2\sigma_0^2}\, du.
$$
This is obtained by reinterpreting an approach of Melbourne and T\"orok, originally applied in the case of suspended flow \cite{MT}.
In particular, they showed that if a discrete system 
  satisfies a central limit theorem with variance $\sigma^2$, then a suitable  suspension flows also satisfy the CLT.
\footnote{There is a mild hypothesis on the roof function $r$ which is satisfied if $r \in L^4$, say. This is the case in our present context.} 
   In the present case one takes $\sigma: E_A \to E_A$ as the discrete transformation and a roof function $r: E_A \to \mathbb R$ defined by
    $r= -\log\big|\phi_{\omega_1}'(\pi_\cS(\sg(\omega)))\big|$. 
 For the suspension space 
 $E_A^r = \{(\omega, u) \hbox{ : } 0 \leq u \leq r(\omega)\}$   
 with the identifications 
 $(\omega, r(\omega)) \sim (\sigma \omega, 0)$
 one can consider the suspension flow $\sigma^r_t:  E_A^r
 \to  E_A^r$ defined by $\sigma^r_t(\omega, u) = (\omega, u+t)$, up to the identifications.
 We can associate to the $\sigma$-invariant probability measure a $\phi$-invariant probability measure 
 $\widehat \mu_\sigma$ defined by 
 $d \widehat \mu_\sigma =  d\mu_\sigma \times dt/\int r d\mu_{\d_\cS}$.
 Given a function $F: E_A^r \to \mathbb R$ the CLT for the flow gives that  
$$
 \lim_{t\to +\infty}\widehat \mu_{\d_\cS}
 \left( \left\{
 (\omega, u) \in E_A^r \hbox{ : } 
 \a \leq 
 \frac{\int_0^t F\circ \phi_s (\omega, u) ds - t \int d\hat \mu_{\d_\cS}}{\sqrt{t}}
 \leq \b\right\}
 \right)=\frac{1}{\sqrt{2\pi} \sigma}\int_\a^\b e^{- u^2/2\sigma_1^2}\, du,
$$
where $\sigma_1^2 = \sigma_0^2/\chi_{\mu_{\d_\cS}}$ cf. \cite{MT}, \S3.We would like to choose $F$ so that $\int_0^t F\circ \phi_s (\omega, u) ds$
 corresponds to $\lambda(\omega|_{n(t, \omega)})$.
     To this end one chooses a function $F$ which integrates to unity on fibers, i.e., $\int_0^{r(\omega)} F(\omega, u) du =1$
     for all $\omega \in \Sigma_A$, and has support close to $E_A\times \{0\}$.
      Thus the Central Limit Theorem for the suspension  flow
        corresponds 
      to the 
Central Limit Theorem formulated above in $t$.     
       The variances are related by a factor of 
       $\int r d\mu_{\d_\cS}$.
        
\end{rem}

We now turn the the parabolic setting.

\section{Central Limit Theorems for Multipliers and Diameters: Parabolic GDMSs with Invariant Measure $\mu{\d_\cS}$}\label{CLT-Parabolic-1}
We want to consider analogous comparison results in the context of 
 parabolic GDMSs.
Following the approach described in Section~\ref{section:parabolic}, given  a parabolic conformal GDMS $S$ we can associate a 
conformal GDMS $S^*$.  In this case the Central Limit Theorem for the 
 measure $\mu_{\d_\cS}^*$ associated to  $\mathcal S^*$ translates into a Central Limit Theorem for the parabolic system $\mathcal S$ and its measure $\mu_{\d_\cS}$.
This leads to the following results, the  first of which  is the analogue of Theorem~\ref{t1ms1}.

\begin{thm}\label{t1ms1-again2} 
If $\cS$ is a finitely irreducible parabolic conformal GDMS with ${\d_\cS} > \frac{2p}{p+1}$, then 
there exists $\sigma^2 > 0$ such that if $G \subset \mathbb R$
is a Lebesgue measurable set with $\hbox{{\rm Leb}}(\partial G) = 0$, then 
$$
\lim_{n \to +\infty}
\mu_{\d_\cS}\left(
\left\{
\omega \in E_A^\infty \hbox{ : } \frac{-\log \big|\phi_{\omega|_n}' (\pi_{\cS}(\sigma^n(\omega)) )\big| - \chi_{\mu_{\d_\cS}} n}{\sqrt{n}}
\in G
\right\}
\right)
= \frac{1}{\sqrt{2\pi}\sigma} \int_G e^{-\frac{t^2}{2\sigma^2}} \,dt.
$$  
In particular, for any $\alpha <\beta$
$$
\lim_{n \to +\infty}
\mu_{\d_\cS}\left(
\left\{
\omega \in E_A^\infty \hbox{ : } \alpha \leq \frac{-\log\big|\phi_{\omega|_n}' (\pi_{\cS}(\sigma^n(\omega)))\big| - \chi_{\mu_{\d_\cS}} n}{\sqrt{n}}
\leq \beta
\right\}
\right)
= \frac{1}{\sqrt{2\pi}\sigma} \int_\a^\b e^{-\frac{t^2}{2\sigma^2}}\, dt.
$$  
\end{thm}

\begin{proof}
By Theorem~\ref{t1_2017_02_18}, the hypothesis that ${\d_\cS} > \frac{2p}{p+1}$ precisely means that measure $\mu_{\d_\cS}$ is finite, and, as always, we normalize it to be a probability measure. Because of Theorem~\ref{t1pc3} and Remark~\ref{r2_2017_02_17} Theorem~\ref{t1ms1-again2} then is a standard consequence of L. S. Young's tower approach \cite{lsy1}, \cite{lsy2}, comp. \cite{gouezel1}, \cite{gouezel1}, and \cite{gouezel1}.
\end{proof} 

The second result is the parabolic analogue of Theorem~\ref{t1ms1-again}.

\begin{thm}\label{t1ms1-again3} 
Let $\cS$ be a finite alphabet irreducible parabolic GDMS with ${\d_\cS} > \frac{2p}{p+1}$. 
Then there exists $\sigma^2>0$ such that if for every $v \in V$, a set $Y_v \subset X_v$ is given having at least two points and whose closure is disjoint from the set of parabolic fixed points $\Om$, then for every Lebesgue measurable set $G \subset \mathbb R$ with $\hbox{\rm Leb}(\partial G) = 0$, we have that
$$
\lim_{n \to +\infty}
\mu_{\d_\cS} \left(
\left\{
\omega \in E_A^\infty \hbox{ : }
 \frac{
 -\log \hbox{\rm diam}(\phi_{\omega|_n}(Y_{t(\omega_n)}))  - \chi_{\mu_t} n
 }{\sqrt{n}}
\in G
\right\}
\right)
= \frac{1}{\sqrt{2\pi}\sigma} \int_G e^{-t^2/2\sigma} dt.
$$  
In particular, for any $\alpha < \beta$
$$
\lim_{n \to +\infty}
\mu_{\d_\cS} \left(
\left\{
\omega \in E_A^\infty \hbox{ : } \alpha \leq 
 \frac{-\log \hbox{\rm diam}(\phi_{\omega|_n}(Y_{t(\omega_n)}))  - \chi_{\mu_t} n}{\sqrt{n}}
\leq \beta
\right\}
\right)
= \frac{1}{\sqrt{2\pi}\sigma} \int_\alpha^\beta e^{-t^2/2\sigma} dt.
$$  
\end{thm}

\begin{proof}
Because of Theorem~\ref{t1ms1-again2}, it suffices to show that 
$$
\lim_{n \to +\infty}
\mu_{\d_\cS}\left(\{\omega \in E_A^\infty \hbox{ : } 
\big| \log \hbox{\rm diam} (\phi_{\omega|_n}(Y_{t(\omega_n)})\big|
- \log\big|\phi_{\omega|_n}'(\pi_{\mathcal S}(\sigma^n(\omega))
\big| \geq n^{1/4}\}
\right) = 0.
$$
To show this, write 
$$
g_n(\omega):= 
\big| \log \hbox{\rm diam} (\phi_{\omega|_n}(Y_{t(\omega_n)})\big|
- \log\big|\phi_{\omega|_n}'(\pi_{\mathcal S}(\sigma^n(\omega))
\big|.
$$
Since the set $E_A^{\mathbb N} \sms E_{A^*}^{\mathbb N}$ is countable
and the measure $\mu_{\d_\cS}$ is atomless, it suffices to deal with the elements of 
$E_A^{\mathbb N} \backslash E_{A^*}^{\mathbb N}$ only. Each such element 
$\omega$ has a unique representation in the form 
$$
\omega = \tau a^j \sigma^n(\omega),
$$
where $\tau\in E^*_{*, A^*}$, $a \in \Omega$
and $j = j(\omega) \in \{0,1, \cdots, n - |\tau|\}$.
Then for every $n \geq 0$ either 
$$
\hbox{\rm diam}(\phi_{\omega|_n} (Y_{t(\omega_n)}))
\asymp
\|
\phi_\tau'
\| (j+1)^{-1/p_a}
\  \  \hbox{ or } \  \
\hbox{\rm diam}(\phi_{\omega|_n} (Y_{t(\omega_n)}))
\asymp
\|
\phi_\tau'
\| (j+1)^{(p_a+1)/p_a},
$$
respectively,  depending on whether $a \in \overline Y_{t(\omega_n)}$ or not.
In either case
$$
\hbox{\rm diam}(\phi_{\omega|n} (Y_{t(\omega_n)}))
\asymp
\|
\phi_\tau'
\| (j+1)^{-\alpha}
$$
where $\alpha \in\{ 1/p_a,(p_a+1)/p_a \}$.
Since $\omega \in E_A^{\mathbb N}\sms E_{A^*}^{\mathbb N}$, there exists a largest (finite) $k \geq 0$ such that 
$$
\omega \in [\tau a^{j+k}].
$$
Then
$$
\big|\phi'_{\omega|_n}\(\pi_{\mathcal S}(\sigma^n(\om)\)\big|
\asymp
\|\phi_\tau'\|
(j+k+1)^{-\frac{p_a+1}{p_a}}(k+1)^{\frac{p_a+1}{p_a}}.
$$
Hence 
$$
g_n(\omega) 
\leq \frac{p_a+1}{p_a} \big(\log(k+1) + \log(j+k+1) + 
\alpha \log(j+1) + \Gamma_+\big)
\le \Ga\log(j+k+1)
$$
where $\Gamma_+\in [0, +\infty)$ and $\Gamma \in [1, +\infty)$ are some universal  constants independent of $\omega$ and $n$. Then
$$
\int_{\omega \in E_{A^*}^\infty:j(\omega) = j}
 g_n(\omega) d\mu_{\d_\cS}(\omega)
\leq \Sigma_j := \Ga
 \sum_{\tau \in E_A^{n-j}(*)}
 \sum_{a \in \Omega} \sum_{b \neq a\atop A_{ab}=1}
 \sum_{k=0}^\infty
 \log (j+k+1) \mu_{\d_\cS}([\tau a^{j+k}b]),
$$
where $E_A^{n-j}(*)$ denotes the set of all finite words of 
``real" length $n-j$ that belong to $E^*_{*A*}$.  
Now represent each element $\tau \in E_A^{n-j}(*)$ uniquely 
as $c^l d\gamma$, where $l \geq 0$, $c \in \Omega$, $d \neq c$.
Then both $c^ld$ and $\gamma$ belong to $E^*_{*A^*}$, and we can write
$$
\Sigma_j =  \Ga\sum_{c\in \Omega}
\sum_{d \neq c\atop A_{cd}=1}
\sum_{\gamma \in E^*_{* A^*}}
\sum_{l=0}^{n-j -1}
\sum_{a\in \Omega}
\sum_{b \neq a\atop A_{ab}=1}
\sum_{k=0}^\infty \log (j+k+1)  \mu_{\d_\cS}([c^l d\gamma a^{j+k}b]).
$$
Now since the Radon-Nikodym derivative $\frac{d\mu_{\d_\cS}}{dm_{\d_\cS}}$ is comparable to $l+1$ on $c^ld$ and since the three words
$c^ld$, $\gamma$ and $a^{j+k}b$, belong to $E_{*A^*}$, we obtain
$$
\begin{aligned}
\Sigma_j
&\asymp
\sum_{c\in \Omega}
\sum_{d \neq c\atop A_{cd}=1}
\sum_{\gamma \in E^*_{* A^*}\atop d\gamma \in E_A^{n-k-l-1}}
\sum_{l=0}^{n-j -1}
\sum_{a\in \Omega}
\sum_{b \neq a\atop A_{ab}=1}
\sum_{k=0}^\infty \log (j+k+1) (l+1)m_{\d_\cS}([c^l d\gamma a^{j+k}b])\cr
&\asymp
\sum_{c\in \Omega}
\sum_{d \neq c\atop A_{cd}=1}\!
\sum_{\gamma \in E^*_{* A^*}\atop d\gamma \in E_A^{n-k-l-1}}\!
\sum_{l=0}^{n-j -1}
\sum_{a\in \Omega}
\sum_{b \neq a\atop A_{ab}=1}
\sum_{k=0}^\infty \log (j+k+1) (l+1)  m_{\d_\cS}([c^l d])
m_{\d_\cS}([\gamma]) m_{\d_\cS}([a^{j+k}b])\cr
&\lek \sum_{c\in \Omega}\sum_{a\in \Omega}\sum_{l=0}^\infty \sum_{k=1}^\infty \log (j+k+1) (l+1)^{1-\frac{p_a+1}{p_a} {\d_\cS}} (j+k+1)^{-\frac{p_a+1}{p_a} {\d_\cS}}\cr
&\asymp\sum_{a\in \Omega}\sum_{k=0}^\infty \log (j+k+1)(j+k+1)^{-\frac{p_a+1}{p_a} {\d_\cS}}
\end{aligned}
$$
where the last comparability sign we wrote because $1- \frac{p_a+1}{p_a} {\d_\cS} < -1$ for all $c \in \Omega$.
Therefore,
$$
\begin{aligned}
\int_{E_A^\infty} g_n d\mu_{\d_\cS}
& = \sum_{j=0}^n 
\int_{\{\omega \in E_{A^*}^\infty:j(\omega) = j\}}g_n(\omega) d\mu_(\omega)\cr
& \lek \sum_{j=0}^\infty  \sum_{a\in \Omega} \sum_{k=1}^\infty \log (j+k) (j+k)^{-\frac{p_a+1}{p_a} {\d_\cS}}\cr
&\asymp D:= \sum_{j=0}^\infty  \sum_{k=1}^\infty \log (j+k) (j+k)^{-\frac{p+1}{p} {\d_\cS}} < +\infty,
\end{aligned}
$$
where, we recall, $p = \max\{ p_a \hbox{ : } a \in \Omega\}$
and the constant $D$ is finite since $\frac{p+1}{p} {\d_\cS}> 2$.
Therefore, Tchebschev's Inequality tells us that 
$$
\mu_{\d_\cS}\big(\{ \omega \in E_A^\infty \hbox{ : } g_n(\omega) \geq n^{1/4}\}\big)
\leq
\frac{\int_{E_A^\infty} g_n d\mu_{\d_\cS}}{n^{1/4}}
\leq D n^{-1/4},
$$
and the proof is complete. 
\end{proof}

\begin{rem}
There are a variety of even stronger results, e.g., Functional Central Limit Theorems and Invariance Principles, based on approximation by Brownian Motion, which should also hold with a little more work.    Similarly, there are other complementary results such as large deviation results.
\end{rem}

\begin{rem}
There are possible stronger results of other kinds as well. For example, in both the hyperbolic and parabolic settings there is the possibility of estimating error terms and obtaining local limit theorems as in \cite{Go-local} and \cite{Go}.
\end{rem}

\section{Central Limit Theorems: Asymptotic Counting Functions for Attracting GDMSs}
In this subsection we work in the setting of attracting GDMSs.
We again fix $\rho \in E_A^\infty$. For any $n \geq 1$ and $\omega \in E_\rho^n$ consider the weights 
$$
e^{-{\d_\cS} \lambda_\rho(\omega)} = |\phi_\omega'(\pi(\rho))|^{\d_\cS}.
$$ 
More precisely, for every set $H \subset E_\rho^n$, we define
\begin{equation}\label{1ms6}
\mu_n(H) := \frac{\sum_{\omega \in H} e^{-{\d_\cS} \lambda_\rho(\omega)}}{\sum_{\omega \in E^n_\rho} e^{-{\d_\cS} \lambda_\rho(\omega)}}
= \frac{\mathcal L_{\d_\cS}^n \1_{[H]}(\rho)}{\mathcal L_{\d_\cS}^n \1(\rho)}.
\end{equation}
Define the function  $\lambda: E_A^\infty \to \mathbb R$ by the  formula:
$$
\lambda(\omega) = -\log |\phi_{\omega_1}'(\sigma(\omega))|.
$$
In particular, for every $\tau \in E^*_\rho$, say $\tau \in E^n_\rho$,
$$
\lambda_{\rho}(\tau) = \sum_{j=0}^{n-1}\lambda\(\sg^j(\tau\rho)\).
$$
We first prove the following.

\begin{thm}\label{t1ms5.2} 
If $\cS$ is a finitely irreducible strongly regular conformal GDMS, then for every $\rho \in E_A^\infty$ we have that

\begin{equation}\label{lms5.2}
\lim_{n \to +\infty} \int_{E_\rho^n} \frac{\lambda_\rho}{n} d\mu_n 
= \chi_{\mu_{\d_\cS}}=\int_{E_\rho^\infty}\lambda d\mu_{\d_\cS}.
\end{equation}
\end{thm}
 
\begin{proof}
The idea of the proof is to represent the integral 
$$
\int_{E_\rho^n} \frac{\lambda_\rho}{n} d \mu_n
$$ 
as the ratio of (sums of) Perron--Frobenius operators, and then to use the spectral properties of the operator $\mathcal L_{\d_\cS}$.
However, there is a difficulty in such an approach which does not appear in the case of a  finite alphabet. The character of this difficulty is that 
although the function $\lambda:E_A^\infty \to \mathbb R$ is always H\"older continuous,  in the case of infinite alphabet it is unbounded.  The remedy comes from the fact that $\mathcal L_{\d_\cS}(\1)$ is a H\"older continuous bounded function. 
Beginning  the  proof, we have
$$
\begin{aligned}
\int_{E_\rho}^n \frac{\lambda_\rho}{n} d\mu_n
&= \frac{\frac{1}{n} \mathcal L_{\d_\cS}^n\(\sum_{j=0}^{n-1} \lambda\circ\sg^j\)(\rho)}{
\mathcal L_{\d_\cS}^n(\1)(\rho)}
= \frac{
\frac{1}{n} \sum_{j=0}^{n-1} \mathcal L_{\d_\cS}^n( \lambda \circ \sigma^j)(\rho)
}{
\mathcal L_{\d_\cS}^n(\1)(\rho)} \cr
&= 
 \frac{\frac{1}{n} \sum_{j=0}^{n-1} \mathcal L_{\d_\cS}^{n-j}\(\mathcal L_{\d_\cS}^j (\lambda \circ \sigma^j)\)(\rho)}{\mathcal L_{\d_\cS}^n(\1)(\rho)}\cr
& = 
 \frac{\frac{1}{n} \sum_{j=0}^{n-1} \mathcal L_{\d_\cS}^{n-j}\(\lambda 
 \mathcal L_{\d_\cS}^j\1\)(\rho)}{\mathcal L_{\d_\cS}^n(\1)(\rho)}\cr
 &
  = 
 \frac{\frac{1}{n} \sum_{j=0}^{n-1} \mathcal L_{\d_\cS}^{n-(j+1)}\(\mathcal L_{\d_\cS}(\lambda\mathcal L_{\d_\cS}^j\1)\)(\rho)}{\mathcal L_{\d_\cS}^n(\1)(\rho)}.
\end{aligned}
$$
Now a straightforward calculation based on the strong regularity of the system $\mathcal S$
shows that the H\"older norms of the functions $\mathcal L_{\d_\cS} (\lambda \mathcal L_{\d_\cS}^i\1)$, $i\ge 0$, are uniformly bounded above. With the fact that the sequence $(\mathcal L_{\d_\cS}^i g)_{i=0}^\infty$ converges uniformly (in fact exponentially fast) to $\int g dm_{\d_\cS} \psi_{\d_\cS}$ for every bounded H\"older continuous function 
$g: E_A^\infty \to \mathbb R$ we conclude that the sequence 
$(\mathcal L_{\d_\cS} (\lambda \mathcal L_{\d_\cS}^j\1))_{j=0}^\infty$ converges uniformly to $\mathcal L_{\d_\cS} (\lambda \psi_{\d_\cS})$.  So, fixing $\epsilon > 0$, we can find $k_1 \geq 1$ such that 
$$
\|
\mathcal L_{\d_\cS} (\lambda \mathcal L_{\d_\cS}^j\1) - \mathcal L_{\d_\cS} (\lambda \psi_{\d_\cS})
\|_\alpha \leq \epsilon
$$
for all $j \geq k_1$.  Furthermore, there exist $N \geq k_2 \geq k_1$ such that for all $n \geq N$ and all $j \leq n - k_2$,
$$
\Big\| \mathcal L_{\d_\cS}^{n-j} (\lambda \mathcal L_{\d_\cS}^j\1) - 
 \int \mathcal L_{\d_\cS} (\lambda\psi_{\d_\cS}) dm_{\d_\cS} \psi_{\d_\cS}
\Big\|_\alpha \leq \epsilon.
$$
But $ \int \mathcal L_{\d_\cS} (\lambda\psi_{\d_\cS}) dm_{\d_\cS}
=  \int \lambda\psi_{\d_\cS}  dm_{\d_\cS} = \int \lambda d\mu_{\d_\cS}
$ and $M := \sup\{ \|\mathcal L_{\d_\cS}^n \1\|_\alpha \hbox{ : } n \geq 0\}$
is finite. So we can conclude that 
$$
\Big\| \mathcal L_{\d_\cS}^{n-(j+1)} \mathcal L_{\d_\cS}(\lambda \mathcal L_{\d_\cS}^j\1) - 
 \int \lambda d\mu_{\d_\cS} \psi_{\d_\cS}
\Big\|_\alpha \leq (1+ M) \epsilon
$$
for all $n \geq N$ and all $k_1 \leq j \leq n-k_2$.
Hence 
$$
 \int \lambda d\mu_{\d_\cS}-   (1+ M) \epsilon
 \leq \liminf_{n \to +\infty} \int_{E_\cS^n} \frac{\lambda_{\d_\cS}}{n} d\mu 
\leq \limsup_{n \to +\infty} \int_{E_\cS^n} \frac{\lambda_{\d_\cS}}{n} d\mu     \leq \int \lambda d\mu_{\d_\cS} + (M+1)\epsilon. 
$$
Letting $\epsilon \to 0$, then concludes the proof.
\end{proof}

Now we are next going to prove  versions of the Central Limit Theorem (CLT) that involve counting.  This requires some preparatory steps.  

We define the functions $\Delta_n: E_\rho^n \to \mathbb R$ by the formulae
\begin{equation}\label{2m56}
\Delta_n(\omega) := \frac{\lambda_\rho (\omega) - \chi_{\d_\cS} n}{\sqrt{n}}
\end{equation}
and consider the sequence $(\mu_n \circ \Delta_n^{-1})_{n=1}^\infty$
of probability distributions on $\mathbb R$. Observe that for every Borel set $F \subset \mathbb R$, we have that 
\begin{equation}\label{3m56}
\begin{aligned}
\mu_n \circ \Delta_n^{-1}(F) =
 \frac{\mathcal L_{\d_\cS}^n \1_{[\Delta_n^{-1}(F)]}(\rho)}{\mathcal L_{\d_\cS}^n \1(\rho)}
& =
  \frac{\mathcal L_{\d_\cS}^n (\1_{F}\circ \Delta_n)(\rho)}{\mathcal L_{\d_\cS}^n \1(\rho)}\cr
&=
\frac{\sum_{\omega \in E^n_\rho} e^{-{\d_\cS} \lambda_p(\omega)}
\1_F(\Delta_n(\omega))
}{\sum_{\omega \in E^n_\rho} e^{-{\d_\cS} \lambda_p(\omega)}}\cr
&=
\frac{\sum_{\omega \in E^n_\rho} e^{-{\d_\cS} \lambda_p(\omega)}
\1_F\left( \frac{\lambda_\rho(\omega) - \chi_{\d_\cS} n}{\sqrt{n}}\right)
}{\sum_{\omega \in E^n_\rho} e^{-{\d_\cS} \lambda_p(\omega)}}
\end{aligned}
\end{equation}
where in the third term the function $\Delta_n$ is considered as defined on $E_A^{\mathbb N}$ by the formula
$$
\Delta_n(\omega)= \frac{\lambda_\rho(\omega|_{n}) - \chi_{\d_\cS} n}{\sqrt{n}}.
$$ 		
Our last counting result for attracting systems is the following.

\begin{thm}\label{t1m58} 
If $\cS$ is a strongly regular finitely irreducible D--generic attracting conformal graph directed Markov system, then the sequence of random variables $(\Delta_n)_{n=1}^\infty$ converges in distribution
to the normal (Gaussian) distribution 
$\mathcal N_0(\sigma)$ with mean value zero and the variance $\sigma^2 = \P''({\d_\cS})$ (the latter being positive because of Remark~\ref{r2_2017_02_17} and since the system $\cS$ is D--generic).  
Equivalently, the sequence 
$(\mu_n \circ \Delta_n^{-1})_{n=1}^\infty$
converges weakly to the normal distribution $\mathcal N_0(\sigma^2)$.  This means that for every Borel set $F \subset \mathbb R$ with 
$\hbox{\rm Leb}(\partial F) = 0$, we have 
\begin{equation}\label{1ms8}
\lim_{n \to +\infty} \mu_n(\Delta_n^{-1}(F))
= \frac{1}{\sqrt{2\pi} \sigma}
\int_F e^{- t^2/2\sigma^2} dt.
\end{equation} 
\end{thm}
\begin{proof}
This theorem is equivalent to showing that the characteristic functions (or Fourier transforms) of the measures $\mu_n \circ \Delta_n^{-1}$ 
converge to the characteristic function of 
$\mathcal N_0(\sigma^2)$, i.e., to the function $\R\ni t\longmapsto e^{-\sigma^2t^2/2}$.
By the formula \eqref{1nh20}  we have for every $t \in \mathbb R$ that 
$$
\begin{aligned}
\int_{\mathbb R} e^{it x} d\mu_n \circ \Delta_n^{-1}(x) 
&=
\int_{E_\rho^n} e^{it \Delta_n(\omega)} d\mu_n(\omega)
= 
\frac{\mathcal L_{\d_\cS}^n(e^{it \Delta_n})(\rho)}{\mathcal L_{\d_\cS}^n \1(\rho)}\cr
&= e^{-t \chi_{\d_\cS} \sqrt{n}} 
\frac{\mathcal L_{{\d_\cS}- \frac{it}{\sqrt{n}}}^n\1(\rho)}{\mathcal L_{\d_\cS}^n \1(\rho)}\cr
&=
 e^{-t \chi_{\d_\cS} \sqrt{n}} 
 \left(
 \frac{\lambda_{{\d_\cS} - \frac{it}{\sqrt{n}}}^n Q_{{\d_\cS} - \frac{it}{\sqrt{n}}}(\1)(\rho) + S^n_{{\d_\cS} - \frac{it}{\sqrt{n}}}\1(\rho) }{\psi_{\d_\cS}(\rho) + S_{\d_\cS}^n\1(\rho)}
\right). \cr
\end{aligned}
$$
It therefore follows from items (4), (5) and (6) following formula \ref{1nh20} that
$$
\lim_{n \to +\infty} \int_{\mathbb R} e^{itx} 
d\mu_n \circ \Delta_n^{-1}(x)
=\lim_{n \to +\infty} 
e^{-it \chi_{\d_\cS} \sqrt{n}} \lambda_{{\d_\cS}- \frac{it}{\sqrt{n}}}^n.
$$
Denote by $\log \lambda_s$, $s$ bbelongingSelonging to some sufficiently small neighborhood of ${\d_\cS}$,
the principle branch of the logarithm of $\lambda_s$, i.e., that  determined by the requirement that $\log \lambda_{\d_\cS} = 0$.  Since $\log \lambda_s = \P(s)$ for real $s > \gamma_s$ and since $\P'(0) = - \chi_{\d_\cS}$, we therefore get that
$$
\lambda_s = \exp(\log \lambda_s)
= \exp
\left(
- \chi_{\d_\cS} (s- \delta)
+ \frac{{\d_\cS}^2}{2}(s-{\d_\cS})^2 + O(|s-{\d_\cS}|^3)
\right).
$$
So for $s = {\d_\cS} - \frac{it}{\sqrt{n}}$
we get 
$$
\lambda_{{\d_\cS} - \frac{it}{\sqrt{n}}}
= \exp
\left(
 i \frac{t \chi_{\d_\cS}}{\sqrt{n}}
 - \frac{\sigma^2 t^2}{2n} + O(n^{-3/2})
\right).
$$
Therefore,
$$
\begin{aligned}
e^{-it \chi_{\d_\cS} \sqrt{n}}
\lambda^n_{{\d_\cS} - \frac{it}{\sqrt{n}}}&
=
e^{-it \chi_{\d_\cS} \sqrt{n}}
\exp \left(
i t \chi_{\d_\cS} \sqrt{n} - \frac{\sigma^2 t^2}{2} + O(n^{-1/2})
\right)\cr
&=
\exp\left(
- \frac{\sigma^2 t^2}{2} + O(n^{-1/2})
\right).
\cr
\end{aligned}
$$
So finally
$$
\lim_{n \to +\infty}
\int_{\mathbb R}  e^{itx} d\mu_n\circ \Delta_n^{-1}(x) = \exp \left( - \sigma^2 t^2/2\right).
$$
Thus since 
$\R\ni t\longmapsto\exp \left( - \sigma^2 t^2/2\right)$
is the characteristic function of the Gaussian distribution $\mathcal N_0(\sigma^2)$, the proof is complete.
\end{proof}

\section{Central Limit Theorems: Asymptotic Counting Functions for Parabolic GDMSs}\label{CLT-Parabolic-2}
We want to extend the Central Limit Theorem for counting functions from the previous (attracting GDNSs) subsection to the case of parabolic GDMSs. 
We are in the same setting as in Section~\ref{section:parabolic} i.e., $\cS = \{\phi_e\}_{e \in E}$ is a finite
irreducible conformal parabolic GDMS. Furthermore,  the 
functions $\Delta_n$ and measures $\mu_n$ have formally 
the same definitions as their ``attracting''
counterparts given in Subsection \ref{subsection:CLT-hyperbolic}
respectively by formulae (\ref{2m56}) and (\ref{1ms6}). We start with the following analogue of Theorem~\ref{t1ms5.2}.

\begin{thm}\label{t1ms15.1} 
If $\cS$ is a finitely irreducible parabolic conformal GDMS
for which 
$$
\d_\cS > \frac{2p_\cS}{p_\cS+1},
$$
i.e the invariant measure $\mu_{\d_\cS}$ is finite (so a probability after normalization), then for every $\rho \in E_{A^*}^\infty$
$$
\lim_{n \to + \infty} \int \frac{\lambda_\rho}{n}d\mu_n
=
\int_{E_A^\infty} \lambda d\mu_{\d_\cS}
= \chi_{\d_\cS}.
\eqno{\label{1ms15.1}}
$$
\end{thm}

\begin{proof}
Since the behavior of iterates of the Perron--Frobenius operator $\mathcal L_{\d_\cS}$ is now (in the parabolic context) more complicated than in the attracting case, we need to  provide a conceptually different proof than that of Theorem~\ref{t1ms5.2}. We will make an
essential use of Birkhoff's Ergodic Theorem instead. 

Firstly, we  fix $\epsilon > 0$.  Then it follows from Birkhoff's Ergodic  Theorem, along with both Lusin's Theorem and Egorov's Theorem, that there exists an integer $N_\epsilon  \geq 1$ and a measurable set $F(\epsilon) \subset E_A^\infty$ such that 
$m_{\d_\cS}(F(\epsilon)) > 1 - \epsilon$ (remembering that $m_{\d_\cS}$ is equivalent to $\mu_{\d_\cS}$) for every $\tau \in F(\epsilon)$
and every integer $n \geq N_{\epsilon}$, 
$$
\left|
\frac{\sum_{j=0}^{n-1} \lambda\circ\sg^j(\tau)}{n} - \chi_{\d_\cS}
\right| \leq \epsilon.
$$
For all $n \geq N_2$ let
$$
F_\rho(\epsilon, n):=
\{
\omega \in E_\rho^n \hbox{ : } \omega\rho \in F(\epsilon)
\}
\  \  \hbox{ and } \  \
F_\rho^c(\epsilon, n):=
\{
\omega \in E_\rho^n \hbox{ : } \omega\rho \in F^c(\epsilon)
\}
$$
Then 
 \begin{equation}\label{2ms15.1}
 \begin{aligned}
\Bigg|
 \sum_{\omega \in F_\rho(\epsilon, n)}
 \frac{\lambda_\rho(\omega)}{n} &
 \frac{|\phi_\omega'(\pi(\rho))|^{\d_\cS}}{\mathcal L_{\d_\cS}^n \1 (\rho)}
 - 
  \sum_{\omega \in F_\rho(\epsilon, n)}
\chi_{\d_\cS}
 \frac{|\phi_\omega'(\pi(\rho))|^{\d_\cS}}{\mathcal L_{\d_\cS}^n \1 (\rho)}
 \Bigg|= 
\cr 
 &= 
\left|  \sum_{\omega \in F_\rho(\epsilon, n)}
\left( \frac{\lambda_\rho(\omega)}{n}
-
\chi_{\d_\cS}
\right)
 \frac{|\phi_\omega'(\pi(\rho))|^{\d_\cS}}{\mathcal L_{\d_\cS}^n \1 (\rho)} \right|
 \cr
& =  
\left|\frac{\lambda_\rho(\omega)}{n}
- \chi_{\d_\cS}
\right| 
 \sum_{\omega \in F_\rho(\epsilon, n)}
 \frac{|\phi_\omega'(\pi(\rho))|^{\d_\cS}}{\mathcal L_{\d_\cS}^n \1 (\rho)}
 \cr
& =  \left|\frac{\lambda_\rho(\omega)}{n}
- \chi_{\d_\cS}
\right| \leq \epsilon.
\end{aligned}
\end{equation}
 Now given a positive number $M$
 and an arbitrary function $g:E_A^\infty \to \mathbb R$ for which $|g| \leq M$, we have that 
 $$
 \begin{aligned}
\lt|\sum_{\omega \in F_\rho^c(\epsilon, n)}
 g(\omega \rho)
 \frac{|\phi_\omega'(\pi(\rho))|^{\d_\cS}}{\mathcal L_{\d_\cS}^n \1 (\rho)}\rt|
 &\leq 
 \frac{M}{\mathcal L_{\d_\cS}^n \1 (\rho)}
  \sum_{\omega \in F_\rho^c(\epsilon, n)}
|\phi_\omega'(\pi(\rho))|^{\d_\cS} 
\leq 
 \frac{M'}{\mathcal L_{\d_\cS}^n \1 (\rho)}
  \sum_{\omega \in F_\rho^c(\epsilon, n)}
m_{\d_\cS}([\omega])\cr
&=  \frac{M'}{\mathcal L_{\d_\cS}^n \1 (\rho)}
m_{\d_\cS}(F_\rho^c(\epsilon, n))
\cr
&\leq 
\frac{M'}{\mathcal L_{\d_\cS}^n \1 (\rho)} \epsilon,
\end{aligned}
 $$
with some appropriate constant $M'>0$.   
 Now it follows from Theorem E of \cite{Hu}
 that there exists a constant $Q_\rho \geq 1$, depending on $\rho$ (in fact depending 
 only on $\hbox{\rm dist}(\pi(\rho), \Omega)$)
 such that 
 $$
 Q_\rho^{-1} \leq \mathcal L_{\d_\cS}^n(\rho)\leq Q_\rho
 $$
 for every integer $n \geq 0$.   We therefore get
 \begin{equation} \label{1ms15.2}
 \left| \sum_{\omega \in F_\rho^c(\epsilon, n)}
 g(\omega \rho)
 \frac{|\phi_\omega'(\pi(\rho))|^{\d_\cS}}{\mathcal L_{\d_\cS}^n \1 (\rho)}
 \right| \leq M' Q_\rho \epsilon.
\end{equation}
 Since 
 $$
 0 \leq \frac{1}{n}\sum_{j=0}^{n-1}\lambda\circ\sg^j \leq M
 $$
 for every $n \geq 1$, applying (\ref{2ms15.1})
 and also (\ref{1ms15.2}) for both 
$$
g = \frac{1}{n}\sum_{j=0}^{n-1}\lambda\circ\sg^j   
\  \  \  {\rm  and }  \  \  \
g = \chi_{\d_\cS},
$$ 
we get the following bound:
$$
 \begin{aligned}
 &\left|
 \int_{E_\rho^n} \frac{\lambda_\rho}{n} d\mu_n - 
 \chi_{\d_\cS}
 \right|\le \cr
&\leq 
  \left|
  \left(
 \int_{F_\rho(\epsilon, n)} \frac{\lambda_\rho}{n} d\mu_n
  - \int_{F_\rho(\epsilon, n)} \chi_{\d_\cS} d\mu_n
 \right) 
 + 
  \left(
 \int_{F_\rho^c(\epsilon, n)} \frac{\lambda_\rho}{n} d\mu_n
  - \int_{F_\rho^c(\epsilon, n)} \chi_{\d_\cS} d\mu_n
 \right) 
 \right|
 \cr
  &\leq 
  \left|
 \int_{F_\rho(\epsilon, n)} \frac{\lambda_\rho}{n} d\mu_n
  - \int_{F_\rho(\epsilon, n)} \chi_{\d_\cS} d\mu_n
 \right|
 + 
  \left|
 \int_{F_\rho(\epsilon, n)^c} \frac{\lambda_\rho}{n} d\mu_n
  - \int_{F_\rho(\epsilon, n)^c} \chi_{\d_\cS} d\mu_n
 \right| 
 \cr
&\leq 
\left|
\sum_{\omega \in F_s(\epsilon, n)}
\frac{\lambda_\rho(\omega)}{n}
\frac{|\phi_{\omega}'(\pi(\rho))|^{\d_\cS}}{\mathcal L_{\d_\cS}^n 1(\rho)}
-
\sum_{\omega \in F_s(\epsilon, n)}
\chi_{\d_\cS}
\frac{|\phi_{\omega}'(\pi(\rho))|^{\d_\cS}}{\mathcal L_{\d_\cS}^n 1(\rho)}
\right| 
+ \left| \int_{F_\rho^c(\epsilon, n)} \frac{\lambda_\rho}{n} d\mu_u\right|
+ 
\left| \int_{F_\rho^c(\epsilon, n)} \chi_{\d_\cS} d\mu_u\right|
\cr
&\leq 
\epsilon + 
\left|
\sum_{\omega \in F_s^c(\epsilon, n)}
\frac{\lambda_\rho(\omega)}{n}
\frac{|\phi_{\omega}'(\pi(\rho))|^{\d_\cS}}{\mathcal L_{\d_\cS}^n \1(\rho)}
\right| + \left|
\sum_{\omega \in F_s^c(\epsilon, n)}
\chi_{\d_\cS}
\frac{|\phi_{\omega}'(\pi(\rho))|^{\d_\cS}}{\mathcal L_{\d_\cS}^n \1(\rho)}
\right| \cr
&\leq \epsilon + M' Q_\rho \epsilon +  M' Q_\rho \epsilon \cr
&\leq (1 + 2 M'Q_\rho)\epsilon.
\end{aligned}
 $$
 Hence, letting $\epsilon \to 0$ we obtained 
 $$
 \int_{E_\rho^n} \frac{\lambda_\rho}{n} d\mu_n 
 =  \chi_{\d_\cS}
 $$
 and the proof is complete.
\end{proof}

Our main theorem in this subsection is the following.

\begin{thm}\label{t1ms16} 
If $\cS$ is a finitely irreducible parabolic conformal GDMS
for which 
$$
\d_\cS > \frac{2p_\cS}{p_\cS+1},
$$
i.e the invariant measure $\mu_{\d_\cS}$ is finite (so a probability after normalization), then the sequence of random variables $(\Delta_n)_{n=1}^\infty$ converges in distribution
to the normal (Gaussian) distribution 
$\mathcal N_0(\sigma^2)$ with mean value zero and the variance $\sigma^2 = P''_*({\d_\cS})>0$.  
Equivalently, the sequence 
$(\mu_n \circ \Delta_n^{-1})_{n=1}^\infty$
converges weakly to the normal distribution $\mathcal N_0(\sigma^2)$.  This means that for every Borel set $F \subset \mathbb R$ with 
$\hbox{\rm Leb}(\partial F) = 0$, we have 
\begin{equation}\label{1ms16}
\lim_{n \to +\infty} \mu_n(\Delta_n^{-1}(F))
= \frac{1}{\sqrt{2\pi} \sigma}
\int_F e^{- t^2/2\sigma^2} dt.
\end{equation}

\end{thm}

\begin{proof}
Using our previous notation recall that
$$
\psi_{\d_\cS} =\frac{d\mu_{\d_\cS}}{dm_{\d_\cS}}.
$$
Then 
$$
\mathcal L_{\d_\cS} \psi_{\d_\cS} = \psi_{\d_\cS},
$$
and we can define the operator $\widehat {\mathcal L}_{\d_\cS} : L^1(\mu_{\d_\cS})  \to L^1(\mu_{\d_\cS})$ by the formulae
$$
\widehat {\mathcal L}_{\d_\cS} (g) = \frac{1}{\psi_{\d_\cS}}
\mathcal L_{\d_\cS}(g \psi_{\d_\cS}).
$$
Then 
$$
\widehat {\mathcal L}_{\d_\cS}(\1) = \1
$$
and $\widehat {\mathcal L}_{\d_\cS}$ is the Perron--Frobenius operator 
associated to the measure--preserving symbolic dynamical system $(\sigma, \mu_{\d_\cS})$.  Following Gou\"ezel \cite{gouezel3}, for every integer $q \geq 1$ we consider the set 
$$
Z_q := \bigcup_{b \in \Omega} \bigcup_{e \in E\sms\{b\}\atop A_{be}=1}
\{b^ke \hbox{ : } 1 \leq k \leq q \} \cup (E\sms\Omega)
$$
and the first return map $\sigma_q: Z_q \to Z_q $.
Still following \cite{gouezel3}, given an integer $n \geq 1$ we define an operator $\widehat {\mathcal L}_{\d_\cS}^{(n)}: L^1(\mu_{\d_\cS}) \to L^1(\mu_{\d_\cS})$ by the formula
$$
\widehat {\mathcal L}_{\d_\cS}^{(n)}(g) 
:= \1_{Z_q} \widehat{\mathcal L}_{\d_\cS}^{n}(g \1_{Z_q}).
$$ 
Now our setting entirely fits into the hypothesis of section 2, 3 and 4 of Gou\"ezel's paper \cite{gouezel3}.  
In particular, Theorem 2.1  (especially its formula (2)), Theorem 3.7 and Lemma 4.4 of \cite{gouezel3} apply to give 
(compare the last formula of the proof of Proposition~4.6 in \cite{gouezel3}) for any $\tau \in Z_q$ and any $t \in \mathbb R$ that 
\begin{equation}\label{1ms17}
\lim_{n \to +\infty}
 \left|
\widehat {\mathcal L}_{\d_\cS}^{(n)}  (e^{it \Delta_n})(\rho)
- \mu_{\d_\cS}(Z_q)^2 e^{-\sigma^2/2t^2}
\right|=0.
\end{equation}
Now  there exists $q_0 \geq 0$ such that
 $\rho \in Z_{q_0}$.  
 Fix $\epsilon > 0$.   
 Take $q \geq q_0$   sufficiently large, say, 
 $q \geq q_1 \geq q_0$ that 
 \begin{equation} \label{5ms17}
 1 - \mu_{\d_\cS}(Z_q)^2 < \epsilon.
\end{equation}
Then by (\ref{1ms17})
\begin{equation}\label{2ms17}
\limsup_{n \to +\infty} \left| 
\widehat L_{\d_\cS}^{(n)}  (e^{it \Delta_n})(\rho)
-  e^{- \sigma^2/2t^2}
\right|
\leq \epsilon e^{-\sigma^2/2t^2}.
\end{equation}
Now define $\mu_n'$ analogously to (\ref{1ms6}), i.e., 
for $H \subset E^n_{*,\rho}$:
$$
\mu_n'(H) = \sum_{\omega \in H} e^{-{\d_\cS} \lambda_\rho(\omega)}.
$$
Then the same calculation as (\ref{1ms8}) gives 
$$
\int_{\mathbb R} e^{itx} d\mu_n' \circ \Delta_n^{-1}(x)
= \widehat {\mathcal L}_{\d_\cS}^n(e^{it \Delta_n})(\rho)
=  \widehat {\mathcal L}_{\d_\cS}^{(n)}(e^{it \Delta_n})(\rho)
+  \widehat {\mathcal L}_{\d_\cS}( \1_{Z_q^c}e^{it \Delta_n})(\rho).
$$ 
But 
\begin{equation}\label{4ms17}
\left|
\widehat {\mathcal L}_{\d_\cS}( \1_{Z_q^c}e^{it \Delta_n})(\rho)
\right|
\leq 
\left|
\widehat {\mathcal L}_{\d_\cS}( \1_{Z_q^c})(\rho)
\right|
= \widehat {\mathcal L}_{\d_\cS}( \1_{Z_q^c})(\rho),
\end{equation}
and according to Theorem E in \cite{Hu} we can write  
$$
\lim_{n \to +\infty} 
\widehat {\mathcal L}_{\d_\cS}( \1_{Z_q^c})(\rho)
= \mu_{\d_\cS}(\1_{Z_q^c}) = 1- \mu_{\d_\cS}(\1_{Z_q}).
$$
Combining this along with (\ref{1ms17}), (\ref{2ms17}) and 
(\ref{4ms17}) gives
$$
\limsup_{n \to +\infty}
\left|
\int_{\mathbb R} e^{itx} d\mu_n' \circ \Delta_n^{-1}(x)
- e^{- \sigma^2/2t^2}
\right|
\leq 
\epsilon  e^{- \sigma^2/2t^2} + 1 - \mu_{\d_\cS}(Z_q)
\leq (1 + e^{-\sigma^2/2t^2})\epsilon.
$$
Hence 
$$
\lim_{n \to +\infty}  \int_{\mathbb R} e^{it x}  d\mu_n'\circ \Delta_n^{-1}(x) = e^{-\sigma^2/2t^2}.
$$
Therefore, formula (\ref{1ms16}) holds with $\mu_n$
replaced $\mu_n'$. Because of this, because the measures $\mu_n$ and 
$\mu_n'$ are equivalent for all $n \geq 1$, and since, by Theorem E of \cite{Hu} again, for the sequence $(\mu_n')_{n=1}^\infty$,
$$
\lim_{n \to +\infty}  \frac{d\mu_n}{d\mu_n'} (x)= 1
$$
uniformly with respect to all $x \in \mathbb R$, we finally conclude 
that the formula (\ref{1ms16}) holds for measures $\mu_n$, $n\ge 1$.
Thus the proof of Theorem \ref{t1ms16} is complete.
\end{proof}

\part{{\Large Examples and Applications, I}}

\sp\section{{\large{\bf Attracting/Expanding Conformal Dynamical Systems}}}\label{AECDS}

In this section we deal with some conformal dynamical systems that are expanding and we show that their, appropriately organized, inverse holomorphic branches form conformal attracting GDMSs. We also examine in greater detail some special countable alphabet conformal attracting GDMSs.

\subsection{Conformal Expanding Repellers}\label{CER}

\sp In this section we deal with conformal expanding repellers. We do it by applying the theory developed in the previous sections. In fact it suffices to work here with conformal GDMSs modeled on finite alphabets $E$. However, most of the results proved in this section are entirely new. 

Let us start with the the definition of a conformal expanding repeller, the primary object of interest in this subsection.

\begin{dfn}\label{exprep}
Let $U$ be an open subset of $\R^d$, $d\ge 1$. Let $X$ be a compact subset of $U$. Let $f:U\to\R^d$ be a conformal map. The map $f$ is called a conformal expanding repeller if the following conditions are satisfied:
\begin{enumerate}
\item{}$f(X)=X$,

\sp\item{} 
$|f'|_X|>1$,

\sp\item{} there exists an open set  $V$ such that $\overline{V}\subset U$ and
$$
X=\bigcap_{k=0}^\infty f^{-n}(V),
$$
and
\item{}the map $f|_X:X\to X$ is topologically transitive.
\end{enumerate}

\sp\fr Note that $f$ is not required to be one-to-one; in fact usually it is not one-to-one. Abusing  notation slightly  we frequently refer also to the set $X$ alone as a conformal expanding repeller. In order to use a uniform terminology we also call $X$ the limit set of $f$.
\end{dfn}

\fr Typical examples of conformal expanding repellers are provided by the following.

\begin{prop}\label{p1_2015_05_11} 
If $f:\hat\C\to\hat\C$ is a rational function of degree $d\ge 2$, such that the map $f$ restricted to its Julia set $J(f)$ is expanding, then $J(f)$ is a conformal expanding repeller.
\end{prop}

The basic concept associated with such repellers which will be used in this section is given by the following definition.

\begin{dfn}\label{3.5.1}
A finite cover $\cR=\{R_e:e\in F\}$ of $X$ is
said  to be a Markov partition of the space $X$ for the mapping $T$ if the following conditions are satisfied.

\sp\begin{itemize}
\item[(a)] \ $R_e=\overline{\Int R_e}$ \ for all $e\in F$.

\sp\item[(b)] \ $\Int R_a\cap \Int R_b = \es$ \ for all $a\ne b$.

\sp\item[(c)] \ $\Int R_b\cap f(\Int R_a) \ne \es \  \Longrightarrow \  R_b\subset f(R_a)$ \ whenever  \  $a, b\in F$.
\end{itemize}
\end{dfn}

\sp\fr The elements of a Markov partition will be called cells in the sequel. The basic theorem about Markov partitions proved, for ex. in \cite{PU}, is this.

\begin{thm}\label{t_Markov_Partitions}
Any conformal expanding repeller $f:X\to X$ admits Markov partitions of arbitrarily small diameters.
\end{thm}

Fix $\b>0$ so small that for every $x\in X$ and every $n\ge 0$ there exists $f_x^{-n}:B(f^n(x),4\b)\to\R^d$, a unique continuous branch of $f^{-n}$ sending $f^n(x)$ to $x$. Theorem~\ref{t_Markov_Partitions} guarantees us the existence of $\cR=\{R_j:j\in F\}$, a Markov partition of $f$ with all cells of diameter smaller than $\b$. Having such a Markov partition $\cR$ we now associate to it a finite graph directed Markov system. The set of vertices is equal to $\cR$ while the alphabet $E$ is defined as follows.
$$
E:=\big\{(i,j)\in F\times F:\Int R_j\cap f(\Int R_i) \ne \es \big\}.
$$
Now, from the above for every $(i,j)\in E$ there exists a unique conformal map $f_{i,j}^{-1}:B(R_j,\b)\to\R^d$ such that
$$
f_{i,j}^{-1}(R_j)\sbt R_i.
$$
Define the incidence matrix $A:E\times E\to\{0,1\}$ by
$$
A_{(i,j) (k,l)}=
\begin{cases}
1 \ &\text{{\rm  if }} \ l=i \\
0 \ &\text{{\rm  if }} \ l\ne i.
\end{cases}
$$
We further define:
$$
t(i,j)=j \  \ \text{{\rm and }} \  \  i(i,j)=i. 
$$
Of course 
$$
\cS_\cR=\{f_{i,j}^{-1}:(i,j)\in E\}
$$ 
forms a finite conformal directed Markov system, and $\cS_\cR$ is irreducible since the map $f:X\to X$ is transitive. Let 
$$
\pi_\cR:=\pi_{\cS_\cR}:E_A^\infty\to X
$$ 
be the canonical projection onto the limit set $J_\cS$ of the conformal GDMS $\cS$ which is easily seen to be equal to $X$.

\sp Returning  to the actual topic of the paper,  i.e., counting inverse images and periodic points, we fix a point $\xi \in X$, a Markov Partition
$$
\mathcal R = \{R_e:e\in F\},
$$
with 
\beq\label{1_2017_01_25}
\xi \in \bu_{e\in F}{\rm Int} (R_e).
\eeq
So, there exists a unique element $e(\xi)\in F$ such that $\xi \in \Int(R_{e(\xi)})$, and we fix a radius $\a>0$ so small that 
$$
B(\xi,\a) \subset R_{e(\xi)}.
$$
Furthermore, there exists a unique code of $\xi$, i.e. a unique infinite word $\rho \in E_A^\infty$ such that
$$
\pi_\cR(\rho) = \xi.  
$$
Using our usual notation we set 
\beq\label{3ex2}
\lambda(z) = \log |(f^{n(z)})'(z)|,
\eeq
where $z$ is an inverse image of $\xi$ under an iterate of $f$ and the integer $n(z) \geq 0$ is uniquely determined by the following two conditions:
\beq\label{1ex2}
f^{n(z)}(z) = \xi  
\eeq
and
\beq\label{2ex2}
f^{k}(z) \neq \xi \  \hbox{ for every integer \ $0 \leq k < k(z)$}. 
\eeq
We immediately note that if $\xi$ is not periodic then condition \eqref{1ex2} alone determines $n(z)$  uniquely. We further note that that if $\om\rho$ is a (unique by \eqref{1_2017_01_25}) coding of $z$ ($\om\in E_\rho^*$) then
$$
\lambda(z) = \lambda_\rho(\omega).
$$
We denote the set of all inverse images of $\xi$ under iterates
of $f$ by $f^{-*}(\xi)$, i.e.
$$
f^{-*}(\xi):= \bu_{n=0}^\infty f^{-n}(\xi).
$$
We call $z:=(x,n) \in X \times \mathbb N$, a periodic pair of $f$ (of period $n$) if
$$
f^n(x) = x.
$$
We then denote $x$ by $\hat z$ and $n$ by $n(z)$. Of course $x$ is a periodic point of $f$ (of period $n$). We emphasize that we do not assume $n$ to be a prime (least) period of $x$. We set
$$
\lambda_p(z) := \log |(f^{n(z)})'(\hat z)|.
$$
We denote by $\widehat{\hbox{Per}}(f)$ (respectively $\widehat{\hbox{Per}}_n(f)$)  the set of all periodic pairs (of period $n$) and by $\hbox{Per}(f)$
(respectively $\hbox{Per}_n(f)$) the set of all periodic points (of period $n$) of $f$.

Given $T \geq 0$  we set
$$
\pi_{\xi}(f,T):= \{ z\in f^{-*}(\xi): \lambda(z) \leq T\}
$$
and 
$$
\pi_{p}(f,T) = \{ z\in \widehat{ \hbox{Per}}(f):\lambda_p(z) \leq T\}.
$$
Furthermore, given a set $B \subset X$, we denote 
$$
\pi_\xi(f,B,T):= B \cap \pi_{\xi}(f,T)
\  \hbox{ and } \
\pi_p(f;B,T):= B \cap \pi_p(f,T).
$$
As in the case of graph directed Markov systems we denote
$$
N_\xi(f,T):= \#\pi_{\xi}(f,T), \  \  N_\xi(f;B,T):= \#\pi_{\xi}(f;B, T)
$$
and
$$
N_p(f,T):= \#\pi_{p}(f,T), \  \  N_p(f,B,T):= \#\pi_{p}(f,B, T).
$$
Given a set $Y \subset B(\xi,\a)$ we denote
$$
\mathcal D^\xi_Y (f; B, T)
:= \{z \in f^{-*}(\xi)\cap B \hbox{ : } \log \hbox{diam}\(f_{\hat z}^{-n(z)}(Y)\) \leq T\},
$$
$$
\mathcal E^\xi_Y (f; B, T)
:= \big\{z \in f^{-*}(\xi) \hbox{ : } \log \hbox{diam}\(f_{\hat z}^{-n(z)}(Y)\) \leq T \  \ {\rm and } \  \  f_{\hat z}^{-n(z)}(Y)\cap B\ne\es\big\},
$$
and then
$$
D^\xi_Y (f; B, T):= \# \mathcal D^\xi_Y (f; B, T) \  \
{\rm and } \  \
E^\xi_Y (f; B, T):= \# \mathcal E^\xi_Y (f; B, T).
$$

Now we record a straightforward, but basic observation which links this section to the previous ones. It is the following.

\begin{observation}\label{o1ex4}
If $f: X \to X$ is a conformal expanding repeller, then with the notation as above
$$
N_\xi(f;B,T)= N_\rho(B,T), \  \  D_Y^\xi(f;B,T)= D_Y^\rho\rho(B,T)
$$
and 
$$
\Gamma N_p(B,T) \leq N_p(f; B,T) \leq N_\rho(B,T)
$$
with some universal constant $\Gamma \in (0,+\infty)$. In addition, 
$$
N_p(f; B, T) = N_p(B,T)
$$
whenever $B \sbt \bu_{e\in F}{\rm Int} (R_e)$.
\end{observation}

We call a conformal expanding repeller $f:X\to X$ $D$--generic if and only if the additive group generated by the set
$$
\{\lam_p(z):z\in \widehat{ \hbox{Per}}(f)\}
$$
is not cyclic. It is immediate from the definition of the graph directed Markov system $\cS_\cR$ and Proposition~\ref{p1nh13} that we have the following. 

\bprop\label{p1ex4.1}
A conformal expanding repeller $f:X\to X$ is $D$--generic if and only if the conformal graph directed Markov system $\cS_\cR$ is $D$--generic.
\eprop

A concept of essentially non--linear conformal expanding repellers was introduced by Dennis Sullivan in \cite{Sullivan_1} and explored in detail in \cite{PU}. One  of its many characterizations (see \cite{PU} for them) is that there is no conformal atlas covering $X$ with respect to which the map $f$ is affine, i.e. a similarity composed with a translation. Analogously as for graph directed Markov systems, with the help of Chapter~10 from \cite{PU}, we get the following proposition, which adds considerably to our knowledge that $D$--generic conformal expanding repellers abound.

\bprop\label{p1_2017_01_25}
An essentially non--linear conformal expanding repeller $f:X\to X$ is $D$--generic.
\eprop

As a fairly direct consequence of Theorem~\ref{dyn} and Theorem~\ref{t1da7}, we get the following.

\begin{thm}\label{t1ex4}
Let $f: X\to X$ be a $D$--generic conformal expanding repeller and let $\delta:= \HD(X)$. 

\sp \begin{enumerate}

\item Let $m_\delta$ be the unique $\delta$-conformal measure for $f$ on $X$, which  coincides with the normalized $\d$--dimensional Hausdorff measure on $X$. 

\sp\item Let $\mu_\delta$ be the unique 
$f$-invariant Borel probability measure on $X$ absolutely continuous (in fact known to be equivalent) with respect to $m_\delta$. 

\sp\item
Let $\psi_\delta:= \frac{d\mu_\delta}{d m_\delta}$.  

\sp\item Fix $\xi \in X$ arbitrarily and $Y \subset B(\xi, \a)$, an arbitrary set consisting of at least two distinct points.  

\sp\item
Let $B \subset X$ be an arbitrary Borel set such that $m_\delta(\partial B)=0$ (equivalently that $\mu_\delta(\partial B)=0$).   
\end{enumerate}

\sp Then
\beq\label{1ex4}
\lim_{T \to +\infty}  \frac{N_\xi(f; B,T)}{e^{\delta T}} = \frac{\psi_\delta(\xi)}{\delta \chi_\delta} m_\delta(B),
\eeq
\beq\label{2ex4}
\lim_{T \to +\infty}  \frac{N_p(f; B,T)}{e^{\delta T}} = \frac{1}{\delta \chi_\delta} \mu_\delta(B),
\eeq
and 
\beq\label{1ex5}
\lim_{T \to +\infty}  \frac{D_Y^\xi(f; B,T)}{e^{\delta T}} 
= \lim_{T \to +\infty}  \frac{E_Y^\xi(f; B,T)}{e^{\delta T}} 
=C_\xi(Y) m_\delta(B),
\eeq
where $C_\xi(Y)\in (0,+\infty)$ is a constant depending only on the repeller $f$, the point $\xi\in X$, and the set $Y$. In addition
\beq\label{1da7.1_A}
K^{-2\d}(\d\chi_\d)^{-1}\diam^\d(Y)
\le C_\xi(Y)
\le K^{2\d}(\d\chi_\d)^{-1}\diam^\d(Y),
\eeq
and the function 
$$
\xi\longmapsto C_\xi(Y)\in (0,+\infty)
$$
is locally constant on some sufficiently small neighborhood of $Y$.
\end{thm} 

\begin{proof}
By making use of Observation \ref{o1ex4}, formulae \eqref{1ex4} and \eqref{1ex5} are immediate consequences of formula \eqref{3_2016_01_30} of Theorem~\ref{dyn}, along with Theorem~\ref{t1da7} and  Theorem~\ref{t1ma1}, once we notice that the measures $m_\delta$ and $\mu_\sigma$ are respectively $\delta$-conformal and invariant, equivalent  to $m_\delta$, for both the conformal expanding repeller $f: X \to X$ and the associated conformal GDMS $\mathcal \cS_\cR$. In order to get
formula \eqref{2ex4} one uses formula \eqref{3_2016_01_30E} of Theorem~\ref{dyn}, and also, in a straightforward way, the fact that $\mu_{\mathcal S}(\partial \mathcal R)=0$. The fact the function $\xi\longmapsto C_\xi(Y)$
is locally constant follows from Remark~\ref{r1_2017_03_20}.
\end{proof} 


From the results of Section~\ref{Section-CLT}, 
in particular the  versions of the Central Limit Theorem, proved for attracting conformal GDMSs, we directly get the following consequences for expanding repellers.

\begin{thm}\label{t1_2017_04_03} 
Let $f: X\to X$ be a $D$--generic conformal expanding repeller and let $\delta:= \HD(X)$. 
With notation of Theorem~\ref{t1ex4}, there exists $\sigma^2 > 0$ (in fact $\sg^2=\P''(0)> 0$) such that if $G \subset \mathbb R$
is a Lebesgue measurable set with $\hbox{{\rm Leb}}(\partial G) = 0$, then 
$$
\lim_{n \to +\infty}
\mu_\d\left(
\left\{
z\in X \hbox{ : } \frac{\log \big|(f^n)'(z)\big| - \chi_{\mu_\d} n}{\sqrt{n}}
\in G
\right\}
\right)
= \frac{1}{\sqrt{2\pi}\sigma} \int_G e^{-\frac{t^2}{2\sigma^2}} \,dt.
$$  
In particular, for any $\alpha <\beta$
$$
\lim_{n \to +\infty}
\mu_\d\left(
\left\{
z\in X \hbox{ : } \alpha \leq \frac{\log\big|(f^n)'(z)\big| - \chi_{\mu_\d} n}{\sqrt{n}}
\leq \beta
\right\}
\right)
= \frac{1}{\sqrt{2\pi}\sigma} \int_\a^\b e^{-\frac{t^2}{2\sigma^2}}\, dt.
$$  
\end{thm}

\sp\fr For every point $z\in X$ and every integer $n\ge 0$ let $e(z,n)\in F$ be such that
$$
f^n(z)\in R_e.
$$ 

\begin{thm}\label{t2_2017_04_03} 
Let $f: X\to X$ be a $D$--generic conformal expanding repeller and let $\delta:= \HD(X)$. 
With notation of Theorem~\ref{t1ex4}, there exists $\sigma^2 > 0$ (in fact $\sg^2=\P''(0)> 0$) such that if $G \subset \mathbb R$
is a Lebesgue measurable set with $\hbox{{\rm Leb}}(\partial G) = 0$, then 
$$
\lim_{n \to +\infty}
\mu_\d\left(
\left\{
z\in X \hbox{ : }
 \frac{
 -\log \hbox{\rm diam}\(f_x^{-n}(Y_{e(z,n)})\)  - \chi_{\mu_\d} n
 }{\sqrt{n}}
\in G
\right\}
\right)
= \frac{1}{\sqrt{2\pi}\sigma} \int_G e^{-\frac{t^2}{2\sigma^2}}\, dt.
$$  
In particular, for any $\alpha < \beta$
$$
\lim_{n \to +\infty}
\mu_\d\left(
\left\{
z\in X \hbox{ : } \alpha \leq 
 \frac{-\log \hbox{\rm diam}\(f_x^{-n}(Y_{e(z,n)})\) - \chi_{\mu_\d} n}{\sqrt{n}}
\leq \beta
\right\}
\right)
= \frac{1}{\sqrt{2\pi}\sigma} \int_\alpha^\beta e^{-\frac{t^2}{2\sigma^2}}\,dt.
$$  
\end{thm}

The next result is a law of the iterated logarithm.

\begin{thm}\label{t3_2017_04_03}
Let $f: X\to X$ be a $D$--generic conformal expanding repeller and let $\delta:= \HD(X)$. 

For every $e \in F$ let $Y_e \subset R_e$ be a set with at least two points. 
There exists $\sigma^2 > 0$ 
(in fact $\sigma^2:=\P''(0) > 0$) such that 
for $\mu_\d$--a.e. $z\in X$, we have that
$$
\limsup_{n \to +\infty}
\frac{\log \big|(f^n)'(z)\big| - \chi_{\mu_\d} n}{\sqrt{n\log\log n}}
= \sqrt{2\pi} \sigma
$$
and 
$$
\limsup_{n \to +\infty}\frac{-\log \hbox{\rm diam}\(f_x^{-n}(Y_{e(z,n)})\)              - \chi_{\mu_\d}n}{\sqrt{n\log\log n}}
= \sqrt{2\pi} \sigma.
$$
\end{thm}

\sp Let  $\xi\in X$ be fixed.  For every set $H \subset f^{-n}(\xi)$, define
\begin{equation}\label{1_2017_04_03}
\mu_n(H) := \frac{\sum_{z \in H} \big|(f^n)'(z)\big|^{-\delta}}{\sum_{z\in f^{-n}(\xi)}
\big|(f^n)'(z)\big|^{-\delta}}.
\end{equation}

\begin{thm}\label{t4_2017_04_03}
If $f: X\to X$ is a conformal expanding repeller, then for every $\xi\in X$, we have that
\begin{equation}\label{2_2017_04_03}
\lim_{n \to +\infty} \int_{f^{-n}(\xi)} \frac{\log\big|(f^n)'\big|}{n} d\mu_n 
= \chi_\delta.
\end{equation}
\end{thm} 

\sp\fr Analogously to \eqref{2m56} we define the functions $\Delta_n: f^{-n}(\xi)\to \mathbb R$ by the formulae
\begin{equation}\label{1_2017_04_04}
\Delta_n(z) := \frac{\log\big|(f^n)'(z)\big| - \chi_{\mu_\d} n}{\sqrt{n}}
\end{equation}
and consider the sequence $(\mu_n \circ \Delta_n^{-1})_{n=1}^\infty$
of probability distributions on $\mathbb R$.

\fr We have the following. 
\begin{thm}\label{t1_2017_04_04} 
If $f: X\to X$ is a $D$--generic conformal expanding repeller, then the sequence of random variables $(\Delta_n)_{n=1}^\infty$ converges in distribution to the normal (Gaussian) distribution 
$\mathcal N_0(\sigma)$ with mean value zero and the variance $\sigma^2 = P''(\delta)>0$.  
Equivalently, the sequence 
$(\mu_n \circ \Delta_n^{-1})_{n=1}^\infty$
converges weakly to the normal distribution $\mathcal N_0(\sigma^2)$.  This means that for every Borel set $F \subset \mathbb R$ with 
$\hbox{\rm Leb}(\partial F) = 0$, we have 
\begin{equation}\label{1_2017_04_04}
\lim_{n \to +\infty} \mu_n(\Delta_n^{-1}(F))
= \frac{1}{\sqrt{2\pi} \sigma}
\int_F e^{- t^2/2\sigma^2} dt.
\end{equation} 
\end{thm}

\subsection{1--Dimensional Attracting Conformal GDMSs and 1--Dimensional Conformal Expanding Repellers}

In this subsection we briefly discuss $1$--Dimensional systems. We start with the following.

\begin{example}\label{r1ex5} Theorem~\ref{dyn}, Theorem~\ref{t1da7}, and Theorem~\ref{t1ma1} hold in particular if a system $\mathcal S$ in  one--dimensional, i.e., if $X$ is a compact interval of $\mathbb{R}$.  Perhaps the the best known and one of the most often considered, is the infinite IFS $\mathcal{G}$ formed by all continuous inverse branches of the Gauss map 
$$
G(x) = x - [x].
$$
So $\mathcal G$ consists of the maps 
$$
[0,1]\ni x \longmapsto g_n(x):=\frac{1}{x+n}, \  \ n \in \mathbb N.
$$
and with $q=2$ in the sense of Remark~\ref{r1_2017_04_01} it becomes a conformal IFS. 
\end{example}

\fr Looking at the fixed points of $g_1$, $g_2$, and $g_3$ one immediately concludes that the Gauss system $\mathcal G$ is $D$--generic. It is also known (see ex. \cite{MU_TAMS}) to be strongly regular, even more, in the terminology of \cite{MU_GDMS}, it is hereditarily regular. So, Theorem~\ref{dyn}, \ref{t1da7} and \ref{t1ma1} do indeed apply to this system. Because of importance of the Gauss map we formulate below all the above mentioned applications expressed in the language of the Gauss map itself rather than the associated IFS $\mathcal G$. We adopt the, naturally adjusted, notation of Subsection~\ref{CER}.

We begin with the growth estimates.

\bthm\label{t1_2017_04_10}
If $G:[0,1]\to[0,1]$ is the Gauss map, then with notation of subsection \ref{CER} we have the following. Fix $\xi\in [0,1]$. 
If $B \subset [0,1]$ is a Borel set such that $\Leb(\bd B)=0$ and $Y \subset [0,1]$ is any set having at least two elements, then
$$
\lim_{T \to +\infty}  \frac{N_\xi(G; B,T)}{e^{T}} = \frac{\psi_1(\xi)}{\chi_1} \Leb(B),
$$
$$
\lim_{T \to +\infty}  \frac{N_p(G; B,T)}{e^{T}} = \frac1{\chi_1}\mu_1(B),
$$
and 
$$
\lim_{T \to +\infty}  \frac{D_Y^\xi(G; B,T)}{e^{T}} = 
\lim_{T \to +\infty}  \frac{E_Y^\xi(G; B,T)}{e^{T}} = C(Y)\Leb(B),
$$
where $C(Y)\in (0,+\infty]$ is a constant depending only on the map $G$ and the set $Y$.
\ethm

We next formulate a Central Limit Theorem for diameters. 

\begin{thm}\label{t2_2017_04_10} 
Let $G:[0,1]\to[0,1]$ be the Gauss map. Let $\sigma^2:=\P''(0)> 0$. With the notation of Theorem~\ref{t1ex4} we have the following. Let $Y\subset [0,1]$ be a set with at least two points. If $H \subset \mathbb R$ is a Lebesgue measurable set with $\hbox{\rm Leb}(\partial H) = 0$, then 
$$
\lim_{n \to +\infty}
\mu_1\left(
\left\{
z\in [0,1] \hbox{ : }
 \frac{
 -\log \hbox{\rm diam}\(G_x^{-n}(Y)\)  - \chi_{\mu_1} n
 }{\sqrt{n}}
\in H
\right\}
\right)
= \frac{1}{\sqrt{2\pi}\sigma} \int_H e^{-\frac{t^2}{2\sigma^2}}\, dt.
$$  
In particular, for any $\alpha < \beta$
$$
\lim_{n \to +\infty}
\mu_1\left(
\left\{
z\in[0,1] \hbox{ : } \alpha \leq 
 \frac{-\log \hbox{\rm diam}\(G_x^{-n}(Y)\) - \chi_{\mu_1} n}{\sqrt{n}}
\leq \beta
\right\}
\right)
= \frac{1}{\sqrt{2\pi}\sigma} \int_\alpha^\beta e^{-\frac{t^2}{2\sigma^2}}\,dt.
$$  
\end{thm}

The law of the iterated logarithm takes the following form.

\begin{thm}\label{t3_2017_04_03}
Let $G:[0,1]\to[0,1]$ be the Gauss map. Let $\sigma^2:=\P''(0)>0$. 
With notation of Theorem~\ref{t1ex4} we have the following.
Let $Y\subset [0,1]$ be a set with at least two points. 
Then for $\Leb$--a.e. $z\in [0,1$, we have that
$$
\limsup_{n \to +\infty}
\frac{\log \big|(G^n)'(z)\big| - \chi_{\mu_1} n}{\sqrt{n\log\log n}}
= \sqrt{2\pi} \sigma
$$
and 
$$
\limsup_{n \to +\infty}\frac{-\log \hbox{\rm diam}\(G_x^{-n}(Y)\)              - \chi_{\mu_1}n}{\sqrt{n\log\log n}}
= \sqrt{2\pi} \sigma.
$$
\end{thm}

Finally, we have a Central Limit Theorem for counting functions.

\begin{thm}\label{t4_2017_10_03}
If $G:[0,1]\to[0,1]$ is the Gauss map, then for every $\xi\in [0,1]$, we have that
$$
\lim_{n \to +\infty} \int_{G^{-n}(\xi)} \frac{\log\big|(G^n)'\big|}{n} d\mu_n 
= \chi_1.
$$
\end{thm} 

\begin{thm}\label{t1_2017_04_04} 
If $G:[0,1]\to[0,1]$ is the Gauss map, then the sequence of random variables $(\Delta_n)_{n=1}^\infty$ converges in distribution to the normal (Gaussian) distribution 
$\mathcal N_0(\sigma)$ with mean value zero and the variance $\sigma^2 = P''(\delta)>0$.  
Equivalently, the sequence 
$(\mu_n \circ \Delta_n^{-1})_{n=1}^\infty$
converges weakly to the normal distribution $\mathcal N_0(\sigma^2)$.  This means that for every Borel set $F \subset \mathbb R$ with 
$\hbox{\rm Leb}(\partial F) = 0$, we have 
\begin{equation}\label{1_2017_04_04}
\lim_{n \to +\infty} \mu_n(\Delta_n^{-1}(F))
= \frac{1}{\sqrt{2\pi} \sigma}
\int_F e^{- t^2/2\sigma^2} dt.
\end{equation} 
\end{thm}

\begin{rem}\label{r2ex5} 
Theorem~\ref{t1ex4} holds in particular if $f: X \mapsto X$ is a conformal expanding repeller with $X$ a compact subset (a topological Cantor set) of $\mathbb R$.
\end{rem}

\subsection{Hyperbolic (Expanding) Rational Functions of the Riemann Sphere $\widehat\C$}
One of the most celebrated conformal expanding repellers are hyperbolic (expanding) rational functions of the Riemann sphere 
$\widehat{\C}$ restricted to the Julia sets and already mentioned in subsection~\ref{CER}. For the sake of completeness and convenience of the reader, let us briefly describe them. Let $f:\oc\to\oc$ be a rational function of degree $d\ge 2$. Let $J(f)$ denote the Julia sets of $f$ and let
$$
\Crit(f):=\{c\in\oc:f'(c)=0\}
$$
be the set of all critical (branching) points of $f$. Put
$$
\PC(f):=\bu_{n=1}^\infty f^n(\Crit(f))
$$
and call it the postcritical set of $f$. The rational map $f:\oc\to\oc$ is said to be hyperbolic (expanding) if the restriction  $f|_{J(f)}: J(f) \to J(f)$ satisfies 
\beq\label{5_2016_07_07}
\inf\{|f'(z)|:z\in J(f)\} > 1
\eeq
or, equivalently, 
\beq\label{6_2016_07_07}
|f'(z)|>1
\eeq
for all $z\in J(f)$. Another, topological, characterization of expandingness is the following.

\bfact
A rational function $f:\oc\to\oc$ is expanding if and only if 
$$
J(f)\cap\ov{\PC(f)}=\es.
$$
\efact
\fr It is immediate from this characterization that all the polynomials $z\mapsto z^d$, $d\ge 2$, are expanding along with their small perturbations $z\mapsto z^d+\e$; in fact expanding rational functions are commonly believed to form a vast majority amongst all rational functions. This is known at least for polynomials with real coefficients. 

\sp It is known from \cite{Zdunik} (see also Section~3 of \cite{PUtame}) that the only essentially linear expanding rational functions are the maps of the form
$$
\oc\ni z\longmapsto f_d(z)=:z^d\in\oc, \  \  |d|\ge 2.
$$
In consequence the only non $D$-generic rational functions of the Riemann sphere $\oc$ are these functions $f_d$. So, as an immediate consequence of Theorem~\ref{t1ex4}, we get the following.

\begin{thm}\label{t1ex6}
Let $f:\oc\to\oc$ be a hyperbolic (expanding) rational function of the Riemann sphere $\widehat{\mathbb{C}}$ not of the form $\oc\ni z\longmapsto z^d\in\oc$, $|d|\ge 2$. Let $\delta:= \HD(J(f))$. 

\sp \begin{enumerate}

\item Let $m_\delta$ be the unique $\delta$-conformal measure for $f$ on the Julia set $J(f)$, which coincides with the normalized $\d$--dimensional Hausdorff measure on $J(f)$. 

\sp\item Let $\mu_\delta$ be the unique 
$f$-invariant Borel probability measure on $J(f)$ absolutely continuous (in fact known to be equivalent) with respect to $m_\delta$. 

\sp\item
Let $\psi_\delta:= \frac{d\mu_\delta}{d m_\delta}$.  

\sp\item Fix $\xi \in J(f)$ arbitrary and $Y \subset B(\xi,\a)$ (where $\a>0$ is sufficiently small as described in subsection~\ref{CER}), an arbitrary set consisting of at least two distinct points.  

\sp\item
Let $B \subset J(f)$ be an arbitrary Borel set such that $m_\delta(\partial B)=0$ (equivalently that $\mu_\delta(\partial B)=0$).   
\end{enumerate}

\sp Then
\beq\label{1ex4M}
\lim_{T \to +\infty}  \frac{N_\xi(f; B,T)}{e^{\delta T}} = \frac{\psi_\delta(\xi)}{\delta \chi_\delta} m_\delta(B),
\eeq
\beq\label{2ex4M}
\lim_{T \to +\infty}  \frac{N_p(f; B,T)}{e^{\delta T}} = \frac{1}{\delta \chi_\delta} \mu_\delta(B),
\eeq
and 
\beq\label{1ex5M}
\lim_{T \to +\infty}  \frac{D_Y^\xi(f; B,T)}{e^{\delta T}} 
= \lim_{T \to +\infty}  \frac{E_Y^\xi(f; B,T)}{e^{\delta T}} 
=C_\xi(Y) m_\delta(B),
\eeq
where $C_\xi(Y)\in (0,+\infty)$ is a constant depending only on the repeller $f$, the point $\xi\in J(f)$, and the set $Y$. In addition
\beq\label{1da7.1_B}
K^{-2\d}(\d\chi_\d)^{-1}\diam^\d(Y)
\le C_\xi(Y)
\le K^{2\d}(\d\chi_\d)^{-1}\diam^\d(Y),
\eeq
and the function 
$$
\xi\longmapsto C_\xi(Y)\in (0,+\infty)
$$
is locally constant on some sufficiently small neighborhood of $Y$.
\end{thm}

Fixing a Markov partition for the map $f:J(f)\to J(f)$, as immediate consequences of Theorems~\ref{t1_2017_04_03} -- \ref{t1_2017_04_04} we get the following  stochastic laws, primarily Central Limit Theorems, for the dynamical system $(f,\mu_\d)$. 

We begin with a Central Limit Theorem for the 
expansion on orbits.  

\begin{thm}\label{t11_2017_04_03} 
Let $f:\oc\to\oc$ be a hyperbolic (expanding) rational function of the Riemann sphere $\widehat{\mathbb{C}}$ not of the form $\oc\ni z\longmapsto z^d\in\oc$, $|d|\ge 2$. With notation of Theorem~\ref{t1ex4} there exists $\sigma^2 > 0$ 
(in fact $\sg^2=\P''(0)> 0$) such that if $G \subset \mathbb R$
is a Lebesgue measurable set with $\hbox{{\rm Leb}}(\partial G) = 0$, then 
$$
\lim_{n \to +\infty}
\mu_\d\left(
\left\{
z\in J(f) \hbox{ : } \frac{\log \big|(f^n)'(z)\big| - \chi_{\mu_\d} n}{\sqrt{n}}
\in G
\right\}
\right)
= \frac{1}{\sqrt{2\pi}\sigma} \int_G e^{-\frac{t^2}{2\sigma^2}} \,dt.
$$  
In particular, for any $\alpha <\beta$
$$
\lim_{n \to +\infty}
\mu_\d\left(
\left\{
z\in J(f) \hbox{ : } \alpha \leq \frac{\log\big|(f^n)'(z)\big| - \chi_{\mu_\d} n}{\sqrt{n}}
\leq \beta
\right\}
\right)
= \frac{1}{\sqrt{2\pi}\sigma} \int_\a^\b e^{-\frac{t^2}{2\sigma^2}}\, dt.
$$  
\end{thm}

We next have a Central Limit Theorem for diameters.

\begin{thm}\label{t12_2017_04_03} 
Let $f:\oc\to\oc$ be a hyperbolic (expanding) rational function of the Riemann sphere $\widehat{\mathbb{C}}$ not of the form $\oc\ni z\longmapsto z^d\in\oc$, $|d|\ge 2$. Let $\sigma^2:=\P''(0)> 0$. With the notation of Subsection~\ref{CER} for every $e \in F$ let $Y_e \subset R_e$ be a set with at least two points and if $G \subset \mathbb R$ is a Lebesgue measurable set with $\hbox{\rm Leb}(\partial G) = 0$, then 
$$
\lim_{n \to +\infty}
\mu_\d\left(
\left\{
z\in J(f) \hbox{ : }
 \frac{
 -\log \hbox{\rm diam}\(f_x^{-n}(Y_{e(z,n)})\)  - \chi_{\mu_\d} n
 }{\sqrt{n}}
\in G
\right\}
\right)
= \frac{1}{\sqrt{2\pi}\sigma} \int_G e^{-\frac{t^2}{2\sigma^2}}\, dt.
$$  
In particular, for any $\alpha < \beta$
$$
\lim_{n \to +\infty}
\mu_\d\left(
\left\{
z\in J(f) \hbox{ : } \alpha \leq 
 \frac{-\log \hbox{\rm diam}\(f_x^{-n}(Y_{e(z,n)})\) - \chi_{\mu_\d} n}{\sqrt{n}}
\leq \beta
\right\}
\right)
= \frac{1}{\sqrt{2\pi}\sigma} \int_\alpha^\beta e^{-\frac{t^2}{2\sigma^2}}\,dt.
$$  
\end{thm}

The following is a version of the law of the iterated 
function scheme.

\begin{thm}\label{t13_2017_04_04}
Let $f:\oc\to\oc$ be a hyperbolic (expanding) rational function of the Riemann sphere $\widehat{\mathbb{C}}$ not of the form $\oc\ni z\longmapsto z^d\in\oc$, $|d|\ge 2$. Let $\sigma^2:=\P''(0)> 0$. With the notation of Subsection~\ref{CER} for every $e \in F$ let $Y_e \subset R_e$ be a set with at least two points and if $G \subset \mathbb R$ is a Lebesgue measurable set with $\hbox{\rm Leb}(\partial G) = 0$, then for  $\mu_\d$--a.e. $z\in J(f)$, we have that
$$
\limsup_{n \to +\infty}
\frac{\log \big|(f^n)'(z)\big| - \chi_{\mu_\d} n}{\sqrt{n\log\log n}}
= \sqrt{2\pi} \sigma
$$
and 
$$
\limsup_{n \to +\infty}\frac{-\log \hbox{\rm diam}\(f_x^{-n}(Y_{e(z,n)})\)              - \chi_{\mu_\d}n}{\sqrt{n\log\log n}}
= \sqrt{2\pi} \sigma.
$$
\end{thm}

\begin{thm}\label{t14_2017_04_03}
If $f:\oc\to\oc$ is a hyperbolic (expanding) rational function of the Riemann sphere $\widehat{\mathbb{C}}$, then for every $\xi\in J(f)$, we have that
\begin{equation}\label{12_2017_04_04}
\lim_{n \to +\infty} \int_{f^{-n}(\xi)} \frac{\log\big|(f^n)'\big|}{n} d\mu_n 
= \chi_\delta.
\end{equation}
\end{thm} 
 
Finally, we have a Central Limit Theorem for counting.   
 
\begin{thm}\label{t24_2017_04_04} 
If $f:\oc\to\oc$ is a hyperbolic (expanding) rational function of the Riemann sphere $\widehat{\mathbb{C}}$ not of the form $\oc\ni z\longmapsto z^d\in\oc$, $|d|\ge 2$, then the sequence of random variables $(\Delta_n)_{n=1}^\infty$ converges in distribution to the normal (Gaussian) distribution 
$\mathcal N_0(\sigma)$ with mean value zero and the variance $\sigma^2 = P''(\delta)>0$.  
Equivalently, the sequence 
$(\mu_n \circ \Delta_n^{-1})_{n=1}^\infty$
converges weakly to the normal distribution $\mathcal N_0(\sigma^2)$.  This means that for every Borel set $F \subset \mathbb R$ with 
$\hbox{\rm Leb}(\partial F) = 0$, we have 
\begin{equation}\label{1_2017_04_04}
\lim_{n \to +\infty} \mu_n(\Delta_n^{-1}(F))
= \frac{1}{\sqrt{2\pi} \sigma}
\int_F e^{- t^2/2\sigma^2} dt.
\end{equation} 
\end{thm}

\sp\section{{\large{\bf Conformal Parabolic Dynamical Systems}}}\label{CPDS}
Now we move  onto dealing with parabolic systems. 
We consider first 1--dimensional examples.

\subsection{1--Dimensional Parabolic IFSs}\label{1dimparabolicexamples}
  
Theorems~\ref{t2pc6}, \ref{t1dp13} and \ref{t1dp13B} hold in particular if a parabolic system $\mathcal S$ is $1$--dimensional, i.e., if $X$ is a compact interval of $\mathbb R$.  Perhaps the best known,  and one of the most often considered, are the $1$-dimensional parabolic IFSs formed by (two) continuous inverse branches of  Manneville--Pomeau maps $f_\a:[0,1] \to [0,1]$ defined by the
$$
f_\a(x) = x + x^{1+\alpha} \hbox{ (mod $1$)},
$$
where $\alpha > 0$ is a fixed number and by the (two) continuous inverse branches of the Farey map (for this one Remark~\ref{r2_2017_02_17} applies with $q=2$)
$$
f(x)
= 
\begin{cases}
\frac{x}{1-x} &\hbox{ if } 0 \leq x \leq \frac{1}{2}\\
\frac{1-x}{x} &\hbox{ if } \frac{1-x}{x} \leq x \leq 1.
\end{cases}
$$
Observe that for parabolic points,
$$
\Om(f)=\Om(f_\a)=\{0\}
$$
for all $\a>0$. Furthermore,
$$
p(f)=1 \  \  {\rm and } \  \  p(f_\a)=\a
$$
for all $\a>0$, and
$$
\Om_\infty(f_\a)=
\begin{cases}
\es    &{\rm if } \  \  \a<1 \\
\{0\} &{\rm  if } \  \  \a\ge 1,
\end{cases}
$$
while
$$
\Om_\infty(f)=\{0\}.
$$
Of course for both systems, arising from $f_\a$ and $f$, the corresponding $\d$ number is equal to $1$ and $m_\d$ is the Lebesgue measure $\Leb$.

Another large class of $1$-dimensional parabolic maps, actually comprising the above, whose continuous inverse branches form a $1$-dimensional parabolic GDMS can be found in 
\cite{U2}. In conclusion, using also Corollary~\ref{c1da12.1J}, we have the following results which apply to all of them. 

\begin{thm}\label{t1ex18}
If $f:[0,1] \to [0,1]$ is the Farey map, then with notation of subsection \ref{CER} we have the following. Fix $\xi\in [0,1]$. 
If $B \subset [0,1]$ is a Borel set such that $\Leb(\bd B)=0$ and $Y\subset [0,1]$ is any set having at least two elements, then
\beq\label{1ex4par}
\lim_{T \to +\infty}  \frac{N_\xi(f; B,T)}{e^{T}} = \frac{\psi_1(\xi)}{\chi_1} \Leb(B),
\eeq
\beq\label{2ex4par}
\lim_{T \to +\infty}  \frac{N_p(f; B,T)}{e^{T}} = \frac1{\chi_1}\mu_1(B),
\eeq
and 
\beq\label{1ex5par}
\lim_{T \to +\infty}  \frac{D_Y^\xi(f; B,T)}{e^{T}} = 
\lim_{T \to +\infty}  \frac{E_Y^\xi(f; B,T)}{e^{T}} = C(Y)\Leb(B),
\eeq
where $C(Y)\in (0,+\infty]$ is a constant depending only on the map $f$ and the set $Y$. In addition $C(Y)$ is finite if 
and only if 
$$
0\notin\ov Y.
$$
\end{thm} 

Although this is not needed for our results in this monograph,
it is interesting that a simple calculation reveals that the attracting ``*" IFS of Section~\ref{section:parabolic} associated with the Farey IFS is just the Gauss IFS $\mathcal G$ described in Remark~\ref{r1ex5}.

As the next theorem shows, the counting situation is more complex in the case of Manneville-Pomeau maps.

\begin{thm}\label{t1ex18B}
If $\a>0$ and $f_\a:[0,1] \to [0,1]$ is the corresponding Manneville-Pomeau map, then with the notation of subsection \ref{CER} we have the following. Fix $\xi\in [0,1]$. 
If $B \subset [0,1]$ is a Borel set such that $\Leb(\bd B)=0$ and $Y \subset [0,1]$ is any set having at least two elements, then
\beq\label{1ex4par}
\lim_{T \to +\infty}  \frac{N_\xi(f_\a; B,T)}{e^{T}} = \frac{\psi_1(\xi)}{\chi_1} \Leb(B),
\eeq
\beq\label{2ex4par}
\lim_{T \to +\infty}  \frac{N_p(f_\a; B,T)}{e^{T}} = \frac1{\chi_1}\mu_1(B),
\eeq
and 
\beq\label{1ex5par}
\lim_{T \to +\infty}  \frac{D_Y^\xi(f_\a; B,T)}{e^{T}} = 
\lim_{T \to +\infty}  \frac{E_Y^\xi(f_\a; B,T)}{e^{T}} = C(Y)\Leb(B),
\eeq
where $C(Y)\in (0,+\infty]$ is a constant depending only on the map $f_\a$ and the set $Y$. In addition $C(Y)$ is finite if and only if either

\begin{enumerate}
\item 
$
0\notin \ov Y
$
or 
\item 
$
\a<1.
$
\end{enumerate}
\end{thm} 

\fr In general, we have the following.

\begin{thm}\label{t1ex18C}
If $f$ is generated by a parabolic Cantor set of \cite{U2}, then with notation of subsection \ref{CER}, we have the following. 

Fix $\xi$ belonging to the limit set of the iterated function system associated to $f$. 
If $B \subset X$ is a Borel set such that $m_\d(\bd B)=0$ and $Y \subset [0,1]$ is any set having at least two elements and contained in a sufficiently small ball centered at $\xi$, then
\beq\label{1ex4par}
\lim_{T \to +\infty}  \frac{N_\xi(f; B,T)}{e^{\delta T}} = \frac{\psi_\d(\xi)}{\d\chi_\d} m_\delta(B),
\eeq
\beq\label{2ex4par}
\lim_{T \to +\infty}  \frac{N_p(f; B,T)}{e^{\delta T}} = \frac1{\d\chi_\d}\mu_\delta(B),
\eeq
and 
\beq\label{1ex5par}
\lim_{T \to +\infty}  \frac{D_Y^\xi(f; B,T)}{e^{\delta T}} = 
\lim_{T \to +\infty}  \frac{E_Y^\xi(f; B,T)}{e^{\delta T}} = C_\xi(Y) m_\delta(B),
\eeq
where $C_\xi(Y)\in (0,+\infty]$ is a constant depending only on the map $f$, the point $\xi$, and the set $Y$. In addition $C_\xi(Y)$ is infinite if and only if
$$
\xi\in \Om_\infty(f)\cap\ov Y 
\  \  {\rm and } \  \ 
p(\xi)\le \d.
$$
\end{thm} 

\sp With respect to the  stochastic laws, as an immediate consequence of the results in Subsections~\ref{CLT-Parabolic-1} and \ref{CLT-Parabolic-2} we get that the following results hold for systems considered in the current subsection.

We begin with a Central Limit Theorem for the expansion along 
orbits.  

\begin{thm}\label{t41_2017_04_03} 
Let $T$ be either a Manneville-Pomeau map $f_\a$ with $\a<1$, or generally, the map generated by a parabolic Cantor set of \cite{U2} with $\Om_\infty(T)=\es$. Let $J$ be either the interval $[0,1]$ (Manneville-Pomeau) or the parabolic Cantor set.
Let $\sigma^2 = \P''(0)> 0$.  With the notation of Subsection~\ref{CER} if $G \subset \mathbb R$ is a Lebesgue measurable set with $\hbox{{\rm Leb}}(\partial G) = 0$, then 
$$
\lim_{n \to +\infty}
\mu_\d\left(
\left\{
z\in J \hbox{ : } \frac{\log \big|(T^n)'(z)\big| - \chi_{\mu_\d} n}{\sqrt{n}}
\in G
\right\}
\right)
= \frac{1}{\sqrt{2\pi}\sigma} \int_G e^{-\frac{t^2}{2\sigma^2}} \,dt.
$$  
In particular, for any $\alpha <\beta$
$$
\lim_{n \to +\infty}
\mu_\d\left(
\left\{
z\in J \hbox{ : } \alpha \leq \frac{\log\big|(T^n)'(z)\big| - \chi_{\mu_\d} n}{\sqrt{n}}
\leq \beta
\right\}
\right)
= \frac{1}{\sqrt{2\pi}\sigma} \int_\a^\b e^{-\frac{t^2}{2\sigma^2}}\, dt.
$$  
\end{thm}

We next have a Central Limit Theorems for diameters.

\begin{thm}\label{t42_2017_04_03} 
Let $T$ be either a Manneville-Pomeau map $f_\a$ with $\a<1$, or generally, the map generated by a parabolic Cantor set of \cite{U2} with $\Om_\infty(T)=\es$. Let $J$ be either the interval $[0,1]$ (Manneville-Pomeau) or the parabolic Cantor set.
Let $\sigma^2 = \P''(0)> 0$.  With the notation of Subsection~\ref{CER},  
for every $e \in F$ let $Y_e \subset R_e$ be a set with at least two points, then if  $G \subset \mathbb R$ is a Lebesgue measurable set with $\hbox{\rm Leb}(\partial G) = 0$ we have 
$$
\lim_{n \to +\infty}
\mu_\d\left(
\left\{
z\in J \hbox{ : }
 \frac{
 -\log \hbox{\rm diam}\(T_x^{-n}(Y_{e(z,n)})\)  - \chi_{\mu_\d} n
 }{\sqrt{n}}
\in G
\right\}
\right)
= \frac{1}{\sqrt{2\pi}\sigma} \int_G e^{-\frac{t^2}{2\sigma^2}}\, dt.
$$  
In particular, for any $\alpha < \beta$
$$
\lim_{n \to +\infty}
\mu_\d\left(
\left\{
\omega \in J \hbox{ : } \alpha \leq 
 \frac{-\log \hbox{\rm diam}\(T_x^{-n}(Y_{e(z,n)})\) - \chi_{\mu_\d} n}{\sqrt{n}}
\leq \beta
\right\}
\right)
= \frac{1}{\sqrt{2\pi}\sigma} \int_\alpha^\beta e^{-\frac{t^2}{2\sigma^2}}\,dt.
$$  
\end{thm}

Next, we have a Central Limit Theorem for preimages.

\begin{thm}\label{t44_2017_04_03}
Let $T$ be either a Manneville-Pomeau map $f_\a$ with $\a<1$, or generally, the map generated by a parabolic Cantor set of \cite{U2} with $\Om_\infty(T)=\es$. Let $J$ be either the interval $[0,1$ (Manneville-Pomeau) or the parabolic Cantor set. Then
for every $\xi\in J$, we have that
$$
\lim_{n \to +\infty} \int_{T^{-n}(\xi)} \frac{\log\big|(T^n)'\big|}{n} d\mu_n 
= \chi_\delta.
$$
\end{thm} 
 
Finally, we have a Central Limit Theorem for counting. 
 
\begin{thm}\label{t44_2017_04_04} 
Let $T$ be either a Manneville-Pomeau map $f_\a$ with $\a<1$, or generally, the map generated by a parabolic Cantor set of \cite{U2} with $\Om_\infty(T)=\es$. Let $J$ be either the interval $[0,1$ (Manneville-Pomeau) or the parabolic Cantor set. Then
for every $\xi\in J$ the sequence of random variables $(\Delta_n)_{n=1}^\infty$ converges in distribution to the normal (Gaussian) distribution 
$\mathcal N_0(\sigma)$ with mean value zero and the variance $\sigma^2 = \P''(\delta)>0$. Equivalently, the sequence 
$(\mu_n \circ \Delta_n^{-1})_{n=1}^\infty$
converges weakly to the normal distribution $\mathcal N_0(\sigma^2)$.  This means that for every Borel set $F \subset \mathbb R$ with 
$\hbox{\rm Leb}(\partial F) = 0$, we have 
\begin{equation}\label{1_2017_04_04}
\lim_{n \to +\infty} \mu_n(\Delta_n^{-1}(F))
= \frac{1}{\sqrt{2\pi} \sigma}
\int_F e^{- t^2/2\sigma^2} dt.
\end{equation} 
\end{thm}

\subsection{Parabolic Rational Functions}\label{PRF}

Now we pass to the counting applications for parabolic rational functions.
We recall that if $f: \widehat {\mathbb C} \to  \widehat {\mathbb C}$  is a rational function then $\xi \in \widehat {\mathbb C}$  is called a rationally indifferent 
(or just parabolic) periodic point of $f$ if $f^p(\xi) = \xi$ for some integer $p \geq 1$ and $(f^p)'(\xi)$ is a root of unity.  It is well known and easy to to see that then $\xi \in J(f)$, the Julia set of $f$. The following theorem has been proved in \cite{DU_LMS}.

\begin{thm}\label{t2ex18}
If $f: \widehat {\mathbb C}  \to \widehat {\mathbb C}$ is a rational function, then the following two conditions are equivalent.
\begin{enumerate}
\item $f|_{{\mathcal J}(f)}: J(f) \to J(f)$ is expansive.

\sp\item $|f'(z)| >0$ for all $z \in  J(f)$, i.e. $J (f)$ contains no critical point of $f$.
\end{enumerate}
\end{thm}
In addition, if (a) or (b)  hold then the map 
$\widehat f: \widehat{\mathbb C} \to \widehat {\mathbb C}$ is not expanding iff $J(f)$ contains a parabolic periodic point.
Following \cite{DU_LMS} and \cite{DU_Forum} we then call the map 
$\widehat f: \widehat {\mathbb C} \to \widehat {\mathbb C}$ parabolic.

Probably, the best known example of a parabolic rational function is the polynomial
$$
\widehat \C \ni z \longmapsto f_{1/4}(z) := z^2 + \frac{1}{4}\in \widehat \C.
$$
It has only one parabolic point, namely $z = 1/2$. In fact this is a fixed point of $f_{1/4}$ and $f_{1/4}'(1/2) =1$.  It was independently proved in \cite{U3} and \cite{Zdunik} that 
\beq\label{1ex19}
\delta:= \HD(\mathcal J_{1/4}) > 1.
\eeq
With the arguments parallel to those in the proof of Theorem~ \ref{t1ex4}, as a consequence of Theorem~\ref{t2pc6_B}, Theorem~\ref{t1dp13} and Theorem~\ref{t1dp13B}, we get the following.

\begin{cor}\label{c1ex19}
If $f:\widehat{ \mathbb  C} \to \widehat{ \mathbb  C}$ is a parabolic rational function then with notation of Subsection \ref{CER}, we have the following.

Fix $\xi\in J(f)$. If $B \subset \widehat{ \mathbb  C}$ is a Borel set such that $m_\d(\bd B)=0$ and $Y \subset \widehat{ \mathbb  C}$ is any set having at least two elements and contained in a sufficiently small ball centered at $\xi$, then
\beq\label{1ex4parG}
\lim_{T \to +\infty}  \frac{N_\xi(f; B,T)}{e^{\delta T}} = \frac{\psi_\d(\xi)}{\d\chi_\d} m_\delta(B),
\eeq
\beq\label{2ex4parG}
\lim_{T \to +\infty}  \frac{N_p(f; B,T)}{e^{\delta T}} = \frac1{\d\chi_\d}\mu_\delta(B),
\eeq
and 
\beq\label{1ex5paG}
\lim_{T \to +\infty}  \frac{D_Y^\xi(f; B,T)}{e^{\delta T}} = 
\lim_{T \to +\infty}  \frac{E_Y^\xi(f; B,T)}{e^{\delta T}} = C_\xi(Y) m_\delta(B),
\eeq
where $C_\xi(Y)\in (0,+\infty]$ is a constant depending only on the map $f$, the point $\xi$, and the set $Y$. In addition $C_\xi(Y)$ is infinite if and only if
$$
\xi\in \Om_\infty(f)\cap\ov Y 
\  \  {\rm and } \  \ 
p(\xi)\le \d.
$$
\end{cor} 

\sp As in the previous subsection, the
 stochastic laws appear an immediate consequence of the results in Subsections~\ref{CLT-Parabolic-1} and \ref{CLT-Parabolic-2} we get the following.

We begin with a Central Limit Theorem for the expansion along 
orbits.  
\begin{thm}\label{t51_2017_04_03} 
Let $f:\oc\to\oc$ be a parabolic rational function of the Riemann sphere $\widehat{\mathbb{C}}$ with $\d>\frac{2p}{p+1}$.  With notaion of Theorem~\ref{t1ex4} we have the following. 

There exists $\sigma^2 > 0$ (in fact $\sg^2=\P''(0)> 0$) such that if $G \subset \mathbb R$ is a Lebesgue measurable set with $\hbox{{\rm Leb}}(\partial G) = 0$, then 
$$
\lim_{n \to +\infty}
\mu_\d\left(
\left\{
z\in J(f) \hbox{ : } \frac{\log \big|(f^n)'(z)\big| - \chi_{\mu_\d} n}{\sqrt{n}}
\in G
\right\}
\right)
= \frac{1}{\sqrt{2\pi}\sigma} \int_G e^{-\frac{t^2}{2\sigma^2}} \,dt.
$$  
In particular, for any $\alpha <\beta$
$$
\lim_{n \to +\infty}
\mu_\d\left(
\left\{
z\in J(f) \hbox{ : } \alpha \leq \frac{\log\big|(f^n)'(z)\big| - \chi_{\mu_\d} n}{\sqrt{n}}
\leq \beta
\right\}
\right)
= \frac{1}{\sqrt{2\pi}\sigma} \int_\a^\b e^{-\frac{t^2}{2\sigma^2}}\, dt.
$$  
\end{thm}

We next have a Central Limit Theorem for diameters.

\begin{thm}\label{t12_2017_04_03} 
Let $f:\oc\to\oc$ be a 
parabolic rational function of the Riemann sphere $\widehat{\mathbb{C}}$ with $\d>\frac{2p}{p+1}$. Let $\sigma^2:=\P''(0)> 0$. With notaion of Subsection~\ref{CER} we have the following. 

For every $e \in F$ let $Y_e \subset R_e$ be a set with at least two points. If $G \subset \mathbb R$ is a Lebesgue measurable set with $\hbox{\rm Leb}(\partial G) = 0$, then 
$$
\lim_{n \to +\infty}
\mu_\d\left(
\left\{
z\in J(f) \hbox{ : }
 \frac{
 -\log \hbox{\rm diam}\(f_x^{-n}(Y_{e(z,n)})\)  - \chi_{\mu_\d} n
 }{\sqrt{n}}
\in G
\right\}
\right)
= \frac{1}{\sqrt{2\pi}\sigma} \int_G e^{-\frac{t^2}{2\sigma^2}}\, dt.
$$  
In particular, for any $\alpha < \beta$
$$
\lim_{n \to +\infty}
\mu_\d\left(
\left\{
z\in J(f) \hbox{ : } \alpha \leq 
 \frac{-\log \hbox{\rm diam}\(f_x^{-n}(Y_{e(z,n)})\) - \chi_{\mu_\d} n}{\sqrt{n}}
\leq \beta
\right\}
\right)
= \frac{1}{\sqrt{2\pi}\sigma} \int_\alpha^\beta e^{-\frac{t^2}{2\sigma^2}}\,dt.
$$  
\end{thm}

Finally, we have a Central Limit Theorem for counting.

\begin{thm}\label{t54_2017_04_03}
If $f:\oc\to\oc$ is a parabolic rational function of the Riemann sphere $\widehat{\mathbb{C}}$ with $\d>\frac{2p}{p+1}$, then for every $\xi\in J(f)$, we have that
\begin{equation}\label{12_2017_04_04}
\lim_{n \to +\infty} \int_{f^{-n}(\xi)} \frac{\log\big|(f^n)'\big|}{n} d\mu_n 
= \chi_\delta.
\end{equation}
\end{thm} 
 
\begin{thm}\label{t54_2017_04_04} 
If $f:\oc\to\oc$ is a parabolic rational function of the Riemann sphere $\widehat{\mathbb{C}}$ with $\d>\frac{2p}{p+1}$, then the sequence of random variables $(\Delta_n)_{n=1}^\infty$ converges in distribution to the normal (Gaussian) distribution 
$\mathcal N_0(\sigma)$ with mean value zero and the variance $\sigma^2 = P''(\delta)>0$.  
Equivalently, the sequence 
$(\mu_n \circ \Delta_n^{-1})_{n=1}^\infty$
converges weakly to the normal distribution $\mathcal N_0(\sigma^2)$.  This means that for every Borel set $F \subset \mathbb R$ with 
$\hbox{\rm Leb}(\partial F) = 0$, we have 
\begin{equation}\label{1_2017_04_04}
\lim_{n \to +\infty} \mu_n(\Delta_n^{-1}(F))
= \frac{1}{\sqrt{2\pi} \sigma}
\int_F e^{- t^2/2\sigma^2} dt.
\end{equation} 
\end{thm}

\sp Note that for the map $f_{1/4}:\widehat{ \mathbb  C} \to \widehat{ \mathbb  C}$,
$$
p\(f_{1/4}\)=p_{1/4}=\max\{p(a) \hbox{ : } a \in \Omega\} =1,
$$
so by (\ref{1ex19}) we have that, 
\beq\label{1_2017_04_05}
\delta > p_{1/4}=p\(f_{1/4}\)=\frac{2p_{1/4}}{2p_{1/4}+1}.
\eeq
Thus, we arrive at the following result known from \cite{DU_LMS}, \cite{DU_Forum}, and (for the concluding argument) \cite{ADU}.

\bthm\label{t1_2017_03_20}
For the map $f_{1/4}:\widehat{ \mathbb  C} \to \widehat{ \mathbb  C}$, $\Om_\infty=\es$ and the invariant measure $\mu_\d$ is finite, so a probability after normalization.
\ethm

\fr Recalling also that the set $J(f_{1/4})$ is connected, as a consequence of all in this subsection, we get the following.

\begin{cor}\label{c1ex19H}
If $f_{1/4}: \widehat{ \mathbb  C} \to \widehat{ \mathbb  C}$ is parabolic quadratic polynomial
$$
\widehat \C \ni z \longmapsto f_{1/4}(z) := z^2 + \frac{1}{4}\in \widehat \C,
$$
then with notation of Subsection \ref{CER}, we have the following.

Fix $\xi\in J(f_{1/4})$. 
If $Y \subset \widehat{ \mathbb  C}$ is any set having at least two elements and contained in a sufficiently small ball centered at $\xi$, then there exists a constant $C_\xi(Y)\in (0,+\infty)$ such that if $B \subset \widehat{ \mathbb C}$ is a Borel set with $m_\d(\bd B)=0$, then
\beq\label{1ex4parG}
\lim_{T \to +\infty}  \frac{N_\xi(f_{1/4}; B,T)}{e^{\delta T}} = \frac{\psi_\d(\xi)}{\d\chi_\d} m_\delta(B),
\eeq
\beq\label{2ex4parG}
\lim_{T \to +\infty}  \frac{N_p(f_{1/4}; B,T)}{e^{\delta T}} = \frac1{\d\chi_\d}\mu_\delta(B),
\eeq
and 
\beq\label{1ex5paG}
\lim_{T \to +\infty}  \frac{D_Y^\xi(f_{1/4}; B,T)}{e^{\delta T}} = 
\lim_{T \to +\infty}  \frac{E_Y^\xi(f_{1/4}; B,T)}{e^{\delta T}} = C_\xi(Y) m_\delta(B).
\eeq
\end{cor} 

\brem\label{r1_2017_04_05}
Because of \eqref{1_2017_04_05} all the hypotheses of Theorems~\ref{t51_2017_04_03} -- \ref{t54_2017_04_04} are satisfied for the map $f_{1/4}:\widehat{ \mathbb  C} \to \widehat{ \mathbb  C}$; so, in particular, all these theorems hold for the map $f=f_{1/4}$.
\erem

\sp On the other hand if $f : \widehat{ \mathbb  C} \to \widehat{ \mathbb  C}$ is a parabolic rational function with $\HD(J(f)) \leq 1$, 
which is the case for many maps, in particular those of the form $\widehat {\mathbb C} \ni z \mapsto 2 + 1/z + t$ where $t \in \mathbb R$ or parabolic Blaschke products, then 
$$
\delta \le 1\le p_a
$$
for every point $a\in \Om(f)$. Thus also
$$
\Om_\infty(f)=\Om(f)
$$
and, as an immediate consequence of Corollary~\ref{c1ex19}, we get the following.

\begin{cor}\label{c1ex19I}
If $f:\widehat{ \mathbb  C} \to \widehat{ \mathbb  C}$ is a parabolic rational function with $\HD(J(f)) \leq 1$, then with notation of Subsection \ref{CER}, we have the following.

Fix $\xi\in J(f)$. If $B \subset \widehat{ \mathbb  C}$ is a Borel set such that $m_\d(\bd B)=0$ and $Y \subset \widehat{ \mathbb  C}$ is any set having at least two elements and contained in a sufficiently small ball centered at $\xi$, then
\beq\label{1ex4parI}
\lim_{T \to +\infty}  \frac{N_\xi(f; B,T)}{e^{\delta T}} = \frac{\psi_\d(\xi)}{\d\chi_\d} m_\delta(B),
\eeq
\beq\label{2ex4parI}
\lim_{T \to +\infty}  \frac{N_p(f; B,T)}{e^{\delta T}} = \frac1{\d\chi_\d}\mu_\delta(B),
\eeq
and 
\beq\label{1ex5parI}
\lim_{T \to +\infty}  \frac{D_Y^\xi(f; B,T)}{e^{\delta T}} = 
\lim_{T \to +\infty}  \frac{E_Y^\xi(f; B,T)}{e^{\delta T}} = C_\xi(Y) m_\delta(B),
\eeq
where $C_\xi(Y)\in (0,+\infty]$ is a constant depending only on the map $f$, the point $\xi$, and the set $Y$. In addition $C_\xi(Y)$ is finite if and only if
$$
\xi\notin \Om(f)\cap\ov Y.
$$
\ecor

\part{{\Large Examples and Applications, II: Kleinian Groups}}

In this part we apply our counting results to some large classes of Kleinian Groups. These include all finitely generated Schottky groups and essentially all finitely generated Fuchsian groups. The applications described in this section would actually fit into two previous sections: Convex co-compact (hyperbolic) groups would fit to Section~\ref{AECDS} while parabolic ones would fit to Section~\ref{CPDS}. However, because of their distinguished character and specific methods used to deal with them, we collect all applications to Kleinian groups in one separate part. 

\sp\section{{\large{\bf Finitely Generated Schottky Groups with 
no Tangencies}}}\label{Schottky-No Tangencies}

Fix an integer $d \geq 1$.  Fix also an integer $q \geq 2$.  Let 
$$
B_j, \ \ j = \pm 1, \pm 2, \cdots, \pm q,
$$ 
be open balls in $\R^d$ with mutually disjoint closures.  For every $j=1,2, \cdots, q$ let 
$$
g_j: \widehat {\mathbb R}^d \to \widehat {\mathbb R}^d
$$ 
be a conformal self-map of the one point compactification of $\mathbb R^d$ (thus making $\widehat {\mathbb R}^d$ conformally equivalent to the unit sphere $S^d \subset \mathbb R^{d+1}$) such that 
\beq\label{1ea8}
g_j(B_{-j}^c) = \overline {B}_j.
\eeq
The group $G$ generated by the maps $g_j$, $j=1, \ldots, q$, is called a hyperbolic Schottky group; hyperbolic alluding to the lack of  tangencies. If there is no danger of misunderstanding, we will frequenly skip in this section the adjective ``hyperbolic'',  speaking simply about Schottky groups. Note that if we set 
$$
g_j := g_{-j}^{-1}
$$
for all $j = -1, \ldots, -q$ then \eqref{1ea8} holds for all $j = \pm 1, \pm 2, \cdots, \pm q$.

Denote by $\mathbb H^d$ the space $\mathbb R^d \times (0, +\infty)$ endowed with the Poincar\'e metric. The Poincar\'e Extension Theorem ensures that all the maps $g_j$, $j=1, \ldots, q$, uniquely extend to conformal self-maps of 
$$
\ov{\mathbb H}^d := \mathbb R^d \times [0,1),
$$
also denoted by $g_j$, onto itself. Their restrictions to $\mathbb H^d$, which are again also denoted by $g_j$, are isometries with respect to the Poincar\'e metric $\rho$ on $\mathbb H^d$.  
The group generated by these isometries in discrete, is also denoted by $G$, and is also called the Schottky group generated by the maps $g_j$, $j=1, \ldots, q$. 
For every $j = \pm 1, \pm 2, \cdots, \pm q$ denote by $\hat B_j$ the half-ball in 
$\mathbb H^{d+1}$ with the same center and radius as those of $B_j$. We recall  the following well-known standard fact.

\bfact
The region
$$
\cR:=\mathbb H^{d+1} \sms \bu_{j=1}^q (\hat B_j\cup \hat B_{-j})
$$
is a fundamental domain for the action of $G$ on $\mathbb H^{d+1}$ and 
$$
\hat\R^d\sms \bu_{j=1}^q (B_j\cup B_{-j})
$$
is a fundamental domain for the action of $G$ on $\hat\R^d$.
\efact

For any $z \in \overline{\mathbb H}^{d+1}$ the set of cluster points of the set $Gz$ is contained in 
$$
\bu_{j=1}^q \overline{B}_j\cup \overline{B}_{-j},
$$
and is independent of $z$. We call it the limit set of $G$
and denote it by $\Lambda(G)$.  This set is compact, perfect, $G(\Lambda(G)) = \Lambda(G)$ and $G$ acts minimally on $\Lambda(G)$.  We denote
$$
V:= \{\pm1, \pm 2, \ldots, \pm q\}, \  \ 
E := V\times V \sms \{(i,-i)   \hbox{ : } i \in V\},
$$
and introduce an incidence matrix $A: E \times E  \to \{0,1\}$ by declaring that
$$
A_{(a,b),(c,d)}
=
\begin{cases}
1 & \hbox{ if } b=c \cr
0 & \hbox{ if } b\neq c
\end{cases}
$$
Furthermore, we set for all $(a,b) \in E$, $t(a,b) =b$ and $i(a,b)=a$, and 
$$
g_{(a,b)}:= g_a|_{\overline B_b}: \overline B_b \to \overline B_a.
$$
In this way we have associated to $G$ the conformal graph directed Markov system $$
\mathcal S_G:=\{g_e:e \in E\}.
$$
By the very definition of this system, for every $\om\in E_A^*$, say $\om=(a_1,b_1)(a_2,b_2)\ld(a_n,b_n)$, we have that
$$
g_\om= g_{(a_1,b_1)} \circ g_{(a_2,b_2)} \circ \ldots \circ g_{(a_n,b_n)}|_{\ov B_{b_n}}
= g_{a_1} \circ g_{a_2} \circ \ldots \circ g_{a_n}|_{\overline{B}_{b_n}}:
\overline{B}_{b_n} \to \overline{B}_{a_1}.
$$
Of course,
$$
\Lambda(G) = J_{\cS_G}
$$
and we make the following observation: 
\begin{observation}\label{o1ex9}
The projection map 
$$
\pi = \pi_G:=\pi_{\mathcal S_G}: E_A^{\mathbb N} \to \Lambda(G)
$$
is a homeomorphism, in particular, a bijection.
\end{observation}

We will now make some preparatory comments on our approach 
 to  counting problems for  the group
$G$ by means of the conformal GDMS $\mathcal S_G$. For any element $\xi \in \Lambda(G)$ there exists a unique $k \in V$ such that $\xi \in \overline B_k$ and by Observation~\ref{o1ex9}, a unique $\rho \in E_A^\infty$ such that 
$$
\xi =  \pi_G(\rho);
$$
of course $i(\rho) = k$. Set 
$$
G_\xi := \{g_\omega:\omega \in E_\rho^*\}
= \{g_\omega:\omega \in E_A^*,\, t(\omega) = i(\rho)=k\}:=G_k.
$$
The next obvious observation is the following.

\begin{observation}\label{o2ex9}
The maps 
$$
E_\rho^* \ni \omega \mapsto g_\omega \in G 
\  \ {\rm and } \  \
E_\rho^* \ni \omega \longmapsto g_\omega(\xi) \in G(\xi)
$$ 
are both $1$-to-$1$.
\end{observation}

For every $g=g_\omega \in G_\xi$, $\omega \in E_\rho^*$, we denote 
$$
\lambda_\xi(g) = -\log |g'(\xi)| = -\log |g_\omega'(\xi)| = \lambda_\rho(\omega).
$$
Furthermore, for every set $Y \subset \ov B_k$ we denote 
$$
\Delta_{\omega} (Y) = - \log (\hbox{diam} (g_\omega(Y)))
$$

Now we move onto the discussion of periodic points of the system $\mathcal S_G$ along with periodic orbits of the geodesic flow and closed geodesics on the hyperbolic manifold $\mathbb H^{d+1}/G$.

Indeed, first of all we recall  the following.

\bobs\label{o1ex11.1}
The map $E_p^*\ni\om\longmapsto g_\om\in G$ is 1-to-1.
\eobs
Now, if $\omega \in E_p^*$ then 
$$
g_\omega (\overline{B}_{t(\omega)}) \subset \overline{B}_{t(\omega)}
$$
and the map $g_{\omega}:  \overline{B}_{t(\omega)} \to  \overline{B}_{t(\omega)}$ has a unique fixed point.  Call it $x_\omega$.
We know that the map ${\overline g}_\omega: {\widehat {\mathbb R}}^d \to {\widehat {\mathbb R}}^d$
has exactly one other fixed point.  Call it $y_\omega$.
Denoting by $-\omega$ the word
$$
(-\alpha_n, -\alpha_{n-1}) (-\alpha_{n-1}, -\alpha_{n-2})  (-\alpha_{n-2}, -\alpha_{n-3})
\cdots
 (-\alpha_{2}, -\alpha_{1})  (-\alpha_{1}, -\alpha_{n})
$$
and marking that $\omega =  (\alpha_1, \beta_1) (\alpha_2, \beta_2) \cdots  (\alpha_n, \beta_n)$
belongs to $E_p^*$, we see that $-\omega \in E_p^*$ and $g_{-\omega} = g_{\omega}^{-1}$ as elements of the group 
$G$.  Then $x_{-\omega} \in {\overline B}_{-\alpha_n} \neq {\overline B}_{\beta_n} $.  So as $g_\omega (x_{-\omega}) = x_{-\omega}$ we must have $y_\omega = x_{-\omega}$. Therefore, we have the following.

\bprop\label{p1ex11} If $\omega \in E_p^*$ then $\g_\om$, the geodesic in $\mathbb H^{d+1}$ joining $y_\omega$ and $x_\omega$ (oriented from $y_\omega$ to $x_\omega$), is fixed by $g_\omega$, crosses the fundamental domain $\mathcal R$, $\gamma_\omega/G$ is a closed geodesic on $\mathbb H^{d+1}/G$ with length \beq\label{1ex11}
\lambda_p(\omega)= - \log |g_\omega'(x_\omega)|,
\eeq
and simultaneously represents a periodic orbit of the geodesic flow on the unit tangent bundle of $\mathbb H^{d+1}/G$ with the periodic equal to $\lambda_p(\omega)$.
\eprop

On the other hand, if $\gamma$ is a closed oriented geodesic in $\mathbb H^{d+1}/G$ 
then its full lift $\widetilde \gamma$ in $\mathbb H^{d+1}$ consists of a countable union of mutually disjoint geodesics in $\mathbb H^{d+1}$.  Then the set $\widetilde \gamma \cap \mathcal R$ is not empty and each of its connected components is an oriented geodesic joining two distinct faces of $\mathcal R$.  Fix $\Delta$, one of the such connected components. Let $\widehat \Delta$ be the full geodesic in $\mathbb H^{d+1}$
containing $\Delta$ and oriented in the direction of $\Delta$.
Fix $z \in \widehat \Delta$ arbitrarily. Denote by $l(\g)$ the length of $\g$.  Let $z^*$ be the unique point on $\widehat \Delta$ such that $\rho(z^*,z) = l(\gamma)$ and the segment $[z,z^*]$ is oriented in the direction of $\widehat \Delta$.   Since both points $z$ and $z^*$ project to the same element of $\mathbb H^{d+1}/G$, there exists a unique element $g_{\gamma,\Delta} \in G$ such that 
$g_{\gamma, \Delta}(z) = z^*$.  Since $\gamma$ has no self intersections it follows that 
$$
g_{\gamma, \Delta}(\widetilde \Delta) = \widetilde \Delta.
$$
Denote be $x_{\Delta}$ and $y_{\Delta}$ the endpoints of $\widehat \Delta$ labeled so that the direction of $\widehat \Delta$ is from $y_\Delta$ to $x_\Delta$. Let $a,b$ be unique elements of $V$ such that $x_\Delta \in \overline B_a$ and  $y_\Delta \in \overline B_b$.
Let $\widehat {\om}_\Delta \in E_A^*$ and $k\in V$ be the unique elements respectively of $E_A^*$ and $V$ such that 
$$
g_{\gamma, \Delta} = g_{\widehat{\om}_\Delta}
\  \  \hbox{ and } \  \  
t(\widehat \om_\Delta) = k,
$$
the first equality meant in the group $G$. We will prove the following.

\medskip
\noindent 
{\it Claim 1.}  $k=-b$

\begin{proof}
By our choice of the endpoints $x_\Delta$ and $y_\Delta$, $y_\Delta$ is an attracting fixed point of $g_k^{-1} (g_{\widehat {\omega_\Delta}})^{-1} 
 = (g_{\widehat \omega_\Delta} \circ g_k)^{-1}$.
Since also $y_\De \in \overline B_b$, we thus conclude that 
\beq\label{1ex12}
g_{-k}\circ(g_{\widehat {\omega}_\Delta})^{-1} (\overline B_b) \sbt \overline B_b.
\eeq
Consequently, $-k=b$, and Claim 1 is proved.  
\end{proof}

\fr Since also $a \neq b$ as $\De$ intersects $\mathcal R$, we thus conclude that 
\beq\label{2ex12}
 \omega_{\gamma, \Delta}
:= \widehat\omega_\Delta(-b, a) \in E_A^*
\  \  \hbox{ and } \  \
 g_{\gamma, \Delta}=g_{\om_{\g,\De}}.
\eeq
In addition, by the same token as \eqref{1ex12} we get that $g_{\widehat\omega_\Delta}\circ g_k (\overline{B_a}) \subset \overline {B_a}$.  Thus $i (\widehat \omega_\Delta) = a$. Consequently 
$$
\omega_\Delta \in E_p^*.
\label{3ex12}
$$
In addition,
$$
\lambda_p(\omega_{\gamma, \Delta}) 
 =  \lambda(g_{\omega_{\gamma, \Delta}}) 
 =  \lambda(g_{\gamma,\Delta}) 
= \rho(z^*, z) = l(\gamma) \label{2ex13}
$$
and
$$
 \gamma_{\omega_{\gamma, \Delta}}/G = \gamma.
 \label{1ex13}
$$
 
Denote by $\mathcal C (\gamma)$ the set of all connected components of $\widetilde \gamma \cap \mathcal R$.  Of course we have the following.

\bobs\label{o2ex13}
The function $\mathcal C(\gamma) \ni \Delta \longmapsto \omega_{\gamma, \Delta} \in E_p^*$
is one-to-one.
\eobs

We shall prove the following.

\begin{prop}\label{p1ex13}
The map $E_p^* \longmapsto \gamma_\omega/G$ is a surjection from $E_p^*$
onto $\mathcal C(G)$, the set of all closed oriented geodesics on
 $\mathcal H^{d+1}/G$.  Furthermore,  if $\gamma$ is a closed oriented geodesic on $\mathbb H^{d+1}/G$ then 
$$
\Per(\gamma) := \{\omega \in E_p^*:\gamma_\omega/G = \gamma\}
= \{\omega_{\gamma, \Delta} \in E_p^*:\De\in \mathcal C(\gamma)\}
\label{3ex13}
$$
and $\Per(\gamma)$ forms a full periodic cycle, i.e. the orbit of any element of $\omega \in \Per(\gamma)$ under the map $\sigma^*: \omega \longmapsto \sigma(\omega)\omega_1$.
\end{prop}

\begin{proof}
The first part of this proposition has already been proved.   
More precisely, it is contained in Proposition \ref{p1ex11} and formula \eqref{1ex13}. The inclusion 
$$
\{
\omega_\Delta \in E_p^* \hbox{ : } \Delta \in \mathcal C (\gamma)  
\}
\subset
 \{
\omega \in E_p^* \hbox{ : } \gamma_\omega/G = \gamma 
\}
$$
follows immediately from \eqref{1ex13}. The inclusion  
$$
\Per(\gamma)= 
\{
\omega_\Delta \in E_p^* \hbox{ : } \gamma_\omega/G = \gamma  
\}
\subset
 \{
\omega_{\gamma, \Delta} \in E_p^* \hbox{ : }  \Delta \in \mathcal C(\gamma)
\}
$$
follows from the fact that for each $\omega \in E_p^*$ the geodesic $\gamma_\omega$
crosses $\mathcal R$.  So formula \eqref{3ex13} is established.  Now, 
$$
g_{\sg(\omega)\om_1} (g_{\omega_1}^{-1}(x_\omega))
= g_{\omega_1}^{-1} \circ g_{\omega_1} \circ g_{\sigma(\omega)}(x_\omega)
= g_{\omega_1}^{-1} \circ g_\om(x_\omega)
= g^{-1}_{\omega_1}(x_\omega).
$$
Similarly, 
$$
g_{\sg(\omega))\om_1} (g_{\omega_1}^{-1}(y_\omega))
= g_{\omega_1}^{-1}(y_\omega).
$$
Also, by the Chain Rule,
$$
l(g_{\sigma(\omega)\omega_1})
= \lambda_p (\sigma (\om)\omega_1) = \lambda_p(\omega) = l(g_\omega).
$$
Therefore, noting also that $g_{\omega_1}^{-1}(\gamma_\om)$
crosses $\mathcal R$, we get 
$$
\gamma_{\sigma(\omega)\omega_1} = g_{\omega_1}^{-1}(\gamma_\omega)
\  \  \hbox{ and } \  \
\gamma_{\sigma(\omega)\omega_1}/G = \gamma_\om/G = \gamma.
$$
So, $\sigma(\omega)\omega_1 \in \Per(\gamma)$ and we have proved that $\Per(\gamma)$ is a union of full periodic cycles.  Let $\omega \in \Per(\gamma)$ be arbitrary. Put $n:= |\omega|$. Since 
$$
\sum_{j=0}^{n-1} l(\gamma_{\sigma^{*j}(\omega)}\cap \mathcal R)
= l(\gamma)
= \sum_{\Delta \in \mathcal C(\gamma)}|\Delta|,
$$
since all elements $\gamma_{\sigma^{*j}(\omega)}\cap \mathcal R$
are mutually disjoint, and since 
$\{\gamma_{\sigma^{*j}(\omega)} \cap \mathcal R \hbox{ : } 0 \leq j \leq n-1 \}
\subset \mathcal C(\gamma)$ we can  conclude
$$
\{\gamma_{\sigma^{*j}(\omega)} \cap \mathcal R \hbox{ : } 0 \leq j \leq n-1 \}
= \mathcal C(\gamma).
$$
Along with (\ref{3ex13}) and Observation \ref{o2ex13} this yields the last assertion of Proposition \ref{p1ex13} and the proof of this proposition is complete
\end{proof}

Denote by $\widehat G \subset G$ the set of those elements in $G$ for which 
$\gamma_g$, the oriented geodesic in $\mathbb H^{d+1}$ from its repelling fixed points
$y_g$ to its attracting fixed point $x_g$ crosses the fundamental domain $\mathcal R$. 
We can now complete Observation~\ref{o1ex11.1} by proving the following.

\begin{prop}\label{p1ex14.1}
The map $E_p^* \ni \omega \longmapsto g_\omega \in G$ is a bijection from $E_p^*$ onto $\widehat G$.
\end{prop}
  
\begin{proof}  
Observation~\ref{o1ex11.1} tells us that this map is one-to-one and Proposition~
\ref{p1ex11} tells us that its range is contained in $\widehat G$. Thus, in order to complete the proof we have to show that $\widehat G$ is contained 
in this range. So fix $g \in \widehat G$.  Let $\alpha$ be the projection on 
$\mathbb H^{d+1}/G$ of the geodesic $\gamma_g$ such that $l(\alpha) = \alpha(g)$.
Then $g = g_{\omega_\alpha, \Delta}$ where $\Delta = \gamma_g \cap \mathcal R$.
Since $\omega_{\alpha, \Delta} \in E_p^*$ we are done.
\end{proof}
  
Propositions~\ref{p1ex11} and \ref{p1ex13} provide a full description of closed 
oriented geodesics and periodic orbits of the geodesic flow in terms of symbolic dynamics and graph directed Markov systems. For the picture to be complete we also describe all periodic points of the group $G$.

\begin{prop}\label{p1ex14}
The map 
$$
E_p^* \ni \omega \longmapsto \langle \omega\rangle = \{ g \circ g_\omega \circ g^{-1} \hbox{ : } g \in G\}
$$
has the following properties:
\begin{enumerate}
\item 
$\langle \omega\rangle = \langle \tau\rangle \  \eqv
\  \langle \omega\rangle \cap \langle \tau\rangle \ne\es \  \eqv 
\ \tau = \sigma^{*j}(\omega)$ for some $j \geq 0$.
\item
Each element $g \circ g_\omega \circ g^{-1}$ has precisely two fixed points $g(x_\omega)$ and $g(y_\omega$).  In addition
$$
(g\circ g_\omega \circ g^{-1})'(g(x_\omega)) = g_\omega'(x_\omega)
\  \hbox{ and } \
(g\circ g_\omega \circ g^{-1})'(g(y_\omega)) = g_\omega'(y_\omega)
$$
\item 
For each $h \in G\sms \{\Id\}$ there exists a unique periodic cycle such that 
\begin{enumerate}
\item there exists $\omega \in E_p^*$ in this periodic cycle and a unique $g \in G$, depending on $\omega$, such that $h = g\circ g_\omega \circ g^{-1} $,
\item for each $\omega \in E_p^*$ in this periodic cycle there exists a unique $g \in G$, depending on $\omega$, such that $h = g\circ g_\omega \circ g^{-1} $.
\end{enumerate} 
\end{enumerate}
\end{prop}

\fr The proof of this proposition is straightforward and we omit it.  

\sp Now we pass to the main goal of this monograph, i.e., counting estimates.
We deal with these in the symbol space and on both
$\mathbb H^{d+1}$ and $\mathbb H^{d+1}/G$.  We start with appropriate definitions.

Let $B$ denote a Borel subset of $\R^d$. Set

$$
\pi_\xi (G; T, B):= \{ g \in G_\xi \hbox{ : } \lambda_\xi (g) \leq T \hbox{ and } g(\xi) \in B \}
$$

$$
\pi_\xi (G; T):=\pi_\xi (G; T, \R^d) 
=\{ g \in G_\xi \hbox{ : } \lambda_\xi (g) \leq T \}
$$

$$
\pi_p(G; T, B):= \{ \omega \in E_p^*
\hbox{ : } \lambda_p (\omega) = l(\gamma_\omega) \leq T \hbox{ and } x_\omega \in B \},
$$

$$
\pi_p(G; T):=\pi_p(G; T, \R^d)
=\{ \omega \in E_p^* \hbox{ : } \lambda_p (\omega) = l(\gamma_\omega) \leq T  \},
$$

$$
\widehat \pi_p (G,T) := \{g \in \widehat G \hbox{ : } l(\gamma_g)  \leq T \} 
$$

\sp\fr Having $k \in V = \{\pm j\}_{j=1}^q$ and  $Y\subset\overline B_k$
put  
$$
\Delta_g(Y) : = -\log\(\diam(g(Y))\).
$$
We further denote

$$
\mathcal D_\xi(G; T, B, Y): = \{g \in G_k \hbox{ : } \Delta_g(Y) \leq T \  \hbox{ and } \ g(\xi) \in B\},
$$

$$
\mathcal E_k(G; T, B, Y) = \{ g \in G_k \hbox{ : } \Delta_g(Y)\leq T  \  \hbox{ and } \  g(Y)\cap B\ne\es\},
$$
and 
$$
\mathcal E_k(G; T, Y):=\cE_k(G; T, \R^d, Y)
=\{ g \in G_k \hbox{ : } \Delta_g(Y)\leq T\}.
$$

\sp\fr We denote by 
$N_\xi(G; T, B)$, $N_\xi(G; T)$,
$N_p(G; T, B)$, $N_p(G; T)$,
$\widehat N_p(G; T)$,
$D_\xi(G; T, B, Y)$,
$E_k(G; T, B, Y)$ and $E_k(G; T, Y)$
the corresponding cardinalities.

\sp As an  immediate consequence  of Theorem~\ref{dyn}, Theorem~\ref{t1da7}, and  Theorem~\ref{t1ma1}
along with Observation \ref{o2ex9}, Proposition \ref{p1ex11}, Observation \ref{o1ex11.1},
Observation \ref{o2ex13} and Proposition \ref{p1ex14.1} we get the following.

\begin{thm}\label{t1ex16}
Let $G = \langle g_j\rangle_{j=1}^q$ be a hyperbolic finitely generated Schottky group acting on $\hat\R^d$, $d\ge 2$. 

\sp\begin{itemize}
\item Let $\delta_G$ be the Poincar\'e  exponent of $G$; it is known to be equal to $\HD(\Lambda(G))$.  

\sp\item Let $m_{\delta_G}$ be the Patterson-Sullivan
conformal measure for $G$ on $\Lambda(G)$.  

\sp\item Let $\mu_{\delta_G}$ be the $\mathcal S_G$-invariant measure on $\Lambda(G)$ equivalent to $m_{\delta_G}$.  

\sp\sp\item Fix $k\in\{\pm 1, \pm 2, \cdots, \pm q\}$ and $\xi \in \Lambda(G) \cap \overline B_k$.
\end{itemize}

\fr Let $B\sbt\R^d$ be a Borel set with $m_{\d_G}(\bd B)=0$ (equivalently $\mu_{\d_G}(\bd B)=0$) and let $Y\sbt\ov B_k$ be a set having at least two distinct points. Then with some constant $C_k(Y)\in(0,+\infty)$, we have that
$$
\lim_{T\to+\infty}\frac{N_\xi(G; T, B)}{e^{\d_G T}}
=\frac{\psi_{\d_G}(\xi)}{\d_G\chi_{\d_G}}m_{\d_G}(B),
\  \  \  \  \
\lim_{T\to+\infty}\frac{N_\xi(G; T)}{e^{\d_G T}}
=\frac{\psi_{\d_G}(\xi)}{\d_G\chi_{\d_G}},
$$

$$
\lim_{T\to+\infty}\frac{N_p(G; T, B)}{e^{\d_G T}}
=\frac{1}{\d_G\chi_{\d_G}}\mu_{\d_G}(B),
\  \  \  \  \
\lim_{T\to+\infty}\frac{N_p(G; T)}{e^{\d_G T}}
=\frac{1}{\d_G\chi_{\d_G}},
$$

$$
\begin{aligned}
\lim_{T\to+\infty}\frac{\widehat N_p(G; T)}{e^{\d_G T}}
&=\frac{1}{\d_G\chi_{\d_G}}, \\  \\
\lim_{T\to+\infty}\frac{D_\xi(G; T, B,Y)}{e^{\d_G T}}
&=C_k(Y)m_{\d_G}(B), \\  \\
\lim_{T\to+\infty}\frac{E_k(G; T, B,Y)}{e^{\d_G T}}
&=C_k(Y)m_{\d_G}(B), \\ \\
\lim_{T\to+\infty}\frac{E_k(G; T,Y)}{e^{\d_G T}}
&=C_k(Y).
\end{aligned}
$$
\end{thm}

\sp Theorem~\ref{t1ms1} -- Theorem~\ref{t1ms1-2again} for the conformal GDMS $\cS_G$, associated to the group $G$, are valid without changes. Therefore, we do not 
repeat  them here. However, we present the  appropriate versions of Theorems~\ref{t1ms5.2} and \ref{t1m58} as their formulations are closer to the group $G$. In order to get appropriate expressions in the language of the group $G$ itself, given $\xi\in\La(G)$, and an integer $n\ge 1$, we set
$$
G_\xi^n:=\{g_\om:\om\in E_\rho^n\}\sbt G_\xi.
$$
Furthermore, we define a probability measure $\mu_n$ on $G_\xi^n$ by setting that 

\begin{equation}\label{21_2017_04_04}
\mu_n(H) := \frac{\sum_{g \in H} e^{-\delta \lambda_\xi(g)}}{\sum_{\omega \in G_\xi^n} e^{-\delta \lambda_\xi(g)}}
\end{equation}
for every set $H \subset G_\xi^n$. As an immediate consequence of Theorem~\ref{t1ms5.2} we get the following.

\bthm\label{t31_2017_04_04}
If $G = \langle g_j\rangle_{j=1}^q$ is a hyperbolic finitely generated Schottky group acting on $\hat\R^d$, $d\ge 2$, then for every $\xi\in\La(G)$ we have that
$$
\lim_{n \to +\infty} \int_{G_\xi^n} \frac{\lambda_\xi}{n} d\mu_n 
= \chi_{\mu_\d}.
$$
\ethm

\fr Now define the functions $\Delta_n:G_\xi^n\to \mathbb R$ by the formulae
$$
\Delta_n(g) = 
\frac{ \lambda_\xi(g)- \chi n}{\sqrt{n}},
$$
As an immediate consequence of Theorem~\ref{t1m58} we get the following.

\begin{thm}\label{t32_2017_04_04} 
If $G = \langle g_j\rangle_{j=1}^q$ is a hyperbolic finitely generated Schottky group acting on $\hat\R^d$, $d\ge 2$, then for every $\xi\in\La(G)$
the sequence of random variables $(\Delta_n)_{n=1}^\infty$ converges in distribution to the normal (Gaussian) distribution 
$\mathcal N_0(\sigma)$ with mean value zero and the variance $\sigma^2 = \P_{\cS_G}''(\delta)>0$.  
Equivalently, the sequence 
$(\mu_n \circ \Delta_n^{-1})_{n=1}^\infty$
converges weakly to the normal distribution $\mathcal N_0(\sigma^2)$.  This means that for every Borel set $F \subset \mathbb R$ with 
$\hbox{\rm Leb}(\partial F) = 0$, we have 
\begin{equation}\label{1ms8}
\lim_{n \to +\infty} \mu_n(\Delta_n^{-1}(F))
= \frac{1}{\sqrt{2\pi} \sigma}
\int_F e^{- t^2/2\sigma^2} dt.
\end{equation} 
\end{thm}

\section{{\large{\bf Generalized (allowing tangencies) Schottky Groups}}}\label{tangent Schottky}

In this section we keep to the same setting and the same notation as in Subsection~\ref{Schottky-No Tangencies}.
except that we now do not assume that the closures $\overline B_j$, $j=\pm 1, \cdots, \pm q$ to be disjoint but merely that the open balls $B_j$, $j=\pm 1, \cdots, \pm q$ themselves are mutually disjoint.

\subsection{General Schottky Groups}
We also assume that if an element $g \in G\sms\{\Id\}$ has a fixed point (call it $z_q$) in $\partial B_j$
for some $j \in \{\pm 1, \cdots, \pm q\}$ then $g$ is parabolic.  Then $z_g$ is a unique fixed point of $g$
and there exists a unique $j^*\in \{\pm 1, \cdots, \pm q\}\sms\{j\}$ such that
$$
z_g\in\ov B_j\cap \ov B_{j^*}.
$$ 
We refer to $z_g$ as a parabolic fixed point of $G$ (and of $g$). We denote by $p(g)\ge 1$ its rank. We further denote by $\Om(G)$ the set of all parabolic fixed points of $G$. Any such group $G$ is called a generalized Schottky group (GSG). If $G$ has at least one parabolic element, it is called a parabolic Schottky group (PSG). We associate to the group $G$ the conformal GDMS $\mathcal S_G$ in exactly the same way as for hyperbolic (i.e. without tangencies) Schottky groups in Section~\ref{Schottky-No Tangencies}.
Since any generalized Schottky group $G$ is geometrically finite, the number of conjugacy classes of parabolic elements of $G$ and the number of orbit classes of parabolic fixed points of $G$, i.e. $\Om(G)/G$, are both finite. In consequence, we have the following.

\begin{observation}\label{o1ex25}
The conformal GDMS $\mathcal S_G$ associated to $G$ is attracting if $G$ has no parabolic fixed points and it is (finite) parabolic (in the sense of Remark~\ref{r2_2017_02_17}) if $G$ has some parabolic fixed points.  
\end{observation}

\fr and

\begin{observation}\label{o1ex25J}
We have that: 

\sp
\begin{itemize}
\item Each parabolic fixed point of $G$ has a representative in 
$$
\bu_{-q\le j<k\le q}\ov B_j\cap \ov B_k,
$$
and 
\item
$$
\Om(\cS_G)=\Om(G)\cap \bu_{-q\le j<k\le q}\ov B_j\cap \ov B_k.
$$
\end{itemize}
\end{observation}

\fr We define
\beq\label{2_2017_04_05}
p_G:=p(\cS_G):=\sup\{p(g):g\in\Om(G)\}.
\eeq  
So, as an immediate consequence of Theorems~\ref{t2pc6_B}, 
\ref{t1dp13}, and \ref{t1dp13B}, in the 
same way as Theorem \ref{t1ex16}, i.e., along with 
Observation~ \ref{o2ex9}, Proposition~\ref{p1ex11},
Observation~\ref{o1ex11.1}, Observation~\ref{o2ex13} and Proposition \ref{p1ex14.1}, we get the following. 

\begin{thm}\label{t2ex25} Let $G = \langle g_j \rangle_{j=1}^q$
 be a parabolic Schottky group acting on $\mathbb R^d$, $d \geq 2$. 

\sp\begin{itemize}
\item Let $\delta_G$ be the Poincar\'e  exponent of $G$; which is known to be equal to $\HD(\Lambda(G))$.  

\sp\item Let $m_{\delta_G}$ be the Patterson-Sullivan
conformal measure for $G$ on $\Lambda(G)$.  

\sp\item Let $\mu_{\delta_G}$ be the $\mathcal S_G$-invariant measure on $\Lambda(G)$ equivalent to $m_{\delta_G}$.  

\sp\sp\item Fix $k\in\{\pm 1, \pm 2, \cdots, \pm q\}$ and $\xi \in \Lambda(G) \cap \overline B_k$. 
\end{itemize}

\fr Let $B\sbt\R^d$ be a Borel set with $m_{\d_G}(\bd B)=0$ (equivalently $\mu_{\d_G}(\bd B)=0$) and let $Y\sbt\ov B_k$ be a set having at least two distinct points. 
Then with some constant $C_k(Y)\in(0,+\infty)$, we have that
$$
\lim_{T\to+\infty}\frac{N_\xi(G; T, B)}{e^{\d_G T}}
=\frac{\psi_{\d_G}(\xi)}{\d_G\chi_{\d_G}}m_{\d_G}(B),
\  \  \  \  \
\lim_{T\to+\infty}\frac{N_\xi(G; T)}{e^{\d_G T}}
=\frac{\psi_{\d_G}(\xi)}{\d_G\chi_{\d_G}},
$$

$$
\lim_{T\to+\infty}\frac{N_p(G; T, B)}{e^{\d_G T}}
=\frac{1}{\d_G\chi_{\d_G}}\mu_{\d_G}(B),
\  \  \  \  \
\lim_{T\to+\infty}\frac{N_p(G; T)}{e^{\d_G T}}
=\frac{1}{\d_G\chi_{\d_G}},
$$

$$
\begin{aligned}
\lim_{T\to+\infty}\frac{\widehat N_p(G; T)}{e^{\d_G T}}
&=\frac{1}{\d_G\chi_{\d_G}}, \\  \\
\lim_{T\to+\infty}\frac{D_\xi(G; T, B,Y)}{e^{\d_G T}}
&=C_k(Y)m_{\d_G}(B), \\  \\
\lim_{T\to+\infty}\frac{E_k(G; T, B,Y)}{e^{\d_G T}}
&=C_k(Y)m_{\d_G}(B), \\ \\
\lim_{T\to+\infty}\frac{E_k(G; T,Y)}{e^{\d_G T}}
&=C_k(Y).
\end{aligned}
$$
In addition, $C_k(Y)>0$ is finite if and only if 

\sp\begin{enumerate}
\item 
$$
\ov Y\cap \Om_\infty(\cS_G)=(\ov Y\cap \Om_\infty(\cS_G)\cap \bd B_k)=\es
$$
or 
\item 
$$
\d_G>\max\big\{p(g):z_g\in\bd B_k \big\}.
$$
\end{enumerate}
\end{thm}

As in the case of hyperbolic Schottky groups, there are also Central Limit Theorems on the distribution of the preimages for parabolic Schottky groups.
 Theorem~\ref{t1ms1-again2} and Theorem~\ref{t1ms1-again3} for the parabolic conformal GDMS $\cS_G$, associated to the group $G$, take the same form. Therefore, we do not repeat them here.
However, we present the appropriate versions of Theorems~\ref{t1ms15.1} and \ref{t1ms16} as their formulations are closer to the actual group $G$. As in the case of hyperbolic groups, in order to get appropriate expressions in the language of the group $G$ itself, given $\xi\in\La(G)$, and an integer $n\ge 1$, we set
$$
G_\xi^n:=\{g_\om:\om\in E_\rho^n\}\sbt G_\xi.
$$
Furthermore, we define a probability measure $\mu_n$ on $G_\xi^n$ by setting that 

\begin{equation}\label{21_2017_04_05}
\mu_n(H) := \frac{\sum_{g \in H} e^{-\delta \lambda_\xi(g)}}{\sum_{\omega \in G_\xi^n} e^{-\delta \lambda_\xi(g)}}
\end{equation}
for every set $H \subset G_\xi^n$. As an immediate consequence of Theorem~\ref{t1ms15.1} we get the following.

\bthm\label{t31_2017_04_05}
If $G = \langle g_j\rangle_{j=1}^q$ is a parabolic finitely generated Schottky group acting on $\hat\R^d$, $d\ge 2$, and 
$$
\d_G>\frac{2p_G}{p_G+1},
$$
i.e the invariant measure $\mu_\delta$ is finite (so a probability after normalization), then for every $\xi\in\La(G)$ we have that
$$
\lim_{n \to +\infty} \int_{G_\xi^n} \frac{\lambda_\xi}{n} d\mu_n 
= \chi_{\mu_{\d}}.
$$
\ethm

\fr Again as in the hyperbolic (no tangencies) case, we define the functions $\Delta_n:G_\xi^n \to \mathbb R$, $n\in\N$, by the formulae
$$
\Delta_n(g) = 
\frac{\lambda_\xi(g)- \chi n}{\sqrt{n}}.
$$
As an immediate consequence of Theorem~\ref{t1ms16} we get the following.

\begin{thm}\label{t32_2017_04_05} 
If $G = \langle g_j\rangle_{j=1}^q$ is a parabolic finitely generated Schottky group acting on $\hat\R^d$, $d\ge 2$, and 
$$
\d_G>\frac{2p_G}{p_G+1},
$$
i.e.,  the invariant measure $\mu_\delta$ is finite (thus a probability measure after normalization), then for every $\xi\in\La(G)$
the sequence of random variables $(\Delta_n)_{n=1}^\infty$ converges in distribution to the normal (Gaussian) distribution 
$\mathcal N_0(\sigma)$ with mean value zero and the variance $\sigma^2 = \P_{\cS_G^*}''(\delta)>0$.  
Equivalently, the sequence 
$(\mu_n \circ \Delta_n^{-1})_{n=1}^\infty$
converges weakly to the normal distribution $\mathcal N_0(\sigma^2)$.  This means that for every Borel set $F \subset \mathbb R$ with 
$\hbox{\rm Leb}(\partial F) = 0$, we have 
$$
\lim_{n \to +\infty} \mu_n(\Delta_n^{-1}(F))
= \frac{1}{\sqrt{2\pi} \sigma}
\int_F e^{- t^2/2\sigma^2} dt.
$$
\end{thm}

\subsection{Apollonian Circle Packings}\label{Apollonian Circle Packings}
We now describe the application of Theorem~\ref{t2ex25}
to Apollonian circle packings, as explained in the introduction. 
This can be formulated in the framework we described in the introduction to this section. Some additional information related to the subject of this section and the one following it can be found in works such as \cite{Arnoux}, \cite{BM}, \cite{Fu2}, \cite{GLMWY}, \cite{KO}, \cite{lag}, \cite{MSW}, \cite{OS1}--\cite{OS3}, \cite{Pa}, and \cite{Sa2}. Of course we make no claims for this list to be even remotely complete. 

Let $C_1, C_2, C_3, C_4$ be four distinct circles in the Euclidean (complex) plane, each of which shares a common tangency point with each of the others. We assume that the bounded component of the complement of one of these circles contains the bounded components of the complements of the remining three circles.
Without loss of genrality $C_4$ is this circle enclosing the three other. We refer to such configuration of circles $C_1, C_2, C_3, C_4$ as bounded. This name will be justified in a moment. We can now choose  the new four circles $K_1, K_2, K_3, K_4$ that are dual to the original four tangent circles, i.e., those circles that pass through the three of the four possible tangent points between the initial circles $C_1, C_2, C_3, C_4$. We label them (uniquely) so that
$$
C_i\cap K_i=\es
$$
for all $i=1, 2, 3, 4$. Figure~3
depicts this construction. 

\begin{figure}
\begin{tikzpicture}[scale=0.65]
\draw (0,0) circle [radius=1.0];;
\draw (1,0.6) circle [radius=2.16];;
\draw (2,0) circle [radius=1.0];;
\draw (1.0,1.75) circle [radius=1.0];;
\draw [red] (1.0,0.61) circle [radius=0.62];;
\draw [red] (5, 3) circle [radius=4];;
\draw [red] (-3, 3) circle [radius=4];;
\draw [red] (1, -4) circle [radius=4];;
\node [red] at (-7,0.5) {$K_2$}; 
\node [red] at (0,-7) {$K_1$}; 
\node [red] at (7,0.5) {$K_3$}; 
\node [red]at (1,0.5) {$K_4$}; 
\node at (-1.5,1.5) {$C_4$}; 
\node  at (2.4,0.4) {$C_2$}; 
\node  at (0.2,-0.5) {$C_3$}; 
\node at (1.6,1.9) {$C_1$}; 
\end{tikzpicture}
\caption{The Tangent Circles $C_1, C_2, C_3, C_4$ and Dual Circles $K_1, K_2, K_3, K_4$}
\end{figure}
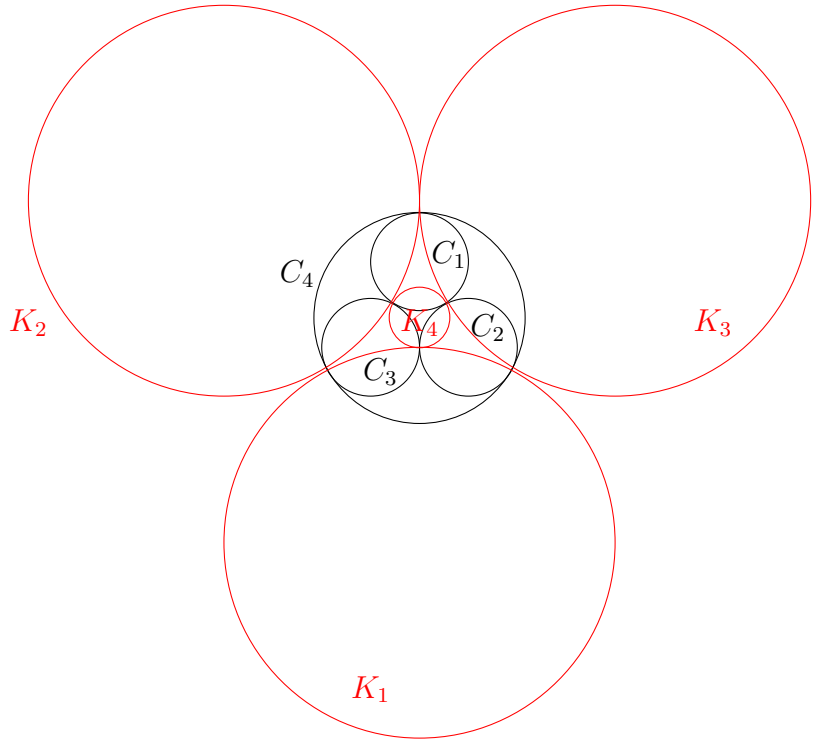

We associate to  the dual circles $K_1, K_2, K_3, K_4$ the respective inversions $g_1, g_2, g_3, g_4$ in these four dual circles. More precisely, if $K_i$, $i=1, 2, 3, 4$, is a circle with center\ $a_i \in \mathbb C$ and radius $r_i>0$ then we define
$$
g_i(z) = \frac{1}{r_i^2} \frac{z-a_i}{|z-a_i|^2}  + a_i, \  
$$
Denote by $B_1, B_2, B_3$ and $B_4$ the open balls (disks) enclosed, respectively, by the circles $K_1, K_2, K_3, K_4$. Let 
$$
G:= \langle 
g_1, g_2, g_3, g_4
\rangle 
$$
be the group generated by the four inversions $g_1, g_2, g_3, g_4$.
Let $\Gamma$ be the subgroup of $G$ consisting of its all orientation preserving elements. Observe that $\Ga$ is a free group generated by three elements, for example by
$$
\gamma_1:= g_{4} \circ g_1, \  \
\gamma_2:= g_{4} \circ g_2, \  \
\gamma_3:= g_{4} \circ g_3.
$$
Now noting that the the balls
$$
B_1, \  \ B_2, \  \  B_3; \ \
B_{-1} := g_4(B_1),\  \ B_{-2} := g_4(B_2), \  \ B_{-3} := g_4(B_3), 
$$
are mutually disjoint (see Figure~4), and that for every $i=1,2,3$:
$$
\gamma_i(\overline B_i)  = g_4 \circ g_i (\overline B_i) = g_4(B_i^c)
= (g_4(B_i))^c = B_{-i}^c
$$
we get the following.
\begin{figure}
\begin{tikzpicture}[scale=2.1]
\draw (0,0) circle [radius=1.0];;
\draw (1,0.6) circle [radius=2.16];;
\draw (2,0) circle [radius=1.0];;
\draw (1,1.75) circle [radius=1.0];;
\draw [red] (0.97,0.6) circle [radius=0.6];;
\draw [blue] (0.615,0.76) circle [radius=0.2];;
\node [blue] at (0.615,0.76) {$B_{-2}$}; 
\draw [red] (3,-0.8)  arc [radius =3.8, start angle = 52, end angle = 120];
\draw [red] (1.14,3.1)  arc [radius = 3.98, start angle = 173, end angle = 240];
\draw [red] (-1.3,-0.5)  arc [radius =3.86, start angle = 292, end angle = 360];
\node at (2.8,2.3) {$C_4$}; 
\node at (0.4,2) {$C_1$}; 
\node at (2.7,0.2) {$C_2$}; 
\node at (-0.5,-0.5) {$C_3$}; 
\node [red] at (2.8,-1.0) {$B_1$}; 
\node [red] at (1.5,3.1) {$B_3$}; 
\node [red] at (-1.5,-0.2) {$B_2$}; 
\draw [dashed] (1.0,-1.1) circle [radius=0.46];;
\node at (1.0,-1.1) {$g_1(C_1)$}; 
\draw [dashed] (2.455,1.4) circle [radius=0.47];;
\draw [dashed] (2.9,0.82) circle [radius=0.23];;
\node at (3.8,0.82) {$g_3 g_1(C_1)$}; 
\node [red] at (1.15,0.95) {$B_4$}; 
\end{tikzpicture}
\caption{Circles, Disks, and Generators of $G$}
\end{figure}
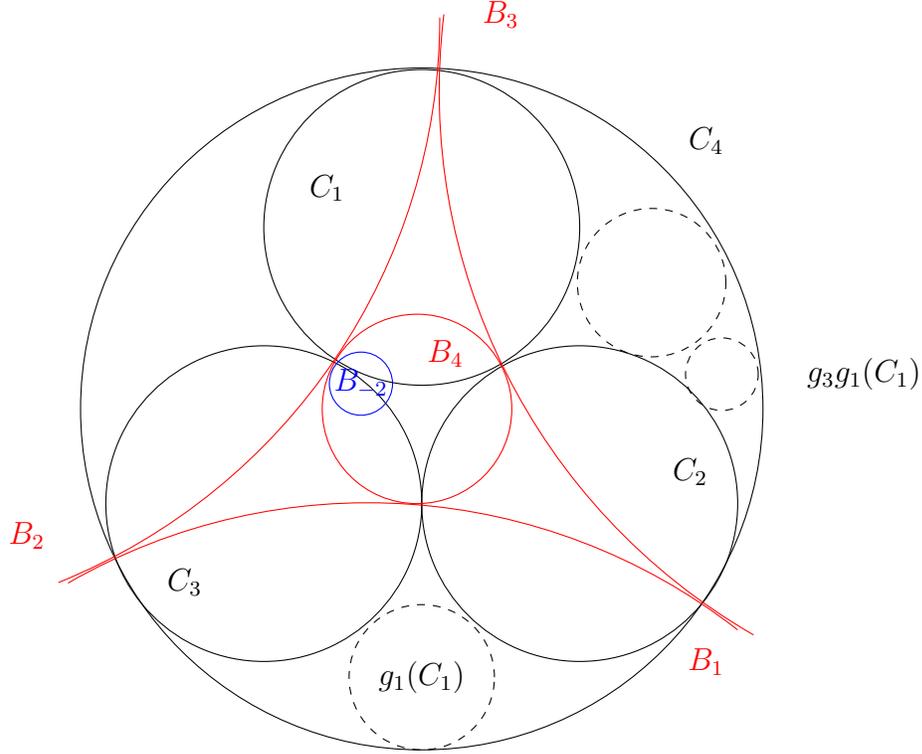 

\begin{observation}\label{o1ex27}
$\Gamma = \langle\gamma_1, \gamma_2, \gamma_3 \rangle$ is a parabolic Schottky group.
\end{observation}

\fr In addition, 

\begin{observation}
The parabolic Schottky group $\Gamma$ has six conjugacy classes 
of parabolic elements whose representatives are
$$
\gamma_1, \  \  \gamma_2, \  \  \gamma_3, \  \  \gamma_1 \gamma_2^{-1}, \  \  \gamma_1 \gamma_3^{-1}, \  \  \gamma_2 \gamma_3^{-1},
$$
with the corresponding parabolic fixed points being the only elements, respectively, of 
$$
\overline B_1 \cap \overline B_4, \ \
\overline B_2 \cap \overline B_4, \  \
\overline B_3 \cap \overline B_4, \  \
\overline B_{-1} \cap \overline B_{-2}, \  \
\overline B_{-1} \cap \overline B_{-3}, \  \
\overline B_{-2} \cap \overline B_{-3}.
$$
\end{observation}

\begin{observation}
The limit set $\La(\Gamma)$ coincides with the residual set of the Apollonian circle packing generated by the circles $C_1, C_2, C_3, C_4$. In addition (see \cite{Boyd}, \cite{MU_Apollo}, and Theorem~\ref{t1_2017_02_18}), we have the following.

\sp\begin{enumerate}\label{o3_2017_04_05}
\item $\delta_\Ga = \HD(\Lambda(\Gamma))>1$,

\sp\item $p(g)=1$ for every parabolic element of $\Ga$, and so 
$$
\d_\Ga> \sup\{p(g)\},
$$
where $g\in\Ga$ ranges of all parabolic elements of $G$,

\sp\item $\Om_\infty\(\cS_\Ga\)=\es$, and so $\mu_{\delta_{\Ga}}$, the probability $\mathcal S_\Ga$-invariant measure on $\Lambda(\Ga)$, is finite, thus probability after normalization.
\end{enumerate}
\end{observation}

\fr Hence, as an immediate consequence of Theorem \ref{t2ex25}, we get the following.

\begin{cor}\label{c2ex28} 
Let $C_1, C_2, C_3, C_4$ be a bounded\footnote{Boundedness of the configuration
$C_1, C_2, C_3, C_4$ guarantees us
that  the group $\Ga$ is  Schottky in the sense of our previous section, and, in particular, all the numbers $N_\xi(\Ga;T)$ and $N_p(\Ga;T)$ are finite.}
  configuration of four distinct circles in the plane, each of which shares a common tangency point with each of the others. Let $\Gamma$ be the corresponding parabolic Schottky group. 

\sp\begin{itemize} 
\item Let $\delta_\Ga$ be the Poincar\'e exponent of $\Ga$; it is known to be equal to $\HD(\Lambda(\Ga))$.  

\sp\item Let $m_{\delta_\Ga}$ be the Patterson-Sullivan
conformal measure for $\Ga$ on $\Lambda(\Ga)$.  

\sp\item Let $\mu_{\delta_{\Ga}}$ be the probability $\mathcal S_\Ga$-invariant measure on $\Lambda(\Ga)$ equivalent to $m_{\delta_\Ga}$. 

\sp\item Fix $k \in \{\pm 1, \pm 2, \pm 3\}$ and $\xi \in \Lambda(\Gamma) \cap \overline B_k$.
\end{itemize}

\sp\fr Then for every set $Y \subset \overline {B_k}$ having at least two distinct points there exists a constant $C_k(Y)\in(0,+\infty)$ such that for every Borel set $B \subset \R^d$ with $m_{\delta_\Ga}(\bd B)=0$ (equivalently $\mu_{\delta_\Ga}(\bd B)=0$), we have that
$$
\lim_{T\to+\infty}\frac{N_\xi(\Ga; T, B)}{e^{\d_\Ga T}}
=\frac{\psi_{\d_\Ga}(\xi)}{\d_\Ga\chi_{\d_\Ga}}m_{\d_\Ga}(B),
\  \  \  \  \
\lim_{T\to+\infty}\frac{N_\xi(\Ga; T)}{e^{\d_\Ga T}}
=\frac{\psi_{\d_\Ga}(\xi)}{\d_\Ga\chi_{\d_\Ga}},
$$

$$
\lim_{T\to+\infty}\frac{N_p(\Ga; T, B)}{e^{\d_\Ga T}}
= \frac{1}{\d_\Ga\chi_{\d_\Ga}}\mu_{\d_\Ga}(B),
\  \  \  \  \
\lim_{T\to+\infty}\frac{N_p(\Ga; T)}{e^{\d_\Ga T}}
=\frac{1}{\d_\Ga\chi_{\d_\Ga}},
$$

$$
\begin{aligned}
\lim_{T\to+\infty}\frac{\widehat N_p(\Ga; T)}{e^{\d_\Ga T}}
&=\frac{1}{\d_\Ga\chi_{\d_\Ga}}, \\  \\
\lim_{T\to+\infty}\frac{D_\xi(\Ga; T, B,Y)}{e^{\d_\Ga T}}
&=C(Y)m_{\d_G}(B), \\  \\
\lim_{T\to+\infty}\frac{E_k(\Ga; T, B,Y)}{e^{\d_\Ga T}}
&=C_k(Y)m_{\d_\Ga}(B), \\ \\
\lim_{T\to+\infty}\frac{E_k(\Ga; T,Y)}{e^{\d_\Ga T}}
&=C_k(Y).
\end{aligned}
$$
\end{cor}

Making use of Observation~\ref{o3_2017_04_05}, as an immediate consequence respectively of Theorem~\ref{t31_2017_04_05} and Theorem~\ref{t32_2017_04_05}, we get the following two theorems.

\bthm\label{t6_2017_04_05}
Let $C_1, C_2, C_3, C_4$ be a bounded configuration of four  distinct circles in the plane, each of which shares a common tangency point with each of the others. If $\Gamma$ is the corresponding parabolic Schottky group, then for every $\xi\in\La(\Ga)$ we have that
$$
\lim_{n \to +\infty} \int_{\Ga_\xi^n} \frac{\lambda_\xi}{n} d\mu_n 
= \chi_{\mu_{\d}}.
$$
\ethm

The next theorem is a Central Limit Theorem for diameters of 
circles in the Apollonian Circle Packing.

\begin{thm}\label{7_2017_04_05} 
Let $C_1, C_2, C_3, C_4$ be a bounded configuration of four distinct circles in the plane, each of which shares a common tangency point with each of the others. If $\Gamma$ is the corresponding parabolic Schottky group, then for every $\xi\in\La(\Ga)$
the sequence of random variables $(\Delta_n)_{n=1}^\infty$ converges in distribution to the normal (Gaussian) distribution 
$\mathcal N_0(\sigma)$ with mean value zero and the variance $\sigma^2 = \P_{\cS_\Ga^*}''(\delta)>0$.  
Equivalently, the sequence 
$(\mu_n \circ \Delta_n^{-1})_{n=1}^\infty$
converges weakly to the normal distribution $\mathcal N_0(\sigma^2)$.  This means that for every Borel set $F \subset \mathbb R$ with 
$\hbox{\rm Leb}(\partial F) = 0$, we have 
$$
\lim_{n \to +\infty} \mu_n(\Delta_n^{-1}(F))
= \frac{1}{\sqrt{2\pi} \sigma}
\int_F e^{- t^2/2\sigma^2} dt.
$$
\end{thm}

In Figure 2 we illustate the Central Limit Theorem for the diameters in the standard Apollonian Circle Packing
in Theorem \ref{t32_2017_04_05}. 
\begin{figure}[h]
\includegraphics[height=6cm]{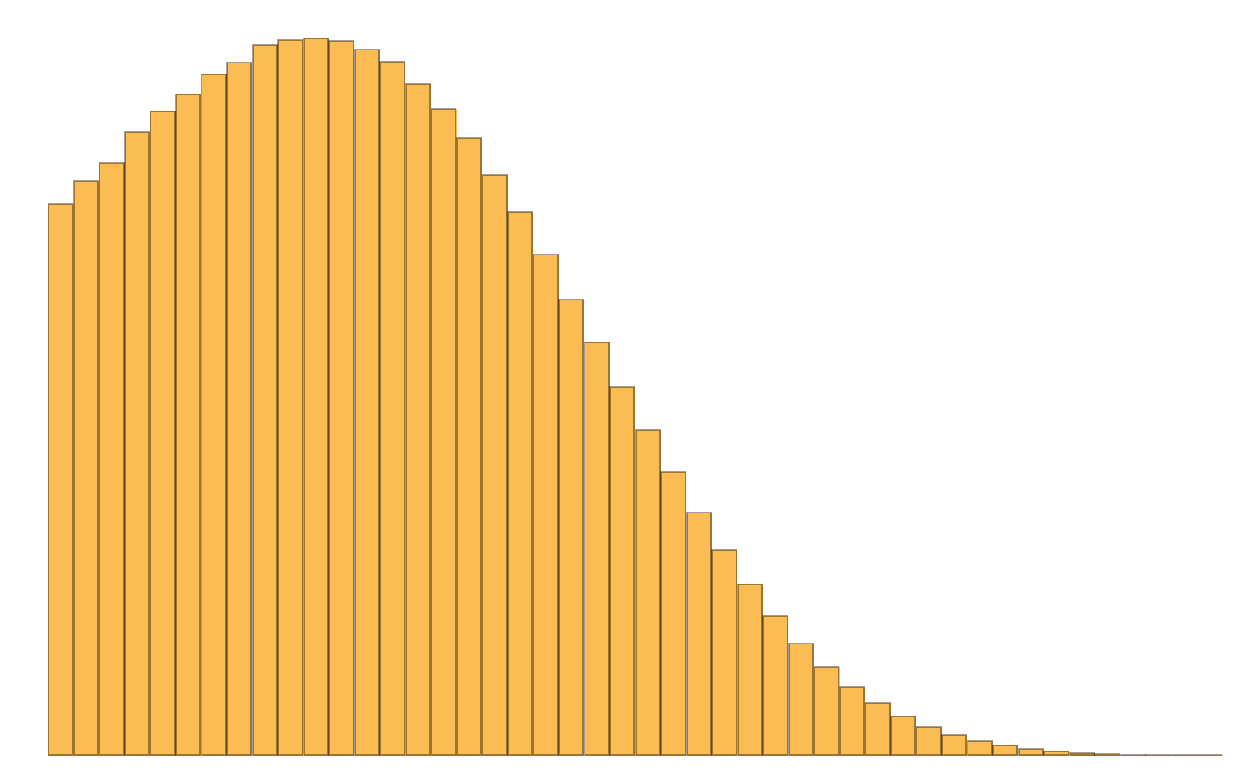}
\caption{We plot a portion of the weighted histogram of the 
6,377,292
values $-\log r$ where $r$ is a circle of generation $n=14$ for
standard  Apollonian circle packing.  There are 46 bins with a weighting corresponding to $r^\delta$.}
\end{figure}

Now, we consider the actual counting of the circles in the Apollonian circle packing generated by the bounded configuration of the circles 
$C_1, C_2, C_3$ and $C_4$. The following immediate observation is crucial to this goal.

\begin{observation}\label{o3ex28}
The elements of $\mathcal A$, the Apollonian circle packing generated by the bounded configuration of the circles $C_1, C_2, C_3, C_4$, is bounded\footnote{This justifies the name ``bounded'' in regards to the configuration $C_1, C_2, C_3, C_4$.} and coincide with the following disjoint union
$$
\begin{aligned}
\{
C_1, C_2, C_3, C_4\}
\cup
\bigcup_{j=1}^3 &(\Gamma_j \cup \{\Id \})(g_j(C_j))
\cup
\bigcup_{j=1}^3 \bigcup_{i=1, i \neq j}^4 (\Gamma_i \cup \{\Id \})(g_i \circ g_j(C_j)) \cup \\ 
&\cup \{g_4(C_4) \}
\cup \cup_{j=1}^3 (\Gamma_j \cup \{\Id \})(g_j\circ g_4)(C_4),
\end{aligned}
$$
and for $j=1,2,3$ and $i \in \{1,2,3\}\sms \{ j \}$ we have that 
$$
g_j(C_j) \subset \overline B_j, \  \
g_i \circ g_j(C_j) \subset \overline B_i, \  \
g_4 \circ g_j(C_j) \subset \overline B_{-j}, \  \
g_j \circ g_4(C_4) \subset \overline B_{j}.
$$
\end{observation}

\sp\fr For every $T > 0$ and every set $B \subset \mathbb C$, we denote 
$$
\mathcal D(T; B) := \big\{ C \in \mathcal A \hbox{ : } - \log \hbox{\diam}(C) \leq T \hbox{ and } C \cap B \neq \emptyset\big\},
$$
$$
\mathcal D(T) := \mathcal D(T; \C)
$$
$$
N_\cA(T; B):=\#\mathcal D(T; B)
\  \  \hbox{ and } \  \
N_\cA(T):=\#\mathcal D(T).
$$
As an immediate consequence of Corollary~\ref{c2ex28} and Observation~\ref{o3ex28} we get the following 
result proved in \cite{KO} (see also \cite{OS1}--\cite{OS3}) by entirely different methods.  

\begin{thm}\label{t1ex29}
Let $C_1, C_2, C_3, C_4$ be a bounded configuration of four distinct circles in the plane, each of which shares a common tangency point with each of the others. Let $\mathcal A$ be the corresponding circle packing.

Let $\delta = 1.30561\ldots$ be the Hausdorff dimension of the residual set of $\mathcal A$ and let $m_\delta$ be the Patterson-Sullivan measure of the corresponding parabolic Schottky group $\Gamma$.  

Then the limit 
$$
\lim_{T\to +\infty} \frac{N_{\mathcal A}(T)}{e^{\delta T}}
$$
exists, is positive, and finite. Moreover, there exists a constant $C \in (0, +\infty)$ such that 
$$
\lim_{T \to +\infty} 
 \frac{N_{\mathcal A}(T; B)}{e^{\delta T}} = C m_\delta(B)
$$
for every Borel set $B \subset \mathbb C$ with $m_\delta(\partial B) = 0$.
\end{thm}
 
\subsection{Apollonian Triangle} 

\sp Now we consider the Apollonian triangle. Let $C_1, C_2, C_3$ be three mutually tangent circles in the plane having mutually disjoint interiors. Let $C_4$ be the circle tangent to all the circles $C_1, C_2, C_3$ and having all of them in its interior, i.e. the configuration $C_1, C_2, C_3, C_4$ is bounded.

We look at the curvilinear triangle $\mathcal T$ formed by the three edges joining the three tangency points of $C_1, C_2, C_3$
and lying on these circles. The bounded collection 
$$
\mathcal G := \{ C \in \mathcal A \hbox{ : } C \subset \mathcal T\}
$$
 is called the Apollonian gasket generated by the circles $C_1, C_2, C_3$.  Since $\partial \mathcal T \cap \Lambda(\Gamma)=\partial \mathcal T$ has Hausdorff dimension $1$, since $\d>1$ and since $m_\d$ is a constant multiple of $\d$--dimensional Hausdorff measure restricted to $\La(\Ga)$, we have that $m_\delta(\partial T)=0$. Another, a more general argument for this, would be to invoke Corollary~1.4 from \cite{FS}. Therefore, as an immediate consequence of Theorem~\ref{t1ex29} we get the following result, also proved by Kontorovich and Oh in \cite{KO} (see also \cite{OS1}--\cite{OS3}) with entirely different methods.
 
\begin{cor}\label{c1ex30}
Let $C_1, C_2, C_3$ be three mutually tangent circles in the plane having mutually disjoint interiors. Let $C_4$ be the circle tangent to all the circles $C_1, C_2, C_3$ and having all of them in its interior, i.e. the configuration $C_1, C_2, C_3, C_4$ is bounded. Let $\mathcal A$ be the corresponding (bounded) circle packing.

Let $\delta = 1.30561 \ldots$ be the Hausdorff dimension of the residual set of $\mathcal A$ and let $m_\delta$ be the Patterson-Sullivan measure of the corresponding parabolic Schottky group $\Gamma$.  

If $\mathcal T$ is the curvilinear triangle formed by $C_1$, $C_2$ and $C_3$, then the limit
$$
\lim_{T\to +\infty} \frac{N_{\mathcal A}(T; \mathcal T)}{e^{\delta T}}
$$
exists, is positive, and finite; we just count the elements of $\mathcal G$. Moreover, there exists a constant $C \in (0, +\infty)$, in fact the one of Theorem~\ref{t1ex29}, such that 
$$
\lim_{T \to +\infty}  \frac{N_{\mathcal A}(T; B)}{e^{\delta T}} = C m_\delta(B)
$$
for every Borel set $B \subset \mathcal T$ with $m_\delta(\partial B) = 0$.
\end{cor}

Now we will provide a somewhat different proof of Corollary~
\ref{c1ex30}, by appealing directly to the theory
 of parabolic conformal IFSs and avoiding the intermediate 
step of parabolic Schottky groups. Indeed, let $C_0$ be the circle 
inscribed in $\mathcal T$ and tangent to the circles 
$C_1$, $C_2$ and $C_3$.  Let $x_1, x_2$ and $x_3$ be the vertices
of the curvilinear triangle $\mathcal T$, i.e., for 
$i=1,2,3$, $x_i$ is the only element of the intersection $K_i \cap K_4$.
Let 
$$
\phi_i: \widehat {\mathbb C} \to  \widehat {\mathbb C}
$$
be the M\"obius transformation fixing the point $x_i$ and mapping
the other vertices $x_j$ and $x_k$, respectively, onto the only points of the intersections $C_0 \cap C_j$ and $C_0 \cap C_k$.  
Then 
$$
\mathcal S = \{\phi_1, \phi_2, \phi_3\}
$$ 
is a parabolic IFS defined on $\overline B_4$, $x_i$ is a parabolic fixed point 
of $\phi_i$, $i=1,2,3$, and 
$$
\mathcal G = \{\phi_\om (C_0) \hbox{ : } \om \in \{1,2,3\}^* \},
$$
see Figure~5. We therefore obtain Corollary~\ref{c1ex30} immediately from  Theorem~\ref{c1da12.1J}.

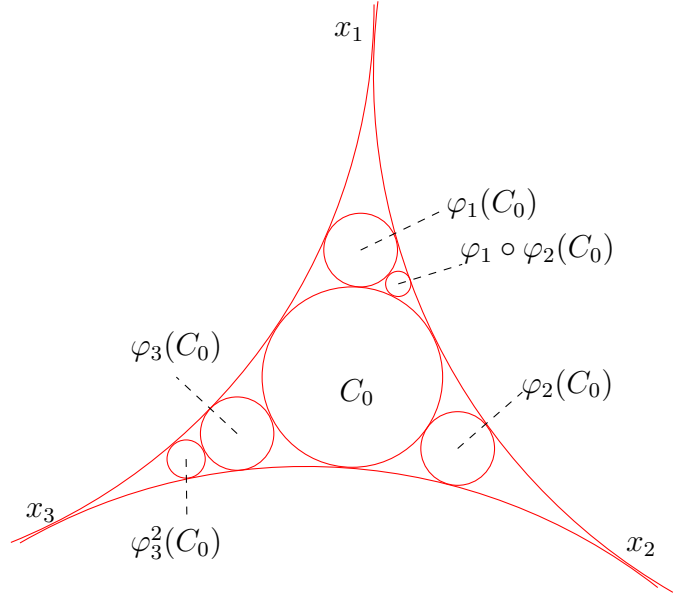
\begin{figure}
\begin{tikzpicture}[scale=2.0]
\draw [red] (0.97,0.6) circle [radius=0.6];;
\draw [red] (3,-0.8)  arc [radius =3.8, start angle = 52, end angle = 120];
\draw [red] (1.14,3.1)  arc [radius = 3.98, start angle = 173, end angle = 240];
\draw [red] (-1.3,-0.5)  arc [radius =3.86, start angle = 292, end angle = 360]; 
\draw [red] (1.67,0.125) circle [radius=0.245];;
\draw [dashed] (1.67,0.125)  -- (2.1, 0.5);
\draw [red] (0.204, 0.224) circle [radius=0.244];;
\draw [dashed](0.204, 0.224)  -- (-0.2, 0.6);
\draw [red] (1.025, 1.445) circle [radius=0.245];;
\draw [dashed] (1.025, 1.445)  -- (1.55, 1.7);
\draw [red] (-0.135, 0.055) circle [radius=0.127];;
\draw [dashed] (-0.135, 0.055)  -- (-0.13, -0.35);
\draw [red] (1.275, 1.22) circle [radius=0.084];;
\draw [dashed] (1.275, 1.22)  -- (1.7, 1.35);
\node at (-1.1, -0.3) {$x_3$};;
\node at (2.89, -0.53) {$x_2$};;
\node at (0.95, 2.9) {$x_1$};;
\node at (-0.2, 0.8) {$\phi_3(C_0)$};;
\node at (-0.2, -0.5) {$\phi_3^2(C_0)$};;
\node at (2.4, 0.55) {$\phi_2(C_0)$};;
\node at (2.2, 1.45) {$\phi_1 \circ \phi_2(C_0)$};;
\node at (1.9, 1.75) {$\phi_1(C_0)$};;
\node at (1.0, 0.50) {$C_0$};;
\end{tikzpicture}
\caption{Apollonian Triangle}
\end{figure}

\begin{rem}
In the context of limit sets, such as circle packings,  there is scope for finding error terms in the above asymptotic formulae, see ex. \cite{LO} and \cite{Pan}. It could also done using the techniques worked out in our present manuscript.  However, in the general setting of conformal graph directed Markov systems quite delicate technical hypotheses might well be required. 
\end{rem}

\begin{rem}
 For these analytic maps it would be equally possible to work with  Banach spaces of analytic functions, rather than H\"older continuous functions.
 This would have the advantage that the transfer operator operator is compact (even trace class or nuclear) and might help to simplify some of the arguments as well as being useful in explicit numerical computations.    On the other hand, working with H\"older functions allows the results to be applied to a  far  greater range of examples.
\end{rem}

\begin{rem}
In higher dimensions, we can consider the packing of the sphere $S^d$ by mutually tangent $d$-spheres.  The same analysis gives a corresponding asymptotic for the diameters of spheres. In an overlapping setting and with entirely different metods this question has been addressed in Oh's paper \cite{Oh}. 
\end{rem}

\sp
\section{{\large{\bf Fuchsian Groups}}}\label{Fuchsian Groups}

We recall that a Fuchsian group $\Gamma$ is a discrete 
group of orientation preserving Poincar\'e isometries acting 
on the unit disk 
$$
\mathbb D= \{z \in \mathbb C \hbox{ : } |z| < 1\}
$$ 
in the complex plane.
A Poincar\'e isometry  means that the Poincar\'e metric 
$$
\frac{|dz|}{1- |z|^2}
$$ 
is preserved, equivalently the map is 
a holomorphic homeomorphism of the disk $\mathbb D$ onto
itself.  The limit set $\Lambda (\Gamma)$ of $\Gamma$ is a 
compact perfect subset of $S^1 =  \bd\mathbb D =  
\{z \in \mathbb C \hbox{ : } |z| = 1\}$.
Assume that $\Gamma$ is finitely generated and denote a 
minimal (in the sense of inclusion) set of its generators by 
$\{g_j\}_{j= \pm 1}^{\pm q}$ where $g_j = g_j^{-1}$.  
Assume that $q \geq 2$, so that $\Gamma$ is non-elementary.  
Following \cite{series1} (see also \cite{series2}) we call $\Gamma$ non-exceptional if at least one of the following conditions holds (corresponding to conditions (10.1)-(10.3) from \cite{lalley}):

\sp\begin{enumerate}
\item
$\mathbb D/\Gamma$ is not compact; 

\sp\item
The generating set has at least $5$ elements (i.e., $q \geq 5$)
and every non-trivial relation has length $5$; and

\sp\item
At least $3$ of the generating relations have length at least $7$. 
\end{enumerate}

\sp\fr In particular finitely generated every parabolic Fuchsian group is non-exceptional as the condition (1) above is satisfied. In the language of conformal GDMSs, C. Series proved 
in \cite{series1} (see also \cite{series2}) the following:

\begin{thm}\label{t1_2017_04_13}
If $\Gamma$ is a non-exceptional Fuchsian group then there 
exists a finite irreducible pre-parabolic GDMS $\mathcal 
S_\Gamma$ with an incidence matrix $A$, a finite set of 
vertices $V$ and a finite alphabet $E = \{\pm 1, \pm 2, \cdots, 
\pm q\}$ such that 
\begin{enumerate}
\item  For every $j \in E$ the corresponding element of 
$\mathcal S_\Gamma$ is $g_j:X_{t(j)} \to X_{i(j)}$

\sp\item All sets $X_v$, $v \in V$ are closed subarcs of $S^1$

\sp\item The map $E_A^* \ni \omega \mapsto g_\omega \in \Gamma$ is a bijection 

\sp\item $\Lambda(\Gamma) = \mathcal J_{\mathcal S_\Gamma}$

\sp\item The map $\pi_{\mathcal S_\Gamma}: E_A^\infty \to \mathcal J_{\mathcal S_\Gamma} = \Lambda(\Gamma)$ is a continuous surjection and it is $1$-to-$1$ except at countably many points, where it is 
$2$-to-$1$.
\end{enumerate}
\end{thm}

Similarly to (but not exactly) as in Section~\ref{Schottky-No 
Tangencies}, given $e\in E$ we define 
$$
\Gamma_e:= \{\gamma \omega \in E_A^* \  \hbox{ and }  \  \om_1 = e\}. 
$$
Then having $\rho \in A_A^{\mathbb N}$ we set 
$$
\Gamma_\rho := \Gamma_{\rho_1}
$$
Again, similarly as in Section~\ref{Schottky-No Tangencies}, we denote 
$$
\lambda_{\rho}(\gamma)
= - \log |\gamma'(\pi_\Gamma(\rho))|
= - \log |\gamma_\omega'(\pi_\Gamma(\omega))|
= \lambda_{\rho}(\omega)
$$
for every $\omega \in E^*_\rho$ ($\gamma = \gamma_\omega
\in \Gamma_{\rho_1} = \Gamma_\rho$) and 
$$
\Delta(Y):= - \log (\hbox{\rm diam} (\gamma_{\omega}(Y)))
$$
if $Y \subset X_{t(\rho_1)}$.  Also 
$$
\lambda_p(\omega) = - \log |\gamma_\omega'(x_\omega)|
$$
if $\omega \in E_p^*$

Let $B$ denote a Borel subset of the set $S^1$. Set
$$
\pi_\rho (\Ga; T, B):= \{\g \in \Ga_\rho \hbox{ : } \lambda_\rho (\g) \leq T \hbox{ and } \g(\pi_\Ga(\rho)) \in B \}
$$

$$
\pi_\rho (\Ga; T):=\pi_\xi (\Ga; T, S^1) 
=\{ \g \in \Ga_\rho \hbox{ : } \lambda_\rho (\g) \leq T \}
$$

$$
\pi_p(\Ga; T, B):= \{ \omega \in E_p^*
\hbox{ : } \lambda_p (\omega) = l(\gamma_\omega) \leq T \hbox{ and } x_\omega \in B \},
$$

$$
\pi_p(\Ga; T):=\pi_p(\Ga; T, S^1)
=\{ \omega \in E_p^* \hbox{ : } \lambda_p (\omega) = l(\gamma_\omega) \leq T  \},
$$

$$
\widehat \pi_p (\Ga,T) := \{\g \in \widehat\Ga \hbox{ : } l(\gamma_\g)  \leq T \} 
$$
With $e:=\rho_1$ we further denote
$$
\mathcal D_\rho(\Ga; T, B, Y): = \{\g \in \Ga_e \hbox{ : } 
\Delta_\g(Y) \leq T \  \hbox{ and } \  \g(\pi_\Ga(\rho)) \in B\},
$$

$$
\mathcal E_e(\Ga; T, B, Y) = \{\g \in \Ga_e \hbox{ : } 
\Delta_\g(Y)\leq T  \  \hbox{ and } \  \g(Y)\cap B\ne\es\},
$$
and 
$$
\mathcal E_e(\Ga; T, Y):=\cE_e(\Ga; T,S^1, Y)
=\{\g \in \Ga_e \hbox{ : } \Delta_\g(Y)\leq T\}.
$$

\sp\fr We denote by 
$N_\xi(\Ga; T, B)$, $N_\xi(\Ga; T)$,
$N_p(\Ga; T, B)$, $N_p(\Ga; T)$,
$\widehat N_p(\Ga; T)$,
$D_\xi(\Ga; T, B, Y)$,
$E_e(\Ga; T, B, Y)$ and $E_e(\Ga; T, Y)$
the corresponding cardinalities.

\sp As immediate consequences of  Theorem~\ref{dyn}, 
Theorem~\ref{t1da7}, Theorem~\ref{t1ma1}, 
Theorems~\ref{t2pc6_B}, \ref{t1dp13}, and \ref{t1dp13B},
along with Theorem~\ref{t1_2017_04_13} and Fuchsian 
counterparts of Proposition \ref{p1ex11},
Observation \ref{o2ex13} and Proposition \ref{p1ex14.1}, 
following from \cite{series1} and \cite{series2}, we 
get the following. 

\begin{thm}\label{t1ex16_04_14_2017}
Let $\Ga = \langle \g_j\rangle_{j=1}^q$ be a  finitely 
generated non-exceptional Fuchsian group. 

\sp\begin{itemize}
\item Let $\delta_\Ga$ be the Poincar\'e  exponent of $\Ga$; 
it is known to be equal to $\HD(\Lambda(\Ga))$.  

\sp\item Let $m_{\delta_\Ga}$ be the Patterson-Sullivan
conformal measure for $G$ on $\Lambda(\Ga)$.  

\sp\item Let $\mu_{\delta_\Ga}$ be the $\mathcal 
S_\Ga$-invariant measure on $\Lambda(\Ga)$ equivalent to 
$m_{\delta_\Ga}$.  

\sp\sp\item Fix $e\in E = \{\pm 1, \pm 2, \cdots, 
\pm q\}$ and $\rho \in E_A^\infty$ with $\rho_1=e$.
\end{itemize}

\fr Let $B\sbt S^1$ be a Borel set with $m_{\d_\Ga}(\bd B)=0$ 
(equivalently $\mu_{\d_\Ga}(\bd B)=0$) and let 
$Y\sbt X_{t(e)}$ be a set having at least two distinct 
points. Then with some constant $C_e(Y)\in(0,+\infty]$, we have that
$$
\lim_{T\to+\infty}\frac{N_\xi(\Ga; T, B)}{e^{\d_\Ga T}}
=\frac{\psi_{\d_\Ga}(\xi)}{\d_\Ga\chi_{\d_\Ga}}m_{\d_\Ga}(B),
\  \  \  \  \
\lim_{T\to+\infty}\frac{N_\xi(\Ga; T)}{e^{\d_\Ga T}}
=\frac{\psi_{\d_\Ga}(\xi)}{\d_\Ga\chi_{\d_\Ga}},
$$

$$
\lim_{T\to+\infty}\frac{N_p(\Ga; T, B)}{e^{\d_\Ga T}}
=\frac{1}{\d_\Ga\chi_{\d_\Ga}}\mu_{\d_\Ga}(B),
\  \  \  \  \
\lim_{T\to+\infty}\frac{N_p(\Ga; T)}{e^{\d_\Ga T}}
=\frac{1}{\d_\Ga\chi_{\d_\Ga}},
$$

$$
\begin{aligned}
\lim_{T\to+\infty}\frac{\widehat N_p(\Ga; T)}{e^{\d_\Ga T}}
&=\frac{1}{\d_\Ga\chi_{\d_\Ga}}, \\  \\
\lim_{T\to+\infty}\frac{D_\xi(\Ga; T, B,Y)}{e^{\d_\Ga T}}
&=C_e(Y)m_{\d_\Ga}(B), \\  \\
\lim_{T\to+\infty}\frac{E_k(\Ga; T, B,Y)}{e^{\d_\Ga T}}
&=C_e(Y)m_{\d_\Ga}(B), \\ \\
\lim_{T\to+\infty}\frac{E_k(\Ga; T,Y)}{e^{\d_\Ga T}}
&=C_e(Y).
\end{aligned}
$$
In addition, $C_e(Y)>0$ is finite if and only if 
$$
\ov Y\cap \Om(\cS_\Ga)=\es,
$$
in particular if $\Ga$ has no parabolic points, i.e. if it is 
convex co-compact.
\end{thm}

\sp Theorem~\ref{t1ms1} -- 
Theorem~\ref{t1ms1-2again} hold for the conformal GDMS $\cS_\Ga$, 
associated to the group $\Ga$, without changes. 
Therefore, we do not repeat them here. However, as in 
Section~\ref{Schottky-No Tangencies}, we present  the 
appropriate versions of Theorems~\ref{t1ms5.2} and 
\ref{t1m58} as their formulations are closer to the group 
$\Ga$. In order to get appropriate expressions in the 
language of the group $\Ga$ itself, given $\rho\in E_A^\infty$, 
and an integer $n\ge 1$, we set
$$
\Ga_\rho^n:=\{\g_\om:\om\in E_\rho^n\}\sbt \Ga_\rho.
$$
Furthermore, we define a probability measure $\mu_n$ on 
$\Ga_\rho^n$ by setting that 

\begin{equation}\label{21_2017_04_14}
\mu_n(H) 
:= \frac{\sum_{\g \in H} 
e^{-\delta \lambda_\rho(\g)}}{\sum_{\g\in\Ga_\rho^n} 
e^{-\delta \lambda_\rho(\g)}}
\end{equation}
for every set $H \subset \Ga_\rho^n$. As an immediate 
consequence of Theorem~\ref{t1ms5.2} we get the following.

\bthm\label{t31_2017_04_14} 
If $\Ga = \langle \g_j\rangle_{j=1}^q$ is a finitely 
generated non-exceptional convex co-compact (i. e. without 
parabolic fixed points) Fuchsian group, then for every 
$\rho\in E_A^\infty$ we have that
$$
\lim_{n \to +\infty} \int_{\Ga_\rho^n} \frac{\lambda_\rho}{n} d\mu_n 
= \chi_{\mu_\d}.
$$
\ethm

\fr Now define the functions $\Delta_n:\Ga_\rho^n\to \mathbb
R$ by the formulae
$$
\Delta_n(\g) = 
\frac{ \lambda_\xi(\g)- \chi n}{\sqrt{n}}.
$$
As an immediate consequence of Theorem~\ref{t1m58} we get the following.

\begin{thm}\label{t32_2017_04_14} 
If $\Ga = \langle \g_j\rangle_{j=1}^q$ is a finitely 
generated non-exceptional convex co-compact (i. e. without 
parabolic fixed points) Fuchsian group, then for every 
$\rho\in E_A^\infty$
the sequence of random variables $(\Delta_n)_{n=1}^\infty$ converges in distribution to the normal (Gaussian) distribution 
$\mathcal N_0(\sigma)$ with mean value zero and the variance 
$\sigma^2 = \P_{\cS_\Ga}''(\delta)>0$.  
Equivalently, the sequence 
$(\mu_n \circ \Delta_n^{-1})_{n=1}^\infty$
converges weakly to the normal distribution $\mathcal N_0(\sigma^2)$.  This means that for every Borel set $F \subset \mathbb R$ with 
$\hbox{\rm Leb}(\partial F) = 0$, we have 
\begin{equation}\label{1ms8}
\lim_{n \to +\infty} \mu_n(\Delta_n^{-1}(F))
= \frac{1}{\sqrt{2\pi} \sigma}
\int_F e^{- t^2/2\sigma^2} dt.
\end{equation} 
\end{thm}

\sp
\subsection{Hecke Groups}
A special class of Fuchsian parabolic (so non-exceptional) groups are Hecke groups. These are easiest to express in the Lobachevsky model of hyperbolic geometry and plane rather than in the Poincar\'e one. The 2-dimensional hyperbolic (Lobachevsky) plane is the set 
$$
\bH:=\{z\in\C: \im z>0\}
$$
endowed with the Riemannian metric
$$
\frac{|dz|}{\im z}.
$$
Given $\e>0$ the corresponding Hecke group is defined as follows
$$
\Gamma_\epsilon:= \big\langle z\mapsto -1/z,\, z\mapsto z+ 1+\epsilon \big\rangle.
$$
This group has an elliptic element order $2$ which is the map $z\longmapsto -1/z$ and one (conjugacy class) of parabolic elements which is the map $z\longmapsto z+ 1+\epsilon$. Its (parabolic) fixed point is $\infty$. In particular all the limit sets $\La(\Ga_\e)$ are unbounded, and therefore the Hecke groups $\Gamma_\e$ do not really fit into the setting of our current manuscript. However, any M\"obius transformation
$$
H:\D\to\bH
$$
is an isometry with respect to corresponding Poincar\'e metrics and the map
$$
\Ga_\e\ni \g\longmapsto H^{-1}\circ\g\circ H
$$
establishes an algebraic isomorphism between $\Ga_\e$ and the group 
$$
\hat\Ga_\e:=\{H^{-1}\circ\g\circ H:\g\in\Ga_e\}.
$$
Of course, the conjugacy $H$ between $\hat\Ga_\e$ and $\Ga_\e$ congregates elements of $\hat\Ga_\e$ and $\Gamma_\e$ viewed as isometric actions. The groups $\hat\Ga_\e$ are Fuchsian parabolic (so non-exceptional) groups acting on $\D$ and perfectly fit into the setting of Section~\ref{Fuchsian Groups}. In particular, Theorem~\ref{t1ex16_04_14_2017} holds for them.

\end{document}